\documentclass[edeposit,fullpage,11pt]{uiucthesis2009}

\usepackage{amsthm,epsfig, epstopdf, tikz, tikz-cd, tcolorbox, subcaption, amssymb, setspace, graphicx, hyperref, enumitem, mathrsfs, bbm,cleveref, mathtools,color, stmaryrd, comment, mathpazo, euler, subcaption}

\usepackage[american]{babel}

\usetikzlibrary{decorations.pathmorphing}
\setlist[enumerate]{topsep = -5pt, itemsep = 0pt}
\theoremstyle{remark} 
\usepackage[margin=1in]{geometry}
\theoremstyle{remark}
\newtheorem{theorem}{Theorem}[section]

\numberwithin{equation}{theorem}

\newtheorem{definition}[theorem]{Definition}
\newtheorem{prop}[theorem]{Proposition}
\newtheorem{lemma}[theorem]{Lemma}
\newtheorem{corollary}[theorem]{Corollary}
\newtheorem{example}[theorem]{Example}
\newtheorem{remark}[theorem]{Remark}

\newtheorem{fact}[theorem]{Fact}
\newtheorem{notation}[theorem]{Notation}
\newtheorem{assumption}[theorem]{Assumption}

\newcommand{\nat}[6][large]{%
  \begin{tikzcd}[ampersand replacement = \&, column sep=#1]
    #2\ar[bend left=40,""{name=U}]{r}{#4}\ar[bend right=40,',""{name=D}]{r}{#5}\& #3
          \ar[shorten <=10pt,shorten >=10pt,Rightarrow,from=U,to=D]{d}{~#6}
    \end{tikzcd}
}

\newcommand{\dCrl}[0]{\mathfrak{dCrl}}

\newcommand{\defeq}{\vcentcolon=}

\newcommand{\R}{\mathbbm{R}}
\newcommand{\C}{\mathbbm{C}}

\newcommand{\N}{\mathbbm{N}}

\newcommand{\pdiv}[2]{\frac{\partial{#1}}{\partial{#2}}}
\newcommand{\discatp}{\displaystyle\bigsqcap}
\newcommand{\discats}{\displaystyle\bigsqcup}

\newcommand{\uF}[1]{\underline{\mathbbm{F}}}
\newcommand{\one}{\mathbbm{1}}
\newcommand{\two}{\mathbbm{2}}

\newcommand{\disp}{\displaystyle\prod}
\newcommand{\disu}{\displaystyle\bigcup}

\newcommand{\diss}{\displaystyle\sum}
\newcommand{\disg}{\displaystyle\int}

\newcommand{\disbop}{\displaystyle\bigotimes}
\newcommand{\disbos}{\displaystyle\bigoplus}
\newcommand{\dissup}{\displaystyle\sup}

\newcommand{\dismax}{\displaystyle\max}
\newcommand{\dismin}{\displaystyle\min}

\newcommand{\Ab}[0]{\mathbbm{A}}
\newcommand{\Bb}[0]{\mathbbm{B}}

\newcommand{\Sb}[0]{\mathbbm{S}}
\newcommand{\Tb}[0]{\mathbbm{T}}
\newcommand{\Ub}[0]{\mathbbm{U}}

\newcommand{\sF}[1]{\mathsf{#1}}

\newcommand{\Ac}[0]{\mathcal{A}}

\newcommand{\Cc}[0]{\mathcal{C}}
\newcommand{\Dc}[0]{\mathcal{D}}

\newcommand{\Fc}[0]{\mathcal{F}}
\newcommand{\Gc}[0]{\mathcal{G}}
\newcommand{\Hc}[0]{\mathcal{H}}
\newcommand{\Ic}[0]{\mathcal{I}}

\newcommand{\Lc}[0]{\mathcal{L}}

\newcommand{\Oc}[0]{\mathcal{O}}

\newcommand{\Sc}[0]{\mathcal{S}}
\newcommand{\Tc}[0]{\mathcal{T}}
\newcommand{\Uc}[0]{\mathcal{U}}

\newcommand{\Xc}[0]{\mathcal{X}}

\newcommand{\Xf}[0]{\mathfrak{X}}

\newcommand{\af}[0]{\mathfrak{a}}
\newcommand{\bff}[0]{\mathfrak{b}}

\newcommand{\ef}[0]{\mathfrak{e}}
\newcommand{\ff}[0]{\mathfrak{f}}
\newcommand{\gf}[0]{\mathfrak{g}}

\newcommand{\ifrak}{\mathfrak{i}}

\newcommand{\pf}[0]{\mathfrak{p}}

\newcommand{\zf}[0]{\mathfrak{z}}

\newcommand{\scB}[0]{\mathscr{B}}
\newcommand{\scC}[0]{\mathscr{C}}

\newcommand{\scF}[0]{\mathscr{F}}
\newcommand{\scG}[0]{\mathscr{G}}
\newcommand{\scH}[0]{\mathscr{H}}

\newcommand{\scJ}[0]{\mathscr{J}}

\newcommand{\scM}[0]{\mathscr{M}}

\newcommand{\scP}[0]{\mathscr{P}}

\newcommand{\scS}[0]{\mathscr{S}}
\newcommand{\scT}[0]{\mathscr{T}}
\newcommand{\scU}[0]{\mathscr{U}}

\newcommand{\fA}[0]{\mathsf{A}}
\newcommand{\fB}[0]{\mathsf{B}}
\newcommand{\fC}[0]{\mathsf{C}}
\newcommand{\fD}[0]{\mathsf{D}}
\newcommand{\fE}[0]{\mathsf{E}}

\newcommand{\fO}[0]{\mathsf{O}}

\newcommand{\fS}[0]{\mathsf{S}}

\newcommand{\fX}[0]{\mathsf{X}}
\newcommand{\fY}[0]{\mathsf{Y}}

\newcommand{\fa }[0]{\mathsf{a}}
\newcommand{\fb }[0]{\mathsf{b}}
\newcommand{\fc }[0]{\mathsf{c}}
\newcommand{\fd}[0]{\mathsf{d}}
\newcommand{\fe}[0]{\mathsf{e}}

\newcommand{\fs}[0]{\mathsf{s}}
\newcommand{\ft }[0]{\mathsf{t}}

\newcommand{\fx}[0]{\mathsf{x}}
\newcommand{\fy}[0]{\mathsf{y}}
\newcommand{\fz}[0]{\mathsf{z}}

\renewcommand{\dh}[0]{\dot{h}}

\newcommand{\dt}[0]{\dot{t}}

\newcommand{\dv}[0]{\dot{v}}

\newcommand{\dx}[0]{\dot{x}}

\newcommand{\ox}[0]{\overline{x}}

\renewcommand{\a}{\alpha}
\renewcommand{\b}{\beta}
\renewcommand{\d}{\delta}
\newcommand{\e}{\varepsilon}

\newcommand{\g}{\gamma}

\renewcommand{\i}{\iota}
\renewcommand{\k}{\kappa}
\renewcommand{\l}{\lambda}
\newcommand{\m}{\mu}

\newcommand{\p}{\pi}
\newcommand{\ph}{\varphi}

\renewcommand{\r}{\rho}
\newcommand{\s}{\sigma}
\renewcommand{\t}{\tau}

\newcommand{\z}{\zeta}
\newcommand{\G}{\Gamma}

\begin{document}

\title{Morphisms of Networks of Hybrid Open Systems} 
\author{Arnold James Schmidt}
\department{Mathematics} 
\schools{B.A., University of Notre Dame, 2014\\
		B.S., University of Notre Dame, 2014\\
		M.S., University of Illinois at Urbana-Champaign, 2016}
		\phdthesis
		\adviser{Professor Eugene Lerman}
		\degreeyear{2019}
		\committee{Professor Lee DeVille, Chair \\ Professor Eugene Lerman, Director of Research\\
		 Assistant Professor Dan Berwick-Evans \\
		Visiting Scholar Eddie Nijholt}
		\maketitle
		
\frontmatter
\begin{abstract}

This dissertation develops a theory of networks of hybrid open systems and morphisms.  It builds upon a framework of networks of continuous-time open systems as product and interconnection.  We work out categorical notions for hybrid systems, deterministic  hybrid systems,  hybrid open systems, networks of hybrid open systems, and morphisms of networks of hybrid open systems.  

We also develop categorical notions for abstract systems, abstract open systems, networks of abstract open systems, and morphisms of networks of abstract open systems. We show  that a collection of relations holding among pairs of systems induces a relation between interconnected systems. We use this result for abstract systems to prove a corresponding result for networks of hybrid systems. 

This result  translates as saying that our procedure for building networks preserves morphisms of open  systems: a collection of morphisms of (sub)systems is sent to a morphism of  networked systems.  We thus   both justify our formalism and concretize the intuition that a network is a collection of systems pieced together in a certain way. 
\end{abstract}
\newpage 
\chapter*{Acknowledgments}\normalsize

I owe an immense debt of gratitude to the people in my  life who have supported  me throughout my education and career.  First, to my adviser, Eugene, whose direction, input,  critique,  bottomless patience, and friendship have sustained and shaped me as a researcher and mathematician (and climber).  I  thank my committee---Lee, Dan, and Eddie---for their time, encouragement,  and guidance with my thesis.  Next, Guosong Yang---my academic big brother in engineering---I looked and look up to you for your rigor in detail and your  patience in explanation. I value the time we had in graduate school together, and treasure the advice and encouragement you have given me.  To  Daniel Liberzon and Ali Belabbas, thank you for your understanding and support during my work in engineering.  I call to mind professors who mentored me during my undergraduate years: Sam Evens,  Michael Seelinger, Curtis Franks, Arlo Caine, Dave Hoelzle, and Cameron Hill.  You guided me through my undergraduate journey, gave me opportunity, and helped me follow through. I also thank my high school calculus teacher, Mr.\ Wahl who first introduced me and inspired me to appreciate the beauty of mathematics.  I am still inspired.  

Finally, and most importantly, I thank my family.  You have consistently encouraged me to pursue my dreams, buttressed my efforts with your unfailing support, and celebrated my accomplishments as your own. Your love keeps me strong.

\tableofcontents
\mainmatter
\singlespacing

\chapter{Introduction}\label{section:Intro}\label{ch1}
\section{Motivation} 

Networks of systems are ubiquitous in engineering.  In fact, elementary examples of higher order ordinary differential equations (ODEs)---e.g.\ coupled oscillators (\cite[\S7.1]{boyce}, \cite[\S6.1]{goodwine}) and systems of ODEs generally (\cite[\S8]{arnold})---arise naturally as networks.  Informally, a network interconnects a collection of open systems, systems with external inputs, to make up a higher dimensional closed system.   Networks are not, however, merely  high dimensional dynamical systems (\cite{golubitsky06nonlineardynamics}): the process of combining subsystems together reveals structure and regularity (symmetry or synchronous behavior, for example) which may be isolated, formalized, and separately understood. 

Many perspectives have been developed in the literature for building a  theory of networks of dynamical systems.  A common one, including work of Golubitsky and Steward, uses graph theoretic formalism  (\cite{golubitsky06nonlineardynamics}, \cite{golubitskyRecentAdvances}): nodes depict a space where dynamics live, and edges depict interaction among or information flow between spaces.  This framework allows for insight into the structure of network dynamics to be seen through discrete patterns  in graphs, for example to highlight synchrony in coupled oscillators (\cite{golubitskyCPG}). Work by Lerman and DeVille (\cite{lermannetworks}, \cite{dl1}) capitalizes on and extends this approach by considering a special class of \textit{maps} or morphisms of networks---as maps of graphs carrying information from the underlying dynamical systems---which induce related interconnected dynamical systems. 

Though graphs seem like a natural structure with which to represent networks---especially in engineering since combinatorial  properties are both interpretable and  computable---other perspectives in the literature view networks instead from a category theoretic angle.  One reason is that category theory is particularly apt for isolating features of  mathematical phenomena at a certain, and usually correct, level of generality: the universal property of product, e.g.,\ ``really'' defines  products. Following the trope that ``networks of systems are not just bigger systems,'' a categorical approach may aim to define what precisely the nuance is, whereas a combinatorial approach may, on the other hand, look for ways to compute particular network-specific behavior.  After all, both products and monoidal products arise often in many systematic investigations of networks, concepts which are at home in the study of category theory.   For example, work in \cite{lermanSpivakVagner} uses wiring diagrams to model input-output relations among a collection of systems and to formally interconnect  those systems. Wiring diagrams---which define a monoidal category (\cite{lermanSpivakVagner})---represent ways of ``interaction,'' or input/output relations, and an algebra on these diagrams introduces a notion of state or space with  which dynamics may be included.  This is one approach for isolating network defining features and---more to the point---provides an instance of the idea that networks may be well modeled as objects or morphisms in a monoidal category.  Additional work treating wiring diagrams includes \cite{SpivakOperadWiring}, 
\cite{spivak2016nesting}, and  \cite{SpivakOperadWiring2} and uses them  to depict both  visually and categorically  the notion of building a larger object from a collection of smaller ones.  Other examples of the monoidal viewpoint may also be found among \cite{baezPetri}, \cite{fongAlgebraOpen}, and \cite{baezMarkov}. We develop yet another one here.

Our starting point for  a theory of networks is developed in \cite{lermanopennetworks}, which treats in particular networks of  open dynamical systems.  This theory differs   from the theory of wiring diagrams in two notable respects.  First, we focus  on a fixed space with a collection of dynamics (vector fields or control)  over the space, instead of scrutinizing  input-output relations for varying spaces.  Secondly, and related, a notion 
 of morphism or map of networks of systems is introduced: fixing the space and looking at the collection of possible dynamics over the space makes way for the notion of \textit{related} dynamics between maps of spaces, and related dynamics on networks.   One reason for studying this notion of morphism of networks arises   from the central role which (the analogous notion of) morphism of dynamical systems plays in the theory of dynamical systems.  Not only do maps of dynamical systems preserve structural information (in a way analogous to homomorphisms of algebraic objects), but also  intrinsic properties of dynamical systems may in fact be encoded \textit{as} maps of dynamical systems.  Equilibria, periodic orbits, and---more generally---invariant subspaces are instances of (arise as) maps of  systems.  Integral curves  are also a special kind of map of dynamical systems, and are moreover preserved by maps: maps of dynamical systems send integral curves to integral curves. Synchronous behavior of subsystems is an analog of invariance for networks, and may---in our framework---be realized as a map of networks.

Another  motivation for the ``morphism-centric'' perspective is the categorical worldview which says that objects are known by the class of morphisms into or out of them: the Yoneda embedding is fully faithful (\cite{mazurGrothendieck}).  These motivations are related. For example, it turns out the  functor from the category of complete dynamical systems to the category of sets which drops both dynamics and smooth structure is \textit{representable} (\cref{def:functorRepresentable}, \cref{prop:existenceAndUniquenessRepresentable}).  At its core, this is the    fundamental existence and uniqueness theorem  from the theory of dynamical systems for complete dynamical systems (\cref{theorem:E&U}) masquerading as a categorical statement:  an initial element in the category of elements of the forgetful functor (\cref{prop:functorRepresentableIFFcategoryOfElementsHasUniversal}, \cref{prop:existenceAndUniquenessRepresentable}) is the pair $\big((\R,\frac{d}{dt}),0\big)$, a morphism from which is a solution (integral curve) to some dynamical system together with a choice of initial condition.

 While it is encouraging that category theoretic formulations  exist for otherwise concrete theorems in dynamical systems, we emphasize again that category theory provides a useful conceptual apparatus for suitably isolating the relevant aspects of a mathematical topic, and how  similar kinds of ideas and  arguments may extend to other settings. In our case, we develop  a theory for network of \textit{deterministic hybrid system},  piggybacking on \cite{lermanopennetworks} for networks of (continuous-time) open  systems.  We work out hybrid analogs for continuous-time systems concepts; a hybrid phase space is the hybrid version of a manifold, a hybrid system is the hybrid version of a continuous-time dynamical system, etc. Each of these comprises the object part of some category, and we define the corresponding morphisms and categories.  We show that there are ``structure preserving'' functors from the relevant hybrid category to the non-hybrid category, and use said functors to deduce results for the hybrid setting from those in the non-hybrid setting. 
 
 This project leads to a categorical abstraction of the notion of network, which is applicable to other kinds of dynamical systems. In fact, we also abstract the very notion of system.  These tasks are interrelated: the abstraction is justified in part by its effectiveness in applying to both continuous-time systems and hybrid systems, while derived out of what we perceived to be exclusively hybrid versus what turns out to be general, again illustrative of the power of a categorical  viewpoint. 
 
 We spell out our main result and highlight where the abstraction occurs.  Our conception of network takes a collection of disparate spaces and combines them together into one, which we realize by taking a product in a monoidal category, and interconnecting (\cref{def:HybridInterconnection}, \cref{def:SSubInterconnection}). (We thus see a monoidal viewpoint centrally present in our work as well.)   Interpreted  in the category of continuous-time open systems,   the main result (\cite[Theorem 9.3]{lermanopennetworks}, restated in \cref{theorem:mainTheoremLermanOpenNetworks}) says that if two distinct collections of open systems are ``pairwise related,'' then interconnection on the product systems of each collection results in a pair of open systems which are related.  This generalizes the main result in \cite{lermannetworks} described in terms of graph fibrations for the existence of maps between closed systems, and allows for more complex (e.g.\ non-diagonal) examples of subsystem invariance.  We will make precise both ``relatedness'' and  ``pairwise,'' but a quick takeaway  from this appetizer  on networks is that information of the pieces (open subsystems) can be systematically pieced together to provide information about the whole (the networked system), and moreover, that the category ``holds together'' under the process of taking networks, i.e.\ morphisms still make sense and behave as we would like.  Our main theorems for a class of  hybrid open systems (\cref{theorem:mainTheoremFordeterministicHybridSystems}) and for abstract systems (\cref{theorem:mainTheoremAbstract}), generally, say the same thing, modulo the category in which each respective statement is made. 
 
 Here is  how we  piggyback on \cite{lermanopennetworks}:  since we claim that the notion of network is essentially categorical, and  as we construct a category of hybrid systems,   we  should be able to obtain \textit{mutatis mutandis} a similar result with  similar arguments for hybrid systems,   and indeed we can.  Though functors automatically and always preserve some structure, namely identity and composition (and therefore, e.g., also isomorphism  (\cref{fact:functorsPreserveIsos})), we show that a functor from some hybrid category to its  non-hybrid analog preserves products (\cref{lemma:PUUPLite}) as well.   This is  the structure preservation  we need in order to translate results about networks of continuous-time systems to similar ones for hybrid systems. A forthcoming version of our work (\cite{lermanSchmidt1}) produces  this  translation directly. 
 
 In this work, we instead opt to rework  the underlying framework.  We develop a theory of system as object and section of a split epimorphism, which epimorphism comes from a natural transformation between two functors, the source of which we may think of as a ``tangent'' functor and the target as an underlying phase space. The target of the functors represents a category whose ``dynamics'' we treat as proxy for dynamics in the source.  Thus  one approach to presenting a theory for networks of hybrid systems  is to lay the groundwork for mapping hybrid to non-hybrid and then \textit{to cite} the original theorem from the continuous-time case (\cite{lermanSchmidt1}). By contrast, the presentation we develop  here builds from scratch---using the continuous-time case as strong guidance---a general theorem and cites \textit{it} to produce as individual instances results for discrete-time, continuous-time, hybrid, and deterministic hybrid networks.

\section{Summary of Results}

We  present two detailed contributions.  We conclude the thesis with a third, which uses the morphism-centric viewpoint of continuous-time dynamical systems from a different angle in application.

First, we  develop a categorical notion of deterministic hybrid systems and networks of deterministic hybrid systems.  One aspect of working categorically is  that we define morphisms for each concept we introduce.    We state a result for maps of  networks of deterministic hybrid open systems (\cref{theorem:mainTheoremFordeterministicHybridSystems}) which illustrates a concrete extension of \cite[theorem 9.3]{lermanopennetworks}: two distinct collections of \textit{deterministic hybrid} open systems which are ``pairwise related'' induce a pair of related \textit{deterministic hybrid} open systems after interconnection.   While we state \cref{theorem:mainTheoremFordeterministicHybridSystems} in \cref{ch3}, we give its proof as an instance of a more general result in \cref{ch4}.

 Secondly, we develop an abstract  notion  of system and network of systems which makes sense apart from ordinary continuous-time dynamical systems and even hybrid systems.   
  We state and prove the main result  (\cref{theorem:mainTheoremAbstract}) which says that two distinct collections of abstract open systems which are ``pairwise related'' induce a pair of related abstract open systems after interconnection.  This is the general  result we use to prove \cref{theorem:mainTheoremFordeterministicHybridSystems} as a corollary.   Put differently, \cref{theorem:mainTheoremAbstract} is a theorem of which \cref{theorem:mainTheoremFordeterministicHybridSystems} is an example.\footnote{Treating theorems as examples of other theorems (as opposed to illustrating said theorems with examples) is reminiscent of a  quip attributed to David Spivak.} We also deduce  \cite[theorem 9.3]{lermanopennetworks}  and  \cite[theorem 6.19]{lermanSchmidt1}---an analogous non-deterministic hybrid networks result---as corollaries of \cref{theorem:mainTheoremAbstract}.  We end with a version of the main theorem for networks of discrete-time systems.  We abstracted  enough to prove four distinct, yet similar, network statements, while leaving enough room for further application for different ``kinds of systems.'' 

We thus both formulate an abstract framework and work out details for a particular class to which the abstraction applies.  We discuss these in turn.   

\subsection{Networks of Systems and Interconnection}\label{subsubNetworksVanillaInIntro}

We start with a concrete example from continuous-time dynamical systems to introduce our overall  approach to networks.  For us, a continuous-time dynamical system $(M,X)$ is a pair where   $M$ is a manifold and $X\in \Xf(M)$  is a vector field over the manifold. Consider a two-dimensional dynamical system $(\R^2,X)$ with  vector field $X\in \Xf(\R^2)$,   a  smooth map $X:\R^2\rightarrow T\R^2$ sending $(x,y)\mapsto X(x,y)\in T_{(x,y)}\R^2\cong T_x\R\times T_y\R$. Thus $X(x,y) = (X_1(x,y),X_2(x,y))$ with $X_1(x,y)\in T_x\R$ and $X_2(x,y)\in T_y\R$.  The maps $X_1$ and $X_2$ are not vector fields over $\R$:  the tangent vector $X_1(x,y)\in T_x\R$, e.g., depends on a variable other than $x$.  Instead they make up two \textit{open} systems over surjective submersion $\R^2\xrightarrow{p_i}\R$, namely  maps $X_i:\R^2\rightarrow T\R$ compatible with $p_i$: $\t_\R\circ X_i = p_i$, where $\t_\R:T\R\rightarrow\R$ is the canonical projection of the tangent bundle and $p_i$ the projection onto the $i$th component. These open systems induce a product open system $X_1\times X_2$ over surjective submersion $(\R^2\rightarrow\R)\times (\R^2\rightarrow\R)\cong (\R^4\rightarrow\R^2)$ which is defined by $(x,y,x',y')\mapsto (X_1(x,y),X_2(x',y'))\in T_x\R\times T_{y'}\R$.

We recover the original vector field $X$ as a (closed) system over $\R^2$ by ``interconnection.''  We define an embedding $\i:\R^2\hookrightarrow\R^4$ sending $(x,y)\mapsto \i(x,y)\defeq (x,y,x,y)$.  Precomposing the open system $X_1\times X_2$ with interconnection returns $X = X_1\times X_2 \circ \i$. We call $\i$ \textit{interconnection} because through it we may  ``interconnect'' two separate open systems  $(\R^2\xrightarrow{p_i}\R, X_i)$ into a single (closed) dynamical system $\big(\R^2, \i^* (X_1\times X_2)\big) = (\R^2, X)$.  This  example provides a quick taste of how systems are constructed out of a collection of systems, and this process of interconnecting leads to our concept of networks. We will formalize this model and show that it applies to other kinds (e.g.\ hybrid) of systems.   

\subsection{Hybrid Phase Spaces}\label{subsec:intoHYPH}
A hybrid system consists of both continuous and discrete behavior. We will successively construct various notions of hybrid systems, the principal building block of which is a hybrid phase space. Just as a continuous-time system consists of manifold and vector field (``space'' and ``dynamics''), similarly many of our hybrid systems notions are defined as a hybrid space with  other data specifying hybrid dynamics. 

A hybrid phase space $\left(\begin{tikzcd}G_1\arrow[r,shift left,"s"]\arrow[r,shift right,swap,"t"] & G_0\arrow[l]\end{tikzcd}, \scH:G\rightarrow\sF{RelMan}\right)$ may be described by a \textit{directed reflexive (multi) graph} $G$ (with nodes $G_0$, edges $G_1$, source and target maps $s,t: \begin{tikzcd}G_1\arrow[r,shift left]\arrow[r, shift right] &  G_0\end{tikzcd}$ and unit  map $u:G_0\rightarrow G_1$), and assignments of manifolds $\scH(g)$ to each node $g\in G_0$. For us, all manifolds are manifolds-with-corners (even if there are no corners) and edges encode relations: $\scH(\g)\subseteq \scH(g)\times \scH(g')$ for each edge $g\xrightarrow{\g}g'$ of $G_1$. We require unit relation $\scH(u(g)) = \Delta(\scH(g)) \defeq \big\{ (x,x')\in \scH(g)\times \scH(g):\, x = x'\big\}$ for technical reasons which we will elaborate and make use of later: first this condition allows us to define a global  jump map, and secondly it allows us to take products of hybrid systems (a key ingredient to building networked systems from a collection) while circumventing a highly restrictive, and otherwise unrealistic, constraint of simultaneous state transitions. Each manifold $\scH(g)$ represents (part of) a phase space where a state may flow continuously,  whereas the relations $\scH(\g)$ represent possible jumps.  We are rather flexible about the topological nature of the relations.  

We present an example which serves as a phase space for the dynamic behavior of a standard hybrid system, a thermostat, which controls the temperature of a room.  A heater turns on to drive  temperature up to some specified level at which point the heater turns off and the temperature decreases until it falls below some threshold. For now, we only discuss the space. 
Consider directed graph $\begin{tikzcd}
	0\arrow[r,bend left,"e_{1,0}"]  & 1\arrow[l,bend left, "e_{0,1}"] 
\end{tikzcd}$ (note that we do not explicitly display unit edges $u(i)$) and manifold assignment $\scH(i) \defeq \R\times \{i\}$.  To edges, we assign relations $\scH(e_{1,0})\defeq \big\{\big((x,0),(x,1)\big):\, t\leq -1\big\}$ and similarly $\scH(e_{0,1}) \defeq \big\{\big((x,1),(x,0)\big):\, t\geq 1\big\}$.  The physical interpretation is the following: variable $x\in \R$ represents the temperature of some room, while the second  variable $i = 0,1$ will indicate whether the heater is on or off. So far we have not specified \textit{how} this discrete change is realized in practice, nor have we explicitly defined (continuous) dynamics.   We will formalize this example in \cref{example:ThermostatVanilla} after defining hybrid systems.  Before continuing with the conversational version of this example, we discuss dynamics and determinism in hybrid systems.

\subsection{Deterministic Hybrid Systems}\label{subsec:introDetHySys}

From the data of a hybrid phase space, we may recover an underlying manifold by taking the disjoint union $\discats_{g\in G_0}\scH(g)$ over nodes $g\in G_0$ of each manifold $\scH(g)$. It turns out that this operation, the coproduct,  is functorial (to the category of manifolds).  This coproduct functor serves two purposes for us: first we  may define many hybrid notions as a hybrid phase space plus some data in the category of manifolds, which data is related in some way to the hybrid phase space by said functor. Secondly, the functor gives us a way to apply results  from a more general theory of networks, without requiring ad hoc modifications rehashing the original theory to make it fit. 
In light of the first point, we define a hybrid system as a hybrid phase space with a vector field on the underlying manifold.  Having a notion of hybrid dynamical system,  there is an analogous notion of integral curve, or \textit{execution}, which is something like a ``piecewise continuous'' curve satisfying the governing dynamics (differential equation or  vector field equation) and at end points is compatible with relations (the pair of left and right limit points is  an element of one of the relations).  In fact, executions are a special class of map of hybrid dynamical systems---a map of hybrid phase spaces relating dynamics---and they provide concrete validation that the formalism is capturing familiar notions. 

One departure  of (this definition of) executions from a classical theory of dynamical systems is that executions are not unique.  We introduce a mechanism which, under mild conditions, enforces uniqueness of executions.  Recall that a vector field $X\in \Xf(M)$ of a smooth manifold $M$ is a smooth section of the tangent bundle, an element of $\G(TM\xrightarrow{\t_M} M)$. We construct a corresponding notion of section on a ``continuous-discrete bundle'' $\Tb \scH$ of a hybrid phase space $\scH$: $\Tb\scH$ is the product of the (ordinary) tangent bundle $T\discats_{g\in G_0}\scH(g)$ of the underlying (disjoint union) manifold and the underlying manifold itself $\discats_{g\in G_0}\scH(g)$, as a set.  There is a natural projection to the underlying manifold, obtained by first projecting to the tangent bundle, then taking the canonical projection of tangent bundle to the underlying manifold.  A section, therefore, assigns a tangent vector and element of the underlying space to each point in the underlying space.  This pair captures both continuous and discrete behavior at once: the tangent vectors should be smoothly varying and indicates a direction to flow, while the point assignment is required only to be compatible with relations, and represents discrete jumps.  In principle, each point is assigned a jump; however, our previous stipulation  that each node has a unit arrow permits (and we make use of)  trivial jumps $x\mapsto x$.  We return to the thermostat to understand deterministic hybrid systems in an example.

\subsection{Digital Control and Interconnection}\label{subsec:digitalControl}
   We model the dynamics of a thermostat as an interconnection of a continuous hybrid open system and a discrete one.  

Let the temperature be represented by $x\in \R$ and whether the heater is on or off by $i\in \{0,1\}$.  Together we have states $(x,i)\in \R\times\{0,1\}$ and we define two open systems $X_1(x,i) \defeq (-1)^{1-i}\in T_x\R$ and $X_2(x,i) \defeq 0  \in T_i\{0,1\}$.  Each vector represents the continuous dynamics: $X_1$ is positive when heat is on and negative otherwise, and the discrete space $\{0,1\}$ has no (or better: zero) continuous dynamics.  On the other hand, we define open jump maps to capture  switching behavior.  Temperature never jumps, indicated by  $\r_1(x,i) \defeq x$ and digital control is defined by $$\r_2(x,i) \defeq\left\{\begin{array}{ll} i & \mbox{if} \;(-1)^{1-i}x < 1 \\ 1-i & \mbox{if}\;(-1)^{1-i}x \geq 1. \end{array}\right.$$  Here $x=1$ represents the upper threshold for the heater to turn off, and $x=-1$ the lower threshold for the  heater to turn (back) on. 

We represent these hybrid open systems by $(\R\times\{0,1\}\xrightarrow{p_1}\R, X_1,\r_1)$ and $(\R\times\{0,1\}\xrightarrow{p_2}\{0,1\}, X_2,\r_2)$, where $p_i$ denotes the projection onto the $i$th factor. These induce a product open system $(\R\times \{0,1\}\times \R\times\{0,1\}\xrightarrow{p_1\times p_2}\R\times \{0,1\}, X_1\times X_2, \r_1\times \r_2)$  where $(p_1\times p_2)(x,i,x',i') = (x,i')$. In a manner analogous to the two-dimensional system $X\in\Xf(\R^2)$ outlined in \cref{subsubNetworksVanillaInIntro}, we define interconnection $\i:\R\times \{0,1\}\rightarrow\R\times \{0,1\}\times \R\times \{0,1\}$ sending $(x,i)\mapsto (x,i,x,i)$ and obtain the thermostat as a result of precomposing interconnection  $\i$ with the deterministic closed system: $\i^*(X_1\times X_2)(x,i) = ((-1)^{1-i},0)$ (governing the vector field on temperature) and $$\i^*(\r_1\times \r_2)(x,i) = \left\{\begin{array}{lll}(x,i) & \mbox{if} & (-1)^{1-i}x< 1 \\ 
(x,1-i) & \mbox{if} & (-1)^{1-i}x \geq 1, 
\end{array}\right.$$
governing digital control.

In presenting this example, we specified numerous data: a phase space where the state lives, a (smooth) map (control) which determines continuous flow, and a jump map representing discrete behavior.  We can now comment on how  the relations of a hybrid phase space constrain the jump map: according to the relation $\scH(e_{1,0})$ we defined above (in  \cref{subsec:intoHYPH}),  a non-trivial jump for $(x,i)$ may only occur if both $x\leq -1$ and $i=0$. We see that the jump map does nothing to the continuous variable $x$: $\r_1(x,i) = p_1 \circ i^*(\r_1\times \r_2)(x,i)=x$. In this way, interconnection represents digital control, as the  discrete change merely switches between continuous dynamics, and otherwise does nothing to the state.

\subsection{Abstract Systems}

We believe that a categorical theory of networks extends beyond the notion of ``network of dynamical systems.''  In fact, even the notion  of system may be appropriately generalized.  Recall that a continuous-time system for us is a manifold $M$ and   a (smooth) section $X\in \G(TM)$ of the tangent bundle. We only  need  a category and some sort of fibered space on which to take sections; relatedness of sections straightforwardly generalizes relatedness of vector fields (some diagram commutes).  Our abstraction is roughly as follows: we consider two functors $\Tc, \Uc:\fC\rightarrow\fD$ between concrete and locally small categories, and natural transformation  $\nat{\fC}{\fD}{\Tc}{\Uc}{\t}$ which we called ``fibered'' if  each morphism $\t_\fc:\Tc \fc\rightarrow\Uc\fc$ is a split epimorphism. We denote sections by $\G_\t(\fc)\defeq \big\{X:\Uc\fc\rightarrow\Tc \fc:\, \t_\fc\circ X = id_{\Uc \fc}\big\}$. 

Since $\G_\t(\fc)$ is nonempty for each object $\fc\in \fC$, we define an abstract system $(\fc,X)$ as a pair where $X\in \G_\t(\fc)$. Supposing that integral curves realize the dynamics portion of a (continuous-time) dynamical system, we take our cue from the Yoneda version of existence and uniqueness to justify this abstraction (\cref{prop:existenceAndUniquenessRepresentable}): a ``solution'' of system is a quasi-initial element in the category of elements $\disg_{\fC}\upsilon$ of a faithful functor $\upsilon:\fC\rightarrow\sF{Set}$ (\cref{def:abstractSolution}).

The functors $\Tc, \Uc:\fC\rightarrow\fD$ allow us to use dynamics in one category (in $\fD$) as proxy for structure in the other (in $\fC$).  In the case of \textit{non}-deterministic hybrid systems, $\sF{C}$ is the category of hybrid phase spaces, while $\fD$ is that of manifolds, in which we already have a theory for dynamical systems (as object and  section of the tangent bundle pair).  The functor $\Uc$ is the forgetful functor  alluded to in \cref{subsec:introDetHySys}, and $\Tc$ the composition of the tangent endofunctor $T:\sF{Man}\rightarrow\sF{Man}$ with forgetful functor $\Uc$. 

The setup for  \textit{deterministic} hybrid systems is slightly different.  We keep the category $\fC$ the same (hybrid phase spaces), but this time map to $\fD=\sF{Set}$.  Now $\Tc$ is $\Tb$, the c.d.\ endofunctor we introduced in \cref{subsec:introDetHySys}, and $\Uc$ maps a hybrid phase space to the underlying set of the underlying manifold.  The reason for working in $\sF{Set}$ has to do with taking sections of the bundle $\t:\Tc\Rightarrow \Uc$: namely, the jump map $\r$ is not required to be smooth or even continuous. 

We want a notion of open system, which we obtain by considering a subcategory $\fA\subseteq \sF{Arrow(C)}$. In the continuous-time case, $\fA$ is the category of surjective submersions.  Additionally, we define interconnection abstractly as a subcategory $\fA_{int}\subseteq\fA$, whose objects are the objects of $\fA$ but whose morphisms are isomorphism on codomain. This data assembles to form a double category $\fA^\square$ whose $0$-morphisms are interconnection and whose $1$-morphisms are objects of $\fA$. We extend  sections  to a (double) functor $\G_\t:A^\square\rightarrow\sF{Set}^\square$, which loosely parses as: interconnection preserves related open systems.  This makes up half of the proof of our main theorem. Next, we need to interpret interconnection in the context of networks.

\subsection{Abstract Networks}

Construction of networks occurs by putting and connecting multiple open systems together.  We take this to mean applying a monoidal product over a finite indexed list of objects in a monoidal category and taking interconnection, respectively.  In our example of interconnection in \cref{subsubNetworksVanillaInIntro}, we considered two open systems and obtained another one by taking the cartesian product of both, e.g.\ $(\R^2\xrightarrow{p_1}\R, X_1)\times (\R^2\xrightarrow{p_2}\R,X_2) \cong (\R^4\xrightarrow{p_1\times p_2} \R^2,X_1\times X_2)$.   To apply this procedure in general,  we require  $\Tc, \Uc:\fC\rightarrow\fD$ to be (strong)  \textit{monoidal} functors, and we require the target space $\fD$ to be cartesian.  Again, $\fD$ is the space in which we indirectly (functorially) make sense of dynamics in $\fC$.  We show that a collection of open systems induces a single open system on the product.  This is the step for making one system from many. We prove, additionally, that this mechanism preserves related systems.  In slightly more detail, if two collections are ``pairwise'' related, then taking the product on each collection results in a pair of related systems.  This makes up the second half of the proof for our main theorem. 

Piecing together product (2nd half) and interconnection (1st half), we  present the main abstract result: a pair of collections of open systems related in some way induces related open systems after interconnection.  This result is similar to but generalizes the main theorem of \cite{lermanopennetworks}---which could be stated as an instance of ours with $\fC = \fD = \sF{Man}$, $\Uc = id_{\sF{Man}}$, and $\Tc = T$---and captures the intuition we hinted at earlier.  Namely: a network pieces subsystems together, but more importantly, the fact that relations are preserved in the networked systems suggests that our formalism is on the right track.

\subsection{Application: Maps of Systems and Stability} 
We end this dissertation with a more grounded investigation into maps of dynamical systems as a means of verifying concrete systems properties.  This result is not about networks, but we present it as a springboard for future development in the context of networks.   We use maps of dynamical systems to answer questions about stability of dynamical systems.  Stability, for us, means that integral curves stay close to each other if they have initial conditions close to each other.  Traditionally, techniques for studying stability include Lyapunov's first and second methods, namely linearization of nonlinear dynamics at a point, or construction of a decreasing-along-solutions Lypanov or energy function.  We introduce a result which says that a class of maps of dynamical systems preserves stability; in other words, a stable point in the domain is sent to a stable point in the image.  The advantage of this method, akin to Lyapunov's second method, is that stability questions  which can be answered for some system may under suitable mapping be translated to another hitherto inscrutable dynamical system.  
\section{Outline of Dissertation}

In \cref{ch2} we review mathematical notions and theorems requisite for understanding the machinery built up in this thesis.  This includes category theory, at the level of basic definitions, universal properties of product and coproduct, the arrow category, Yoneda lemma, double categories, and monoidal categories.  We also review manifolds, manifolds with corners,  and dynamical systems, and provide proofs of basic facts from geometry using category theory as reference for some of the techniques employed later on.  This review is rather extensive, for the purposes of making this thesis as  self-contained as possible. 

 In \cref{ch3} we develop a concrete categorical theory of hybrid and deterministic hybrid system.  This includes constructing a category of hybrid phase spaces, hybrid surjective submersions, a  control/sections functor, a monoidal product, and a strong monoidal functor $\Ub:\sF{HyPh}\rightarrow\sF{Man}$ from the category of hybrid phase spaces to the category of manifolds.  
 
 In \cref{ch4} we work out abstract categorical notions of system (as pair of object  and section of a bundle) and networks of abstract systems (as interconnection to a monoidal product).  We then state and prove the main abstract result that a morphism of networks of systems induces a 1-morphism of related sections. We conclude with various concrete examples of morphisms of networks of open continuous-time systems, of  hybrid open  systems, and of deterministic hybrid open systems as instances of the abstract result.  

We end in \cref{ch5} with a light appetizer on a more applied direction of the morphism-centric viewpoint in the study of dynamical systems.   We review the notion of Lyapunov stability definable as ``continuity of a [certain kind of] map,'' and prove a preliminary result that a certain class of maps of dynamical systems preserves stability. Our goal is to use this perspective as guidance for continued study of stability for hybrid systems, networks of continuous-time systems (especially string stability, an oft used notion for autonomous traffic flow algorithms, e.g.\ \cite{hedrickabc}), and networks of hybrid systems.

\chapter{Background}\label{subsec:Prelims}\label{ch2}

\section{Introduction}
Almost all of \cref{ch2} is background material used in the development of theory in \cref{ch3} and \cref{ch4}.  We review ordinary categories, arrow categories, double categories, and monoidal categories, in addition to some light differential geometry, and dynamical systems theory.  Because the bulk of our work is categorical, we emphasize details on the category theory, while keeping details on  geometry to a minimum.

\section{Category Theory}\label{section:CategoryTheory}
\subsection{Ordinary Categories}\label{subsection:BasicCategoryTheory}
Concepts in \cref{subsection:BasicCategoryTheory} are standard and may be found in \cite{riehlb}, \cite{awodey}, \cite{maclane}, or \cite{leinsterCategory}. Pacing is brisk, and we refer the reader to these texts for more complete treatment of categories. Our goal here is to set notation and highlight the relevant concepts we will be using.

\begin{definition}\label{def:category}
	A \textit{category} $\fC = \left(\begin{tikzcd} \fC_1\arrow[r,shift left,"\sF{dom}"]\arrow[r,shift right,swap,"\sF{cod}"] & \fC_0\arrow[l] \end{tikzcd}\right)$ consists of a collection of \textit{objects} $\fC_0$ and collection of \textit{morphisms} $\fC_1$, assignments $\sF{dom},\sF{cod}:\fC_1\rightarrow \fC_0$ (domain and codomain) and  map $id_{(\cdot)}:\fC_0\rightarrow\fC_1$ (the \textit{identity} or 	\textit{unit}). We may write morphism $f$ as $\sF{dom}(f)\xrightarrow{f}\sF{cod}(f)$ to specify domain and codomain. For each object $\fc\in \fC_0$, $\fc = \sF{dom}(id_\fc) =\sF{cod}(id_\fc)$, and for morphisms $f,g\in \fC_1$ with $\sF{cod}(f) = \sF{dom}(g)$, there is morphism $\sF{dom}(f)\xrightarrow{g\circ f} \sF{cod}(g)$.\footnote{Implicit here is that $\sF{dom}(g\circ f)  = \sF{dom}(f)$ and $\sF{cod}(g\circ f ) = \sF{cod}(g)$.} 
	 
	Finally, categories satisfy the following two axioms. 
	\begin{enumerate}
		\item For each morphism $f \in \fC_1$, $f = id_{\sF{cod}(f)} \circ f = f\circ id_{\sF{dom}(f)}$. 
		\item  Composition of morphisms is \textit{associative}: for every triple $f,g,h\in \fC_1$ of composable  morphisms,\footnote{In other words, $\sF{cod}(f) = \sF{dom}(g)$ and $\sF{cod}(g) = \sF{dom}(h)$.} $(h\circ g) \circ f = h\circ (g\circ f).$ 
	\end{enumerate}
When we write a morphism $f$ as $\fc\xrightarrow{f}\fc'$, it is clear that $\fc = \sF{dom}(f)$ and $\fc'= \sF{cod}(f)$. For objects $\fc, \fc'\in \fC$, we will let $\fC(\fc,\fc')$ denote the collection of morphisms from $\fc$ to $\fc'$.
\end{definition}

\begin{example}\label{ex:examplesOfCategories}
	Examples of categories include $\sF{Set}$ (objects are sets, morphisms are functions), $\sF{Man}$ (objects are smooth manifolds with corners (\cref{def:manifoldsWithCorners}), morphisms are smooth maps),  $\sF{Top}$ (objects are topological spaces, morphisms are continuous maps), and $\sF{Cat}$ (objects are small categories (\cref{def:smallCategory}), and morphisms are \textit{functors} (\cref{def:functor})).  Oftentimes, objects in a category are sets with additional structure and morphisms are structure preserving maps, but they need not always be (e.g.\ a group may succinctly be defined as a category with one object all of whose morphisms are invertible).  
\end{example}

The next example, the category of relations, will show up again in our development of dynamical systems.  Here we define relations as an ordinary category. Later we will redefine  a double category of relations (\cref{subsection:doubleCat}, \cref{def:categoryRelSet}).

\begin{example}\label{ex:(single)CategoryOfRelations}
	We define the (ordinary) category $\sF{Rel}$ of relations as follows:\begin{enumerate}
		\item Objects $X$ are sets $\{x\in X\}$.
		\item Morphisms $X\xrightarrow{R}Y$ are \textit{relations} $R\subseteq X\times Y$, subsets of the product. 
	\end{enumerate}
	 Composition is defined as follows: for relations $R\subseteq X\times Y$ and $S\subseteq Y\times Z$, we define \begin{equation}\label{eq:definingCompositionOfRelations} S\circ R\defeq\{(x,z)\in X\times Z:\, \exists\, y\in Y\;\mbox{with}\;(y,z)\in S\,\mbox{and}\, (x,y)\in R\}.\end{equation}  For set $X$, the identity relation is $id_X\defeq \Delta(X) = \{(x',x)\in X^2:\, x'= x\}$.  It is a formal verification that this category satisfies the defining axioms of a category (\cref{def:category}). 
\end{example}
\begin{remark}
	Defining a category requires specifying objects and morphisms, as well as verifying conditions 1 and 2 in \cref{def:category}.  Often, the checks of conditions are routine and we skip them.  Our one exception is the definition of (the category of) hybrid phase spaces (\cref{def:hybridPhaseSpace0}, \cref{def:categoryHyPh}). 
\end{remark}
Emphasis on morphisms  is central in the philosophy of category theory.  For example, the identity morphism is defined according to how it behaves when composed with other morphisms, as opposed to where it sends elements.  Similarly,  a surjection in set theory is a map for which each preimage is nonempty.  The analogous notion in category theory is \textit{epimorphism}. 

\begin{definition}\label{def:epimorphism}
A morphism $\fc\xrightarrow{p}\fc'$ in category $\fC$ is called an \textit{epimorphism} if for any pair of morphisms $\begin{tikzcd} \fc'\arrow[r,shift left,"g"]\arrow[r,swap,shift right,"h"] & \fc'',\end{tikzcd}$ $g\circ p = h\circ p$ implies that $g= h$. 
\end{definition}

\begin{definition}\label{def:splitEpi}
A morphism $\fc'\xrightarrow{p}\fc$ in category $\fC$ is called a \textit{split epimorphism} if there is a right inverse of $p$, i.e.\ a morphism $s:\fc\rightarrow \fc'$ such that $p \circ s = id_\fc$.
The morphism $s$ is called a \textit{section of} $p$, and we will generally denote the collection of such maps by $$\G(p)\defeq\big\{s:\sF{cod}(p)\rightarrow\sF{dom}(p):\, p\circ s = id_\fc\big\}.$$ 	
\end{definition}
Precomposing a section $s\in \G(p)$ to both sides of $h\circ p = g\circ p$ shows that a split epimorphism $p$ is an epimorphism. 

There is also a categorical version of injection or embedding: 
\begin{definition}\label{def:monomorphism}
	A morphism $\fc'\xrightarrow{i}\fc$ is called a \textit{monomorphism} if for any pair of morphisms $\begin{tikzcd} \fc''\arrow[r,shift left,"g"]\arrow[r,swap,shift right,"h"] & \fc',\end{tikzcd}$ $i\circ g = i\circ h$ implies that $g=h$. When there is monomorphism $\fc'\xrightarrow{i}\fc$, we also say that $\fc'$ is a \textit{subobject} of $\fc$. 
\end{definition}
	\begin{definition}\label{def:locallySmall}\label{def:smallCategory}
	A category $\fC$ is said to be \textit{locally small} if $\fC(\fc,\fc')$ is a \textit{set} for every pair of objects $\fc, \fc'\in \fC$. A category is \textit{small} if $\fC_1$ is a set.  
\end{definition}
We do not preoccupy ourselves with size issues, and will generally assume enough smallness whenever necessary (e.g.\ anything involving Yoneda, \cref{lemma:YonedaLemma}). 

\begin{definition}\label{def:functor}
A (covariant) \textit{functor} $\Fc:\fC\rightarrow\fD$ between two categories assigns an object $\Fc(\fc)\in \fD_0$ to every object $\fc\in \fC_0$ and a morphism $\Fc(\fc)\xrightarrow{\Fc(f)}\Fc(\fc')\in \fD_1$ to every morphism $\fc\xrightarrow{f}\fc'\in \fC_1$. These assignments  satisfy the following two conditions: \begin{enumerate}
	\item $\Fc(id_\fc) = id_{\Fc(\fc)}$ for every object $\fc\in \fC_0$ and
	\item $\Fc(g\circ f) = \Fc(g)\circ \Fc(f)$ for every composition of morphisms $g\circ f\in \fC_1$. 
\end{enumerate}
\end{definition}
\begin{definition}\label{def:contravariantFunctor}
	A \textit{contravariant functor} $\Fc:\fC^{op}\rightarrow \fD$ from $\fC$ to $\fD$ assigns an object $\Fc(\fc)\in \fD_0$ to each object $\fc\in \fC_0$ and morphism $\Fc(\fc)\xrightarrow{\Fc(f)}\Fc(\fc')$ to each morphism $\fc\xleftarrow{f}\fc'\in \fC_1$. These assignments  satisfy the following two conditions: \begin{enumerate}
	\item $\Fc(id_\fc) = id_{\Fc(\fc)}$ for every object $\fc\in \fC_0$ and
	\item $\Fc(g\circ f) = \Fc(f)\circ \Fc(g)$ for every composition of morphisms $g\circ f\in \fC_1$. 
\end{enumerate} 
\end{definition}
\begin{remark}
	$\fC^{op}$ is the opposite category of $\fC$, with the same objects but with the direction of all morphisms flipped:  $\fc\xrightarrow{f}\fc'\in \fC_1$ if and only if $\fc'\xrightarrow{f}\fc\in \fC_1^{op}.$
\end{remark}

\begin{remark}\label{remark:functoriality}
	We encode conditions 1 and 2 in \cref{def:functor} and \cref{def:contravariantFunctor} by the following two phrases:
\textit{functoriality on identity} refers to satisfaction of condition 1, and \textit{functoriality on composition} refers to satisfaction of condition 2.   
\end{remark}
Consider the well-known fact that functors preserve isomorphism. Many proofs in mathematics that an isomorphism in one category induces an isomorphism in another amount to demonstrating  a functorial relationship between both categories.  First we recall the definition of isomorphism in a category.  
\begin{definition}\label{def:isomorphism}
Let $\fC$ be a category and $\fc\xrightarrow{f}\fc'$ a morphism in $\fC$.  We say that $f$ is an \textit{isomorphism} if $f$ has a (necessarily unique) left and right inverse: a morphism $\fc'\xrightarrow{g}\fc$ such that $id_\fc = g\circ f$ and $id_{\fc'} = f\circ g$. 
\end{definition}
\begin{remark}
	Uniqueness of the inverse follows directly from the axioms of a category (\cref{def:category}): suppose that $g',g:\fc'\rightarrow\fc$ are both inverses of $f$.  Then $$g' = g' \circ id_{\fc'} = g'\circ (f\circ g) = (g'\circ f)\circ g = id_{\fc}\circ g = g.$$
\end{remark}

\begin{fact}\label{fact:functorsPreserveIsos}
	Let $\Fc:\fC\rightarrow\fD$ be a functor (\cref{def:functor}) and $\fc\xrightarrow{f} \fc'$ an isomorphism in $\fC$ (\cref{def:isomorphism}).  Then $\Fc\fc \xrightarrow{\Fc f} \Fc\fc'$ is an isomorphism in $\fD$. 
\end{fact}
We prove this rudimentary fact only to demonstrate use of  terminology in \cref{remark:functoriality}. \begin{proof}
	Let $g:\fc'\rightarrow\fc$ be the inverse of $f$, so that $id_\fc = g\circ f$.  Then $$\Fc g\circ \Fc f = \Fc(g\circ f)= \Fc(id_\fc) =id_{\Fc\fc }.$$ The first equality follows  by \textit{functoriality on composition}, and the last one by \textit{functoriality on identity}.  A formally identical computation shows that $id_{\Fc \fc'} = \Fc f \circ \Fc g$, and hence that $\Fc f:\Fc\fc\xrightarrow{\sim}\fc'$ is an isomorphism in $\fD$. 
\end{proof}

Another useful consequence of functoriality is that commutative diagrams are preserved by functors.
\begin{lemma}\label{lemma:functorsPreserveDiagrams}
Functors preserve commutative diagrams. 
\end{lemma}
\begin{proof}
	See \cite[Lemma 1.6.5]{riehlb}].
\end{proof}

We catalog a few additional properties of functors. 
\begin{definition}\label{def:faithfulFunctor}\label{def:fullFunctor}
Let  $\Fc:\fC\rightarrow\fD$ be a functor between locally small categories $\fC$ and  $\fD$ (\cref{def:locallySmall}). We say that functor  $\Fc$ is \textit{faithful} if for each pair of objects $\fc, \fc'\in \fC_0$, the map $\Fc:\fC(\fc,\fc')\rightarrow \fD(\Fc\fc,\Fc,\fc')$ is injective.  We say that $\Fc$ is \textit{full} if for each pair of objects $\fc, \fc'\in \fC_0$, the map $\Fc:\fC(\fc,\fc')\rightarrow \fD(\Fc\fc,\Fc,\fc')$ is surjective. Finally, we say that $\Fc$ is \textit{fully faithful} if $\Fc$ is both full and faithful. 
\end{definition}

\begin{remark}\label{remark:concreteCategory}
	In \cref{ex:examplesOfCategories}, we noted that oftentimes categories have sets with additional structure as their objects.  Such categories are said to be \textit{concrete}. The formal definition of a concrete category is a category which admits a \textit{faithful} (\cref{def:faithfulFunctor}) \textit{functor} (\cref{def:functor}) $\scU:\fC\rightarrow \sF{Set}$ to the category of sets. Every category we will encounter is concrete.  In practice, we will treat objects in a concrete category $\fC$ as sets-with-structure, whose faithful functor $\scU:\fC\rightarrow \sF{Set}$ forgets the structure.  
\end{remark}

\begin{definition}\label{def:naturalTrans}\label{def:naturalIso}
	Let $\Fc,\Gc:\fC\rightarrow\fD$ be two functors.   A \textit{natural transformation} $\nat{\fC}{\fD}{\Fc}{\Gc}{\a}$ (also denoted by $\a:\Fc\Rightarrow \Gc$) is an assignment $\a:\fC_0\rightarrow \fD_1$ which sends $\fc\mapsto \big(\Fc\fc\xrightarrow{\a_\fc}\Gc\fc\big)$ satisfying the following \textit{naturality} condition:\begin{enumerate}
		\item[] For each morphism $\fc\xrightarrow{f}\fc'$ in $\fC$, the diagram $\begin{tikzcd}
	\Fc\fc \arrow[d,"\Fc f"] \arrow[r,"\a_\fc"] & \Gc\fc \arrow[d,swap,"\Gc f"] \\
	\Fc\fc'\arrow[r,"\a_{\fc'}"] & \Gc \fc'
\end{tikzcd}$ commutes in $\fD$. 	\end{enumerate} 

	 We say that $\a:\Fc\Rightarrow \Gc$ is a natural \textit{isomorphism} if each component $\g_\fc:\Fc\fc\xrightarrow{\sim} \Gc\fc $ is an isomorphism in $\fD$. 
\end{definition}

\begin{example}\label{def:functorCategory}
	Let $\fC$ and $\fD$ be two locally small categories.  We define the \textit{functor category} $\fD^\fC$ as follows: \begin{enumerate}
		\item Objects $\Fc\in (\fD^\fC)_0$ are functors $\Fc:\fC\rightarrow\fD$ (\cref{def:functor}).
		\item Morphisms $\big(\a:\Fc\Rightarrow\Gc\big)\in(\fD^\fC)_1$ are natural transformations (\cref{def:naturalTrans}) between functors $\Fc$ and $\Gc$. 
	\end{enumerate}
	That $\fD^\fC$ is a category follows from the axioms defining categories (\cref{def:category}): each functor $\Fc:\fC\rightarrow\fD$ has an identity transformation $id_\Fc:\Fc\Rightarrow\Fc$ because each object $\fd\in \fD_0$ has an identity morphism $id_\fd$.  Therefore, we define natural transformation $id_{\Fc}$ on components by $(id_\Fc)_\fc\defeq  id_{\Fc \fc}$. The same idea implies that $(\g\circ \b ) \circ \a = \g\circ (\b\circ \a)$ for natural transformations $\a:\Fc\Rightarrow\Gc$, $\b:\Gc\Rightarrow\Hc$, and $\g:\Hc\Rightarrow\Ic$. 
\end{example}

We define a more general version of the functor category $\fD^\fC$.
\begin{definition}\label{def:generalizedFunctorCategory}
	Let $\fE$ be a category.  We define the category $(\sF{Cat}/\fE)^\Rightarrow$ as follows: 
	\begin{enumerate}
		\item Objects are functors $\upsilon:\fC\rightarrow\fE$ from a locally small category $\fC$ to $\fE$. 
		\item For objects $\upsilon:\fC\rightarrow\fE$, $\z:\fD\rightarrow\fE$, a morphism $\upsilon\xrightarrow{(\a,\af)}\z$ is a pair, where $\a:\fC\rightarrow\fD$ is a functor and $\af:\upsilon \Rightarrow \z\circ \a$ is a natural transformation. 
	\end{enumerate}
	The category $(\sF{Cat}/\fE)^\Leftarrow$ has the same objects.  Now a morphism $\upsilon\xrightarrow{(\b,\bff)}\z$  still is a pair with  functor $\b:\fC\rightarrow\fD$, but the  natural transformation $\bff:\z\circ \b\Rightarrow \upsilon$ goes the other direction.
\end{definition}
\begin{remark}\label{remark:generalizedFunctorCategoryMod}
	We will often use some subcategory of $(\sF{Cat}/E)^\bullet$. The category of hybrid phase spaces (\cref{def:categoryHyPh}) and the category of lists (\cref{def:categoryOfLists}) may be interpreted as some variant of this category. 
\end{remark}
\begin{fact}[Ex.\ 1.4.i \cite{riehlb}]\label{fact:inverseOfNaturalIsoIsNatural}
	Let $\g:\Fc\Rightarrow \Gc$ be a natural isomorphism.  Then the inverse $\g^{-1}:\Gc\Rightarrow \Fc$ defined on components as $(\g^{-1})_\fc\defeq (\g_\fc)^{-1}$ is a natural transformation (and therefore, natural isomorphism). 
\end{fact}

Finally, we single out a special class of functors.  Suppose that $\fC$ is locally small and let $\fc\in \fC_0$ be an object.  Then there is functor $\fC(\fc,\cdot):\fC\rightarrow\sF{Set}$ sending an object $\fc'\mapsto \fC(\fc,\fc')$ the set of morphisms from from $\fc$ to $\fc'$. Functors naturally isomorphic to $\fC(\fc,\cdot)$ have a special name: 

\begin{definition}\label{def:functorRepresentable}
	A covariant functor $\Fc:\fC\rightarrow\sF{Set}$ from is \textit{representable} if there is a natural isomorphism $\g:\fC(\fc,\cdot) \Rightarrow \Fc $ for some object $\fc\in \fC_0$. 
\end{definition}

\subsubsection{Limits, Colimits, and Universal Properties}
We recall universal properties of products and coproducts in a category $\fC$.  

\begin{definition}\label{def:terminalObject}\label{def:initialObject}
	An object $\fc_t\in \fC$ is said to be \textit{terminal} if for each object $\fc\in \fC$, there is exactly one morphism $\fc\xrightarrow{}\fc_t$ in $\fC$. Similarly, an object $\fc_i\in \fC$ is \textit{initial} if for each object $\fc\in \fC$, there is exactly one morphism $\fc_i\rightarrow\fc$. 
\end{definition}

\begin{definition}\label{def:products}
 Let $\fc,\, \fc'\in \fC$ be objects in a category.  A \textit{product} $\fc\times \fc'$ is an object in $\fC$  equipped with two maps $\fc\times \fc'\xrightarrow{p_\fc}\fc$ and $\fc\times \fc'\xrightarrow{p_{\fc'}} \fc'$ satisfying the following universal property: for any object $\fz\in \fC$ and pair of maps $\fz\xrightarrow{f_\fc}\fc$, $\fz\xrightarrow{f_{\fc'}} \fc'$, there is a unique map $f:\fz\dashrightarrow\fc\times \fc'$ through which $f_\fc$ and $f_{\fc'}$ factor.  In other words, $p_\fc\circ f = f_\fc$ and $p_{\fc'}\circ f = f_{\fc'}$ in the following diagram: $$
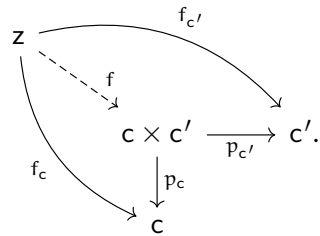
\begin{tikzcd}
	\fz \arrow[rrd,bend left,"f_{\fc'}"]\arrow[ddr,swap,bend right,"f_\fc"]\arrow[dr,dashed,near end,"f"] & & \\ & \fc\times \fc' \arrow[d,"p_\fc"] \arrow[r,swap,"p_{\fc'}"] & \fc'.\\
	& \fc & 
\end{tikzcd}$$
\end{definition}

A product is not only an object, but an object with some morphisms satisfying some universal property. This is a standard trope in category theory. 
\begin{definition}\label{def:categoryWithProducts}
	We say that a category $\fC$ has finite products if for any two objects $\fc,\,\fc'\in \fC$, the product $\fc\times \fc'\in \fC$. 
\end{definition}

We will use the universal property frequently, so it may be helpful to recall how it works in a routine application. 
\begin{prop}
	Let $\fc\times \fc'$ be a product of $\fc$ and $\fc'$ in category $\fC$.  If another object (label it) $\fc\fc'$ with two maps $\fc\fc'\xrightarrow{P_\fc}\fc$ and $\fc\fc'\xrightarrow{P_{\fc'}}\fc'$ satisfy the universal property of products (\cref{def:products}), then $\fc\fc' \cong \fc\times \fc'$ and the isomorphism is unique. 
\end{prop}
\begin{proof} 
The universal property of $\fc\times \fc'$ implies that there is a unique map $f:\fc\fc'\dashrightarrow \fc\times \fc'$ with $$\left\{\begin{array}{ll}p_\fc\circ f &= P_\fc\\ p_{\fc'} \circ f & = P_{\fc'}.\end{array}\right.$$ Consider, then, the following diagram  \begin{equation}\label{diagram:product} \begin{tikzcd}
	\fc\times \fc'\arrow[dr,dashed,"g"]\arrow[dddrr,bend right,swap,"p_\fc"]\arrow[ddrrr,bend left,"p_{\fc'}"] & & & \\ & \fc\fc'\arrow[dr,dashed,"f"]\arrow[drr,bend left,swap,"P_{\fc'}"] \arrow[ddr,bend right,"P_\fc"] & & \\ & & \fc\times \fc'\arrow[d,"p_\fc"]\arrow[r,swap,"p_{\fc'}"]& \fc'.\\
	& & \fc & 
\end{tikzcd} \end{equation} The unique map $g:\fc\times \fc'\dashrightarrow \fc\fc'$ similarly arises from  the universal property for $\fc\fc'$ and satisfies $$\left\{\begin{array}{ll}P_\fc\circ f &= p_\fc\\ P_{\fc'} \circ f & = p_{\fc'}.\end{array}\right.$$ Composing, we have a map $f\circ g: \fc\times \fc'\rightarrow\fc\times \fc'$ with $$\left\{\begin{array}{ll}p_\fc\circ f \circ g  &= p_\fc\\ p_{\fc'} \circ f \circ g & = p_{\fc'}.\end{array}\right.$$ The map $f\circ g$ is unique, also by the universal property.  But $id_{\fc\times \fc'}:\fc\times \fc'\rightarrow \fc\times \fc'$ also satisfies $$\left\{\begin{array}{ll}p_\fc\circ  id_{\fc\times \fc'} &= p_\fc\\ p_{\fc'} \circ id_{\fc\times \fc'}  & = p_{\fc'},\end{array}\right.$$ so $f\circ g = id_{\fc \times \fc'}.$ 
An identical argument, swapping $\fc\fc'$ and $\fc\times \fc'$ in diagram \eqref{diagram:product}, shows that $g\circ f = id_{\fc\fc'}$, and hence that $\fc\fc' \cong \fc\times \fc'$ (\cref{def:isomorphism}).  Moreover, since the morphisms $f$ and $g$ making diagram $\eqref{diagram:product}$ commute are unique, $\fc\fc'$ and $\fc\times \fc'$ are isomorphic up to unique isomorphism. 	\end{proof}
\begin{notation}\label{notation:projForProduct}
  For product $\fc\times \fc'\in \fC$, we will generally reserve $p_\fc$  as the notation for the canonical projections $p_\fc:\fc\times \fc'\rightarrow \fc$ to the $\fc$th component.  
\end{notation}

Another easy fact: 
\begin{prop}\label{prop:productIsosIsIso}
	Suppose $\fc\xrightarrow{f}\fc'$ and $\fd\xrightarrow{g}\fd'$ are isomorphisms in category $\fC$. Then the product morphism $\fc\times \fd\xrightarrow{f\times g} \fc'\times \fd'$ is an isomorphism.  Moreover, the inverse is given by $(f\times g)^{-1}= f^{-1}\times g^{-1}$.
\end{prop}
\begin{proof}
The statement follows immediately from the universal property of product (\cref{def:products}).  The following diagram commutes. 
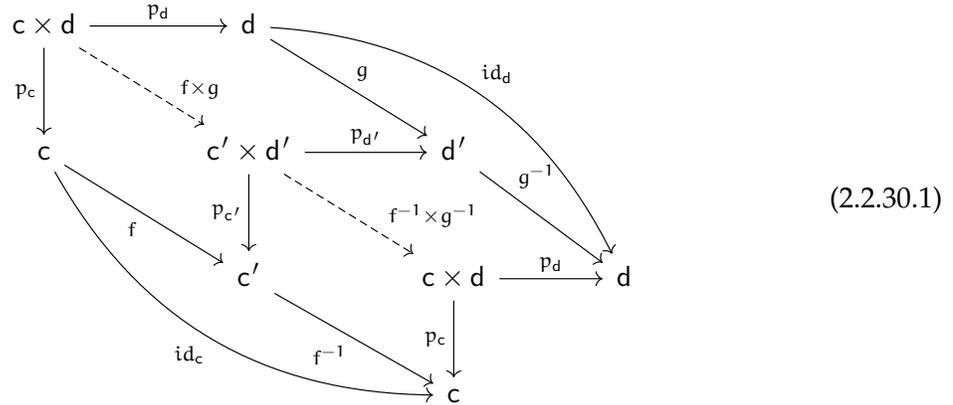
\begin{equation}\begin{tikzcd}[column sep = large, row sep = large]\label{diagram:productIsomorphism}
	\fc\times\fd\arrow[r,"p_\fd"]\arrow[d,swap,"p_\fc"] \arrow[dr,dashed,near end,"f\times g"] & \fd\arrow[dr,"g"]\arrow[ddrr,bend left,"id_\fd"] & & \\
	\fc \arrow[dr,swap,"f"]\arrow[ddrr,bend right,swap,"id_\fc"] & \fc'\times \fd' \arrow[r,"p_{\fd'}"]\arrow[d,swap,"p_{\fc'}"] \arrow[dr,dashed,near end,"f^{-1}\times g^{-1}"] & \fd'\arrow[dr,near start,"g^{-1}"] & \\
	& \fc'\arrow[dr,swap,"f^{-1}"] & \fc\times \fd\arrow[r,"p_\fd"] \arrow[d,swap,,"p_\fc"] & \fd\\
	& &  \fc& 
\end{tikzcd}\end{equation} 
Thus the composition of maps $(f^{-1}\times g^{-1})\circ (f\times g) = id_{\fc\times \fd}$.  Swapping $\fc\times \fd$ and $\fc'\times \fd'$ in diagram \eqref{diagram:productIsomorphism} shows that $(f\times g) \times (f^{-1}\times g^{-1}) = id_{\fc'\times \fd'}$.  Together, these equalities prove   that $\fc\times \fd\cong \fc'\times \fd'$ and that $f^{-1}\times g^{-1} = (f\times g)^{-1}$.\end{proof}

That $(f\times g)^{-1} = f^{-1}\times g^{-1}$ is a consequence of a more general fact. 

\begin{lemma}\label{prop:productBifunctor}
	Let $\fC$ be a category with products (\cref{def:categoryWithProducts}).   Then the product $\times:\fC\times \fC\rightarrow\fC$ sending $(\fc,\fc')\mapsto \fc\times \fc'$ is functor. 
\end{lemma}
 \begin{proof}[Proof Sketch.]
 We display the relevant commuting diagrams depicting the universal property. For identity, the commuting diagram $$
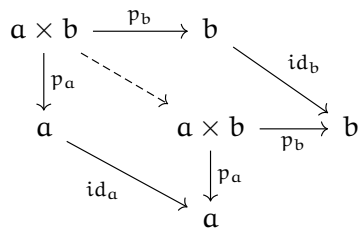
\begin{tikzcd}
	a\times b\arrow[r,"p_b"]\arrow[d,"p_a"]\arrow[dr,dashed] & b\arrow[dr,"id_b"] & \\
	a\arrow[dr,swap,"id_a"] & a\times b\arrow[r,swap,"p_b"]\arrow[d,"p_a"] &b \\
	 & a & 
\end{tikzcd}$$ shows that $id_a\times id_b = id_{a\times b}$.  

Similarly, the commuting diagram $$
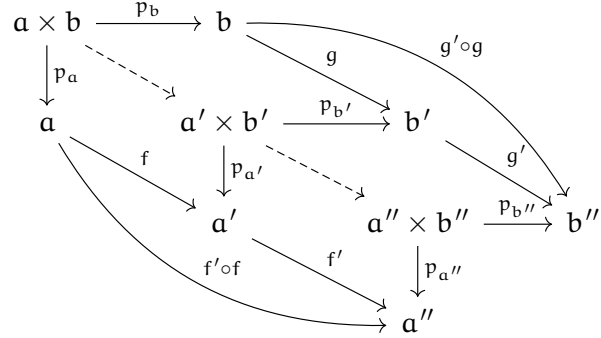
\begin{tikzcd}
	a\times b\arrow[r,"p_b"]\arrow[d,"p_a"]\arrow[dr,dashed] & b\arrow[dr,"g"] \arrow[ddrr,bend left,"g'\circ g"] & & \\
	a\arrow[dr,"f"] \arrow[ddrr,bend right,"f'\circ f"] & a'\times b'\arrow[r,"p_{b'}"] \arrow[d,"p_{a'}"] \arrow[dr,dashed] & b'\arrow[dr,"g'"] & \\
	& a'\arrow[dr,"f'"] & a''\times b''\arrow[d,"p_{a''}"] \arrow[r,"p_{b''}"] & b''\\
	& & a'' & 
\end{tikzcd}$$
shows that $(f'\times g')\circ (f\times g) = (f'\circ f)\times (g'\circ g)$. \end{proof}
 
\begin{remark}
	We  observe from a different angle that $\fc\times \fd\cong \fc'\times \fd'$ when $\fc\cong \fc'$ and $\fd\cong \fd'$, this time by  \cref{fact:functorsPreserveIsos} and \cref{prop:productBifunctor}. The fact that the proofs of \cref{prop:productIsosIsIso} and \cref{prop:productBifunctor} are formally identical is no accident. 
\end{remark}

Reversing all the arrows in \cref{def:products} gives us the \textit{coproduct}: 
\begin{definition}\label{def:coproducts}
 Let $\fc,\, \fc'\in \fC$ be objects in a category.  A \textit{coproduct} $\fc\sqcup \fc'$ is an object in $\fC$  equipped with two maps $\fc\sqcup \fc'\xleftarrow{i_\fc}\fc$ and $\fc\sqcup \fc'\xleftarrow{i_{\fc'}} \fc'$ satisfying the following universal property: for any object $\fz\in \fC$ and pair of maps $\fz\xleftarrow{f_\fc}\fc$, $\fz\xleftarrow{f_{\fc'}} \fc'$, there is a unique map $\fz\dashleftarrow\fc\sqcup \fc':f$ through which $f_\fc$ and $f_{\fc'}$ factor: $$
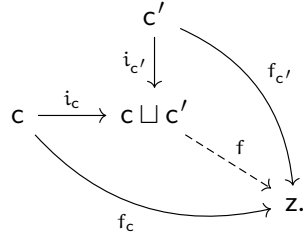
\begin{tikzcd}
	& \fc'\arrow[ddr,bend left,"f_{\fc'}"] \arrow[d,swap,"i_{\fc'}"] & \\
	\fc\arrow[drr,bend right, swap,"f_\fc"] \arrow[r,"i_\fc"] & \fc\sqcup \fc' \arrow[dr,dashed,"f"] & \\
	& & \fz. 
\end{tikzcd}$$
\end{definition}

Now we make a more technical observation, namely that coproducts and products ``interact.''  Interaction in one direction always exists and is canonical  (e.g.\ \cref{prop:canonicalMapCoproduct2Product}).  Interaction in the other direction, if it exists, generally uses category-specific information (e.g.\ \cref{prop:coproductProductCommuteInSet}). 
	\begin{prop}\label{prop:canonicalMapCoproduct2Product}
		Let $\fC$ be a category with finite products and coproducts, $J$, $\{K_j\}_{j\in J}$ be  finite sets, and $\left\{\fc_{k}^{j}\right\}_{k\in K_j, j\in J}$ an indexed collection of objects in $\fC$. Then there is a canonical map \begin{equation}\label{eq:canonicalMapCoproduct2Product} \Omega:\discats_{(k_{j}) \in \bigsqcap_{j\in J} K_{j}}\discatp_{j\in J} \fc_{k_{j}^j} \dashrightarrow \discatp_{j\in J}\discats_{k\in K_j} \fc_{k}^j.\end{equation} 
	\end{prop}
	Compare with a more general statement for limits and colimits in \cite[Lemma 3.8.3]{riehlb}.

Because notation is unwieldy, we first provide the argument in a concrete case. The general version is very similar. Let $J = \{a,b\}$, $K_a = \{1,2,3\}$, $K_b = \{1,2\}$, and $\{\fc_1^a,\fc_2^a,\fc_3^a,\fc_1^b,\fc_2^b\}$ be a collection of objects in $\fC$.  Then \cref{prop:canonicalMapCoproduct2Product} says that there is canonical map \small $$\Omega:(\fc_1^a\times \fc_1^b)\sqcup (\fc_1^a\times \fc_2^b)\sqcup (\fc_2^a\times \fc_1^b)\sqcup(\fc_2^a\times \fc_2^b)\sqcup (\fc_3^a\times \fc_1^b)\sqcup(\fc_3^a\times \fc_2^b)\rightarrow\big(\fc_1^a\sqcup\fc_2^a\sqcup \fc_3^a\big)\times \big(\fc_1^b\sqcup\fc_2^b\big).$$
\normalsize 
We see how this map is constructed. Because the right-hand side is a product, the map $\Omega$ is defined uniquely by a pair of maps $\Omega^a$ and $\Omega^b$ from the left-hand side \textit{to} each component $(\fc_1^a\sqcup\fc_2^a\sqcup\fc_3^a)$ and $(\fc_1^b\sqcup\fc_2^b)$, respectively. Start with $\Omega^a$.  And because the left-hand side is a coproduct, $\Omega^a$ is defined uniquely by collection of maps $$\big\{\Omega^a_{(k_a,k_b)}:  \fc_{k_a}^a\times \fc_{k_b}^b\rightarrow (\fc_1^a\sqcup\fc_2^a\sqcup\fc_3^a)\big\}_{(k_a,k_b)\in \{1,2,3\}\times\{1,2\}}$$ from each component $\fc_{k_a}^a\times \fc_{k_b}^b$. We thus define $\Omega_{(k_a,k_b)}^a\defeq i_{k_a}\circ p^a$, where $p^a:\fc_{k_a}^a\times \fc_{k_b}^b\rightarrow \fc_{k_a}^a$ is the canonical projection of product, and $i_{k_a}:\fc_{k_a}^a\hookrightarrow \fc_1^a\sqcup\fc_2^a\sqcup\fc_3^a$ is the canonical injection of coproduct.  This collection of maps $\big\{\Omega^a_{(k_a,k_b)}\big\}$  induces unique map $\Omega^a$, and $\Omega^b$ is defined similarly.  

Before reproducing this argument  for arbitrary finite sets $J$, $\{K_j\}_{j\in J}$, consider that we could have first defined collection of maps $$\left\{\omega_{(k_a,k_b)}:\fc_{k_a}^a\times\fc_{k_b}^b\rightarrow \big(\fc_1^a\sqcup\fc_2^a\sqcup\fc_3^a\big)\times \big(\fc_1^b\sqcup\fc_2^b\big)\right\}_{(k_a,k_b)\in \{1,2,3\}\times\{1,2\}}$$ (using the universal property of coproduct) and \textit{then}  maps $\omega_{(k_a,k_b)}^a$, $\omega_{(k_a,k_b)}^b$ to each factor $(\fc_1^a\sqcup\fc_2^a\sqcup\fc_3^a)$ and $\fc_1^b\sqcup\fc_2^b$, respectively.  As we define these maps identically---i.e.\ $\omega_{(k_a,k_b)}^a \defeq i_{k_a}\circ p^a  = \Omega_{(k_a,k_b)}^a$---it turns out that the induced map from the coproduct of products to product of coproducts is identical as well.  
We now give the general construction in both ways and show that they are equal. 

	\begin{proof}[Proof of \cref{prop:canonicalMapCoproduct2Product}.]
	
	We start with notation. Let $(k_j)_{j\in J}\in \discatp_{j\in J} K_j$ denote a $J$-tuple of indices, where $k_j\in K_j$.  We will write $(k_j)$ for the tuple and $k_j$ for an element (index) of $K_j$. 
	
	A map \begin{equation}\label{eq:constructOmega} \Omega:\discats_{(k_j)\in \bigsqcap K_j}\discatp_{j\in J}\fc_{k_j}^j\longrightarrow \discatp_{j\in J}\discats_{k\in K_j} \fc_k^j\end{equation} is uniquely defined (c.f.\ \cref{def:products}) by collection of maps \begin{equation}\label{eq:constructOmega2} \left\{\Omega^{j'}:\discats_{(k_j)\in \bigsqcap K_j}\discatp_{j\in J}\fc_{k_j}^j\longrightarrow \discats_{k\in K_{j'}}\fc_k^{j'}\right\}_{j'\in J},\end{equation} since the the right-hand side of \eqref{eq:constructOmega} is a product. Similarly, for each $j'\in J$, the map $\Omega^{j'}$ is uniquely defined (\cref{def:coproducts}) by collection of maps \begin{equation}\label{eq:constructOmega3}
		\left\{\Omega^{j'}_{(k'_j)}:\discatp_{j\in J}\fc_{k'_j}^j\longrightarrow\discats_{k\in K_{j'}}\fc_k^{j'}\right\}_{(k'_j)\in \bigsqcap K_j}.\footnote{The prime in $k_j'$ modifies  $k$, not $j$.}
	\end{equation} 
	We set \begin{equation}\label{eq:constructOmega4}
	\Omega_{(k'_j)}^{j'}\defeq i^{j'}_{k'_{j'}}\circ p^{j'}_{(k'_j)},
	\end{equation} where $p^{j'}_{(k'_j)}:\discatp_{j\in J}\fc_{k'_j}^j\twoheadrightarrow \fc_{k'_{j'}}^{j'}$ is the canonical projection of product, and $i^{j'}_{k'_{j'}}:\fc_{k'_{j'}}^{j'}\hookrightarrow\discats_{k\in K_{j'}}\fc_k^{j'}$ is the canonical inclusion of coproduct.\footnote{Notice that the \textit{super}script of $p$ and the \textit{sub}script of $i$ indicate the canonical map; the subscript of $p$ and superscript of $i$ only play the role of  distinguishing the domains and codomains.}   Having thus defined collection \begin{equation}\label{eq:constructOmega5}
	\left\{\Omega^{j'}_{(k'_j)}:\discatp_{j\in J}\fc_{k'_j}^j\longrightarrow\discats_{k\in K_{j'}}\fc_k^{j'}\right\}_{j'\in J, \,(k'_j)\in \bigsqcap K_j},
	\end{equation}
	we conclude there is a unique (canonical) map $\Omega$ in  \eqref{eq:constructOmega}.  This map $\Omega$ satisfies $p^{j'}\circ \Omega = \Omega^{j'}$, where $p^{j'}:\discatp_{j\in J}\discats_{k\in K_j}\fc_k^j\rightarrow\discats_{k\in K_{j'}} \fc_k^{j'}$ is the canonical projection, and $\Omega^{j'}$ satisfies $\Omega_{(k'_j)}^{j'} = \Omega^{j'} \circ i_{(k'_j)}$, where $i_{(k'_j)}:\discatp_{j\in J}\fc_{k'_j}^j\hookrightarrow\discats_{(k_j)\in \bigsqcap K_j}\discatp_{j\in J} \fc_{k_j}^j$ is the canonical inclusion of coproduct. 
	
	Now, we begin with the universal property of coproduct, and then use the universal property of product.  A  map \begin{equation}\label{eq:constructVarOmega0} \omega:\discats_{(k_j)\in \bigsqcap K_j}\discatp_{j\in J}\fc_{k_j}^j\longrightarrow \discatp_{j\in J}\discats_{k\in K_j}\fc_k^j
 \end{equation}
 is uniquely defined (\cref{def:coproducts}) by collection of maps \begin{equation}\label{eq:constructVarOmega1}\left\{\omega_{(k'_j)}:\discatp_{j\in J} \fc_{k'_j}^j\longrightarrow\discatp_{j\in J}\discats_{k\in K_j}\fc_k^j\right\}_{(k'_j)\in \bigsqcap K_j}.\end{equation}
	Similarly, for each $(k'_j)\in \discatp_{j\in J}K_j$, the map $\omega_{(k'_j)}$ is uniquely defined (\cref{def:products}) by collection of maps \begin{equation}\label{eq:constructVarOmega2}
 	\left\{\omega_{(k'_j)}^{j'}:\discatp_{j\in J}\fc_{k'_j}^j\longrightarrow\discats_{k\in K_{j'}}\fc_k^{j'}\right\}_{j'\in J},
 \end{equation}
which we also (c.f.\ \eqref{eq:constructOmega4}) define by \begin{equation}\label{eq:constructVarOmega3}
	\omega_{(k'_j)}^{j'}\defeq  i^{j'}_{k'_{j'}}\circ p^{j'}_{(k'_j)} = \Omega_{(k'_j)}^{j'}.
\end{equation}
Again, the collection 
\begin{equation}\label{eq:constructVarOmega4}
	 	\left\{\omega_{(k'_j)}^{j'}:\discatp_{j\in J}\fc_{k'_j}^j\longrightarrow\discats_{k\in K_{j'}}\fc_k^{j'}\right\}_{j'\in J,\, (k'_j)\in \bigsqcap K_j}
\end{equation}
	uniquely defines the map $\omega$ in \eqref{eq:constructVarOmega0}, and we have equalities $\omega\circ i_{(k'_j)} = \omega_{(k'_j)}$ and $p^{j'}\circ \omega_{(k'_j)} = \omega_{(k'_j)}^{j'}$. 
	
	Now we verify that $\omega = \Omega$. We use the following diagram as reference: \begin{equation}
		\begin{tikzcd}[column sep = huge, row sep = large]
			\discats_{(k_j)\in \bigsqcap K_j}\discatp_{j\in J} \fc_{k_j}^j \arrow[d,shift right, swap,"\omega"]\arrow[d,shift left, "\Omega"] & \discatp_{j\in J}\fc_{k'_j}^j\arrow[r,"p_{(k'_{j})}^{j'}"]\arrow[l,hook ,swap,"i_{(k'_j)}"] \arrow[d, shift right, swap, "\omega_{(k'_j)}^{j'}"]\arrow[d,shift left, "\Omega_{(k'_j)}^{j'}"] & \fc_{k'_{j'}}^{j'}\arrow[dl,hook,"i_{k'_{j'}}^{j'}"] \\
			\discatp_{j\in J}\discats_{k\in K_j}\fc_k^j \arrow[r,"p^{j'}"] & \discats_{k\in K_{j'}}\fc_k^{j'}. & 
		\end{tikzcd}
	\end{equation} Since $\sF{cod}(\omega) = \sF{cod}(\Omega)$ are both products, it suffices to show that $p^{j'}\circ \omega = p^{j'}\circ \Omega$ for every $j\in J$, where $p^{j'}:\discatp_{j\in J}\discats_{k\in K_j}\fc_k^j\rightarrow\discats_{k\in K_{j'}}\fc_k^{j'}$ is the canonical projection of product. And since $\sF{dom}(p^{j'}\circ \omega)=\sF{dom}(p^{j'}\circ\Omega)$ are both coproducts, it suffices to show that $(p^{j'}\circ\omega)\circ i_{(k'_j)} = (p^{j'}\circ\Omega)\circ i_{(k'_j)}$ where $i_{(k'_j)}:\discatp_{j\in J}\fc_{k'_j}^j\hookrightarrow\discats_{(k_j)\in \bigsqcap K_j}\discatp_{j\in J}\fc_{k_j}^j$ is the canonical inclusion of the coproduct.  But by construction  of $\omega$ and $\Omega$ (c.f.\ \eqref{eq:constructOmega4} and \eqref{eq:constructVarOmega3}), $$p^{j'}\circ\omega\circ i_{(k'_j)}  = \omega_{(k'_j)}^{j'} = \Omega_{(k'_j)}^{j'} = p^{j'}\circ\Omega\circ i_{(k'_j)},$$  for all $j'\in J$ and $(k'_j)_{j\in J}\in \discatp_{j\in J} K_j$, proving that $\omega = \Omega$. 
	\end{proof}

	\begin{prop}\label{prop:coproductProductCommuteInSet}
	If $\fC$ is $\sF{Set}$, then the map $\Omega$ in \eqref{eq:canonicalMapCoproduct2Product} is a bijection. 
\end{prop}

First a lemma, which says essentially that each element in a coproduct of indexed sets in fact belongs to a particular set in the collection.

\begin{lemma}\label{lemma:sourceSet4Coproduct}
	Let $\{\fc_k\}_{k\in K}$ be an indexed collection of sets.  Then there is a well defined map $$\fs:\discats_{k\in K}\fc_k\xrightarrow{} K$$ satisfying $\fs\circ i_{k'} \equiv k'$, where $i_{k'}:\fc_{k'}\hookrightarrow\discats_{k\in K} \fc_k$ is the canonical inclusion of coproduct. 
\end{lemma}

\begin{proof} 
	An explicit construction of coproduct in $\sF{Set}$ gives a unique isomorphism $$\discats_{k\in K}\fc_k \cong \disu_{k\in K} \fc_k\times \{k\}. $$ Call this map $\eta:\discats_{k\in K }\fc_k\xrightarrow{\sim} \disu_{k\in K} \fc_k\times \{k\}$.  There  is a map $\tilde{p}_2: \disu_{k\in K} \fc_k\times\{k\} \rightarrow \disu_{k\in K} \{k\} = K$ sending $(x,k)\mapsto k$, formally defined by composition $p_2\circ \i$, where $p_2:\left(\disu_{k\in K}\fc_k\right)\times \left(\disu_{k\in K}\{k\}\right)\rightarrow K$ is the canonical projection onto the second factor and $\i:\disu_{k\in K}\fc_k\times\{k\}\hookrightarrow \left(\disu_{k\in K}\fc_k\right)\times \left(\disu_{k\in K}\{k\}\right)$ is  inclusion.  There are canonical  inclusions $i_{k'}:\fc_{k'}\hookrightarrow\discats_{k\in K}\fc_k$ and $\tilde{i}_{k'}:\fc_{k'}\hookrightarrow\disu_{k\in K} \fc_k\times\{k\}$ such  that $\eta\circ i_{k'} = \tilde{i}_{k'}$ for each $k'\in K$.  
	
	We  define $\fs\defeq \tilde{p}_2 \circ \eta$ 	and immediately conclude   that $\fs \circ i_{k'} = (\tilde{p}_2\circ \eta)\circ i_{k'} =  \tilde{p}_2\circ \tilde{i}_{k'} \equiv k'$. \end{proof}
	
	We draw a few  observations from \cref{lemma:sourceSet4Coproduct}. The map \begin{equation}\label{eq:defP1ForSourceMap} \tilde{p}_1:\disu_{k\in K}\fc_k\times\{k\}\rightarrow\disu_{k\in K} \fc_k \end{equation} defined by sending  $\big((x,k)\in \fc_k\times\{k\}\big)\mapsto \big(x\in \fc_k\big)$ is a left inverse of $\tilde{i}_{k'}$, i.e.\ $\tilde{p}_1\circ \tilde{i}_{k'}=id_{\fc_{k'}}$.  Thus, for $t\in \discats_{k\in K}\fc_k$, we  have that \begin{equation} \tilde{p}_1\circ \eta(t) \in \fc_{\fs(t)},\end{equation} which implies further that \begin{equation}\label{eq:XYZrst} i_{\fs(t)}^{-1}(t) \in \fc_{\fs(t)}.\end{equation}

 \begin{proof}[Proof of \cref{prop:coproductProductCommuteInSet}] 
 We define a map $\aleph:\discatp_{j\in J}\discats_{k\in K_j}\fc_k^j\rightarrow \discats_{(k_{j})\in \bigsqcap K_j}\discatp_{j\in J}\fc_{k_{j}}^j$ as follows.  Let $(x_j)_{j\in J}\in \discatp_{j\in J} \discats_{k\in K_j}\fc_k^j$ be an arbitrary element. Then $x_{j'}=p^{j'}\big((x_j)_{j\in J}\big) \in \discats_{k\in K_{j'}}\fc_{k}^{j'}$ where 
 $p^{j'}:\discatp_{j\in J}\discats_{k\in K_j} \fc_{k}^{j}\rightarrow \discats_{k\in K_{j'}}\fc_{k}^{j'} $ is the canonical projection.

  Set $s_{j'}\defeq \fs(x_{j'})$ (\cref{lemma:sourceSet4Coproduct}).  Then $x_{j'} = i_{s_{j'}}^{-1}(x_{j'}) \in \fc_{s_{j'}}^{j'}$ (c.f.\ \eqref{eq:XYZrst}),  where $i_{s_{j'}}:\fc_{s_{j'}}^{j'}\hookrightarrow \discats_{k\in K_{j'}}\fc_{k}^{j'}$ is the canonical injection.\footnote{We implicitly identify element $x\in \fc_{\fs(x)}$ with its image $i_{\fs(x)}(x)\in \discats_{k\in K}\fc_k$.}  Thus we have    $(x_j)_{j\in J} \in \discatp_{j\in J}\fc_{s_j}^j$, and we set $\aleph\big((x_j)_{j\in J}\big)\defeq i_{(s_j)}\big((x_j)_{j\in J}\big)$, where    $i_{(s_{j})}: \discatp_{j\in J}\fc_{s_{j}}^{j}\hookrightarrow \discats_{(k_j)\in \bigsqcap K_j} \discatp_{j\in J} \fc_{k_j}^j$ is the canonical inclusion.

  Having defined the map $\aleph$, the verification that $\aleph$ and $\Omega$ are mutually inverse follows formally by construction (and we omit the details).  We conclude that $\Omega$ is a bijection. 
 \end{proof}
 \begin{remark}
 	We observe that constructing the map in the other direction was not canonical, and did not use universal properties, but did use ``structural'' properties of the ambient category, $\sF{Set}$.  Specifically, we needed to take elements to define the map $\aleph$ (compare with remark at the end of \cite[\S 3.8]{riehlb}). A more categorical proof of the same result uses the Yoneda embedding (\cref{prop:YonedaEmbedding}) and the structural fact that  $\sF{Set}$  is \textit{cartesian closed} (c.f. \cite[Proposition 8.6]{awodey}).  \end{remark}

 \begin{example}\label{ex:vectorSpacesProduct&Coproduct}
 Let $\{V_j\}_{j\in J}$ be a finite collection of finite dimension vector spaces. The category $\sF{Vect}$ of vector spaces has both coproducts $\disbos_{j\in J}V_j$  and products $\disp_{j\in J}V_j$. As sets, there is bijection $\Omega:\disbos_{j\in J}V_j \xrightarrow{\sim} \disp_{j\in J} V_j$ (\cref{prop:coproductProductCommuteInSet}).  In fact, $\Omega$ is linear since both projection $\disp_{j\in J}V_j\xrightarrow{p_{j'}}V_{j'}$ (mapping $(x_j)_{j\in J}\mapsto x_{j'}$) and inclusion $V_{j'}\xrightarrow{i_{j'}}\disbos_{j\in J}V_j$ (mapping $x_{j'}\mapsto \bigoplus_{j\in J}x_{j'}\d_{j,j'}$) are linear, where $$\left(\d_{j,k}\defeq\left\{\begin{array}{ll} 1 & \mbox{if}\; j = k \\ 0 & \mbox{else} \end{array}\right.\right)\in V_j.$$
 \end{example}

\subsubsection{Arrow Category}
We have seen that categories have objects and morphisms. Morphisms may themselves be considered objects, in another category.

\begin{definition}\label{def:arrowCategory}
	Let $\fC$ be a category. The \textit{arrow category} $\sF{Arrow(C)}$ has \begin{enumerate}
		\item objects $\fc\xrightarrow{f}\fc',$ which are  morphisms in $\fC$, and 
		\item morphisms $(\fc\xrightarrow{f}\fc')\xrightarrow{(\a,\a')}(\fd\xrightarrow{g}\fd')$  are pairs of morphisms $\fc\xrightarrow{\a}\fd,$ $\fc'\xrightarrow{\a'}\fd'$ in $\fC$ such that $\a'\circ f = g\circ \a$. 
	\end{enumerate}
	In other words, morphisms in $\sF{Arrow(C)}$ are commuting squares  \begin{equation}\label{eq:morphismInArrowCat} \begin{tikzcd}
	\fc\arrow[d,"f"] \arrow[r,"\a"] & \fd\arrow[d,"g"]\\ \fc'\arrow[r,"\a'"] & \fd'
\end{tikzcd}\end{equation} in $\fC$.  \end{definition}

\begin{remark}\label{remark:alternativeDefOfArrowCat}
We give an alternative equivalent definition of the arrow category which is useful in other contexts (e.g.\ the notion of a hybrid phase space in \cref{def:hybridPhaseSpace0}).

First we define a category $\two$ with objects $\two_0\defeq \{0,1\}$  and a unique morphism $0\xrightarrow{e_{1,0}}1$ between $0$ and $1$ (\cite[\S 5.1]{riehlb}) . We then define \begin{equation}\label{eq:defArrowCatSecondDef} \sF{Arrow(C)}\defeq \fC^{{\two}^{op}}\end{equation} as the functor category from $\two^{op}$ to $\fC$ (\cref{def:functorCategory}). An object $\fa\in \fC^{\two^{op}}$ in this category is a functor $\fa:\two^{op}\rightarrow\fC$---realized as morphism $\fa(0)\xleftarrow{\fa(e_{1,0})} \fa(1)$ in $\fC$---and a morphism is a natural transformation, encoded in the commutative diagram \eqref{eq:morphismInArrowCat}.  These definitions are readily checked to be equivalent.\footnote{Defining this functor category in terms of the opposite category $\two^{op}$ is  intended to signal that $\sF{Arrow(C)}$ is a presheaf. As we make no further use of this observation, we would have been just as well to define $\sF{Arrow(C)}$ as the functor category $\fC^\two$.} When we wish to emphasize that $\sF{Arrow(C)}$ is the \textit{arrow} category of $\fC$, we may denote objects as $\fc\xrightarrow{f}\fc'$. Other times, our focus on  $\fA\defeq \sF{Arrow(C)}$ is as a category in its own right, and objects of $\fA$ may be denoted simply as $\fa\defeq\big(\sF{dom}(\fa)\xrightarrow{\fa}\sF{cod}(\fa)\big)$.
\end{remark}

We note a few simple facts about the arrow category. 
\begin{lemma}\label{fact:terminalArrowObject}
 	A terminal object $\fc_t\in \fC$ (\cref{def:terminalObject}) defines a terminal object $\fc_t\xrightarrow{id_{\fc_t}}\fc_t$ in $\sF{Arrow(C)}$. 
\end{lemma}
\begin{proof}
	Let $\fc\xrightarrow{f}\fc'$ be a morphism in $\fC$.  Terminality of $\fc_t$ implies that there are morphisms  $\fc\xrightarrow{f_\fc}\fc_t$,  $\fc'\xrightarrow{f_{\fc'}}\fc_t$, and that they are unique. But  $f_{\fc'}\circ f$ is also a morphism from $\fc$  to $\fc_t$, and therefore is equal to $f_\fc= id_{\fc_t} \circ f_\fc$.  Thus these morphisms assemble to 
	 define a  unique morphism $(f_\fc,f_{\fc'}):f\rightarrow id_{\fc_t}$ in $\sF{Arrow(C)}$. 
\end{proof}

\begin{prop}\label{prop:catHasProductsImpliesArrowsDoesToo}
	Let $\fA\defeq \sF{Arrow(C)}$ be the arrow category of $\fC$, a category with products.  Then $\fA$ has products as well. 
\end{prop}
\begin{proof}
	Let $\fc\xrightarrow{f}\fc'$ and $\fd\xrightarrow{g}\fd'$ be morphisms in $\fC$.  The product $f\times g\in \fA$ is defined by functoriality of $\times$ (\cref{prop:productBifunctor}). We recall from \eqref{diagram:productIsomorphism} the commuting diagram \begin{equation}
		\begin{tikzcd}
		\fc\times \fd\arrow[r,"p_\fd"]\arrow[d,"p_\fc"]\arrow[dr,dashed,"f\times g"] & \fd\arrow[dr,"g"]& \\
		\fc\arrow[dr,"f"] & \fc'\times \fd'\arrow[r,"p_{\fd'}"]\arrow[d,"p_{\fc'}"] & \fd'\\
		& \fc'. & 
		\end{tikzcd}
	\end{equation} The object $f\times g$ is terminal in $\fA$ with respect to maps to $f$ and $g$.  Indeed, a map $\a:h\rightarrow f$ consists of a pair of maps $\sF{dom}(h)\xrightarrow{\a_0} \fc$ and $\sF{cod}(h)\xrightarrow{\a_1}\fc'$  in $\fC$, and similarly a map $\b:h\rightarrow  g$ is pair $\sF{dom}(h)\xrightarrow{\b_0} \fd,$ and $\sF{cod}(h)\xrightarrow{\b_1}\fd'$.  The universal property of product  in $\fC$ implies there are  unique maps $\sF{dom}(h)\rightarrow \fc\times \fd$ and $\sF{cod}(h)\rightarrow \fc'\times \fd'$, and they factor through  projection.  It follows easily that  $h\rightarrow f\times g$ is a map in $\fA$.
\end{proof}
\begin{remark}
	\label{def:isoArrowCategory}	\label{remark:isoArrowCategory}
We interpret the notion of isomorphism in a category (\cref{def:isomorphism}) for the arrow category. 
	 Two objects $f,f'\in \sF{Arrow(C)}_0$ are isomorphic if there are isomorphisms $\a_0:\sF{dom}(f)\rightarrow\sF{dom}(f')$ and $\a_1:\sF{cod}(f)\rightarrow \sF{cod}(f')$ for which the diagram $$\begin{tikzcd}
	\sF{dom}(f)\arrow[d,"f"]\arrow[r,"\a_0"] & \sF{dom}(f')\arrow[d,"f'"] \\ 
	\sF{cod}(f)\arrow[r,"\a_1"] & \sF{cod}(f')
\end{tikzcd}$$ commutes. This  is  an isomorphism in $\sF{Arrow(C)}$ because the  inverse diagram $$\begin{tikzcd}
	\sF{dom}(f')\arrow[r,"\a_0^{-1}"]\arrow[d,"f'"]& \sF{dom}(f)\arrow[d,"f"] \\
	\sF{cod}(f')\arrow[r,"a_1^{-1}"] & \sF{cod}(f) 
\end{tikzcd}$$ also commutes. Indeed, $f'\circ \a_0 = \a_1\circ f$ implies that  $ f'\circ \a_0\circ \a_0^{-1} = \a_1\circ f \circ \a_0^{-1}$ which implies that  $\a_1^{-1} \circ f' = \a_1^{-1}\circ \a_1\circ f \circ \a_0^{-1}$.\end{remark}

\begin{definition}\label{def:productPreservingFunctor}
	We say that a functor $\Fc:\fC\rightarrow\fD$ is \textit{product preserving} if for each pair of objects $\fc,\,\fc'\in \fC$, there is isomorphism $\Fc \fc\times \Fc \fc' \cong \Fc(\fc\times \fc')$, natural in each factor $\fc, \fc'$. 
\end{definition}
\begin{lemma}\label{lemma:arrowProductPreservingToo}
	A product preserving functor $\Fc:\fC\rightarrow\fD$  extends to a  product preserving functor $\Fc_*:\sF{Arrow(C)\rightarrow Arrow(D)}$.  
\end{lemma}
\begin{proof}
First, set notation $\fA_\fC\defeq \sF{Arrow(C)}$ and $\fA_\fD\defeq\sF{Arrow(D)}$.  On objects of $\fA_\fC$,  $\Fc_*:\fA_\fC\rightarrow\fA_\fD$ is  defined by \begin{equation}\label{eq:extendingFunctorToFunctorOnArrows} \Fc_*\big(\fc\xrightarrow{f}\fc'\big) \defeq   \Fc \fc\xrightarrow{\Fc f}\Fc\fc'.\end{equation}  That $\Fc$ is a functor guarantees this assignment is well-defined. For given a  morphism $(\a,\a'):(\fc\xrightarrow{f}\fc')\rightarrow (\fd\xrightarrow{g}\fd')$  in $\fA_\fC$,  we have that $g\circ\a = \a'\circ f$, and therefore   $\Fc(\a')\circ \Fc(f) = \Fc(\a'\circ f) = \Fc(g\circ \a) = \Fc(g)\circ \Fc(\a)$, again since $\Fc$ is a functor.  In other words, $(\Fc\a,\Fc\a'): \Fc_* f\rightarrow\Fc_* g$ is a morphism in $\fA_\fD$. Alternatively, since $\fA_\fC$ is itself a functor category $\fC^{\two^{op}}$ (c.f.\ \eqref{eq:defArrowCatSecondDef}), we automatically obtain a functor category $\fD^{\two^{op}}$ by post composition with $\Fc$: for object $\fa:\two\rightarrow\fC$ in $\fA_\fC$, we obtain  $\Fc_* \fa :\two^{op}\rightarrow\fD$  by $\Fc_* \fa \defeq  \Fc\circ \fa$. Morphisms in $\fC^{\two^{op}}$ are natural transformations, and they are sent to natural transformations in $\fD^{\two^{op}}$ because functors preserve diagrams (\cref{lemma:functorsPreserveDiagrams}).

We  now show that $\Fc_*(\fa\times \fb) \cong \Fc_*\fa\times \Fc_*\fb$ for $\fa,\fb\in (\fA_\fC)_0$.  By assumption (\cref{def:productPreservingFunctor}), there are  isomorphisms $(\g_{\fa,\fb})_{0}:\Fc(\sF{dom}(\fa)\times \sF{dom}(\fb))\xrightarrow{\sim}\Fc(\sF{dom}(\fa))\times \Fc(\sF{dom}(\fb))$ and $(\g_{\fa,\fb})_{1}:\Fc(\sF{cod}(\fa)\times \sF{cod}(\fb))\xrightarrow{\sim} \Fc(\sF{cod}(a))\times \Fc(\sF{cod}(b))$. Naturality of $\g$  \textit{means} that $$\begin{tikzcd}
	\Fc(\sF{dom}(\fa)\times \sF{dom}(\fb))\arrow[d,"\Fc(\fa\times \fb)"] \arrow[r,"(\g_{\fa{,}\fb})_{0}"] & \Fc\sF{dom}(\fa)\times \Fc\sF{dom}(\fb)\arrow[d,swap,"\Fc\fa\times \Fc\fb"] \\
	\Fc(\sF{cod}(\fa)\times \sF{cod}(\fb))\arrow[r,"(\g_{\fa{,}\fb})_{1}"] & \Fc\sF{cod}(\fa)\times \Fc\sF{cod}(\fb)
\end{tikzcd}$$ commutes.  Thus $\g_{\fa,\fb}:\Fc_*(\fa\times\fb)\rightarrow\Fc_*\fa\times \Fc_*\fb$ defines isomorphism $\Fc_*(\fa\times \fb)\cong \Fc_*\fa\times \Fc_*\fb $ in $\fA_\fD$ (\cref{remark:isoArrowCategory}).

Now we argue for naturality of the isomorphism.  Let $$\begin{array}{lll} (h_0,h_1):& (\sF{dom}(\fa)\xrightarrow{\fa}\sF{cod}(\fa)) &\rightarrow (\sF{dom}(\fc)\xrightarrow{\fc }\sF{cod}(\fc))\\(k_0,k_1): & (\sF{dom}(\fb)\xrightarrow{\fb}\sF{cod}(\fb)) & \rightarrow (\sF{dom}(\fd)\xrightarrow{\fd}\sF{cod}(\fd))\end{array}$$ be morphisms in $\fA_\fC$.  Then\small $$\big(\sF{dom}(\fa)\times\sF{dom}(\fb)\xrightarrow{\fa\times \fb }\sF{cod}(\fa)\times \sF{cod}(\fb) \big)\xrightarrow{h_0\times k_0,h_1\times k_1} \big(\sF{dom}(\fc)\times \sF{dom}(\fd)\xrightarrow{\fc\times \fd}\sF{cod}(\fc)\times\sF{cod}(\fd)\big)$$ \normalsize is a morphism in $\fA_\fC$ (\cref{prop:catHasProductsImpliesArrowsDoesToo}), and the  diagram\small $$
\begin{tikzcd}[ row sep = large]
	\Fc(\sF{dom}(\fa)\times \sF{dom}(\fb))\arrow[rr,near start,"(\g_{\fa{,}\fb})_0"]\arrow[dd,"\Fc(h_0\times k_0)"]\arrow[dr,"\Fc(\fa\times \fb)"] & & \Fc\sF{dom}(\fa)\times \Fc\sF{dom}(\fb)\arrow[dd,near start,swap,"\Fc h_0 \times \Fc k_0 "] \arrow[dr,"\Fc\fa\times \Fc\fb"] & \\
	& \Fc(\sF{cod}(\fa)\times \sF{cod}(\fb))\arrow[dd,near start,"\Fc(h_1\times k_1)"]\arrow[rr,near start,swap,"(\g_{\fa{,}\fb })_1"] & & \Fc\sF{cod}(\fa)\times F\sF{cod}(\fb)\arrow[dd,"\Fc h_1 \times \Fc k_1 "]\\
	\Fc(\sF{dom}(\fc)\times \sF{dom}(\fd))\arrow[rr,near start,"(\g_{\fc{,}\fd})_0"]\arrow[dr,"\Fc(\fc\times \fd)"] & & \Fc\sF{dom}(\fc)\times \Fc\sF{dom}(\fd)\arrow[dr,near start,"\Fc\fc\times \Fc\fd"] & \\
	& \Fc(\sF{cod}(\fc)\times \sF{cod}(\fd))\arrow[rr,"(\g_{\fc{,}\fd})_1"] & & \Fc\sF{cod}(\fc)\times \Fc\sF{cod}(\fd)\end{tikzcd}$$\normalsize is easily seen to  commute.  For example, $\Fc(\fc\times \fd) \circ \Fc(h_0\times k_0) = \Fc(h_1\times k_1)\circ \Fc(\fa\times \fb)$ and $\big(\Fc\fc\times \Fc\fd\big)\circ \big(\Fc h_0\times \Fc k_0\big) = \big(\Fc h_1\times \Fc k_1\big)\circ\big( \Fc\fa\times \Fc\fb\big) $ since $\Fc$ is a functor (\cref{lemma:functorsPreserveDiagrams}). The other four faces commute by naturality of $\g$. In short, the diagram $$\begin{tikzcd}[column sep = large, row sep = large]
	\Fc_*(\fa\times \fb)\arrow[r,"(\g_{\fa{,}\fb})"]\arrow[d,swap,"\big(\Fc(h_0\times k_0){,}\Fc(h_1\times k_1)\big)"] & \Fc_*\fa\times \Fc_*\fb\arrow[d,"\big(\Fc(h_0)\times \Fc(k_0){,}\Fc(h_1)\times \Fc(k_1)\big)"] \\
	\Fc_*(\fc\times \fd)\arrow[r,"(\g_{\fc{,}\fd})"]& \Fc_*\fc\times \Fc_*\fd 
\end{tikzcd}$$  commutes, proving naturality in $\fA_\fD$. 
 \end{proof}

\subsubsection{Yoneda Lemma and the Category of Elements}

\begin{theorem}[Yoneda lemma]\label{theorem:Yoneda}\label{lemma:YonedaLemma}
	Let $\fC$ be locally small (\cref{def:locallySmall}) and $\Fc:\fC\rightarrow\sF{Set}$  a covariant functor (\cref{def:functor}).   For object $\fc\in \fC_0$, there is bijection $\big\{\a:\fC(\fc,\cdot)\Rightarrow \Fc\big\} \cong \Fc\fc$ between the set of natural transformations from represented functor (\cref{def:functorRepresentable}) $\fC(\fc,\cdot)$ to $\Fc$ and the set $\Fc\fc$. This bijection is natural in both $\fc$ and $\Fc$. 
\end{theorem}
The isomorphism sends a natural transformation $\big(\a:\fC(\fc,\cdot)\Rightarrow \Fc\fc\big)$ to $\a_\fc(id_\fc)\in \Fc\fc$. 
\begin{proof}
	See \cite[Theorem 2.2.4]{riehlb}.
\end{proof}

We recall the category of elements (\cite[\S2.4]{riehlb}): 
\begin{definition}\label{def:categoryOfElements}
The category $\disg_\fC \Fc$ of elements has \begin{enumerate}
	\item objects:  pairs $(\fc,x)$ with $\fc\in \fC_0$ an object in $\fC$ and $x\in \Fc\fc$, 
	\item morphisms: $(\fc,x)\xrightarrow{f}(\fd,y)$ with $\fc\xrightarrow{f}\fd\in \fC_1$ a morphism in $\fC$ such that $\Fc f(x) = y$. 
\end{enumerate} 	
\end{definition}

\begin{remark}\label{remark:projectFromCategoryOfElements}
There is forgetful functor $\Pi:\disg_\fC \Fc\rightarrow\fC$ which sends $$\big((\fc,x)\xrightarrow{f}(\fd,y)\big)  \rightsquigarrow \big(\fc\xrightarrow{f}\fd\big).$$
\end{remark}

The category of elements has a close connection to representability (\cref{def:functorRepresentable}):
\begin{prop}\label{prop:functorRepresentableIFFcategoryOfElementsHasUniversal}
	Let $\Fc:\fC\rightarrow\sF{Set}$ be a set valued covariant functor. Then $\Fc$ is representable if and only if the category of elements $\disg_\fC \Fc$ has an initial element.
\end{prop}
\begin{proof}
	See \cite[\S2.4]{riehlb}.
\end{proof}

A useful lifting property (c.f.\ \cite[2.4.viii]{riehlb}): 
\begin{fact}\label{fact:uniqueLiftingCategoryOfElements}
	Let $\Fc:\fC\rightarrow\sF{Set}$ be covariant functor from locally small category $\fC$.  Then for any morphism $\fc\xrightarrow{f}\fd$ in $\fC$ and object $(\fc,x)\in \disg_\fC \Fc$ in the category of elements, there is unique morphism $(\fc,x)\xrightarrow{f}(\fd,y)$ such that $$\Pi\big((\fc,x)\xrightarrow{f}(\fd,y)\big) = \big(\fc\xrightarrow{f}\fd\big).$$
\end{fact}
\begin{proof} Morphisms  $(\fc,x)\xrightarrow{f} (\fd,y)$ in $\disg_\fC \Fc$ are \textit{defined} as morphisms $\fc\xrightarrow{f}\fd$ in $\fC$ with the additional condition that $\Fc f (x) = y$, so a morphism in $\fC$ determines a morphism in $\disg_\fC\Fc$.  Since $\Fc:\fC\rightarrow\sF{Set}$ is a functor, $\Fc f:\Fc\fc\rightarrow\Fc\fd$ is a map of sets, and  $\Fc f(x) \in \Fc \fd$.  Therefore, for $y\defeq \Fc f(x)$, $(\fc, x)\xrightarrow{f}(\fd,y)$ is  a morphism in $\disg_\fC \Fc$. \end{proof}

In \cref{prop:existenceAndUniquenessRepresentable}, we will see a concrete application of the category of elements  applied  to  dynamical systems. 

We state one last relevant Yoneda concept, the \textit{Yoneda embedding}:
\begin{prop}\label{prop:YonedaEmbedding}
	Let $\fC$ be a locally small category and define functor, the Yoneda embedding,  $$\i:\fC^{op}\hookrightarrow\sF{Set^C}$$  by $$\big(\fc\xrightarrow{f}\fc'\big)\rightsquigarrow \big(\fC(\fc,\cdot)\xleftarrow{f^*}\fC(\fc',\cdot)\big).$$ The functor $\i$ is \textit{fully faithful} (\cref{def:fullFunctor}).
\end{prop}
\begin{proof}
	See \cite[Corollary 2.2.8]{riehlb}. Define $\Fc:\fC\rightarrow\sF{Set}$ by $\Fc\defeq \fC(\fc,\cdot)$, so $\Fc(\fc')= \fC(\fc,\fc')$. By Yoneda lemma (\cref{lemma:YonedaLemma}), there is bijection $$\fC(\fc,\fc')\cong \big\{\mbox{natural transformations}\;\a:\fC(\fc',\cdot)\Rightarrow \Fc(\cdot)\big\},$$ which says exactly that $\fC(\fc,\fc')$ is in bijective correspondence with the set $\big\{\fC(\fc',\cdot)\Rightarrow \fC(\fc,\cdot)\big\}$ of morphisms in the functor category $\sF{Set^C}$ (\cref{def:functorCategory}).
\end{proof}

\subsection{Double Categories}\label{subsection:doubleCat}

\subsubsection{Internal Category}
We start by defining pullback. 
\begin{definition}\label{def:pullback}
Let $\fc\xrightarrow{f}\fa$, $\fc'\xrightarrow{f'}\fa$ be two morphisms in category $\fC$.  The \textit{pullback}  is defined to be an object $\fc\times_\fa\fc' $ of $\fC$ \textit{terminal} with respect to pairs of maps $z\xrightarrow{g}\fc$, $z\xrightarrow{g'}\fc'$ such that $f'\circ g' = f\circ g$. In other words, for any object $\fz\in \fC$, the commuting of solid line diagram $
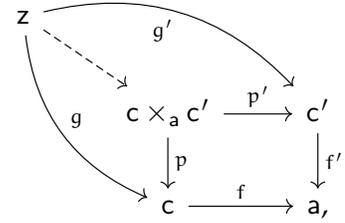
\begin{tikzcd}
	\fz\arrow[dr,dashed]\arrow[drr, bend left,swap,"g'"]\arrow[ddr,bend right, "g"] & & \\
	& \fc\times_\fa\fc'\arrow[r,"p'"]\arrow[d,"p"] & \fc'\arrow[d,"f'"]\\
	& \fc\arrow[r,"f"] & \fa, 
\end{tikzcd}$ implies that there is a unique map $\fz\dashrightarrow \fc\times_\fa\fc$ through which $g$ and $g'$ factor. \end{definition}
The pull back $\fc\times_\fa\fc$ may be depicted in diagram form  by  $$\begin{tikzcd}[column sep = small]\fc\times_\fa\fc'\arrow[r,swap,pos = -1/5, "\big\lrcorner"]\arrow[d] & \fc'\arrow[d,"f'"] \\ \fc\arrow[r,"f"] & \fa.
\end{tikzcd}$$

If $\fC$ is concrete (\cref{remark:concreteCategory}),  the pullback looks like the set $$\fc\times_\fa\fc'= \big\{(x,x')\in \fc\times \fc':\, f(x) = f'(x')\big\}.$$

\begin{definition}\label{def:internalCat}
	Let $\fC$ be a category.  A category $\Ab$ \textit{internal to} $\fC$ consists of objects $\Ab_0, \Ab_1\in \fC_0$ (called the \textit{object of objects} and \textit{object of morphisms}, respectively) with source and target morphisms  $\begin{tikzcd}
		\Ab_1\arrow[r,shift left,"\scS"]\arrow[r,shift right,swap,"\scT"] & \Ab_0
	\end{tikzcd}$, unit morphism $\scU:\Ab_0\rightarrow\Ab_1$,  and composition morphism \newline $\scC:\Ab_1\times_{\Ab_0}\Ab_1\rightarrow\Ab_1,$ where the pullback (\cref{def:pullback}) arises from diagram 
	$$\begin{tikzcd}
	\Ab_1\times_{\Ab_0} \Ab_1\arrow[r,swap,pos = -1/5, "\big\lrcorner"]\arrow[r,"p_2"]\arrow[d,"p_1"] & \Ab_1\arrow[d,"\scT"] \\ \Ab_1\arrow[r,"\scS"] & \Ab_0.
\end{tikzcd}$$ 

 We require morphisms $\scS$, $\scT$, $\scU$, and $\scC$ to make the following diagrams commute $$\begin{array}{lclc}
	1. & \begin{tikzcd}
	\Ab_0\arrow[r,"\scU"]\arrow[d,"\scU"]\arrow[dr,"id_{\Ab_0}"]  & \Ab_1\arrow[d,"\scT"]\\
	\Ab_1\arrow[r,"\scS"] & \Ab_0
\end{tikzcd} & 
2. & \begin{tikzcd}[column sep = large]
\Ab_1\times_{\Ab_0}\Ab_1\times_{\Ab_0}\Ab_1\arrow[r,"\scC\times_{\Ab_0}id_{\Ab_1}"] \arrow[d,"id_{\Ab_1}\times \scC"] & \Ab_1\times_{\Ab_0}\Ab_1\arrow[d,"\scC"]\\
\Ab_1\times_{\Ab_0} \Ab_1 \arrow[r,"\scC"] & \Ab_1
\end{tikzcd} 
\\ 
3. & \begin{tikzcd}
\Ab_1\times_{\Ab_0}\Ab_1\arrow[dr,"\scC"]\arrow[drr,bend left,"p_2"]\arrow[ddr,bend right,"p_1"] & & \\
& \Ab_1\arrow[dr,shift left,"\scT"]\arrow[dr,shift right,swap,"\scS"] &\Ab_1\arrow[d,"\scT"] \\
& \Ab_1\arrow[r,swap,"\scS"]  & \Ab_0
\end{tikzcd}
& 4. & 
 \begin{tikzcd}[column sep = large]
\Ab_1\times_{\Ab_0}\Ab_1\arrow[dr,"\scC"] & \Ab_1\times_{\Ab_0}\Ab_0\arrow[l,swap,"id_{\Ab_1}\times_{\Ab_0} \scU"] \arrow[d,"p_1"]\\
\Ab_0\times_{\Ab_0}\Ab_1\arrow[u,"\scU\times_{\Ab_0}id_{\Ab_1}"] \arrow[r,"p_2"] & \Ab_1.
\end{tikzcd}\end{array}
$$
\end{definition}
\begin{remark}
	Commuting of diagram 1 specifies source and target for the unit; similarly 3 specifies source and target for composition.  Condition 2 says that composition of morphisms is associative, and condition 4 specifies that the unit morphism behaves like an identity. (The other pullbacks $\Ab_1\times_{\Ab_0}\Ab_0$ and $\Ab_0\times_{\Ab_0}\Ab_1$ are defined by pairs of maps $\begin{tikzcd} \Ab_1\times_{\Ab_0}\Ab_0\arrow[r,shift left, "id_{\Ab_0}\circ p_2"]\arrow[r,shift right,swap,"\scS\circ p_1"] & \Ab_0 \end{tikzcd}$ and $\begin{tikzcd} \Ab_0\times_{\Ab_0}\Ab_1\arrow[r,shift right,swap, "id_{\Ab_0}\circ p_1"]\arrow[r,shift left,"\scT\circ p_2"] & \Ab_0.)\end{tikzcd}$  In fact, these are the conditions which axiomatize categories (\cref{def:category}). A small category $\fC$ (\cref{def:smallCategory}), for example, is a category internal to $\sF{Set}$. 
\end{remark}

 \subsubsection{Double Categories}
\begin{definition}\label{def:doubleCat}
	A \textit{(strict) double category} $\Ab$ is a category internal to the category $\sF{CAT}$ of categories (\cref{def:internalCat}). \end{definition}
	\begin{remark}

Morphisms in $\sF{CAT}$ are functors (\cref{ex:examplesOfCategories}), so in \cref{def:internalCat}, each of $\scS$, $\scT$, $\scU$, and $\scC$ is a functor between categories. 

\label{terminology:DoubleCat}
		We call $\Ab_0$ the \textit{object category} and $\Ab_1$ the \textit{arrow category} (c.f.\ terminology for $\Ab_i$ in \cref{def:internalCat}).  Objects of $\Ab_i$ we call $i$-objects, and morphisms of $\Ab_i$ we call $i$-morphisms, for $i=0,1$.  We note that this terminology diverges from others in the literature (e.g.\ \cite{baezPetri}) which calls morphisms in $\Ab_0$ \textit{1-morphisms} and objects of $\Ab_1$ \textit{1-cells}.   We will never use the phrase \textit{1-cell}.  
	\end{remark}

	A modified version of the next example will arise  in our development of networks of systems (\cref{subsection:fiberedCategories}, \cref{def:doubleCategoryASquare}). 
	\begin{example}\label{ex:arrowCatAsDoubleCat}\label{ex:subArrowsCatAsDoubleCat} 
		The arrow category $\fA = \sF{Arrow(C)}$ of category $\fC$ may be interpreted as a double category $\Ab$. The object category is $\Ab_0=\fC$, the original category. $\Ab_1$ is $\fA$: $1$-objects are morphisms $\fc\xrightarrow{f}\fd$ in $\fC$ and $1$-morphisms $f\xrightarrow{(\a,\b)}g$ are  commuting squares: $$\begin{tikzcd} \fc\arrow[d,"\a"]\arrow[r,"f"{name = foo1, below}] & \fd\arrow[d,swap,"\b"] \\
\fc'\arrow[r,"g"{name = foo2,above}] & \fd'.\ar[shorten <= 2pt, shorten >= 2pt,Rightarrow, from = foo1, to = foo2]	
\end{tikzcd}$$ In this diagram, $\fc$ is a $0$-object, and $\fc\xrightarrow{\a}\fc'$ a $0$-morphism, $f$ is a 1-object and $(\a,\b):f\Rightarrow f'$ a 1-morphism.  Also, $\scS(f) = \fc$ and $\scS(\a,\b) = \a$.  From diagram $$\begin{tikzcd} \fc\arrow[d,"\a"]\arrow[r,"f"{name = foo1, below}] & \fd\arrow[d,swap,"\b"]\arrow[r,"g"{name = foo3,below}] & e\arrow[d,"\g"] \\
\fc'\arrow[r,"f'"{name = foo2,above}] & \fd'\ar[shorten <= 2pt, shorten >= 2pt,Rightarrow, from = foo1, to = foo2]\arrow[r,"g'"{name = foo4, above}] & \fe'\ar[shorten <= 2pt, shorten >= 2pt, Rightarrow, from = foo3, to = foo4]	
\end{tikzcd}$$ we have composition $\scC(g,f) = g\circ f$ and $\scC\big((\b,\g),(\a,\b)\big) = (\a,\g)$. There is also a \textit{vertical composition} (composition in $\Ab_1$, as $\Ab_1$ is a category): composition of $$\begin{tikzcd} \fc\arrow[d,"\a"]\arrow[r,"f"{name = foo1, below}] & \fd\arrow[d,swap,"\b"] \\
\fc'\arrow[d,"\a'"] \arrow[r,"g"{name = foo2,above}] & \fd'\arrow[d,swap,"\b'"]\ar[shorten <= 2pt, shorten >= 2pt,Rightarrow, from = foo1, to = foo2]\\
\fc''\arrow[r,"h"{name = foo3, above}] & \fd'' \ar[shorten <= 10pt, shorten >= 2pt, Rightarrow, from = foo2, to = foo3]
\end{tikzcd}$$  results in 1-morphism $$\begin{tikzcd}[column sep = large]\fc\arrow[d,"\a'\circ\a"]\arrow[r,"f"{name = foo1, below}] & \fd\arrow[d,swap,"\b'\circ\b"] \\
\fc''\arrow[r,"h"{name = foo2,above}] & \fd''.\ar[shorten <= 2pt, shorten >= 2pt,Rightarrow, from = foo1, to = foo2]	
\end{tikzcd}$$ Finally, for any 0-object $\fc\in \Ab_0$, $\scU(\fc) = id_\fc\in \Ab_1$ (a $1$-object) and for 0-morphism $\fc\xrightarrow{\a}\fc'$, $\scU(\a)$ is the $1$-morphism $$\begin{tikzcd} \fc\arrow[d,"\a"]\arrow[r,"id_\fc"{name = foo1, below}] & \fc\arrow[d,swap,"\a"] \\
\fc'\arrow[r,"id_{\fc'}"{name = foo2,above}] & \fc.\ar[shorten <= 0pt, shorten >= 0pt,Rightarrow, from = foo1, to = foo2]	
\end{tikzcd}$$  Observe that for $1$-object $\fc\xrightarrow{f}\fc$, the identity $$\begin{tikzcd} \fc\arrow[d,"id_{\fc}"]\arrow[r,"f"{name = foo1, below}] & \fd\arrow[d,swap,"id_\fd"] \\
\fc\arrow[r,"f"{name = foo2,above}] & \fd\ar[shorten <= 2pt, shorten >= 2pt,Rightarrow, from = foo1, to = foo2]	
\end{tikzcd}$$
 in $\Ab_1$ is not the same  as the unit $\scU(f)$ for $0$-morphism $\fc\xrightarrow{f}\fc$. While both are $1$-morphisms,   the latter  acts as identity on vertical composition in the category $\Ab_1$, and the first  acts as identity on horizontal or $\scC$-composition in the double category $\Ab$. 
	\end{example}
	
In our study of hybrid systems, we will define hybrid phase spaces (\cref{def:hybridPhaseSpace0}) with \textit{discrete} double categories.
\begin{definition}\label{def:discreteDoubleCat}
	Let $\Ab$ be a double category (\cref{def:doubleCat}).  We say that $\Ab$ is \textit{discrete} if the only $0$-morphisms and $1$-morphisms are the identity.  \end{definition}

\begin{remark}\label{remark:discreteCategory} Having defined discrete double categories, we may now implicitly define  discrete (ordinary) categories  by saying that a double category is discrete if both categories $\Ab_0$ and $\Ab_1$ are discrete. In other words,  a discrete  category $\fC$ is one for which the only morphisms are identity. 
\end{remark}
\begin{remark}\label{def:ordinaryCatasDoubleCat}
	Let $\fC$ be a category.  We may realize $\fC$ as a discrete double category $\C$
as follows: the objects of the object category $\C_0$ are the objects $\fC_0$ of $\fC$, and the objects of the arrow category $\C_1$ are the morphisms $\fC_1$ of $\fC$.  This contrived construction will make sense of functors to $\sF{Rel}$ (\cref{ex:functorToRel}), which we turn now to defining. 
\end{remark}
 
We  make the category of relations (\cref{ex:(single)CategoryOfRelations}) into a  double category.
\begin{definition}
\label{def:categoryRelations}\label{def:categoryOfRelations}\label{def:Rel}
	We define the (double) category $\sF{Rel}$ of relations: $\sF{Rel}_0$ is the discrete category (\cref{remark:discreteCategory}) whose objects are sets, while $\sF{Rel}_1$ has relations as $1$-objects and \textit{inclusions} of relations as $1$-morphisms. For example, $$\begin{tikzcd} \fX\arrow[d,"id_{\fX}"]\arrow[r,"R"{name = foo1, below}] & \fY\arrow[d,swap,"id_\fY"] \\
\fX\arrow[r,"S"{name = foo2,above}] & \fY\ar[shorten <= 2pt, shorten >= 2pt,Rightarrow, from = foo1, to = foo2]	
\end{tikzcd}$$ is a $1$-morphism with both $\fX$, $\fY$ sets, and $R,\, S\subseteq \fX\times \fY$. The 1-morphism  $R\Rightarrow S$ means that $R\subseteq S$. 
\end{definition}

There is a similar double category $\sF{RelSet}$ whose $0$-morphisms are functions: 
\begin{definition}\label{def:categoryRelSet}
	We define double category $\sF{RelSet}$ by the following: 
	\begin{enumerate}
		\item The object category $\sF{RelSet}_0 = \sF{Set}$ is the category of sets and maps of sets.
		\item The arrow category $\sF{RelSet}_1$, like $\sF{Rel}_1$, has relations for $1$-objects and inclusions for $1$-morphisms: a $1$-morphism $$\begin{tikzcd} \fX\arrow[d,"f"]\arrow[r,"R"{name = foo1, below}] & \fY\arrow[d,swap,"g"] \\
\fX'\arrow[r,"R'"{name = foo2,above}] & \fY'\ar[shorten <= 2pt, shorten >= 2pt,Rightarrow, from = foo1, to = foo2]	
\end{tikzcd}$$ is an inclusion \begin{equation}\label{eq:inc123} \big(f\times g\big)( R)\subseteq R'.\end{equation} 
	\end{enumerate} In other words, for any pair $(x,y)\in R$ of $R$-related elements, $(f(x),g(y))\in R'$.
\end{definition}

The double category $\sF{RelSet}$ generalizes for any concrete category. 
\begin{example}\label{def:relC}
	Let $\fC$ be a \textit{concrete} category (\cref{remark:concreteCategory}).  We define double category $\sF{RelC}$ by: \begin{enumerate}
 \item The object category $\sF{RelC}_0= \fC$ is the original category $\fC$
 \item  The arrow category $\sF{RelC}_1$ has relations (between objects of $\fC$) as $1$-objects and inclusions like that in \eqref{eq:inc123}.   Precisely, a $1$-morphism $$\begin{tikzcd} \fc\arrow[d,"f"]\arrow[r,"R"{name = foo1, below}] & \fd\arrow[d,swap,"g"] \\
\fc'\arrow[r,"R'"{name = foo2,above}] & \fd'\ar[shorten <= 2pt, shorten >= 2pt,Rightarrow, from = foo1, to = foo2]	
\end{tikzcd}$$ in $\sF{RelC}$ is inclusion $(f\times g)(R)\subseteq R'$.  Here $\fc\xrightarrow{f}\fc'$ and $\fd\xrightarrow{g}\fd'$ are morphisms in $\fC$. 
 \end{enumerate}
	\end{example}
\begin{remark}
	We require that $\fC$ is concrete to make sense of relations (set membership and inclusion) in \cref{def:relC}. It should be noted---though we do not belabor the formlism---that secretly relations live in the underlying sets, and the 1-morphism $f\times g$ is the underlying set map.  
\end{remark}
\begin{example}\label{def:relMan}\label{ex:relMan}\label{def:relManDouble}
 Associated to the category $\sF{Man}$ of manifolds with corners and smooth maps  (\cref{ex:examplesOfCategories}) is the double category $\sF{RelMan}$.  We will encounter this category in the definition of hybrid phase space  (\cref{def:hybridPhaseSpace0}) as the target of a  functor of double categories. We now define functors of double categories. 
\end{example}

\begin{definition}\label{def:functorDoubleCat}\label{remark:varianceForDoubleFunctors}\label{def:arbitrary12349}
	A \textit{strict functor} $\scF:\Ab\rightarrow\Bb$ of double categories (or: \textit{strict double functor}) is a pair of functors $\scF_0:\Ab_0\rightarrow\Bb_0$ and $\scF_1:\Ab_1\rightarrow \Bb_1$ such that \begin{enumerate}
		\item  $\scS \circ \scF_1 = \scF_0\circ \scS$
		\item $\scT \circ \scF_1 = \scF_0\circ \scT$
		\item $\scC(\scF_1(\cdot),\scF_1(\cdot)) = \scF_1\circ \scC$ 
		\item $\scU_{\scF_0(\cdot)} = \scF_1\circ \scU$. 
	\end{enumerate}
	
A double functor $\scF$ is said to be \textit{lax} (instead of strict) if there is (not necessarily invertible) morphism \begin{enumerate}
\item[3'.] $\scC(\scF_1(\cdot),\scF_1(\cdot))\Rightarrow \scF_1\circ \scC$ or 
\item[4'.] $\scU_{\scF_0(\cdot)}\Rightarrow\scF_1\circ \scU$.
\end{enumerate}
replacing either (or both) conditions 3 or 4 above. 

We say that functor $\scF$ is \textit{covariant}  if $\scF$ sends $1$-morphism 
	 $$\begin{tikzcd}
	a\arrow[r,"f"{name = U, below}]\arrow[d,"\a"] & b\arrow[d,swap,"\b"]\\
	a'\arrow[r,swap,"g"{name = D, above}] & b' \ar[shorten <= 2pt, shorten >= 2pt,Rightarrow, from = U, to = D]
\end{tikzcd}   \;\; \rightsquigarrow \;\; \begin{tikzcd}
	\scF(a)\arrow[r,"\scF(f)"{name = U, below}]\arrow[d,"\scF(\a)"] & \scF(b)\arrow[d,swap,"\scF(\b)"]\\
	\scF(a')\arrow[r,swap,"\scF(g)"{name = D, above}] & \scF(b'), \arrow[Rightarrow, from = U, to = D]
\end{tikzcd}$$  and otherwise call $\scF$ \textit{contravariant} if $\scF$ maps the  first $1$-morphism to $$\begin{tikzcd}
	\scF(a)\arrow[r,"\scF(f)"{name = U, below}] & \scF(b)\\\
	\scF(a')\arrow[u,swap,"\scF(\a)"]\arrow[r,swap,"\scF(g)"{name = D, above}] & \scF(b').\arrow[u,"\scF(\b)"] \arrow[Rightarrow, from = D, to = U]
\end{tikzcd}$$
\end{definition}

\begin{example}\label{ex:functorToRel}
	It will be useful for us to consider functors $\Fc:\fC\rightarrow\sF{Rel}$ from an ordinary category $\fC$ to $\sF{Rel}$ (\cref{def:Rel}).  Recall that $\fC$ may be thought of as a discrete double category (\cref{def:ordinaryCatasDoubleCat}).  A (double) functor (\cref{def:functorDoubleCat}) $\Fc$ from $\fC$ to $\sF{Rel}$ assigns a set $\Fc(\fc)$ to each object of $\fc\in \fC_0$ and a relation $\Fc(f)\subseteq \Fc(\fc)\times \Fc(\fc')$ to each morphism $\fc\xrightarrow{f}\fc'$ of $\fC_1$.  Usually these functors will be lax, as there will be inclusion $\Fc(g)\circ \Fc(f)\subseteq \Fc(g\circ f)$ (a 1-morphism in $\sF{Rel}$) for composition of morphisms $\fc\xrightarrow{f}\fc'\xrightarrow{g}\fc''$.

	 Indeed, the composition (c.f.\ \eqref{eq:definingCompositionOfRelations}) \small $$\Fc(g)\circ \Fc(f)\defeq \big\{(x,z)\in \Fc(\sF{dom}(f))\times \Fc(\sF{cod}(g)):\, (x,y)\in \Fc(f),\,(y,z)\in \Fc(g)\,\mbox{for some }y\in \Fc(\sF{cod}(f))\big\},$$\normalsize  and simply $\Fc(g\circ f)\subseteq \Fc(\sF{dom}(f))\times \Fc(\sF{cod}(g))$.  There is no apriori reason why $\Fc(g)\circ\Fc(f) = \Fc(g\circ f)$.  In fact, there is no reason why there should be any inclusion relation whatsoever in either direction.  Still, in our examples (e.g.\ \cref{def:deterministicControlLax}), we will always see the inclusion $\Fc(g)\circ \Fc(f)\subseteq \Fc(g\circ f)$. 
\end{example}

\begin{definition}\label{def:naturalTransformDoubleCat}
	Let $\scF,\scG:\Ab\rightarrow\Bb$ be two functors of double categories $\Ab$ and $\Bb$.  A (strict) \textit{natural transformation} $\nat{\Ab}{\Bb}{\scF}{\scG}{\g}$ is a pair of (ordinary) natural transformations (\cref{def:naturalTrans}) $\big(\g_0:\scF_0\Rightarrow\scG_0, \g_1:\scF_1\Rightarrow\scG_1\big)$ compatible with structure functors (``$\chi\circ \g = \g\circ \chi''$ for $\chi =\scS, \scT, \scU, \scC$) of $\Bb$. 
	 Precisely, for $\fx,\fy\in \Ab$: 
	 \begin{enumerate}
	 	\item $\scS(\g_\fx) = \g_{\scS(\fx)}$, 
	 	\item $\scT(\g_\fx) = \g_{\scS(\fx)}$, 
	 	\item $\scU(\g_\fx) = \g_{\scU(\fx)}$,
	 	\item $\scC(\g_\fx,\g_\fy) = \g_{\scC(\fx,\fy)}$. 
	 \end{enumerate}
\end{definition}

We parse this definition.  Let $\begin{tikzcd} \fc\arrow[d,"\a"]\arrow[r,"f"{name = foo1, below}] & \fd\arrow[d,swap,"\b"] \\
\fc'\arrow[r,"g"{name = foo2,above}] & \fd'\ar[shorten <= 2pt, shorten >= 2pt,Rightarrow, from = foo1, to = foo2]	
\end{tikzcd}$ be  a $1$-morphism in $\Ab$ and $\nat{\Ab}{\Bb}{\scF}{\scG}{\g}$ a double natural transformation.

 Consider   diagram $$\small 
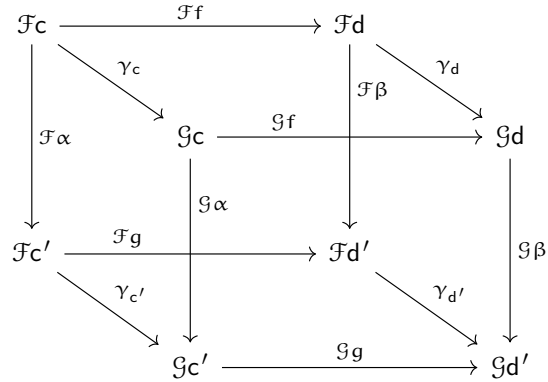
\begin{tikzcd}[column sep = large, row sep = large] 
\scF \fc\arrow[rr,"\scF f"]\arrow[dr,"\g_\fc"]\arrow[dd,"\scF\a"] & & \scF \fd \arrow[dd,near start,"\scF\b"]\arrow[dr,"\g_\fd"] & \\
 & \scG \fc\arrow[dd,near start,"\scG \a"]\arrow[rr,near start,"\scG f"] & & \scG \fd\arrow[dd,"\scG \b"] \\
 \scF \fc'\arrow[rr,near start,"\scF g"]\arrow[dr,"\g_{\fc'}"] & & \scF \fd'\arrow[rd,"\g_{\fd'}"] & \\
 & \scG \fc'\arrow[rr,"\scG g"] & & \scG \fd'
\end{tikzcd}$$ \normalsize in $\Bb$. Equalities $\scG\a \circ \g_\fc = \g_{\fc'} \circ \scF \a$ and $\scG\b\circ \g_\fd = \g_{\fd'} \circ \scF \b$ hold by naturality of $\g_0:\scF_0\Rightarrow\scG_0$. On the other hand, $(\g_\fc,\g_\fd):\scF f\Rightarrow \scG f$ and $(\g_{\fc'},\g_{\fd'}):\scF g\Rightarrow\scG g$ are $1$-morphisms, and the diagram \small 
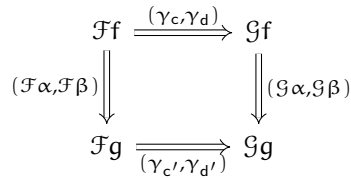
$$\begin{tikzcd}[column sep = large, row sep = large]
	\scF f \arrow[r,Rightarrow,"(\g_\fc{,}\g_\fd)"]\arrow[d,Rightarrow,swap,"(\scF \a{,}\scF \b)"] & \scG f\arrow[d,Rightarrow,"(\scG\a{,}\scG\b)"] \\
	\scF g\arrow[r,Rightarrow,swap,"(\g_{\fc'}{,}\g_{\fd'})"] & \scG g 
\end{tikzcd}$$\normalsize commutes by naturality of $\g_1:\scF_1\Rightarrow \scG_1$.

\subsubsection{Category $\sF{Rel}$ and Category of Elements in $\sF{Rel}$}

\begin{definition}\label{def:categoryOfRelatedElements}
	Let $\Fc:\fC\rightarrow\sF{Rel}$ be a lax functor, we define the \textit{category of related elements} by $$\disg_\fC \Fc\defeq\big\{(\fc,x):\, \fc\in \fC_0,\, x\in \Fc\fc\big\}.$$ Morphisms are relations: $(\fc,x)\xrightarrow{f}(\fc',x')$ is a morphism in $\disg_\fC \Fc$ if $\fc\xrightarrow{f}\fc'$ is a morphism in $\fC$ and $(x,x')\in \Fc f$. 
\end{definition}

\subsection{Monoidal Categories}\label{subsection:MonoidalCat}
\begin{definition}\label{def:monoidalCat}[c.f.\ \cite[\S VII.1]{maclane}]
	A \textit{monoidal category}  $(\fC,\otimes_\fC,1_\fC)$ is a category $\fC$ equipped with bifunctor $\otimes_\fC:\fC\times \fC\rightarrow\fC$ and object $1_\fC\in \fC_0$, together with three natural isomorphisms  $\nat{\fC}{\fC,}{1_\fC\otimes_\fC(\cdot)}{id_\fC}{\l }$  $\nat{\fC}{\fC,}{(\cdot)\otimes 1_\fC}{id_\fC}{\r  }$ and $\nat{\fC\times \fC\times\fC}{\fC,}{(\cdot)\otimes_\fC((\cdot)\otimes_\fC(\cdot))}{((\cdot)\otimes_\fC(\cdot))\otimes_\fC(\cdot)}{\a}$ satisfying coherence conditions \begin{enumerate}
		\item $\a\circ\a = (\a\otimes id_\fC)\circ \a\circ (id_\fC\otimes \a)$.
		\item $(\r\otimes id_\fC)\circ \a = id_\fC\otimes \lambda$. 
	\end{enumerate}
	Pictorially, condition 1 says that for every  collection of objects $\fa, \fb,\fc, \fd \in \fC$, we have commuting pentagonal diagram $$\begin{tikzcd}
	& (\fa\otimes \fb)\otimes(\fc\otimes \fd)\arrow[r,"\a"]&  ((\fa\otimes \fb)\otimes\fc)\otimes\fd \\
	\fa\otimes(\fb\otimes(\fc\otimes\fd))\arrow[rd,"id_\fC\otimes\a"]\arrow[ru,"\a"] & & \\
	& \fa\otimes ((\fb\otimes\fc)\otimes \fd)\arrow[r,"\a"] & (\fa\otimes(\fb\otimes\fc))\otimes \fd\arrow[uu,"\a\otimes id_\fC"] 
\end{tikzcd}$$

	When $\a, \r,$ and $\l$ are all identity, we say that $\fC$ is \textit{strict monoidal}. 
\end{definition}
Let  $\fA\defeq \sF{Arrow(C)}$ be the arrow category of some category $\fC$.  We saw $\fA$ has products whenever $\fC$ has products (\cref{prop:catHasProductsImpliesArrowsDoesToo}). There is a similar statement for monoidal products.
\begin{prop}\label{prop:MonoidalProductInducesMonoidalOnArrows}
	Let $\fA=\sF{Arrow(C)}$ be arrow category of category $\fC$, which is also monoidal $(\fC,\otimes_\fC,1_\fC)$.  Then $\fA$ is monoidal category $(\fA,\otimes_\fA,1_\fA)$.  
\end{prop}
\begin{proof}[Proof Sketch] First, we define the monoidal product $\otimes_\fA$: for $\fa, \fa'\in \fA$, we set $$\fa\otimes_\fA\fa'\defeq \begin{tikzcd}\sF{dom(a)}\otimes_\fC \sF{dom(a')}\arrow[d,"\fa\otimes_\fC\fa'"] \\
\sF{cod(a)}\otimes_\fC \sF{cod(a')},\end{tikzcd}
$$ which is well defined since $\otimes_\fC:\fC\times \fC\rightarrow \fC$ is a bifunctor. 
Let $\fe\defeq 1_\fC$ be the monoidal unit of $\fC$, and set $1_\fA\defeq \fe\xrightarrow{id_\fe}\fe$.  We claim that $1_\fA$ is monoidal unit of $\fA$.  For example, $$\begin{tikzcd}
	\sF{dom(a)}\otimes_\fC\fe\arrow[d,"\fa\otimes_\fC id_\fe"]\arrow[r,"\r_{\sF{dom(a)}}"] & \sF{dom(a)}\arrow[d,"\fa"] \\
	\sF{cod(a)}\otimes_\fC\fe \arrow[r,"\r_\sF{cod(a)}"] & \sF{cod(a)}
\end{tikzcd}$$ is an isomorphism in $\fA$ because each $\rho_{\sF{dom(a)}}$ and $\r_{\sF{cod(a)}}$ is isomorphism in $\fC$, and the diagram commutes since $\r$ is a natural transformation (\cref{remark:isoArrowCategory}). This defines $\r_\fA$ as a morphism (and therefore isomorphism)  in $\fA$.  Transformations $\l_\fA$ and $\a_\fA$ are defined similarly.  Coherence is a formal consequence of coherence in $\fC$ and naturality (in $\fC$) of $\r$, $\l$, and $\a$. 
\end{proof}
\begin{remark}\label{remark:notGonnaDenoteArrowMonoidalProduct}
	At the risk of abusing notation, we may  denote the monoidal product $\otimes_\fA$ of $\fA$ by the monoidal product $\otimes_\fC$ from which it is induced (if, e.g., we wish to emphasize that $\otimes_\fC$ is a functor).
\end{remark}

\begin{definition}\label{def:strongMonoidalFunctor}\label{def:monoidalFunctor}
	A \textit{monoidal functor} $\Fc:(\fC,\otimes_\fC,1_\fC)\rightarrow(\fD,\otimes_\fD,1_\fD)$ is a functor $\Fc:\fC\rightarrow\fD$ together with natural transformations $\nat{\fC\times\fC}{\fD}{\Fc(\cdot)\otimes_\fD \Fc(\cdot)}{\Fc\big((\cdot)\otimes_\fC(\cdot)\big)}{\eta}$ and morphism $1_\fD\xrightarrow{e} \Fc(1_\fC)$.  Naturality of $\eta$ implies, for example, that diagram $\begin{tikzcd}
	\Fc\fc\otimes_\fD\Fc\fd\arrow[r,"\eta_{\fc{,}\fd}"]\arrow[d,"\Fc f\otimes_\fD\Fc g"] & \Fc(\fc\otimes_\fC\fd)\arrow[d,"\Fc(f\otimes_\fC g)"]\\
	\Fc\fc'\otimes_\fD\Fc\fd'\arrow[r,"\eta_{\fc'{,}\fd'}"]& \Fc(\fc'\otimes_\fC\fd')
\end{tikzcd}$ commutes for every pair of morphisms $\fc\xrightarrow{f}\fc'$, $\fd\xrightarrow{g}\fd'$. z

	 When both $\eta:F(\cdot)_\fD \Fc(\cdot)\Rightarrow \Fc( \cdot \otimes_\fC \cdot)$    and $e:1_\fD\xrightarrow{\sim} \Fc(1_\fC)$ are natural  isomorphisms, we say that $\Fc$ is a \textit{strong monoidal functor}  (c.f.\ \cite[\S XI.2]{maclane}).
\end{definition}

\begin{definition}\label{def:monoidalTransformation}
	Let $\nat{\fC}{\fD}{\Fc}{\Gc}{\a}$ be a natural transformation (\cref{def:naturalTrans}) and $\Fc,\Gc:(\fC,\otimes_\fC,1_\fC)\rightarrow(\fD,\otimes_\fD,1_\fD)$ monoidal functors (\cref{def:monoidalFunctor}). We say that  $\a$ is a \textit{monoidal transformation} if for each $\fc, \fc'\in \fC$, $$\begin{tikzcd}[column sep = large]
	\Fc(\fc)\otimes_\fD\Fc(\fc')\arrow[r,"\a_\fc\otimes_\fD\a_{\fc'}"] \arrow[d,"\eta^\Fc_{\fc{,}\fc'}"] & \Gc(\fc)\otimes_\fD\Gc(\fc')\arrow[d,"\eta^\Gc_{\fc{,}\fc'}"]\\
	\Fc(\fc\otimes_\fC\fc')\arrow[r,"\a_{\fc\otimes_\fC\fc'}"] & \Gc(\fc\otimes_\fC\fc')
\end{tikzcd}
$$ commutes. 
\end{definition}

Since $\Fc$ and $\Gc$ are monoidal (functors),   for every pair of morphisms $\fc\xrightarrow{f}\fd$, $\fc'\xrightarrow{f'}\fd$ in $\fC$, we have commuting diagram   $$\begin{tikzcd}
	\Fc\fc\otimes_\fD\Fc\fc'\arrow[rr,"\a_\fc\otimes_\fD\a_{\fc'}"]\arrow[dr,"\eta_{\fc{,}\fc'}^\Fc"] \arrow[dd,"\Fc f\otimes_\fD\Fc f'"] & & \Gc \fc\otimes_\fD\Gc \fc'\arrow[dr,"\eta_{\fc{,}\fc'}^\Gc"] \arrow[dd,near start,swap,"\Gc f\otimes_\fD\Gc f'"] & \\
	& \Fc(\fc\otimes_\fC \fc')\arrow[rr,near start,swap,"\a_{\fc\otimes_\fC\fc'}"]\arrow[dd,near start,"\Fc(f\otimes_\fC f')"] & & \Gc(\fc\otimes_\fC\fc')\arrow[dd,"\Gc(f\otimes_\fC\fc')"] \\
	\Fc\fd\otimes_\fD\Fc\fd'\arrow[rr,near start,"\a_\fd\otimes_\fD\a_{\fd'}"]\arrow[dr,"\eta_{\fd{,}\fd'}^\Fc"] & & \Gc\fd\otimes_\fD\Gc\fd'\arrow[dr,"\eta_{\fd{,}\fd'}^\Gc"] & \\
	& \Fc(\fd\otimes_\fC\fd')\arrow[rr,near start,"\a_{\fd\otimes_\fC\fd'}"] & & \Gc(\fd\otimes_\fC\fd')
\end{tikzcd}$$
where $\eta^\Fc:\Fc(\cdot)\otimes_\fD\Fc(\cdot)\Rightarrow\Fc((\cdot)\otimes_\fC (\cdot))$ and $\eta^\Gc:\Gc(\cdot)\otimes_\fD\Gc(\cdot)\Rightarrow\Gc((\cdot)\otimes_\fC(\cdot))$ are the natural transformations of \cref{def:monoidalFunctor}.  This diagram displays various monoidal properties: the left and right face commute because $\Fc$ and $\Gc$ are monoidal, the top  and bottom face commute because $\a$ is monoidal (transformation), and the front and back face commute by naturality of $\eta$. 

\begin{definition}\label{def:monoidalProductCartesian} 
	Let $(\fC,\otimes_\fC,1_\fC)$ be a monoidal category (\cref{def:monoidalCat}).  We say that monoidal product $\otimes_\fC$ is \textit{induced-cartesian} if $\fC$ is a subcategory of a category $\fC'$ with finite products (\cref{def:categoryWithProducts}) in which $\fc\otimes_\fC \fc' = \fc \times \fc'$.  If $\fC' = \fC$, we may also say that $\fC$ is \textit{cartesian} monoidal. 
\end{definition}

\begin{remark}\label{fact:terminalObjectInCartesianIsMonoidalUnit}
	A terminal object $1\in \fC'$ defines a monoidal unit in $\fC$, which we denote as $1_\fC$.  \textit{We will assume that the monoidal unit of any induced-cartesian monoidal category is a terminal object in the supercategory}, but not necessarily terminal in the submonoidal category.  
	
	As an example of where this assumption may fail, consider the cartesian category $(\sF{Set},\times,1)$ and a subcategory  all of whose objects (sets) have cardinality $\aleph_0$, countable infinity. Since $\aleph_0\times \aleph_0\cong \aleph_0$, a monoidal unit---even in induced cartesian category---need not be terminal in the supercategory.\footnote{To explicitly construct this example, suppose all objects are of the form $\aleph_0\times \cdots \times \aleph_0$, and the only morphism $f:\aleph_0\rightarrow\aleph_0$ is $f = id_{\aleph_0}$. In this case, the only nontrivial morphisms are monoidal products of $id_{\aleph_0}$, $\a$, $\r$, and $\l$.}  
\end{remark}
\begin{prop}\label{prop:inducedCartesianCatImpliesArrowIsInducedCartesian}
	Suppose that $(\fC,\otimes_\fC,1_\fC)$ is induced-cartesian monoidal category.  Then $\fA = \sF{Arrow(C)}$ has monoidal structure and is also induced-cartesian. 
\end{prop}
\begin{proof}
	We have already shown that $(\fA, \otimes_\fA,1_\fA)$ is monoidal (\cref{prop:MonoidalProductInducesMonoidalOnArrows}). Now let $\fC'\supseteq \fC$ be the supercategory with products whose products define the monoidal product in $\fC$ (\cref{def:monoidalProductCartesian}).  Then $\fA'\defeq \sF{Arrow(C')}$ has products (\cref{prop:catHasProductsImpliesArrowsDoesToo}) and is supercategory of $\fA$.  Therefore, $\fA$ is induced-cartesian (\cref{def:monoidalProductCartesian}). 
\end{proof}

The analog of product preservation (\cref{def:productPreservingFunctor}) for monoidal functors \textit{strong} monoidal functoriality (\cref{def:monoidalFunctor}).  In \cref{lemma:arrowProductPreservingToo}, we showed that product preserving functors extend naturally to product preserving functors on the arrow categories. However,  the universal property of product was nowhere used in the proof: we only needed functoriality and naturality. We therefore state without proof the following analogous statement for strong monoidal functors, which can be readily obtained by replacing instances of `$\times$' in the proof of \cref{lemma:arrowProductPreservingToo} with `$\otimes$' where appropriate. 
\begin{prop}\label{prop:arrowMonoidalProductPreserving}
	Suppose that $\Fc:(\fC,\otimes_\fC,1_\fC)\rightarrow (\fD,\otimes_\fD,1_\fD)$ is a strong monoidal functor (\cref{def:monoidalFunctor}), and let $\fA_\fC\subseteq \sF{Arrow(C)}$ be a subcategory of the arrow category of $\fC$, which is monoidal.  Then $\Fc$ extends to a strong monoidal functor $$\Fc_*:(\fA_\fC,\otimes_{\fA_\fC},1_{\fA_\fC})\rightarrow(\fA_\fD,\otimes_{\fA_\fD},1_{\fA_\fd})$$ as well, where $\fA_\fD\defeq \sF{Arrow(D)}$. 
\end{prop}
\begin{lemma}\label{lemma:productMonoidalIsMonoidal}
Let $(\fA,\otimes_\fA,1_\fA)$ and $(\fB,\otimes_\fB,1_\fB)$ be two monoidal categories.  Then there is monoidal category $(\fA\times \fB, \otimes_{\fA\times \fB},1_{\fA\times \fB})$ (c.f.\ \cite[\S 7.1]{maclane}).  	
\end{lemma}
	Given objects $(\fa, \fb),\, (\fa',\fb') \in \fA\times \fB$,  define the monoidal product \begin{equation}\label{eq:defMoniodalProductOfProduct}(\fa,\fb)\otimes_{\fA\times \fB}(\fa',\fb')\defeq (\fa\otimes_\fA \fa',\fb\otimes_\fB \fb'),\end{equation}  and monoidal unit $1_\fA\times 1_\fB\defeq (1_\fA,1_\fB)$.  The definition of product \eqref{eq:defMoniodalProductOfProduct} ensures that $1_{\fA\times \fB}$ is indeed a monoidal unit. Natural isomorphisms $\a_{\fA\otimes \fB}$, $\l_{\fA\otimes \fB}$, and $\r_{\fA\otimes \fB}$ are defined similarly, and coherence is a formal and straightforward verification. 

\subsubsection{Category of Lists} 
Let $\fA$ be a category. We  introduce the category of lists of $\fA$-objects.   The category of lists  is a categorical way  taking an $\fX$-indexed set $\{\fa_\fx\}_{\fx\in \fX}$  of $\fA$ objects.  In other words, the category of lists includes a notion of morphisms between indexed collections. Such morphisms are maps of the index sets together with collection morphisms in the ambient category. For now $\fA$ is any monoidal category, but we use this notation (instead of $\fC$) because in \cref{ch4} we will work with an arrow (sub)category. 
\begin{definition}\label{def:categoryOfLists}
 Let $(\fA,\otimes_\fA,1_\fA)$ be a monoidal category (\cref{def:monoidalProductCartesian}).  We define the \textit{category of lists of $\fA$-objects} $\sF{FinSet/A}^\Leftarrow$ by the following. \begin{enumerate}

\item An object $\Ac_\fX:\fX\rightarrow \fA$ is a functor from finite discrete category $\fX$ (\cref{remark:discreteCategory}). We often just write $\big\{\Ac_{\fX,\fx}\big\}_{\fx\in \fX}$.
\item A morphism $(\ph,\Phi):\big(\Ac_\fX:\fX\rightarrow\fA\big)\rightarrow\big(\Ac_\fY:\fY\rightarrow\fA\big)$   is a pair where $\ph:\fX\rightarrow \fY$ is a functor and $\Phi:\Ac_\fY\circ \ph\Rightarrow \Ac_\fX$ is a natural transformation, as indicated in diagram  $\begin{tikzcd}[column sep = large]\fX\arrow[r,"\ph"]\arrow[dr,swap,"\Ac_\fX",""{name = foo, above}] & \fY\arrow[d,"\Ac_\fY."]\arrow[Rightarrow, to = foo,"\Phi"] \\
  & \fA
\end{tikzcd}$	
 \end{enumerate}
\end{definition}

\begin{remark}\label{remark:altDefinitionOfFinSetCat}
In other words,  $\fX$ is a finite set  and $\Ac$ assigns an $\fA$-object $\Ac_\fX(\fx)$---also denoted by $\Ac_\fx$---to each $\fx\in \fX$.  For morphisms, there is a map of sets $\ph:\fX\rightarrow\fY$ and for each $\fx\in \fX$, a morphism $\Phi_\fx:\Ac_{\fY,\ph(\fx)}\rightarrow\Ac_{\fX,\fx}$ in $\fA$. 
Because $\fX$ is discrete (the only morphisms in $\fX$ are $\fx\xrightarrow{id_\fx}\fx$),  the naturality condition  $\Phi:\Ac_\fY\circ \ph\Rightarrow \Ac_\fX$ is trivial. 
\end{remark}

Since $(\fA,\otimes_\fA,1_\fA)$ is monoidal, from the assignment $\Ac_\fX:\fX\rightarrow\fA$, there is object \begin{equation}\label{eq:productFinSet} \Pi(\Ac_\fX)\defeq \disbop_{\fx\in \fX}\Ac(\fx)\in \fA.\end{equation}

We are interested in when the map $\Pi$ defined in \eqref{eq:productFinSet} extends to a functor.  Here is one case. 
\begin{prop}\label{prop:NietzscheProp}\label{prop:otimesACartesianImpliesProductIsFunctor}\label{prop:productIsFunctorialForCartesianMonoidal}
	Suppose that monoidal category $(\fA,\otimes_\fA,1_\fA)$ is cartesian (\cref{def:monoidalProductCartesian}). Then there is contravariant functor $\Pi:\left(\sF{FinSet/A}^\Leftarrow\right)^{op}\rightarrow \fA$. 
\end{prop}
\begin{proof}
	The assignment on objects is given in \eqref{eq:productFinSet}. We now define the assignment on morphisms. Let $(\ph,\Phi):\Ac_\fX\rightarrow \Ac_\fY$ be a morphism in $\sF{FinSet/A}^\Leftarrow$.  Since $\otimes_\fA$ is cartesian,  $\disbop_{\fx\in \fX}\Ac_\fX(\fx) = \discatp_{\fx\in \fX}\Ac_\fX(\fx)$ and  $\disbop_{\fy\in \fY}\Ac_\fY(\fy) = \discatp_{\fy\in \fY}\Ac_\fY(\fy)$.  Then $\Pi(\ph,\Phi):\Pi(\Ac_\fY)\rightarrow\Pi(\Ac_\fX)$ is   uniquely induced  by the collection of maps $\Phi_{\fx'}:\Ac_\fY(\ph(\fx'))\rightarrow \Ac_\fX(\fx')$ (\cref{def:products}). This is illustrated in the following commuting diagram: \begin{equation}\label{eq:NietzscheEq}\begin{tikzcd}[column sep = large]
	\discatp_{\fx\in \fX}\Ac_\fX(\fx)\arrow[d,"p_{\fx'}"] & \discatp_{\fy\in \fY}\Ac_\fY(\fy)\arrow[l,dashed,"\Pi(\ph{,}\Phi)"]\arrow[d,"p_{\ph(\fx')}"] \\
	\Ac_\fX(\fx') & \Ac_\fY(\ph(\fx')).\arrow[l,"\Phi_{\fx'}"] 
\end{tikzcd}\end{equation}

Functoriality of $\Pi$ is a formally identical check as the proof of \cref{prop:productBifunctor}, which says that $\times$ is a bifunctor. \end{proof}



\begin{remark}\label{remark:catListsCoproduct}
There is also a category $\big(\sF{FinSet}/\fA\big)^{\Rightarrow}$, whose objects are the same as objects of $\big(\sF{FinSet}/\fA\big)^\Leftarrow$, but whose morphisms   are pairs $(\ph,\Phi):\Ac_\fX\rightarrow\Ac_\fY$ where still $\ph:\fX\rightarrow\fY$ is functor of discrete categories, and now $\Phi:\Ac_\fX\Rightarrow\Ac_\fY\circ\ph$ is morphism in $\fA^\fX$ which goes in the other direction (\cref{remark:altDefinitionOfFinSetCat}). If $\fA$ has coproducts and  monoidal product $\otimes_\fA$ is a coproduct  then the assignment $\Ub:\big(\sF{Set/A}\big)^\Rightarrow\rightarrow A$ sending $(\Ac_\fX:\fX\rightarrow \fA\big) \rightsquigarrow  \Ub \big(\Ac_\fX\big) \defeq  \discats_{\fx\in \fX}\Ac_\fX(\fx)$ extends to a \textit{covariant} functor.  The argument is nearly identical to the proof of \cref{prop:productIsFunctorialForCartesianMonoidal}, using the properties of coproduct instead. 	
\end{remark}

\section{Review of Geometry and Dynamical Systems}
\subsection{Differential Geometry}
In this section we  recall notions from geometry, and provide categorical proofs for elementary facts, as a first step toward setting the tone for our categorical investigation into networks of systems. 

First of all, for us \textit{manifold} will always mean \textit{manifold with corner}.  A manifold in the traditional sense is a manifold-with-corners which has no corners. Manifolds with corners are much like manifolds with boundaries.  Instead of charts in half space $\big\{(x_1,\ldots,x_n)\in \R^n:\, x_1\geq 0\big\}$, charts for manifolds with corners live in $\R^n_+\defeq \big\{(x_1,\ldots,x_n)\in \R^n:\, x_i\geq 0:\, \forall \, i = 1,\ldots, n\big\}$. Just as charts of manifolds with boundary need not be at the boundary $\big\{x_1\geq0\big\}$, charts for corners need not have corners.  Compatibility of charts will be defined after we clarify the notion of smooth maps on arbitrary sets.

\begin{definition}\label{def:smoothMapArbitrarySubset}
Let $V\subseteq \R^n$ be an arbitrary subset.  A function  $f:V\rightarrow\R$ is said to be \textit{smooth at} $x\in V$ if there is an open neighborhood $U\subseteq \R^n$ of $x$ and smooth function $\tilde{f}:U\rightarrow\R$ for which $\tilde{f}|_{U\cap V} = f|_{U\cap V}$.  The collection of smooth functions $f:V\rightarrow \R$ is denoted by $\Cc^\infty(V)$. 
\end{definition}

\begin{definition}\label{def:manifoldsWithCorners}
	A second countable Hausdorff  space $M$ is said to be a \textit{(smooth) manifold with corners} if $M$ is equipped with a maximal atlas: a collection of homeomorphisms $\Ac= \big\{\ph_\a:U_\a\hookrightarrow \R^n_+\big\}_{\a\in A}$ from open subsets $U_\a\subseteq M$ of $M$ to open subsets of $\R^n_+$ such that whenever $U_\a\cap U_\b\neq \varnothing$, the map $\ph_\b\circ \ph_\a^{-1}|_{\ph_\a(U_\a\cap U_\b)}:\ph_\a\big(U_\a\cap U_\b)\rightarrow \ph_\b(U_\a\cap U_\b)$ is a diffeomorphism.  Maximality of $\Ac$ means that if $\ph:U\hookrightarrow R^n_+$   is a diffeomorphism from open $U\subseteq M$ onto an open subset of $\R^n_+$ compatible with each $(\ph_\a,U_\a)\in \Ac$, then$(\ph,U)\in \Ac$. 
\end{definition}
\begin{remark}
The notion of sameness for two manifolds with corners is \textit{diffeomorphism},  a smooth invertible map whose inverse is also smooth.  Smoothness of maps for manifolds with corners is always smoothness in the sense of \cref{def:smoothMapOfManifolds}.
\end{remark}
\begin{definition}\label{def:smoothMapOfManifolds}
	Let $M$ and $N$ be manifolds with corners (\cref{def:manifoldsWithCorners}).  We say that a map $f:M\rightarrow N$ is \textit{smooth}  if for every smooth function $\ph\in \Cc^\infty(N)$ (c.f.\ \cref{def:smoothMapArbitrarySubset}), $f^*\ph\in \Cc^\infty(M)$, in other words, if $f^*\big(\Cc^\infty(N)\big)\subseteq \Cc^\infty(M)$. 
\end{definition}

\begin{remark}\label{remark:categoryManifolds}
	 Manifolds with corners and smooth maps between them form a category, which we denote by $\sF{Man}$. Most of the standard notions from the theory of manifolds ``without corners'' applies with minimal modification to this category, e.g.\ the tangent space $T_xM$ is still a vector space of all point derivations of germs of functions passing through $x$, even if $x$ is a corner or boundary point, and the tangent bundle $TM$ has the smooth structure of a  manifold with corners. 
\end{remark}

\begin{fact}\label{fact:productManifolds}\label{def:productManifolds}
	A product $M\times N$ of smooth manifolds $M$ and $N$ is a smooth manifold. 
\end{fact}
\begin{proof}
See \cite[Proposition 5.18]{tu}.	
\end{proof}

\begin{prop}\label{prop:productManifoldsCategorical}
Let $M, N$ be smooth manifolds.  The product  $M\times N$  of manifolds (\cref{def:productManifolds}) satisfies the universal property of product (\cref{def:products}). 
\end{prop}
\begin{proof}
	First of all, the projections $M\times N\xrightarrow{\p_M}M$ and $M\times N\xrightarrow{\p_N}N$ are smooth maps (\cite[Example 6.17]{tu}).  As \textit{sets}, $M\times N$ satisfy the universal property of product, so  pair of smooth maps $f_M:P\rightarrow M$ and $f_N:P\rightarrow N$ induces a unique map $f:P\rightarrow M\times N$. What is left to show is that the map $f$ is smooth. This follows by \cite[Ex.\ 6.18]{tu}, smoothness on components, together with $\pi_M \circ f = f_M$, $\pi_N\circ f = f_N$. 
\end{proof}

Now we examine the product of tangent spaces.  It is a well-known fact that the product of tangent spaces is canonically isomorphic to the tangent space of products of manifolds.  We prove this to illustrate the universal property of product in a specific  category, the category of manifolds. 
\begin{prop}\label{prop:isoProductTangentSpaces}
	Let $M$, $N$ be manifolds of dimension $m$ and $n$, respectively.  Then for any $x\in M$ and $y\in N$, there is canonical  isomorphism of tangent spaces $T_xM\times T_yN \cong T_{(x,y)}M\times N$, natural in $M$ and $N$. 
\end{prop}
First an elementary fact about the differential of a map. 
\begin{lemma}\label{prop:differentialIsFunctorial}
The differential which assigns the tangent bundle $TM$ to manifold $M$ is functorial. 	On maps $M\xrightarrow{f}N$, we have $TM\xrightarrow{Tf}TN$. 
\end{lemma}
\begin{proof}
	See \cite[\S10.2]{tu}.  Functoriality on composition (\cref{remark:functoriality}) is the chain rule, in categorical dress. 
\end{proof}
\begin{proof}[Proof of \cref{prop:isoProductTangentSpaces}.]
	There are canonical projections $p_M:M\times N\rightarrow M$ and $p_N:M\times N\rightarrow N$ (\cref{prop:productManifoldsCategorical}).  Applying the differential to each map at point $(x,y)\in M\times N$, there are   maps $Tp_{M,(x,y)}:T_{(x,y)}M\times N\rightarrow T_xM$ and $Tp_{N,(x,y)}:T_{(x,y)}M\times N\rightarrow T_yN$ of  tangent spaces.  These maps induces a unique map $P\defeq Tp_{M,(x,y)}\times Tp_{N,(x,y)}  : T_{(x,y)}M\times N\rightarrow T_xM\times T_yN$ (\cref{def:products}, \cref{ex:vectorSpacesProduct&Coproduct}).  
	
	We define inverse map $I:T_xM\times T_yN\rightarrow T_{(x,y)}M\times N$ as follows.  There are maps of manifolds $i^y:M\rightarrow M\times N$ sending $x\mapsto (x,y)\in M\times N$ and $i^x:N\rightarrow M\times N$ sending $y\mapsto (x,y)$, which induce maps on the tangent spaces: \begin{equation}\begin{array}{lrlllrl}
		Ti^y_x:  & T_xM & \rightarrow T_{(x,y)}M\times N  &\mbox{and} &  Ti^x_y: & T_yN &\rightarrow T_{(x,y)}M\times N\\
		 & v & \mapsto  (Ti^y)_x(v) & & & w & \mapsto (Ti^x)_y(w) \end{array}
	\end{equation}Since $p_M\circ i^y =id_M$ and $p_N\circ i^x = id_N$, we observe that  \begin{equation}\label{eq:3421} T(p_M\circ i^y)_x = T p_{M,(x,y)} \circ T i^y_x = id_{T_xM}\end{equation} and similarly  $T(p_N\circ i^x)_y = id_{T_yN}$.  The first equality in \eqref{eq:3421} follows by \cref{prop:differentialIsFunctorial}, that $T$ is a functor (more directly $T_x(id_M) = id_{T_xM}$.)   Thus, the composition of maps $$T_xM\times T_yN\xrightarrow{I} T_{(x,y)}M\times N\xrightarrow{P} T_xM\times T_yN$$ is the identity (\cref{prop:productIsosIsIso}).  Since $\dim(T_xM\times T_yN) = m+n = \dim(T_xM)+\dim(T_yN)$, and each $T_xM\times T_yN$ and $T_{(x,y)}M\times N$ are linear spaces, we conclude that  $T_xM\times T_yN \cong T_{(x,y)}M\times N$. 
	
	Naturality is a formal consequence of functoriality (\cref{prop:differentialIsFunctorial}): diagrams $$\begin{array}{ccc} \begin{tikzcd}[column sep = large]
M\times N\arrow[r,"p_M"]\arrow[d,"f\times g"] & M \arrow[d,"f"] \\
M'\times N'\arrow[r,"p_{M'}"] & M'	
\end{tikzcd}  &\mbox{and} & \begin{tikzcd}[column sep = large]
 	M\times N\arrow[r,"p_N"] \arrow[d,"f\times g"] & N \arrow[d,"g"] \\
 	M'\times N'\arrow[r,"p_{N'}"] & N'
 \end{tikzcd}\end{array}$$ commute. Let $(x,y)\in M\times N$ and set $x'\defeq f(x)$, $y'\defeq g(y)$. Applying $T_{(\cdot)}$ to each diagram, we have commuting diagrams 
 $$\begin{array}{ccc} \begin{tikzcd}[column sep = huge]
T_{(x,y)}\big(M\times N\big)\arrow[r,"T_{(x{,}y)}p_M"]\arrow[d,"T_{(x{,}y)}f\times g"] & T_xM \arrow[d,"T_xf"] \\
T_{(x',y')}\big(M'\times N'\big)\arrow[r,"T_{(x'{,}y')}p_{M'}"] & T_{x'} M'	
\end{tikzcd}  &\mbox{and} & \begin{tikzcd} [column sep = huge]
 	T_{(x,y)}\big(M\times N\big)\arrow[r,"T_{(x{,}y)}p_N"] \arrow[d,"T_{(x{,}y)}f\times g"] & T_yN \arrow[d,"T_yg"] \\
 	T_{(x',y')}\big(M'\times N'\big)\arrow[r,"T_{(x'{,}y')}p_{N'}"] & T_{y'}N',
 \end{tikzcd}\end{array}$$ which by the universal property implies that diagram $$\begin{tikzcd}[column sep = huge, row sep = huge]
	T_{(x,y)}M\times N\arrow[rr,"T_{(x{,}y)}p_M\times T_{(x{,}y)}p_N"]\arrow[d,"T_{(x{,}y)}f\times g"]& &  T_xM\times T_yN\arrow[d,"T_xf\times T_yg"] \\
	T_{(x',y')}M'\times N'\arrow[rr,"T_{(x'{,}y')}p_{M'}\times T_{(x'{,}y')} p_{N'}"] & &  T_{x'}M'\times T_{y'}N'
\end{tikzcd}$$ also commutes.
\end{proof}
Since we have isomorphism of tangent spaces $T_{(x,y)} M\times N \cong T_xM\times T_yN$, and $T:\sF{Man}\rightarrow\sF{Man}$ is a functor  (\cref{prop:differentialIsFunctorial}), the bijection as sets $T p_M\times Tp_N:T(M\times N)\rightarrow TM\times TN$ is smooth and therefore the tangent bundles are diffeomorphic, as manifolds:
\begin{corollary}\label{corollary:TProductDiffeoProductT}
Let $M$, $N$ be smooth manifolds.  Then there is canonical diffeomorphism $T (M\times N)\xrightarrow{Tp_M\times Tp_N} TM\times TN$. 	
\end{corollary}

\begin{prop}[Proposition 3.42 (b) \cite{leeTopologicalManifolds}] \label{prop:coproductCanonicalInjectionOpen}
	Let $\{M_k\}_{k\in K}$ be a finite collection of smooth manifolds.  Then the coproduct $M\defeq \discats_{k\in K} M_k$ is a smooth manifolds and the canonical injections $i_k:M_k\hookrightarrow M$ are open embeddings.
\end{prop}

\begin{proof}
Fix $k\in K$.  We show that $i_k:M_k\hookrightarrow M$ is an open embedding. First we define a map $f:M\rightarrow M_k$, which by the universal property is defined by collection of maps $\left\{M_{k'}\xrightarrow{f_{k'}}M_k\right\}_{k\in K}$.  For any distinguished point $x_0\in M_k$, and $k'\neq k$ we define $f_{k'}(x) = x_0$.  Otherwise, for $k'=k$, we set $f_{k'}(x) = x$.  Clearly, each map $f_{k'}:M_{k'}\rightarrow M_k$ is smooth. This defines $f$ uniquely so that $f\circ i_k = f_k=id_{M_k}$, which proves that $i_k$ is an embedding. Also a consequence is that $f = i_k^{-1}$ or $f^{-1} = i_k$ (on the appropriate restriction) which proves that $i_k$ is open as well. 
\end{proof}

\begin{fact}\label{fact:tangentProjSplitEpi}\label{fact:naturalityOfProjTangentBundle}
The canonical projection $\t_M:TM\rightarrow M$ of the tangent bundle is a split epimorphism (\cref{def:splitEpi}) and natural. 
\end{fact}
\begin{proof}
	There is canonical zero section $s:M\rightarrow TM$ sending $x\in M$ to the zero vector $0\in T_xM$.  Naturality restates that $Tf_x:T_xM\rightarrow T_{f(x)}N$ for map of manifolds $f:M\rightarrow N$.\end{proof}

 \begin{prop}\label{prop:productOpenEmbeddings}
 Let $i_M:M\hookrightarrow M'$ and $i_N: N\hookrightarrow N'$ be open embeddings.  Then the induced map $i_M\times i_N:M\times N\hookrightarrow M'\times N'$ is an open embedding. 
 \end{prop}
\begin{proof}
By definition, $i_M:M\hookrightarrow M'$ and $i_N:N\hookrightarrow N'$ are diffeomorphisms to their images. Therefore, $i_M\times i_N:M\times N\hookrightarrow M'\times N'$ is diffeomorphism to its image (\cref{prop:productIsosIsIso}). That $i_M\times i_N$ is open follows immediately from the definition of product topology. \end{proof}

\subsection{Continuous-Time Dynamical Systems}

\begin{definition}\label{def:ctDySys}
	For us, a  continuous-time dynamical system is a pair $(M,X)$ where $M$ is a smooth manifold and $X\in \Xf(M)$ a smooth vector field on $M$. 
\end{definition}

\begin{definition}\label{def:relatedVectorFields} Let $f:M\rightarrow N$ be a map of manifolds.  We say that vector fields $X\in \Xf(M)$ and $Y\in \Xf(N)$ are $f$-\textit{related} if $Tf\circ X = Y\circ f$.	
\end{definition}

\begin{definition}\label{def:ctDySysMorphism}
	Let $(M,X)$ and $(N,Y)$ be two continuous-time dynamical systems.  A morphism $(M,X)\xrightarrow{f}(N,Y)$ of systems is a map $M\xrightarrow{f}N$ of manifolds such that $(X,Y)$ are $f$-related (\cref{def:relatedVectorFields}). 
\end{definition}

\begin{definition}\label{def:solutionToDynamicalSystem}
Let $(M,X)$ be a continuous-time dynamical system.  A  \textit{solution} $\ph_{X}$ of $(M,X)$---also called \textit{integral curve}---is a map $\ph_{X}:(-\e,\e)\rightarrow M$, for some $\e>0$, such that $\frac{d}{dt}\ph_{X}(t) = X(\ph_{X}(t))$ for all $t\in (-\e,\e)$. 

A solution may have non-symmetric domain $(-\d,\e)$,  and we say that $\ph_X$ is maximal if the domain may not be extended, i.e.\ if there is no $(-\d',\e')\supsetneq (-\d,\e)$ for which $\psi_X:(-\d',\e')\rightarrow M$ is an integral curve. \end{definition}

\begin{remark}
	There are systems  whose maximal solution has domain  all of  $\R$.  Such solutions are said to be \textit{complete}. \end{remark}
	
	 Every dynamical system $(M,X)$ has  solutions.  Moreover, solutions are usually said to be unique, with the specification of initial condition $\ph_X(0) = x_0\in M$.

\begin{theorem}\label{theorem:E&U}
	Let $(M,X)$ be a dynamical system, and $x_0\in M$. Then there is $\e>0$ for which a smooth map $\ph_{X,x_0}:(-\e,\e)\rightarrow M$ is the solution to $(M,X)$ with initial condition $x_0$.   In other words,  $\ph_{X,x_0}(0) = 0$ and $\frac{d}{dt} \ph_{X,x_0}(t) = X(\ph_{X,x_0}(t))$.  We say ``the solution'' because this map is unique: if smooth map  $\psi:(-\e,\e)\rightarrow M$ satisfies $$
		\begin{array}{lll}\psi(0) = x_0 & \mbox{and} & \frac{d}{dt} \psi(t) = X(\psi(t))\;\,\forall \, t\in(-\e,\e),\end{array}$$
	then $\psi = \ph_{X,x_0}$. 
\end{theorem}
\begin{proof}
	See \cite[\S 14.3]{tu}.
\end{proof}

An equivalent defintion of integral curves: 
\begin{definition}\label{def:integralCurveAsMap}
	Let $(M,X)$ be a continuous-time dynamical system.  A \textit{solution}  (or \textit{integral curve}) \textit{of system} $(M,X)$  is a map $\ph_{X,x_0}:((-\e,\e),\frac{d}{dt})\rightarrow (M,X)$ of dynamical systems from the dynamical system $((-\e,\e),\frac{d}{dt})$ with constant vector field $\frac{d}{dt}\in \Xf(\R)$ sending $t\mapsto 1\in T_t\R$. 
\end{definition}
Equivalence of \cref{def:solutionToDynamicalSystem} and \cref{def:integralCurveAsMap} follows from \cref{def:ctDySysMorphism}, since $X\circ \ph_{X,x_0} = T\ph_{X,x_0} \left(\frac{d}{dt}\right) = \frac{d}{dt} \ph_{X,x_0}$.

\begin{definition}\label{def:SSub}
	A smooth map $p:M'\rightarrow M$ of manifolds is said to be a \textit{surjective submersion} if $p:M'\twoheadrightarrow M$ is surjective, as a map of sets, and the differential $Tp_x:T_xM'\twoheadrightarrow T_{p(x)}M$ is surjective at each point $x\in M'$. 
\end{definition}

\begin{definition}\label{def:openDySys}
   We define an \textit{open system} $(M'\xrightarrow{p}M ,X)$ as a pair where $p:M'\rightarrow M$ is a surjective submersion  of manifolds (\cref{def:SSub}) and $X:M'\rightarrow TM$ is a smooth map of manifolds such that $\t_M\circ X = p$, where $\t_M:TM\rightarrow M$ is the canonical projection of the tangent bundle.
\end{definition}

\begin{definition}\label{def:openDySysMorphism}
	Let $(M'\xrightarrow{p_M}M, X)$ and $(N'\xrightarrow{p_N}N,Y)$ be two open systems (\cref{def:openDySys}).  A \textit{map} $(p_M,X)\xrightarrow{f} (p_N,Y)$ \textit{of open systems} is a pair of maps $(f':M'\rightarrow N', f:M\rightarrow N)$ such that \begin{equation}\label{eq:openRelatedness} Y\circ f' = Tf\circ X. 	
 \end{equation}
 Whenever two open systems $X:M'\rightarrow TM$ and $Y:N'\rightarrow TN$ satisfy the equality \eqref{eq:openRelatedness}, we say that $(X,Y)$ are $f$-\textit{related} (compare with relatedness of vector fields \cref{def:relatedVectorFields}). 
\end{definition}
We will consolidate this definition.  First a few more. 

\begin{definition}\label{def:SSubMorphism}
	Let $M'\xrightarrow{p_M}M$ and $N'\xrightarrow{p_N}N$ be two surjective submersions (\cref{def:SSub}).  We define a \textit{morphism} $f:p_M\rightarrow p_N$ \textit{of surjective submersions} to be a pair $f':M'\rightarrow N'$ and $f:M\rightarrow N$ of manifolds such that $p_N \circ f' = f \circ p_M$. 
\end{definition}
\begin{remark}\label{remark:SSubIsACategory}
	It is easy to see that surjective submersions and morphisms of surjective submersions form a category $\sF{SSub}$, which is in fact a full subcategory of $\sF{Arrow(Man)}$.  We will henceforth denote surjective submersions by $M\defeq \big(M_{tot}\xrightarrow{p_M}M_{st}\big)$.  The domain $M_{tot}$ is called the \textit{total} space and codomain $M_{st}$ is the \textit{state}.  A morphism $f:M\rightarrow N$ consists of maps $f_{tot}:M_{tot}\rightarrow N_{tot}$ and $f_{st}:M_{st}\rightarrow N_{st}$. 
\end{remark}
\begin{definition}\label{def:relatedControl}
	Let $(M,X)$ and $(N,Y)$ be open systems (\cref{def:openDySys}).  We say that $(X,Y)$ are $f$-\textit{related} if $Tf_{st}\circ X = Y \circ f_{tot}$. 
\end{definition}
\begin{remark}
	In particular, $f:M\rightarrow N$ defines a morphism of surjective submersions. 
\end{remark}
We now redefine morphisms of surjective submersions (\cref{def:openDySysMorphism}):
\begin{definition}\label{def:openDySysMorphism2}
	Let $(M,X),\; (N,Y)$ be two open systems (\cref{def:openDySys}, \cref{remark:SSubIsACategory}).  We define a \textit{morphism} $(M,X)\xrightarrow{f}(N,Y)$ \textit{of open systems} to be a morphism $f:M\rightarrow N$ of surjective submersions (\cref{def:SSubMorphism}) such that $(X,Y)$ are $f$-related (\cref{def:relatedControl}). 
\end{definition}

\subsubsection{Existence and Uniqueness in Category of Elements}
Recall existence and uniqueness, which says that every continuous-time dynamical system $(M,X)$ and choice of initial condition $x_0\in M$ determines a  unique integral curve passing through $x_0$ at time $0$.   
\begin{definition}\label{def:completeDySys}
	A \textit{complete (continuous-time)  dynamical system} $(M,X)$ is a pair where $M$ is a manifold and $X\in \Xf(M)$ is a smooth vector field on $M$ such that for each initial condition $x_0\in M$, there is a \textit{complete} integral curve, i.e.\ a map $\ph_{X,x_0}:\R\rightarrow M$ satisfying $\frac{d}{dt} \ph_{X,x_0}(t) = X(\ph_{X,x_0}(t))$ for all $t\in R$. 
	
	\textit{Morphisms of complete dynamical systems} are the same as maps of dynamical systems: $(M,X)\xrightarrow{f} (N,Y)$ is a \textit{morphism} if $f:M\rightarrow N$ is a smooth map of manifolds and $(X,Y)$ are $f$-related, i.e.\ $Tf\circ X = Y \circ f$. 
\end{definition}
Complete dynamical systems and their morphisms form a category, which we here denote by $\sF{DySys}$.  It is a full subcategory of the category of dynamical systems whose objects may not have complete integral  curves. 

In the category $\sF{DySys}$, existence and uniqueness can be formulated in Yoneda categorical dress: 

\begin{prop}\label{prop:existenceAndUniquenessRepresentable} 
	The forgetful functor $\upsilon:\sF{DySys}\rightarrow\sF{Set}$---sending continuous-time dynamical system $(M,X)\mapsto \{x\in M\}$ to the underlying set---is representable. 
\end{prop}
\begin{proof}
	We argue that the element $\big((\R,\frac{d}{dt}),0\big) \in \disg_{\sF{DySys}}\upsilon$ is initial in the category of elements. By assumption, given dynamical system $(M,X)$ and element $x_0\in \upsilon (M)$, there is morphism $f:(\R,\frac{d}{dt})\rightarrow (M,X)$ with $f(0) = x_0$ (existence).  In fact, there is only one (uniqueness) (\cref{theorem:E&U}).  This proves that $ \big((\R,\frac{d}{dt}),0\big)$ is initial in the category of elements, and hence that $\upsilon:\sF{DySys}\rightarrow\sF{Set}$ is representable (\cref{prop:functorRepresentableIFFcategoryOfElementsHasUniversal}). 
\end{proof}

\subsection{Networks of Open Systems}
This section is review of \cite{lermanopennetworks}.

\begin{definition}\label{def:SSubInterconnection}
We say that morphism $M\xrightarrow{f}N$ of surjective submersions  (\cref{def:SSubMorphism}, \cref{remark:SSubIsACategory})	is an \textit{interconnection} if $f_{st}:M_{st}\xrightarrow{\sim} N_{st}$ is a diffeomorphism of manifolds.
\end{definition}

We now discuss networks of systems.  We introduced the category of lists in \cref{subsection:MonoidalCat}.  Now we assign content to this category as well as interpret in the context of dynamical systems. 

Let $\{M_\fx\}_{\fx\in \fX}$ be a collection of surjective submersions.  We define the product $M\defeq \discatp_{\fx\in \fX}M_\fx$ as follows: $M_{tot} \defeq \discatp_{\fx\in \fX}M_{\fx,tot}$, and similarly $M_{st} \defeq \discatp_{\fx\in \fX}M_{\fx,st}$.  The submersion  $M_{tot}\xrightarrow{p_M}M_{st}$ is uniquely induced by the universal property (alternatively, since $\times$ is a functor (\cref{prop:productBifunctor})).

Let $\fX$ be a finite set and $\Sc_\fX:\fX\rightarrow\sF{SSub}$ assign a surjective submersion $\Sc_\fX(\fx)$ to each $\fx\in \fX$.  
\begin{definition}\label{def:networkOpenSystems}
	A (concrete) \textit{network of open systems} $\left(\Sc_\fX:\fX\rightarrow\sF{SSub}, \i:M\rightarrow\discatp_{\fx\in \fX} \Sc_\fX(\fx)\right)$ is a pair, where $\Sc_\fX:\fX\rightarrow\sF{SSub}$ assigns to each object $\fx\in \fX$ a surjective submersion $\Sc_\fX(\fx)$ and $\i:M\rightarrow\discatp_{\fx\in \fX}\Sc_\fX(\fx)$  is an interconnection of surjective submersions (\cref{def:SSubInterconnection}).
\end{definition}

\begin{definition}\label{def:networkOpenSystemsMorphism}
	Let $\left(\Sc_\fX:\fX\rightarrow\sF{SSub}, \i_\fX:M\rightarrow\discatp_{\fx\in \fX} \Sc_\fX(\fx)\right)$ and $\left(\Sc_\fY:\fY\rightarrow\sF{SSub}, \i_\fY:N\rightarrow\discatp_{\fy\in \fY} \Sc_\fY(\fy)\right)$ be two networks of open systems.  A \textit{morphism} $$((\ph,\Phi),f):\left(\Sc_\fX:\fX\rightarrow\sF{SSub}, \i_\fX:a\rightarrow\discatp_{\fx\in \fX} \Sc_\fX(\fx)\right)\rightarrow \left(\Sc_\fY:\fY\rightarrow\sF{SSub}, \i_\fY:b\rightarrow\discatp_{\fy\in \fY} \Sc_\fY(\fy)\right)$$\textit{of networks of open systems} is a pair where $(\ph,\Phi):(\fX\rightarrow\sF{SSub})\rightarrow(\fY\rightarrow\sF{SSub})$ is a morphism of lists of surjective submersions (a map $\ph:\fX\rightarrow\fY$ of finite sets with smooth map $\Phi_\fx:\Sc_\fY(\ph(\fx))\rightarrow\Sc_\fX(\fx)$ of surjective submersions for each $\fx\in \fX$ \cref{def:categoryOfLists}) compatible with each interconnection, i.e.\ $\i_\fX\circ f = \Pi(\ph,\Phi) \circ \i_\fY$. 
\end{definition}
Recall the definition of control (\cite[\S2]{lermanopennetworks}):
\begin{definition}\label{def:controlOnSubmersions}
	Let $M_{tot}\xrightarrow{p_M}M_{st}$ be a surjective submersion.  We define $$\sF{Crl}(p_M)\defeq\big\{X:M_{tot}\rightarrow TM_{st}:\, p_M = \t_{M_{st}}\circ X\big\},$$ where $\t_M:TM\rightarrow M$ is the canonical projection of the tangent bundle. 
\end{definition}
The main result \cite[Theorem 9.3]{lermanopennetworks}: 
\begin{theorem}\label{theorem:mainTheoremLermanOpenNetworks}
	A morphism \small $$\big((\ph,\Phi),f\big):\left(\Sc_\fX:\fX\rightarrow \sF{SSub}, \i_\fX:a\rightarrow\discatp_{\fx\in \fX}\Sc_\fX(\fx)\right)\rightarrow\left(\Sc_\fY:\fY\rightarrow\sF{SSub},\i_\fY:b\rightarrow\discatp_{\fy\in \fY}\Sc_\fY(\fy)\right)$$ induces a 1-morphism $$\begin{tikzcd}[column sep = large, row sep = large]
	\discatp_{\fy\in \fY}\sF{Crl}(\Sc_\fY(\fy))\arrow[r,swap,""{name = foo1, below}]\arrow[d] & \discatp_{\fx\in \fX}\sF{Crl}(\Sc_\fX(\fx))\arrow[d]\\
	\sF{Crl}(b)\arrow[r,""{name = foo2,above}] & \sF{Crl}(a)
	\ar[shorten <= 15pt, shorten >= 15pt, Rightarrow, from = foo1, to = foo2]
\end{tikzcd}
$$ in $\sF{Set}^\square$. 
\end{theorem}

We will reproduce this statement in a much more general setting in \cref{ch4}. We interpret it now as saying the following: a pair of collections of open systems which are pairwise related induce a pair of related open systems.  We refer the reader to \cite{lermanopennetworks} for more details in the context of continuous-time systems.  At this point, we state the fact for reference, but we will explain the intuition and details of (a version of) its proof throughout this thesis.

\chapter{Hybrid Systems}\label{ch3}

\section{Introduction} 
We develop a categorical study of hybrid systems.  The reason for doing so is twofold.  First, working in the categorical setting allows us to isolate  which concepts are specifically hybrid from those which are not.  For example, many formulations of hybrid systems loosely consider them to be dynamical systems which exhibit both continuous and discrete (discontinuous) behavior.  As we saw previously, a continuous-time system $(M,X)$ consists of a space and  a vector field which specifies the behavior of dynamics in said space.  We will similarly define a hybrid version of space, in or over which it will make sense to speak of a dynamics-governing object. 

Here is a concrete way the category theory arises. After constructing hybrid phase spaces, we observe that there is additionally a notion of morphism or map between hybrid phase spaces. The collection of hybrid phase spaces and morphisms forms a category, and there is a functor from the category of hybrid phase spaces to the category of manifolds.  One way of defining a hybrid system is as a pair $(a,X)$ where $a$ is a hybrid phase space and $X$ is a vector field on the underlying manifold.   According to this definition, the only exclusively ``hybrid'' aspect of hybrid systems is the underlying space!  At first glance, this framing seems counterintuitive, but it does not tell the whole story.  For maps of hybrid systems are first and foremost maps of hybrid phase spaces, satisfying other conditions, conditions which specify coherence of dynamics.   For example, in the continuous-time case,  integral curves---trajectories of continuous-time systems---are simply  special  maps of dynamical systems (namely, ones from an interval $(-\e,\e)$ with constant vector field $\frac{d}{dt}$), or maps of manifolds which preserve the dynamics.  We will define an analogous  class of hybrid phase spaces representing hybrid time, and then executions---the hybrid version of integral curve---as a map from a hybrid phase space in this class, together with some underlying dynamics representing the  passage of time. In this way, we mirror the theory of dynamical systems, while minimizing the formal modifications to make this theory hybrid.  Hybrid phase space is the principle distinguishing feature of hybrid concepts, but its consequences are far reaching.

 The second benefit of the categorical approach is that once we suitably formalize each hybrid notion, we may import results on networks of continuous-time systems to the hybrid setting without extra ad hoc maneuvering. The value of this approach will be especially apparent when we develop the abstract theory of systems in \cref{ch4}.  To recapitulate, the first benefit of category theory is that it helps us clearly define concepts at an  appropriate  level of generality and abstractness. Measures of ``appropriate'' include both how well resulting definitions capture intuition and how much extra work is needed to translate similar results from similar domains.  Which leads to the second benefit: the theory of networks of systems we use is fundamentally categorical. As such, we expect a category-theoretic version of hybrid systems to fit in with this theory of networks, without requiring us to ``reinvent the wheel.'' 
 
  We enumerate the ideas we need; some are needed merely to make sense of the statement of \cref{theorem:mainTheoremFordeterministicHybridSystems}, and some are required to prove  this theorem, a task we undertake in a general setting in \cref{ch4}.
 \begin{enumerate}
  \item We start by defining  non-deterministic hybrid notions.  Taking our cue from the theory of continuous-time dynamical systems, we first work out a notion of hybrid phase space as the basic building block of other hybrid  constructions, which are ``hybrid spaces with other data.'' In the non-deterministic setting, ``other data'' will generally be defined---using the functorial approach---in the category of manifolds. 
 \begin{enumerate}
 	\item We develop the notion of hybrid phase space as collection of manifolds and collection of relations. These data are indexed by  nodes and edges of a directed graph, respectively.  To nodes we associate manifolds and to edges we associate relations between the source and target node manifolds. An element of a relation is a pair, whose first member we think of as a point ``before'' a jump and whose second member is the point ``after'' jump. Our formalism is actually carried out in terms of (double) categories (which may be realized as path categories of directed graphs).  The abstractness is used both for packaging (to circumvent enumerating an unwieldy list of data and conditions) as well as to represent phenomena we intuitively expect hybrid behavior to have; we will elaborate on these properties momentarily. 
 	 	\item We develop a notion of map of hybrid phase space, and show that hybrid phase spaces with their morphisms form a category.  Moreover, there is a way to recover an underlying manifold from a hybrid phase space. We take the coproduct of manifolds indexing over nodes, and show that this operation---denoted by $\Ub$---is functorial. 
 	\item We use the functor $\Ub$ from (the category of) hybrid phase spaces to the category of manifolds to define a hybrid system $(a,X)$ as a pair where $a$ is a hybrid phase space and $X$ is a vector field on the underlying manifold $\Ub(a)$.  We extend the notion of morphism of hybrid phase spaces to that of hybrid systems, by importing the analogous notion from continuous-time dynamical systems: a map of hybrid systems is a morphism of hybrid phase spaces for which the vector fields on underlying manifolds are map-related. 
 	\item We define a notion of surjective submersion in the hybrid setting. Functoriality of $\Ub:\sF{HyPh}\rightarrow\sF{Man}$ makes another appearance: we define a hybrid surjective submersion as a morphism $f:a\rightarrow b$ of hybrid phase spaces such that $\Ub(f):\Ub a\rightarrow\Ub b$ is a surjective submersion in the category of manifolds. We also define maps of hybrid surjective submersions and observe that these morphisms with their objects form a category.  
 	\item We show that the category of hybrid phase spaces has products, and use this in two ways: (1) to show that along with the terminal hybrid phase space, the category $\sF{HyPh}$ is cartesian monoidal and (2) to provide nontrivial examples of hybrid surjective submersion:  the projection maps $p_a:a\times b\rightarrow a$ and $p_b:a\times b\rightarrow b$ are both hybrid surjective submersions. We will see in products a nontrivial consequence of having defined hybrid phase spaces categorically.  Unit edges at nodes correspond to diagonal relations, a consequence of which  for products corresponds to a decoupling of relations. Concretely, two systems considered together need not have simultaneous (discrete) state transitions.  
 	\item We extend the notion of open system to that of hybrid open system: a pair $(a,X)$ such that $a$ is a hybrid surjective submersion (which itself is a map $\sF{dom}(p_a)\xrightarrow{p_a}\sF{cod}(a)$ for which $\Ub(p_a)$ is a surjective submersion) such that  $(\Ub(a),X)$ is an open system (in the category of manifolds). 
 	 \end{enumerate}
 	 \item We next turn to determinism in hybrid  systems. As before, hybrid phase spaces appear in each notion, along with other data.  In the non-deterministic setting, we defined a hybrid category $\sF{HyC}$ as some object or morphism in the category of hybrid phase spaces together with some map---or otherwise satisfying some condition---in the category of manifolds.  In the deterministic setting, we work instead in the category of sets.  
 	 \begin{enumerate}
 	 	\item The key ingredient which allows us to make sense of determinism is the  \textit{continuous-discrete} (or \textit{c.d.}) bundle, defined over hybrid phase space $a$ as $\Tb a\defeq T\Ub a\times \Ub a$, the product of the tangent bundle on the underlying manifold and the  underlying manifold itself.  On its own, this construction is vacuous. We extract usefulness  from it by taking sections: maps $(X,\r):\Ub a\rightarrow\Tb a$ sending a point $x\in \Ub a$ to a pair of points $X(x)\in T_x\Ub a$ and $\r(x)\in \Ub(a)$.  We require that the first component $X$ of this section varies smoothly with $x\in \Ub a$ and that $(x,\rho(x))$ is an element of relation $a(\g_x)$ for some edge $\sF{dom}(\g_x)\xrightarrow{\g_x}\sF{cod}(\g_x)$. As before, $X$ represents continuous-time dynamics; now  $\r$ indicates discrete behavior.  The requirement that $(x,\r(x))\in a(\g_x)$ expresses a constraint that relations of the hybrid phase space impose on  possible jumps. That $\r:\Ub a\rightarrow\Ub a$ is a \textit{function} illustrates where determinism arises: each point of $\Ub a$ jumps to a specified point, even if that point is to itself.  The (everywhere) possibility of sending a point to itself is another consequence of our categorical definition of hybrid phase space (that each object has identity morphism).
 	 	\item We use determinism to define a hybrid analog of integral curve, \textit{executions}. We first demarcate a special class of deterministic  hybrid systems, as follows.  We start by  fixing an increasing set of points $t_0< t_1< t_2< \cdots $ in $\R$; these will be transition times.  We consider  the disjoint union $M\defeq \disu_{i\in \N}[t_i,t_{i+1}]\times\{i\}$ of closed intervals defined by this sequence as the underlying manifold of the hybrid phase space. Essentially we are cutting  $\R$ into countably many snippets.  On $[t_i,t_{i+1}]\times\{i\}$, we attach constant vector field $\frac{d}{dt}$, and we define jump map $\r:M\rightarrow M$ by $$\r(t,i) = \left\{\begin{array}{ll} (t,i) & \mbox{if}\; t\in [t_i,t_{i+1})\times \{i\}\\
(t,i+1) & \mbox{if}\; t = t_{i+1}\end{array}\right..$$	 At each moment (time), $t$ both flows in the positive direction at constant rate, while \textit{also} ``jumping'' to itself, unless we have hit the endpoint of the interval, at which point time jumps to the next interval.   Having thus defined this special class of deterministic  hybrid systems, we define an execution of  deterministic  hybrid system $(a,X,\r)$ as a map from one system in the special class to this one. This extension of  the notion of integral curve as map $(\R,\frac{d}{dt})\rightarrow (M,X)$ provides a concrete rationale for our notion of map of deterministic  hybrid system.  

 	 	\item Much of what follows is variation on a theme.  Using hybrid surjective submersions, we may define a deterministic  hybrid \textit{open} system as  a map $(X,\r):\Ub a_{tot}\rightarrow \Tb a_{st}$ where   
 	 	$p_a:a_{tot}\rightarrow a_{st}$ is a hybrid surjective submersion.  We impose compatibility conditions on $X$ and $\r$. For $X$, we ask that $\t_{\Ub a_{st}} \circ X = \Ub(p_a)$, where $\t_M:TM\rightarrow M$ is the canonical projection of the tangent bundle.  For the jump map $\r$, we ask  for each $x\in \Ub a_{tot}$ that there is some edge $\g_x\in \Sb^{a_{st}}$ such that $\left(\Ub(p_a)(x), \r(x)\right)\in a(\g_x)$. This somewhat opaque condition expresses both the jump constraint from relations of the phase space together with compatibility of the hybrid surjective submersion.  
 	 	\item We next define a deterministic-control sections functor $\dCrl:\sF{HySSub}\rightarrow \sF{Set}$ sending a hybrid surjective submersion to its collection of deterministic  hybrid open systems. There are two ways to define this functor on morphisms.  In one direction,  for morphism $a\xrightarrow{f} b$ of hybrid surjective submersions, we define $\dCrl(f)$ as a relation, pairs of deterministic  control on $a$ and on $b$ which are $f$-related. Relatedness of deterministic  control is analogous to relatedness of vector fields. This functor is lax and  maps into the category $\sF{Rel}$.  In another direction, we consider the sub-collection of morphisms $a\xrightarrow{f}b$ of hybrid surjective submersions  which are isomorphisms on state ($f_{st}:a_{st}\xrightarrow{\sim}b_{st}$), and  define $\dCrl(f):\dCrl(b)\rightarrow \dCrl(a)$ as a pullback.  This assignment is (strictly)  functorial. 
 	 	\item We present a slew of examples demonstrating how deterministic  control interacts with interconnection.  These examples are networks, in a sense similar to that defined for the continuous-time open system case: an indexed collection of hybrid surjective submersions together with an interconnection map to the product. 
 	 	\item This construction is justified in part by the result of applying deterministic  control: a collection of control on each hybrid surjective submersion induces a deterministic  control on the product.  
 	 	\item We also show that the functor $\Ub:\sF{HyPh}\rightarrow\sF{Man}$ is compatible with taking products.  In other words, $\Ub(a\times b) \cong \Ub a\times \Ub$, natural in $a$ and $b$.  This is a technical fact required for proving \cref{theorem:mainTheoremFordeterministicHybridSystems}, which we prove in \cref{ch4} as a corollary from a more abstract version of the same theorem. 
 	 	\end{enumerate}
 	 	\item 
These are the main ingredients we need for our main theorem (\cref{theorem:mainTheoremFordeterministicHybridSystems}), which says that a morphism of networks of deterministic hybrid open systems induces 1-morphism in $\sF{Set}^\square$.  This theorem encodes the idea that a collection of morphisms of subsystems induces a morphism of the interconnected (networked) systems. 
\end{enumerate}

\section{Hybrid Phase Spaces and Hybrid Systems}

We start with hybrid phase spaces.  They will arise in each version of hybrid system we use. 

As we develop these notions, it may be worthwhile to keep in mind a few classic  examples of hybrid systems.  We will discuss the example of a room whose temperature is regulated by a thermostat, which is driven up by a heater when on and falls due to a lower ambient temperature when off.  We will also consider a bouncing ball, whose position and velocity are continuous, except at the moment of impact, at which point the height remains the same, but velocity jumps discontinuously.  We will discuss these examples, and isolate which aspects from the standard theory of hybrid systems correspond with the notions we are abstractly formalizing. 
\subsection{Hybrid Phase Spaces}\label{subsection:hybridPhaseSpaces}
\begin{definition}\label{def:hybridPhaseSpace0}
	We define a \textit{hybrid phase space} $a:\Sb^a\rightarrow\sF{RelMan}$ to be a functor from discrete double category $\Sb^a$  (\cref{def:doubleCat}) to double category $\sF{RelMan}$ (\cref{def:relManDouble}). We call $\Sb^a$ the \textit{source category} of $a$. 
\end{definition}

Thus,  for each $0$-object $\fs\in \Sb^a_0$, $a(\fs)$ is a manifold with corners, and for each $1$-object $\fs\xrightarrow{\g}\fs' \in \Sb^a_1$, $a(\g) \subseteq a(\fs)\times a(\fs')$ is a relation. Recall that we think of elements of relations as jump points before and after. 

For all intents and purposes, we may think of a hybrid phase space as a directed graph, to whose nodes we assign  manifolds, and to whose edges we assign relations between source and target manifolds.  We will need some additional properties, such as unit relations at each node, which follows automatically when we define hybrid phase spaces as functors.  We may refer to objects of the object category $\Sb^a_0$ as \textit{$0$-objects} or \textit{nodes} and objects of the $1$-category $\Sb^a_1$ as \textit{$1$-objects} or \textit{edges}.

\begin{example}\label{ex:phaseSpaceForBall}
	Imagine a bouncing ball, whose state is represented by height $h$ and velocity $v$.  Velocity is unconstrained, but height is constrained to be above the ground, i.e.\ $h\geq 0$.  We will describe the dynamics portion of  this hybrid system in \cref{ex:bouncingBallVanilla}, but for now we describe the phase space.  Let $c:\Sb^c\rightarrow\sF{RelMan}$ have source category $\Sb^c$ given by $\begin{tikzcd}
	0 \arrow[loop left,"e"]
\end{tikzcd}$ assigning $c(0) \defeq \R^{\geq0 }\times \R$ (with coordinates $(h,v)$) and relation $c(e)\defeq\big\{(h,v,h',v')\in c(0)^2:\, h'= h = 0,\; v'\cdot v < 0\big\}$. Following the idea that relations constrain jumps, we see that at the ground, the jump velocity $v'$ and velocity $v$ have opposite signs (excluding the unit jump $(0,v)\mapsto (0,v)$). Everywhere else, the only possible jump is $(h,v)\mapsto (h,v)$.  
\end{example} 

\begin{remark}
	We say that relations \textit{constrain jumps} and not that they \textit{determine them}. Thus in \cref{ex:phaseSpaceForBall}, we did not specify what $v'$ is, except that it may not have the same sign as $v$.  In the context of determinism (\cref{def:detHySys}, \cref{def:deterministicHyOS}), we will see clearly where jumps arise and how relations play a role in  constraining them.
\end{remark}

\begin{example}\label{ex:phaseSpaceForThermostat}
	We consider another common example, a room whose temperature is regulated by a thermostat which discretely turns a heater on and off. Let hybrid phase space $c:\Sb^c\rightarrow\sF{RelMan}$ have source $\Sb^c$ given by $\begin{tikzcd}
0\arrow[r,bend left,"e_{1,0}"] & 1\arrow[l, bend left, "e_{0,1}"] 	
\end{tikzcd}$ to whose nodes $i$ we assign manifold $c(i)\defeq \R\times\{i\}$ and to whose edges $e_{i,1-i}$ we assign relation $c(e_{i,i-1})= \big\{(t,i-1,t,i):\, (-1)^{i-1}t\geq 1\big\}$. We will interpret this relation in \cref{example:ThermostatVanilla} as saying ``turn heater on or off'' when the temperature is beyond a threshold, which thresholds here are set at $-1$ and $1$, respectively.
\end{example}

There are also maps between hybrid phase spaces.
\begin{definition}\label{def:hybridPhaseSpaceMorphism}
	We define a \textit{morphism} $(\ph,\ff):a\rightarrow b$ \textit{of hybrid phase spaces} with the following data:
	\begin{enumerate}
		\item a functor $\ph:\Sb^a\rightarrow \Sb^b$, as defined in \cref{def:functorDoubleCat}.
		\item a natural transformation $\ff:\Sb^a\Rightarrow \Sb^b\circ \ph$, as defined in \cref{def:naturalTransformDoubleCat}. 
	\end{enumerate}
\end{definition}
We may write such a morphism as  $$\begin{tikzcd}[row sep=3em]
\Sb^{a} \arrow{r}{\ph } \arrow[""{name=foo}]{dr}[swap]{a} & \Sb^{b}\arrow{d}{b} \arrow[Rightarrow, from=foo, swap, "\ff"] \\ & \sF{RelMan}\end{tikzcd}$$\normalsize  or condense the functor and natural transformation into one letter as $f= (\ph,\ff)$ when we do not explicitly work with both pieces of data.  Expanding on condition 2 in \cref{def:hybridPhaseSpaceMorphism}, we require a smooth map of manifolds $\ff_\fs:a(\fs)\rightarrow b(\ph(\fs))$ for each node $\fs\in \Sb^a_0$ such that there is inclusion of relations $$\big(\ff_{\fs}\times \ff_{\fs'}\big) \big(a(\g)\big) \subseteq b(\ph(\g))$$ for each edge $\fs\xrightarrow{\g} \fs'\in \Sb^a_1$. 

Having defined a objects (\cref{def:hybridPhaseSpace0}) and morphisms (\cref{def:hybridPhaseSpaceMorphism}), we observe the \textit{category of hybrid phase spaces}.
\begin{lemma}\label{lemma:HyPhisaCategory}\label{def:categoryHyPh}
Hybrid phase spaces and their morphisms form a category  $\sF{HyPh}$. 
\end{lemma}
\begin{remark} Though the formalism is straightforward, we outline the details for \cref{lemma:HyPhisaCategory}.  For future concepts where we introduce a notion of object and morphism, we will state without explicit argument  that they form a category (e.g.\ hybrid systems (\cref{def:hybridSys}), hybrid surjective submersions (\cref{def:hybridSurjectiveSubmersion}), etc.). 
\end{remark}
\begin{proof}
  We must show that  each object has an identity morphism  and that composition of morphisms is associative.  For hybrid phase space $a:\Sb^a\rightarrow\sF{RelMan}$, the identity morphism $id_a\defeq (id_{\Sb^a}, \one_a)$ consists of the identity functor $id_{\Sb^a}:\Sb^a\rightarrow\Sb^a$  and  natural transformation   $\one_a:a\Rightarrow a$,  defined for each node $\fs\in \Sb^a_0$ by $\one_{a,\fs}\defeq id_{a(\fs)}$. It is easy to verify that $(\a,\af) = (\a,\af)\circ id_a=id_b\circ (\a,\af)$ for every morphism $(\a,\af):a\rightarrow b$.  
  
We now check  that composition is associative.  For a sequence of morphisms $$a\xrightarrow{(\a,\af)}b\xrightarrow{(\b,\bff)}c\xrightarrow{(\g,\gf)}d,$$ since composition of functors is associative $$\begin{tikzcd}
	\Sb^a\arrow[rr,bend left, "\b\circ\a"]\arrow[r,"\a"] & \Sb^b\arrow[r,"\b"]\arrow[rr,bend right,"\g\circ\b"] & \Sb^c\arrow[r,"\g"] & \Sb^d, 
\end{tikzcd}$$ we see that $(\g\circ \b)\circ\a = \g\circ(\b\circ \a)$. 

Composition of maps of manifolds is also associative, so for each node $\fs\in \Sb^a_0$,  $$\gf_{\b(\a(\fs))}\circ\left(\bff_{\a(\fs)}\circ \af_{\fs}\right) = \left(\gf_{\b(\a(\fs))}\circ \bff_{\a(\fs)}\right) \circ \af_\fs. $$ This shows that $$(\g,\gf) \circ \big((\b,\bff)\circ (\a,\af)\big) = \big((\g,\gf)\circ (\b,\bff)\big) \circ (\a,\af),$$ and hence that $\sF{HyPh}$  is  a category. \end{proof}

\subsection{Hybrid Closed Systems}\label{subsection:hybridClosedSystems}
Now we turn to dynamics on hybrid systems.  Here is where  category theory begins to make  an operative appearance.  We have constructed the category of hybrid phase spaces in \cref{subsection:hybridPhaseSpaces}. We will now construct a functor to the category of manifolds and interpret dynamics on a hybrid phase space as dynamics on the underlying manifold.

We start with the functor. 
\begin{prop}\label{prop:ForgetfulFunctor} 
	There is forgetful functor $\Ub:\sF{HyPh}\rightarrow \sF{Man}$ defined on objects by $\Ub a\defeq \discats_{\fs \in \Sb^a_0}a(\fs).$  
\end{prop}

\begin{remark}\label{remark:canonicalInclusionForU:HyPh2Man}
	Let $a$ be a hybrid phase space. Since $\Ub a$ is a coproduct,   there are canonical inclusions $in_{a(\fs')}:a(\fs')\hookrightarrow \Ub a$ (\cref{def:coproducts}).  
\end{remark}
\begin{proof}[Proof of \cref{prop:ForgetfulFunctor}]
We must define $\Ub$ on morphisms, and show that $\Ub$ preserves identity and composition. 	Let $f = (\ph,\ff): a\rightarrow b$ be a morphism of hybrid phase spaces. A map $\hat{f}:\Ub a\rightarrow \Ub b$ \textit{from} a coproduct $\Ub a = \discats_{\fs \in \Sb^a_0}a(\fs)$ is uniquely defined by collection of maps $\left\{a(\fs)\xrightarrow{f_\fs} \Ub b\right\}_{\fs\in \Sb^a_0}$, which satisfy $f_\fs = \hat{f}\circ in_{a(\fs)}$.  Since $\Ub b$ is itself a coproduct, it is equipped with canonical inclusions $in_{b(\ft)}: b(\ft) \hookrightarrow \Ub b$ (\cref{remark:canonicalInclusionForU:HyPh2Man}). For $\fs\in \Sb^a_0$, we   define $f_\fs\defeq in_{b(\ph(\fs))}\circ \ff_\fs: a(\fs)\rightarrow\Ub b$, where $\ff_\fs:a(\fs)\rightarrow b(\ph(\fs))$ (\cref{def:hybridPhaseSpaceMorphism}); this uniquely defines $\Ub f\defeq \big(\hat{f}:\Ub a\rightarrow \Ub b\big)$.

Functoriality on the identity (\cref{remark:functoriality})---namely  $\Ub(id_a) = id_{\Ub a}$---follows from the universal property of coproduct (\cref{def:coproducts}), since both maps $\Ub(id_a)$ and $id_{\Ub a}$  make the  diagram  $$\small\begin{tikzcd}[column sep = large]
	\discats_{\fs\in \Sb^a_0} a(\fs) \arrow[r,shift left,dashed,"\Ub(id_a)"]\arrow[r,shift right,swap,"id_{\Ub a}"] & 	\discats_{\fs\in \Sb^a_0} a(\fs)\\
	a(\fs')\arrow[u,hookrightarrow,"in_{a(\fs')}"] \arrow[ur,hookrightarrow,swap,"in_{a(\fs')}"] & 
\end{tikzcd}$$ \normalsize commute for every $\fs\in \Sb^a_0$. 

Functoriality on composition $\Ub(g \circ f) = \Ub(g)\circ\Ub(f)$ follows similarly, for morphisms $a\xrightarrow{f}b\xrightarrow{g}c$ of hybrid phase spaces.  The diagram $$\begin{tikzcd}
\Ub a\arrow[r,dashed,"\Ub f"]\arrow[rr,bend left,dashed,"\Ub\big(g\circ f\big)"] & \Ub b \arrow[r,dashed,"\Ub g"] & \Ub c\\
	a(\fs')\arrow[r,"\af_{\fs'}"]\arrow[u,hookrightarrow,swap,"in_{a(\fs')}"] \arrow[rr,bend right,"\bff_{\a(\fs')}\circ\af_{\fs'}"] & b(\a(\fs'))\arrow[r,"\bff_{\a(\fs')}"]\arrow[u,hookrightarrow,swap,"in_{b(\a(\fs'))}"] & c(\b(\a(\fs')))\arrow[u,hookrightarrow, swap,"in_{c(\b(\a(\fs')))}"]
\end{tikzcd}$$
is easily seen to commute. Since there is unique map $\Ub a\xrightarrow{\chi} \Ub c$ satisfying $\chi\circ in_{a(\fs')} = in_{c(\b(\a(\fs')))}\circ \bff_{\a(\fs')}\circ \af_{\fs'}$ for each $\fs'\in \Sb^a_0$, we conclude that $\chi = \Ub\big(g\circ f\big) = \Ub g\circ \Ub f$. 
\end{proof}

\begin{remark}
	The definition of functor $\Ub$ is  secretly  subsuming the composition of two functors. On the one hand, we forget edges and relations $\sF{HyPh}\rightarrow \big(\sF{Set/Man}\big)^\Rightarrow$ sending $(a:\Sb^a\rightarrow \sF{RelMan})\mapsto (a\restriction_{\Sb^a_0}:\Sb^a_0\rightarrow \sF{Man})$, and then we take the coproduct in $\sF{Man}$ over set $\Sb^a_0$, both of which operations are functorial  (\cref{remark:catListsCoproduct}). 
\end{remark}

\begin{remark}\label{fact:isoHyPhForgetsToDiffeo}
	Since $\Ub:\sF{HyPh}\rightarrow\sF{Man}$ is a functor (\cref{prop:ForgetfulFunctor}), an isomorphism $i:a\xrightarrow{\sim}b$ of hybrid phase spaces becomes a diffeomorphism $\Ub i:\Ub a\xrightarrow{\sim}\Ub b$ in $\sF{Man}$.  
\end{remark}

As we previously foreshadowed, the functor $\Ub$ may be used to define various hybrid concepts as a hybrid phase space with some data in the category of manifolds.  Our first instance of this idea is the notion of hybrid dynamical system.  

\begin{definition}\label{def:HySys}\label{def:hybridSys}
	We define a \textit{hybrid dynamical system} $(a,X)$   to be a pair where $a$ is a hybrid phase space and $X\in \Xf(\Ub a)$ is a vector field on the underlying manifold of hybrid phase space $a$. 
\end{definition}
\begin{example}\label{ex:bbAsNormalHybridSystem}
Recall the bouncing ball phase space from \cref{ex:phaseSpaceForBall}, with $c:\Sb^c\rightarrow\sF{RelMan}$. The hybrid system $(c,Z)$ is standard in literature (c.f.\ \cite{liberzonswitch}, \cite[\S1.2]{goebelhybrd}).  Dynamics  $Z\in \Xf\big(\Ub c\big) = \Xf\big(\R^{\geq 0}\times \R)$ are defined by  $Z(h,v) = v\pdiv{}{h} - \pdiv{}{v}$, where  we normalized the acceleration of gravity (coefficient of $\pdiv{}{v}$).   Notice that at boundary $\{h=0\}$, $Z(h,v)$ is an outward pointing vector.  
\end{example}
We remark  that jumps do not explicitly appear in the notion of hybrid dynamical systems. We have not yet said what happens when $(h,v) = (0,v)$, other than having defined a relation in the phase space $c(e)$.  We will address this point after we have acquired  the notion of execution (\cref{def:deterministicExecution}, \cref{ex:bouncingBallExecution}).

\begin{example}\label{ex:thermoAsNormalHybridSystem}
	Recall the room-with-thermostat from \cref{ex:phaseSpaceForThermostat}, with hybrid phase space $c:\Sb^c\rightarrow\sF{RelMan}$ which assigns space $\R\times\{i\}$  to node $i$.  We define $Z\in \Xf(\Ub c)$ by $Z(x,i)\defeq (-1)^{1-i}$. We interpret $i=0$ as ``heater off,'' and $i=1$ as  ``heater on.''  There is  supposed to be some intuition lurking in the background that the heater turns off when $x\geq 1$, but hybrid dynamical systems do not specify the mechanism for enacting this switch.  Again, we introduce the fix  in \cref{subsec:determinismHySys}.
\end{example}

We turn to  maps of hybrid systems. 
\begin{definition}\label{def:hybridSysMorphism}
	We define a \textit{map} $(a,X)\xrightarrow{f} (b,Y)$ \textit{of hybrid systems} to be a map $a\xrightarrow{f} b$ of hybrid phase spaces (\cref{def:hybridPhaseSpaceMorphism}) such that $\left(\Ub a,X\right) \xrightarrow{\Ub f} \left(\Ub b, Y\right)$ is a map of dynamical systems (\cref{def:ctDySysMorphism}).
\end{definition}

\begin{remark}\label{remark:relatednessInHyPhandinMan}
	Recall  that a map $f:(M,X)\rightarrow (N,Y)$ of continuous-time dynamical systems  is a map $f:M\rightarrow N$ of manifolds such that the respective vector fields $(X,Y)$ are $f$-related.  Thus according to \cref{def:hybridSysMorphism}, $(X,Y)$ are $\Ub f$-related. However, in the context of \textit{hybrid dynamical systems}, we will simply  say  that $(X,Y)$ are  $f$-related,  which we take to \textit{mean} that $(X,Y)$ (as vector fields on $\Ub a$ and $\Ub b$) are $\Ub f$-related. 
	\end{remark}

\begin{remark}\label{def:categoryHySys}
	It is easy to verify that hybrid systems (\cref{def:hybridSys}) and their morphisms (\cref{def:hybridSysMorphism}) form a category $\sF{HySys}$.
\end{remark}

\begin{remark}\label{remark:executionsForNondeterministic} As previously hinted, the concept of hybrid system seems suspiciously non-hybrid since discrete behavior makes no overt appearance in the definition. 	 Jumps for hybrid systems arise in executions.  We define executions in \cref{def:deterministicExecution} for \textit{deterministic} hybrid systems as a map from a (deterministic) hybrid version of time system. Determinism is not essential for defining executions, but our primary focus is determinism, so we delay development of this idea until \cref{subsec:determinismHySys}.  \end{remark}


\subsection{Hybrid Open Systems}
We now discuss hybrid open  systems.  Open systems are like ordinary systems which can take external input.  In control theory, external input is often user-defined.  External input could also be an external disturbance or noise.  In the context of networks, external input may be interpreted as  states of other (sub)systems.  

Recall that a continuous-time open system is a pair $(a_{tot}\xrightarrow{p_a}a_{st}, X)$, where $p_a:a_{tot}\rightarrow a_{st}$ is a surjective submersion (\cref{def:SSub}), and $X:a_{tot}\rightarrow Ta_{st}$ is a smooth map compatible with $p_a$---namely,  $p_a = \t_{a_{st}}\circ X$, where $\t_{a_{st}}:T_{a_{st}}\rightarrow a_{st}$ is the  canonical projection of the tangent bundle. To define hybrid systems (\cref{def:hybridSys}), we needed a hybrid space $a$ and dynamics-governing vector field $X$ on the underlying  space $\Ub a$. To define hybrid open systems, we need a hybrid version of surjective submersion $p_a$ and dynamics-governing object $X$ on the underlying surjective submersion, compatible with the surjective submersion $p_a$.  We start with hybrid surjective submersions. 

\begin{definition}\label{def:hybridSurjectiveSubmersion}
	We define a \textit{hybrid surjective submersion} $p_a:a_{tot}\rightarrow a_{st}$  to be a morphism  of hybrid phase spaces (\cref{def:hybridPhaseSpaceMorphism}) such that $\Ub p_a:\Ub a_{tot}\rightarrow\Ub a_{st}$ is a surjective submersion (\cref{def:SSub}).  We call the domain $a_{tot}$ the \textit{total (hybrid phase) space} and $a_{st}$ the \textit{state space}.  
\end{definition}

\begin{notation}\label{notation:consolidatingNotationForHySSub}
	We will generally notate a hybrid surjective submersion $a_{tot}\xrightarrow{p_a} a_{st}$ by $a$ or by $p_a$ when we want to emphasize the map.  When needed, we will usually let  $\p_a:\Sb^{a_{tot}}\rightarrow \Sb^{a_{st}}$ denote the functor on source categories and  $\pf_a:a_{tot}\Rightarrow a_{st}\circ\pi $ denote the natural transformation. When we have fixed the hybrid surjective submersion $a$, and there is no ambiguity, we may drop subscripts  and simply write $p = (\p,\pf)$. 
\end{notation}

\begin{remark}\label{remark:unwrapHySSub}
We unwrap   \cref{def:hybridSurjectiveSubmersion}. For each node $\ft\in \Sb^{a_{tot}}_0$ in the source category of the total phase space, there is a map of manifolds $\pf_\sF{t}:a_{tot}(\sF{t})\rightarrow a_{st}(\p(\sF{t}))$ (\cref{notation:consolidatingNotationForHySSub}).  For each fixed node $\fs\in \Sb^{a_{st}}_0$ in the source category of the state phase space, the collection of maps $\left\{a_{tot}(\sF{t})\xrightarrow{\pf_\sF{t}}a_{st}(\sF{s})\right\}_{\sF{t}\in \pi^{-1}(\sF{s})}$ uniquely induces a map $$\pf_{\pi^{-1}(\sF{s})}:\discats_{\sF{t}\in \pi^{-1}(\sF{s})} a_{tot}(\sF{t})\rightarrow   a_{st}(\sF{s}) $$  so that the  diagram $$\begin{tikzcd}
	\discats_{\sF{t}\in \pi^{-1}(\sF{s})} a_{tot}(\sF{t})\arrow[r,dashed,"\pf_{\pi^{-1}(\sF{s})}"] & a_{st}(\sF{s}) \\ a_{tot}(\sF{t'}) \arrow[u,hookrightarrow,"in_{(\ft')}"] \arrow[ur, "\pf_{\sF{t'}}"] &  
\end{tikzcd}$$ commutes for every $\sF{t'}\in \pi^{-1}(\sF{s})$. Surjectivity of $\Ub p_a$ implies both surjectivity of each  $$\pf_{\pi^{-1}(\sF{s})}:\discats_{\sF{t}\in \pi^{-1}(\sF{s})} a_{tot}(\sF{t})\rightarrow   a_{st}(\sF{s}) $$  as a map of manifolds (indexing over  $\fs\in \Sb^{a_{st}}_0$),  and surjectivity on objects of functor $\pi: \Sb^{a_{tot}}_{0}\rightarrow \Sb^{a_{st}}_{0}$. These conditions are equivalent:   surjectivity of each $\pf_{\pi^{-1}(\sF{s})}$ and of $\pi$ imply that $\Ub p_a$ is surjective as well.  This observation provides an operational way to check that a map $p:a'\rightarrow a$ of hybrid phase spaces is a hybrid surjective submersion.  
\end{remark}

\begin{example}\label{ex:hySSubTrivial}
	Let $a:\Sb^a\rightarrow\sF{RelMan}$ be an arbitrary hybrid phase space, and $id_a:a\rightarrow a$ the identity morphism. Since $id_{\Sb^a}:\Sb^a\rightarrow\Sb^a$ is surjective and $id_{a(\fs)}:a(\fs)\rightarrow a(\fs)$ is surjective for each $\fs\in \Sb^a_0$, we readily observe that the identity map is a hybrid surjective submersion. 
\end{example}

\begin{example}
	Let $a:\Sb^a\rightarrow\sF{RelMan}$ be a hybrid phase space and suppose for each node $\fs\in \Sb^a_0$, we have a surjective submersion $a'(\fs)\xrightarrow{\pf_\fs}a(\fs)$ for some indexed collection of manifolds $\{a'(\fs)\}_{\fs\in \Sb^a_0}$. We define hybrid phase space $a':\Sb^a\rightarrow\sF{RelMan}$ (with the same source category as $a$) by the following.  The assignment on nodes is given already by $\fs\mapsto a'(\fs)$.  For edge $\fs\xrightarrow{\g}\fs'\in \Sb^a_1$, we set $$a'(\g)\defeq \big(\pf_{\fs}^{-1}\times \pf_{\fs'}^{-1}\big) (a(\g))\subseteq a'(\fs)\times a'(\fs').$$  It follows that $\big(\pf_\fs\times \pf_{\fs'}\big)(a'(\g))\subseteq a(\g)$ for every edge, so $p_a = (id_{\Sb^a}, \pf):a'\rightarrow a$ is a map of hybrid phase spaces. Therefore, $p_a$ is also a hybrid surjective submersion (\cref{remark:unwrapHySSub}). 
\end{example}

Recall that the projection of a product of manifolds $M\times N\xrightarrow{p_M} M$ is a surjective submersion. We present a similar fact for  the category of hybrid phase spaces, which will gives us a method for generating hybrid surjective submersions from hybrid phase spaces. 

\begin{prop}\label{prop:BinProductHyPh} The category $\sF{HyPh}$ of hybrid phase spaces has binary products, and the projection maps $p_a:a\times b\rightarrow a$, $p_b:a\times b\rightarrow b$ are  hybrid surjective submersions. 
\end{prop}

\begin{notation}\label{notation:consolidateNotationForProductHyPh}
Recall from \cref{notation:projForProduct} that $p_\fc:\fc\times \fc'\rightarrow \fc$ denotes the canonical projection of product in category $\fC$. So for $\fC = \sF{HyPh}$, $p_a:a\times b\rightarrow a$ denotes the projection.  Since---as we will show---the projection is  a morphism in $\sF{HyPh}$, $p_a = (\p_a,\pf_a)$ consists of two maps: a functor $\p_a:\Sb^{a\times b}\rightarrow\Sb^a$ and a natural transformation $\pf_a:a\times b \Rightarrow a\circ \p_a$ (c.f.\ \cref{notation:consolidatingNotationForHySSub}). 
	\end{notation}

\begin{proof}
	Let $a,b$ be two hybrid phase spaces.  To define the product $a\times b:\Sb^{a\times b} \rightarrow\sF{RelMan}$, we must define a source category $\Sb^{a\times b}$, a functor $a\times b:\Sb^{a\times b}\rightarrow\sF{RelMan}$, projection morphisms $p_a= (\p,\pf)_a:a\times b\rightarrow a$ and $p_b= (\p,\pf)_b:a\times b\rightarrow b$,\footnote{In this proof, we explicitly only construct projection to $a$, as the construction for $b$ is identical.}  and finally show that $a\times b$ is terminal with respect to pairs of maps of hybrid phase spaces to $a$ and to $b$.
	
We define the  source category by $\Sb^{a\times b}\defeq \Sb^a\times\Sb^b$  as a product in $\sF{Cat}$ and for each node $\sF{t} = (\sF{s},\sF{s'})\in \Sb^{a\times b}_0$, we define manifold by $\big(a\times b\big) (\sF{t})\defeq a(\sF{s})\times b(\sF{s'})$ as  product in $\sF{Man}$. For each $\eta = (\g,\g')\in \Sb^{a\times b}_1$,   there is canonical isomorphism of manifolds \footnotesize\begin{equation}\label{eq:mapIsoForRelations} a(\sF{dom} (\g)) \times a(\sF{cod}(\g))\times b(\sF{dom}(\g'))\times b(\sF{cod}(\g'))\xrightarrow{\m_\eta} a(\sF{dom}(\g))\times b(\sF{dom}(\g'))\times a(\sF{cod}(\g)) \times b(\sF{cod}(\g'))\end{equation}\normalsize sending $(x,x',y,y')\mapsto (x,y,x',y')$, from which  we define \begin{equation}\label{eq:defRelForProd} \big(a\times b\big)(\eta)\defeq \m_\eta(a(\g)\times b(\g'))\subseteq (a\times b)\big(\sF{dom}(\eta)\big)\times(a\times b)\big(\sF{cod}(\eta)\big).\end{equation}

Now we construct projection morphism $p_a:a\times b\rightarrow a$.  The functor $\pi_a:\Sb^{a\times b}\rightarrow\Sb^a$ is the projection of product in $\sF{Cat}$ mapping $\ft = (\fs,\fs')\mapsto \fs$. Similarly for each node $\ft = (\fs,\fs') \in \Sb^{a\times b}$, the projection $\pf_{a,\sF{t}}:\big(a\times b\big)(\ft)\rightarrow a(\pi_a(\ft))$ is the projection $(x,y)\mapsto x$   in $\sF{Man}$, where $(x,y) \in a(\fs)\times b(\fs')$.   For edge $\eta = (\g,\g') \in \Sb^{a\times b}_1$, it is readily apparent that $$\left(\pf_{a,\sF{dom}(\eta)}\times \pf_{a,\sF{cod}(\eta)}\right)\left(\big(a\times b\big)(\eta)\right)  = a(\g)\subseteq a\big(\p_a(\eta)\big), $$
so these data indeed define   morphism in $\sF{HyPh}$. We observe both that $\p_a:\Sb^{a\times b}\rightarrow \Sb^a$ is surjective on objects and  that  $\pf_{a,\pi_a^{-1}(\sF{s})}:\discats_{\ft\in \pi^{-1}(\fs)}\big(a\times b\big)(\ft) \rightarrow a(\fs)$ is surjective for each $\fs\in \Sb^a_0$.  Therefore    $p_a\defeq (\pi,\pf)_a: a\times b\rightarrow a$ is a hybrid surjective submersion (\cref{remark:unwrapHySSub}).

It remains to check  that $a\times b$ satisfies the universal property of product in $\sF{HyPh}$, i.e.\ that given any pair of morphisms of hybrid phase spaces $z_a:c\rightarrow a$ and $z_b:c\rightarrow b$, there is unique morphism $z:c\dashrightarrow a\times b$ so that the diagram $\begin{tikzcd} c\arrow[drr,"z_a"]\arrow[dr,dashed]\arrow[ddr,swap,"z_b"] & & \\ 
& a\times b\arrow[d,"p_b"] \arrow[r,swap,"p_a"] & 	a\\
& b & 
\end{tikzcd}$ commutes.  Let $(\z,\zf)_\bullet$ denote the functor and natural transformation components  of hybrid phase space morphism $z_\bullet$, for $\bullet = a, b$.  Since  $\Sb^{a\times b}$ is a categorical product, there is unique functor $\z:\Sb^c\rightarrow\Sb^{a\times b}$.  Similarly, since each manifold $\big(a\times b\big)(\ft) = a(\pi_a(\ft)) \times b(\pi_b(\ft))$ is a categorical product, for each  $\fs\in \Sb^c_0$ there is unique map of manifolds $c(\fs)\xrightarrow{\zf_\fs = \zf_{a,\fs}\times \zf_{b,\fs}}\big(a\times b\big) (\z(\fs))$ defined by pair   $(\zf_{a,\fs},\zf_{b,\fs})$ of manifold maps.  We verify that $(\z,\zf):c\rightarrow a\times b$ is a morphism of hybrid phase spaces, namely that $\big(\zf_{\sF{dom}(\k)}\times \zf_{\sF{cod}(\k)} \big)\left(c(\k)\right) \subseteq \big(a\times b\big) (\z(\k))$ for every edge $\k\in \Sb^c_1$. By assumption that $z_a$ and $z_b$ are maps of hybrid phase spaces, $$\begin{array}
	{lll} \big(\zf_{a,\sF{dom}(\k)}\times \zf_{a,\sF{cod}(\k)}\big)(c(\k))  \subseteq & a(\z_a(\k)) & \mbox{and} \\
	\big(\zf_{b,\sF{dom}(\k)}\times \zf_{b,\sF{cod}(\k)} \big)(c(\k)) \subseteq & b(\z_b(\k)). & 
\end{array}$$ Therefore  $$\left(\big(\zf_{a,\sF{dom}(\k)}\times \zf_{a,\sF{cod}(\k)}\big)(c(\k)) \times \big(\zf_{b,\sF{dom}(\k)}\times \zf_{b,\sF{cod}(\k)} \big)(c(\k))\right) \subseteq a(\z_a(\k))\times b(\z_b(\k)).$$  Applying the the canonical isomorphism $\m_\k$ (c.f.\  \eqref{eq:mapIsoForRelations}, eq.\eqref{eq:defRelForProd}) to both sides produces inclusion $$\big(\zf_{a,\sF{dom}(\k)}\times \zf_{b,\sF{dom}(\k)}\big)(c(\k)) \times \big(\zf_{a,\sF{cod}(\k)}\times \zf_{b,\sF{cod}(\k)} \big)(c(\k)) \subseteq \left(a\times b\right) (\z(\k)), $$ the left-hand side of which is  $(z_a,z_b) = \left(\zf_{\sF{dom}(\k)}\times \zf_{\sF{cod}(\k)}\right)(\k)$.  Hence  $z:c\rightarrow a\times b$ is a map of hybrid phase spaces, uniquely defined by the pair $\big((\z,\zf)_a,(\z,\zf)_b\big)$.  We conclude that the product $a\times b$ is terminal with respect to pairs of maps of hybrid phase spaces to $a$ and to $b$, proving that $a\times b$ is categorical product.\end{proof}

\begin{remark}
\label{remark:ProductIngredients}
	We collect ingredients used in the proof of \cref{prop:BinProductHyPh} for future reference.  For two hybrid phase spaces $a$ and $b$, the product $a\times b$ is defined with source category \begin{equation}\label{eq:BinProductSource} 
	\Sb^{a\times b}\defeq \Sb^a\times \Sb^b,
	\end{equation}
	and for each node $(\fc,\fd)\in \Sb^{a\times b}_0$, the assignment of manifolds is \begin{equation}
\label{eq:BinProductManifoldAssignment}
(a\times b)(\fc,\fd) \defeq a(\fc)\times b(\fd).	
 \end{equation}
 For edge $\eta\defeq (\g,\g')\in \Sb^{a\times b}_1$, the assignment of relations is
 \begin{equation}\label{eq:BinProductRelationAssignment}
 	(a\times b)(\eta) \defeq \m_\eta(a(\g),b(\g')), 
 \end{equation}
 where  $\m_\eta$ is canonical isomorphism  in \eqref{eq:mapIsoForRelations}. 
 
 The projection maps  \begin{equation}\label{eq:projForProductHyPh}
 	\begin{array}{lll} p_a:a\times b\rightarrow a & \mbox{and} & p_b:a\times b\rightarrow b\end{array} 
 \end{equation} are $p_a\defeq (\p_a,\pf_a)$ and $p_b\defeq (\p_b,\pf_b)$.
\end{remark}
\begin{remark}\label{remark:hyssubFromHyPh}
 We observe again that the projections $p_a$ and $p_b$ in \eqref{eq:projForProductHyPh} are hybrid surjective submersions.  We thus have a way of generating hybrid surjective submersions from hybrid phase spaces. \end{remark}

\begin{example}
	Consider hybrid phase space $a:\Sb^a\rightarrow\sF{RelMan}$ with source category $\Sb^a$ given by $\begin{tikzcd} 0 \arrow[r,bend left,"e_{1,0}"] & 1.\arrow[l,bend left,"e_{0,1}"] \end{tikzcd}$ We describe the product $a\times a$.   The source category    $\Sb^{a\times a}$ is given (\cref{remark:ProductIngredients}, \eqref{eq:BinProductSource}) by \begin{equation}\label{diagram:productOfA} \begin{tikzcd}
	(0,1)\arrow[r,shift left]\arrow[d,shift left] \arrow[dr, shift left] & (1,1)\arrow[d, shift left]\arrow[l, shift left] \arrow[dl,shift left]\\
	(0,0) \arrow[u,shift left] \arrow[r,shift left]\arrow[ur,shift left] & (1,0),\arrow[u,shift left]\arrow[ul, shift left]\arrow[l, shift left]
	\end{tikzcd}\end{equation}  and has manifold assignments $\big(a\times a\big)(i,j) = a(i)\times a(j)$ for $i, j = 0, 1$ (c.f.\ \eqref{eq:BinProductManifoldAssignment}).  Horizontal and vertical arrows appear in this diagram because $\Sb^a$ is a category, so there are unit arrows $id_i$ for $i=0,\,1$. For example, along the bottom of diagram \eqref{diagram:productOfA}, the  left-to-right arrow is $(0,0)\xrightarrow{(e_{1,0},id_0)}(1,0)$.
	\end{example}
	We present a more concrete example which illustrates why we wanted $\sF{Man}$ to have manifolds with corners as objects (\cref{remark:categoryManifolds}). 
	\begin{example}\label{ex:productsInduceCorners}
		Recall phase space $c:\Sb^c\rightarrow\sF{RelMan}$ from \cref{ex:phaseSpaceForBall}, assigning $c(0)= \R^{\geq 0}\times \R$. This manifold is a manifold with boundary, or a manifold with corners that has no corners. Suppose, now,  that we take two  bouncing balls, which we represent  in the product of phase space $c\times c$,  with source category $\Sb^{c\times c}$: 
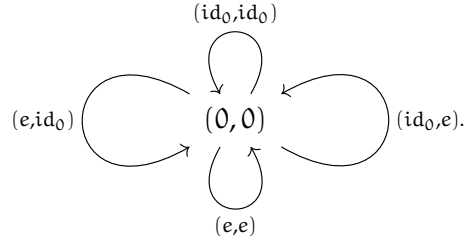
$$\begin{tikzcd}
(0,0)\arrow[out=240,in=300,loop,swap, "(e{,}e)"] \arrow[out=60,in=120,loop, swap,"(id_0{,}id_0)"] 
	\arrow[out=150,in=210,loop,swap, "(e{,}id_0)"] \arrow[out=-30,in=30,loop,swap, "(id_0{,}e)."] 
\end{tikzcd}$$ Observe that $\big(c\times c\big)(0,0) = \big(\R^{\geq 0 }\times \R\big)\times \big(\R^{\geq0}\times \R\big) \cong \big(\R^{\geq0}\big)^2\times \R^2$, is a manifold with corners.  A product of manifolds may have corners, even if the component manifolds do not. 
	\end{example}

\begin{remark}\label{remark:productHySSub}
	Let $a\defeq\big(a_{tot}\xrightarrow{p_a} a_{st}\big) $ and $b\defeq\big( b_{tot}\xrightarrow{p_b} b_{st}\big)$ be hybrid surjective submersions.  Then the product $a\times b\defeq \big(a_{tot}\times b_{tot}\xrightarrow{p_a\times p_b} a_{st}\times b_{st}\big)$ is a hybrid surjective submersion.
\end{remark}
\begin{remark}\label{remark:clearingNotation}
	Notice that we denote a hybrid surjective submersion $a = \big(a_{tot}\xrightarrow{p_a}a_{st}\big)$ by  $p_a$ (\cref{notation:consolidatingNotationForHySSub}) \textit{and} projection of hybrid surjective submersions $a\times b\xrightarrow{p_a}a$ by $p_a$. Since the projection of products of hybrid phase spaces is a surjective submersion, this notation is not inconsistent. On the other hand, not every surjective submersion  arises from the projection of product. For the most part, our examples will be projection of products, and there will be little confusion caused by conflating these notations. 
	\end{remark}

Now we define maps of hybrid surjective submersions. 
\begin{definition}\label{def:hybridSSubMorphism}
	Let $p_a:a_{tot}\rightarrow a_{st}$ and $p_b:b_{tot}\rightarrow b_{st}$ be two hybrid surjective submersions.  We define a \textit{morphism} $f:a\rightarrow b$ \textit{of hybrid surjective submersion} to be  a pair of morphisms $f_{tot}:a_{tot}\rightarrow b_{tot}$, $f_{st}:a_{st}\rightarrow b_{st}$ of hybrid phase spaces (\cref{def:hybridPhaseSpaceMorphism}) such that  $p_b\circ f_{tot} = f_{st}\circ p_a$.\end{definition} 
	
\begin{remark}\label{remark:unpackingHySSubMorphism}
Definition \ref{def:hybridPhaseSpaceMorphism} unpacks as follows. Let $f_\bullet = (\ph,\ff)_\bullet$ (for $\bullet = tot,\, st$) and recall our convention that $p_\bullet = (\p,\pf)_\bullet$ (for $\bullet = a,\, b$). First, there is prism diagram \begin{equation}\label{diagram:prismForHySSubMorphism} \begin{tikzcd}
	\Sb^{a_{tot}}\arrow[rr,"\ph_{tot}"]\arrow[d,swap,"\pi_a"] \arrow[ddr,near start,"a_{tot}"] &  & \Sb^{b_{tot}}\arrow[d,"\pi_b"]\arrow[ddl,swap, near start, "b_{tot}"] \\
	 \Sb^{a_{st}}\arrow[dr,swap,"a_{st}"] \arrow[rr,"\ph_{st}"] & &  \Sb^{b_{st}}\arrow[dl,"b_{st}"]\\
	 & \sF{RelMan}, & 
\end{tikzcd}\end{equation} where the back square face is a commuting  diagram of  functors, and each triangle face is morphism of hybrid phase spaces: $$
	\begin{array}{llll} \begin{tikzcd}[row sep=3em]
\Sb^{a_{tot}} \arrow{r}{\p_a } \arrow[""{name=foo}]{dr}[swap]{a_{tot}} & \Sb^{a_{st}} \arrow{d}{a_{st}} \arrow[Rightarrow, from=foo, swap, "\pf_a"] \\ & \sF{RelMan}\end{tikzcd} & \begin{tikzcd}[row sep=3em]
\Sb^{b_{tot}} \arrow{r}{\p_b } \arrow[""{name=foo}]{dr}[swap]{b_{tot}} & \Sb^{b_{st}}\arrow{d}{b_{st}} \arrow[Rightarrow, from=foo, swap, "\pf_b"] \\ & \sF{RelMan}\end{tikzcd} & 
\begin{tikzcd}[row sep=3em]
\Sb^{a_{tot}} \arrow{r}{\ph_{tot} } \arrow[""{name=foo}]{dr}[swap]{a_{tot}} & \Sb^{b_{st}} \arrow{d}{b_{tot}} \arrow[Rightarrow, from=foo, swap, "\ff_{tot}"] \\ & \sF{RelMan}\end{tikzcd} & \begin{tikzcd}[row sep=3em]
\Sb^{a_{st}} \arrow{r}{\ph_{st}} \arrow[""{name=foo}]{dr}[swap]{a_{st}} & \Sb^{b_{st}}\arrow{d}{b_{st}} \arrow[Rightarrow, from=foo, swap, "\ff_{st}"] \\ & \sF{RelMan}.\end{tikzcd}
\end{array}
$$ 
Requiring that $p_b\circ f_{tot} = f_{st}\circ p_a$ amounts to requiring both that   $\ph_{st}\circ \p_a = \p_b\circ \ph_{tot}$  (an equality of functors), and for each node $\ft\in \Sb^{a_{tot}}_0$, an equality of maps of manifolds \begin{equation}\label{eq:equalityOfHySSubMorph} \ff_{st,\p_a(\ft)}\circ \pf_{a,\ft} = \pf_{b,\ph_{tot}(\ft)}\circ \ff_{tot,\ft}.\end{equation} 

In summary, we require that  \begin{equation}\label{eq:hySSubMorphInManifolds}\Ub f_{st} \circ \Ub p_a = \Ub\big(f_{st}\circ p_a\big) = \Ub\big(p_b\circ f_{tot}\big) = \Ub p_b \circ \Ub f_{tot}.\end{equation}
\end{remark}

\begin{definition}\label{def:HySSubCat}\label{def:categoryHySSub} We define the category $\sF{HySSub}$  whose objects are hybrid surjective submersions (\cref{def:hybridSurjectiveSubmersion})  and whose morphisms are morphisms of  hybrid surjective submersions (\cref{def:hybridSSubMorphism}). \end{definition}

\begin{remark}\label{remark:hyssubIsSubcategoryOfArrow}
	We see that $\sF{HySSub}$ is a subcategory of $\sF{Arrow(HyPh)}$ (\cref{def:arrowCategory}).  In fact, $\sF{HySSub}$ is a full subcategory since the only condition defining morphisms of hybrid surjective submersions is that a square diagram commutes, same as the arrow category of $\sF{HyPh}$. We thus could have alternatively  defined the category $\sF{HyPh}$ as the full subcategory of $\sF{Arrow(HyPh)}$ whose objects are hybrid surjective submersions (\cref{def:hybridSurjectiveSubmersion}). 
\end{remark}

	We showed that $\sF{HyPh}$ has binary products in order to provide nontrivial examples of hybrid surjective submersions.  There is an alternative  reason: product is one piece of the mechanism we use in networks to piece a bunch of spaces together into one.   Formally, we would like $\sF{HyPh}$ to be a monoidal category. 
	
	To this end, we observe that the product  $\times:\sF{HyPh}\times \sF{HyPh}\rightarrow \sF{HyPh}$ defined in \cref{prop:BinProductHyPh} will be the monoidal product (\cref{prop:productBifunctor}).  We now observe the monoidal unit.  The category $\sF{HyPh}$ of hybrid phase spaces has a terminal object.
	\begin{fact}\label{fact:terminalHyPh}
		There is terminal hybrid phase space  $1_\sF{HyPh}$ defined by \begin{enumerate}
			\item $\Sb^{1_\sF{HyPh}}$ is a one object discrete category (a terminal object in the category $\sF{Cat}$).
			\item $1_\sF{HyPh}:\Sb^{1_\sF{HyPh}}\rightarrow\sF{RelMan}$ assigns the one point discrete manifold (a terminal object in the category $\sF{Man}$) to the unique node in $\Sb^{1_\sF{HyPh}}$. 
		\end{enumerate}	It is easy to verify that this data defines a terminal object in $\sF{HyPh}$. \end{fact}

 We collect results to conclude that $\sF{HyPh}$ is a monoidal category. The monoidal product is cartesian, and the monoidal unit is a terminal object (\cref{def:monoidalProductCartesian}):
 
\begin{definition}\label{def:monoidalCategoryHyPh}
   Hybrid phase spaces $(\sF{HyPh}, \otimes_{\sF{HyPh}},1_{\sF{HyPh}})$ form a cartesian monoidal category, with $\otimes_{\sF{HyPh}}=\times$ (\cref{prop:BinProductHyPh}) and  $1_{\sF{HyPh}}$ the monoidal unit (\cref{fact:terminalHyPh}).
\end{definition}

Monoidality of category $\sF{HyPh}$ extends naturally to $\sF{HySSub}$. 
\begin{fact}\label{fact:HySSubIsMonoidal}\label{fact:HySSubIsCartesian}
The category $\big(\sF{HySSub},\otimes_{\sF{HySSub}}, 1_{\sF{HySSub}})$  is cartesian monoidal, with monoidal product $\otimes_{\sF{HySSub}} = \times$ (\cref{remark:productHySSub}) and monoidal unit the terminal object $1_{\sF{HySSub}}$ (coming from $1_{\sF{HyPh}}$ in \cref{fact:terminalHyPh},  \cref{fact:terminalArrowObject}). 
\end{fact}

\begin{remark}
	Since $\fC$ is cartesian, $\fC$ is trivially \textit{induced}-cartesian and  $\fA$ is therefore induced-cartesian as well (\cref{prop:inducedCartesianCatImpliesArrowIsInducedCartesian}).  An arbitrary  subcategory of the arrow category of a cartesian category is not necessarily cartesian. For example, let $\fC$ be cartesian and $\fA\subset \sF{Arrow(C)}$ have as objects all morphisms $\fc\xrightarrow{f}\fc'$ of $\fC$ for which morphisms $f\xrightarrow{(\a_{\sF{dom}},\a_{\sF{cod}})}g$ are  isomorphisms $\a_{\sF{dom}}:\sF{dom}(f)\xrightarrow{\sim}\sF{dom}(g)$, $\a_{\sF{cod}}:\sF{cod}(f)\xrightarrow{\sim}\sF{cod}(g)$ on domain and codomain. Then $\fA$ is not cartesian because  projection and the induced unique map to the product are not  in $\fA$.
\end{remark}

We now define hybrid open systems and their morphisms. 
\begin{definition}\label{def:hybridOS}
	We define a \textit{hybrid open system}  to be a pair $(a,X)$  where  $a$ is a hybrid surjective submersion $p_a:a_{tot}\rightarrow a_{st}$    and $X:\Ub a_{tot}\rightarrow T\Ub a_{st}$ is a smooth map such that \begin{equation}\label{eq:projConditionForOpenSystemsHybrid} \Ub p_a = \t_{a_{st}}\circ X,\end{equation} where $\t_{a_{st}}:T\Ub a_{st}\rightarrow \Ub a_{st}$ is the canonical projection of the tangent bundle.
\end{definition}
A more compact definition of hybrid open system is: a pair  $(a,X)$ where $a\in \sF{HySSub}$ and $(\Ub a, X)$ is an open system (\cref{def:openDySys}). There are also maps of hybrid open systems.

\begin{definition}\label{def:morphismHybridOS}\label{def:hybridOpenSystemMorphism}
	We define a \textit{map} $(a,X)\xrightarrow{f}(b,Y)$  \textit{of hybrid open systems} as a map $a\xrightarrow{f} b$ of hybrid surjective submersions such that $(X,Y)$ are $\Ub f$-\textit{related} (\cref{def:relatedVectorFields}), namely  $(\Ub a,X)\xrightarrow{\Ub f}(\Ub b,Y)$ is a map of continuous-time open systems (\cref{def:openDySysMorphism}).  For hybrid open systems, we also say that $(X,Y)$ are $f$-\textit{related} (dropping `$\Ub$'). 
\end{definition}

While we may develop a theory of networks for hybrid open systems---and indeed we actually do  so in \cite{lermanSchmidt1}---we prefer to modify this notion, in a way that makes sense to speak of unique executions (or unique hybrid integral curves). To this end, we now turn.

\section{Determinism in Hybrid Systems}\label{subsec:determinismHySys}
We now build a theory of  deterministic hybrid systems.  What makes systems deterministic, for us, is that we specify a jump point to  each point of the underlying manifold. By way of comparison, consider that  a vector field $X\in \Xf(M)$ on a manifold $M$ assigns to each point $x\in M$ a tangent vector $X(x)\in T_xM$, or---loosely speaking---a direction to flow. We have already incorporated vector fields into our hybrid apparatus with the notion of a hybrid system (\cref{def:hybridSys}). The vector field of a hybrid system is supposed to depict  continuous behavior.  We vaguely alluded to the hybrid aspect or discrete-jump behavior by pointing to the relations which a hybrid phase space assigns to edges of the source manifold.  Now we repackage this idea as follows: for each point $x\in \Ub a$ in the underlying manifold of a hybrid phase space, we will specify both a tangent vector $X(x)\in T_x\Ub a$ and a jump point $\r(x)\in \Ub a$.  The only constraint we place on $\r$ is that for each $x\in \Ub a$, there is some edge $\g_x\in \Sb^a_1$ for which the pair of points $(x,\r(x))\in a(\g(x))$. 

Ultimately, this construction is used to interpret an analogous notion of integral curve for hybrid systems, which also has some uniqueness property (\cref{theorem:E&U}).   Without a jump map $\r$, a hybrid system (in the sense of \cref{def:hybridSys}) really is like an ordinary continuous-time dynamical system.  On the other hand, a notion of executions for such systems allows discrete jumps, but without determinism those jumps may ``occur anywhere'' there is a relation (c.f.\ \cite[\S 4]{lermanSchmidt1}). By including a jump map, we enforce that a hybrid integral curve jumps everywhere, even if that point is to itself (a self-jump, as we will see in \cref{def:deterministicExecution}, allows for continuous flow).  More importantly, by specifying the jumps, we eliminate the indeterminism of executions inherent in the  non-deterministic version of hybrid system (\cref{def:hybridSys}).  While initially counterintuitive to assign \textit{both} a jump and a tangent vector, we do so for two reasons: (1) the possibility of self-jumps does not impose irregular or erratic ``everywhere discontinuous'' behavior which the notion may otherwise suggest, and (2) having globally defined maps makes the formalism work cleanly.  An alternative is to choose a vector \textit{or} a jump, but this turns out to be problematic when taking products.

We now reinterpret  two examples from \cref{subsection:hybridClosedSystems}---the thermostat and the bouncing ball---to concretize these ideas, before diving deep into the formalism.    In the bouncing ball example (\cref{ex:bbAsNormalHybridSystem}), we have relation $\big\{(h,v,h',v')\in (\R^{\geq0}\times \R)^2:\, h= h'= 0,\, v'\cdot v< 0 \big\}$.  This represents the possibility of a  discontinuous change in velocity at the moment of impact with the ground.  On the other hand, there is \textit{also} identity relation, so even at the ground the state $(0,v)$ may ``jump'' to $(0,v)$ instead of $(0,v')$. While this is  mathematically permissible, any hybrid interpretation of  integral curve would end here, because the state is not defined for $h<0$.  By contrast, jumping to $(0,v')$ both makes physical sense and allows us to have a piecewise continuous curve defined for all time.

Again, imagine the thermostat (\cref{ex:thermoAsNormalHybridSystem}) with relations $\big\{(x,i,x,1-i):\, x\cdot (-1)^{1-i}\geq 1\big\}$.  When $x\geq 1$, say, we may turn the heater off by sending $i=1\mapsto i=0$.  Unlike the bouncing ball, it is permissible (if not ecologically sound) to let the heater run indefinitely.  The relation by itself does not require the jump to happen as soon as temperature enters the region $\{x\geq 1\}$. To enforce said transition, we define the following map $\r:\R\times \{0,1\}\rightarrow \R\times \{0,1\}$ by $$\r(x,i)= (x,i) \cdot \one_{x\cdot (-1)^{1-i}<1} + (x,1-i)\cdot \one_{x\cdot(-1)^{1-i} \geq 1}.$$ Roughly speaking, this translates as: temperature always jumps to itself and the heater remains on or off unless it enters into or is initialized beyond a threshold region.  Now we discuss the details. 

\subsection{Deterministic  Hybrid Closed Systems}\label{subsection:deterministicHybridSystems}
We introduce determinism first for closed systems and then move to open systems.  We need a technical fact which we use to connect the constraints imposed by relations. 

\begin{remark}\label{remark:inclusionForDeterminism}
Recall that a hybrid phase $a:\Sb^a\rightarrow\sF{RelMan}$ space assigns a relation $a(\g) \subseteq a(\fs)\times a(\fs')$ to edge $\fs\xrightarrow{\g}\fs'\in \Sb^a_1$ (\cref{def:hybridPhaseSpace0}). Alternatively, we view this inclusion as a map  $$\i_\g:a(\g)\hookrightarrow a(\sF{dom}(\g))\times a(\sF{cod}(\g))$$ of sets.  Additionally, there are canonical inclusions   $$\begin{array}{lll} i_\g: a(\g)\hookrightarrow\discats_{\g'\in \Sb^a_1}a(\g') &\mbox{and} & i_{(\fs_0,\fs_0')}:a(\fs_0)\times a(\fs_0')\hookrightarrow\discats_{(\fs,\fs')\in \Sb^{a\times a}_0} a(\fs)\times a(\fs') \end{array}$$ into the coproducts. Altogether, these inclusions define a unique map  $\lambda_a$ in commuting diagram $$\begin{tikzcd}[column sep = large, row sep = large]
\discats_{\g'\in \Sb^a_1}a(\g') \arrow[r,dashed,"\lambda_a"] & \discats_{(\fs,\fs')\in \Sb^{a\times a}_0}a(\fs)\times a(\fs') \\
a(\g)\arrow[ru]\arrow[u,hookrightarrow,"i_\g"]\arrow[r,hookrightarrow,"\i_\g"] & a(\sF{dom}(\g))\times a(\sF{cod}(\g)). \arrow[u,swap,hookrightarrow,"i_{\sF{dom}(\g)\times \sF{cod}(\g)}"]
\end{tikzcd}$$ There is also a canonical map (\cref{prop:canonicalMapCoproduct2Product}) \begin{equation}\label{eq:mapOmega} \Omega_a: \discats_{(\fs,\fs')\in \Sb^{a \times a}_0} a(\fs)\times a(\fs') \dashrightarrow \left(\discats_{\fs\in \Sb^a_0}a(\fs)\right)\times  \left(\discats_{\fs\in \Sb^a_0}a(\fs)\right) = \Ub a\times \Ub a. \end{equation}  Composing the dashed-line maps, we obtain a unique map \begin{equation}\label{eq:coproductRelations2Product} \Lambda_a:\discats_{\g\in \Sb^a_1}a(\g)\dashrightarrow \Ub a\times \Ub a,\end{equation}  defined by $\Lambda_a\defeq \Omega_a\circ \lambda_a$.
\end{remark}

We are ready now to formally define deterministic hybrid systems.  We recall and piggyback on material from \cref{subsection:hybridClosedSystems}, in order to minimize redundant enumeration of extra data.  Instead of template ``hybrid phase space with extra data,'' our template will now be   ``hybrid system with extra data.'' Later, we will define the continuous-discrete bundle $\Tb a$ (\cref{def:cdBundle}), at which point we will revert to the template of ``hybrid phase space with data.''

\begin{definition}\label{def:detHySys}
	We define a \textit{deterministic  hybrid system} to be a triple  $(a,X,\r)$ where $(a,X)$ is hybrid system (\cref{def:HySys}) and $\r:\Ub a\rightarrow\Ub a$ is a set map (a morphism in the category $\sF{Set}$)---which we call the \textit{jump map}---satisfying \begin{equation}\label{eq:conditionFordeterministicJumps} \mbox{graph}(\r) \subseteq \Lambda_a\left(\discats_{\g\in \Sb^a_1} a(\g)\right),\end{equation} where $\Lambda_a:\discats_{\g\in \Sb^a_1}a(\g)\dashrightarrow \Ub a\times \Ub a$ (\cref{remark:inclusionForDeterminism}, eq.\ \eqref{eq:coproductRelations2Product}).  
	
\end{definition}
\begin{remark}\label{remark:checkingJumpMap}
Condition \eqref{eq:conditionFordeterministicJumps} 
parses as  saying: for each $x\in \Ub a$, there is an edge $\g_x\in \Sb_1^a$ for which the ordered pair $(x,\r(x)) \in a(\g_x)$.  In fact, we verify condition \eqref{eq:conditionFordeterministicJumps} in practice by finding such $\g_x$. 
\end{remark}
\begin{remark}\label{remark:jumpSet}
	Because $\Sb^a$ is a category, there is identity arrow $id_\fs\in \Sb^a_1$ for each node $\fs\in \Sb^a_0$, and therefore identity relation $a(id_\fs)\defeq \Delta(a(\fs))=\big\{(x,y)\in a(\fs)^2:\, x = y\big\}$. Thus, for  $x\in \Ub a$, it is possible that $\r(x) = x$.  We refer to the points $\scJ_a \defeq \big\{x\in \Ub a :\, \r(x) \neq x\big\}$ as the \textit{jump set}.  We have imposed no smoothness or continuity condition  on the jump map $\r:\Ub a \rightarrow \Ub a$.\footnote{If we were to be pedantic, we would write $\r:\{x\in \Ub a\}\rightarrow \{x\in \Ub\}$, where $\{x\in (\cdot)\}$ denotes the forgetful functor to the category of sets. Because we understand maps to be maps-in-a-category, it is important to draw this distinction: $\r$ is not a map of manifolds and therefore not smooth or even continuous.  However, we prefer a possible ambiguity in notation to over-specification and unnecessary convolutedness.}
\end{remark}
We review in detail and reinterpret the two running examples of thermostat and bouncing ball as deterministic  hybrid systems.  
\begin{example}\label{example:ThermostatVanilla} Recall the thermostat hybrid system $(c,Z)$ from \cref{ex:thermoAsNormalHybridSystem}.  We turn this system into deterministic hybrid system $(c:\Sb^c\rightarrow \sF{RelMan}, Z, \nu)$.  The source category  $\Sb^c$ is generated by graph  $\begin{tikzcd} 0\arrow[r,bend left,"e_{1,0}"] & 1\arrow[l,bend left,"e_{0,1}"] \end{tikzcd}$ and the phase space assigns manifolds $c(i) \defeq \R\times \{i\}$ to $i=0,1$.  The relations are $$c(e_{i,1-i}) \defeq  \left\{(x,1-i,x',i)\in \big(\R\times\{0,1\}\big)^2:\; x = x'\right\}.$$  We define the vector field $Z$ and jump map $\nu$ by  \small \begin{equation}\label{eq:vanillaThermostat}\begin{array}{lrlclrl}
 Z : & \R\times \{0,1\} & \rightarrow T(\R\times \{0,1\}) &\mbox{and} & \nu : & \R\times \{0,1\} & \rightarrow \R\times \{0,1\}\\ 
	& (T,i) & \mapsto  \big((-1)^{1-i},0\big)&      &  & (T,i) &\mapsto \left\{\begin{array}{ll} (T,1-i) &\mbox{if}\; (-1)^{1-i}T\geq  1 \\ (T,i) & \mbox{else.}\end{array}\right.\end{array}\end{equation} \normalsize
	
To be a deterministic  hybrid system, we require the inclusion $$\mbox{graph}(\nu)\subseteq \Lambda_c\left(\discats_{\g\in \Sb^c_1}c(\g)\right).  $$ Either $\nu(T,i)= (T,i)$ in which case $((T,i),\nu(T,i)) \in \Delta\big(\R\times\{i\}\big)=c(id_i)$, or $\nu(T,i) = (T,1-i)$ and $((T,i),\nu(T,i))  \in c(e_{1-i,i})$.  Both cases establish that condition \eqref{eq:conditionFordeterministicJumps}  holds (\cref{remark:checkingJumpMap}) and  therefore that $(c,Z,\nu)$ is a deterministic  hybrid system. 
	 \end{example}

\begin{remark} We now may see how a  deterministic  hybrid system models the behavior of a thermostat.  As before, the first factor $\R$ of $\R\times \{0,1\}$ represents the current temperature, whereas the second factor $\{0,1\}$ represents whether a heater is on or off.  The vector field $X$ governs the continuous dynamics, with positive-direction vectors for ``heat  on'' and negative otherwise.   The jump map represents digital control (a switched system), which discretely turns heat on or off depending on whether some threshold in the temperature has been surpassed. 
\end{remark}

Now we reinterpret  the bouncing ball.  A  ball falls through space with state variables $h$ and $v$ representing position (height) and velocity.  Velocity is unconstrained, but height is nonnegative (initializing the ground as zero, and up as positive).   When the ball hits the ground, there is a sudden  change in state: the height remains the same, but velocity spikes from negative to positive. Supposing   loss of energy at impact,  velocity jumps from $v(t_0^-)$ to $v(t_0^+) = -r v(t_0^-)$ at $t_0$ the time of impact, where $r \in (0,1)$ denotes the coefficient of restitution.  We realize this example in terms of \cref{def:detHySys}.  
\begin{example}\label{ex:bouncingBallVanilla}
We define deterministic  hybrid system $(c:\Sb^c\rightarrow\sF{RelMan}, Z,\mu)$. Hybrid phase space $c$  has   source category $\Sb^c$ given by   $\begin{tikzcd}
	0\arrow[loop left,"e"] 
\end{tikzcd}$  (c.f.\ \cref{ex:phaseSpaceForBall}).
 This phase space assigns manifold $c(0) \defeq \R^{\geq0}\times \R$ and relation $c(e)\defeq \big\{(h,v,h,v'):\, h = h'= 0,\, v\cdot v'< 0\big\}$. 
 Fix $r\in (0,1)$.  We  define control and jump maps by:\begin{equation} \small \begin{array}{lrlclrl}\label{eq:vanillaThermostat}Z : & \R^{\geq 0} \times\R& \rightarrow T(\R^{\geq 0}\times \R) & \mbox{and} & \mu : & \R^{\geq0}\times \R & \rightarrow \R^{\geq 0}\times\R\\
	& (h,v) & \mapsto v\pdiv{}{h}-\pdiv{}{v} & & 
	& (h,v) & \mapsto  \left\{\begin{array}{ll} (0,-rv) &\mbox{if}\; h=0,\,\,v<0 \\ (h,v) & \mbox{else.}\end{array}\right.\end{array}\end{equation} \normalsize
It is clear that $\mbox{graph}(\mu)\subseteq \Lambda_c\left(\discats_{\g\in \Sb^c_1}c(\g)\right)$  since $(0,v,0,-rv)\in c(e)$ when $h =0$ and $v<0$, as $-rv^2<0$ (\cref{remark:checkingJumpMap}).	 \end{example}

There is also a notion of map of deterministic  hybrid systems: 
\begin{definition}\label{def:detHySysMorphism}
	Let $(a,X,\r)$ and $(b,Y,\s)$ be two deterministic  hybrid systems.  We define a \textit{map} $(a,X,\r)\xrightarrow{f} (b,Y,\s)$ \textit{of deterministic hybrid systems} to be a map $(a,X) \xrightarrow{f}(b,Y)$ of hybrid systems (\cref{def:hybridSysMorphism}) such that $(\r,\s)$ are $f$-\textit{related}, namely $\Ub f \circ \r = \s \circ \Ub f$ (compare with \cref{remark:relatednessInHyPhandinMan}). 
\end{definition}

		\begin{remark}\label{remark:compatibility/relatedness}
			Relatedness of jump maps is the hybrid analog of relatedness of vector fields and control.  Both conditions appear in the definition of maps of deterministic  hybrid system: \begin{equation}\label{eq:conditionsForMorphismOfdeterministic System} \begin{array}{lll} T\Ub f\circ X = Y \circ \Ub f& \mbox{and} & \Ub f\circ \r = \s \circ \Ub f.\end{array}\end{equation} 
			In practice, checking relatedness of jump maps amounts to  the same diagram chasing as checking relatedness of vector fields (consider diagrams in \eqref{diagram:Related&CompatibleAreTheSame}, for example).
		\end{remark}
		
		\begin{remark}\label{remark:deterministicHySysProducts}
	Given two deterministic  hybrid open systems $(a,X,\r)$ and $(b,Y,\s)$, there is product system $(a\times b, X\times Y, \r\times \s)$ defined as follows.  The hybrid phase space $a\times b$ is the product of hybrid phase spaces (\cref{prop:BinProductHyPh}).  Vector field $X\times Y$ and jump map $\r\times \s$  are defined by \begin{equation}\label{eq:defProductDetSystem} \begin{array}{lll} (X\times Y) (x,y) \defeq (X(x),Y(y)) \in T_x\Ub a\times T_y\Ub b & \mbox{and} & (\r\times \s)(x,y) \defeq (\r(x),\s(y)) \in \Ub a\times \Ub b.\end{array}\end{equation}
	 We will see in \cref{lemma:PUUPLite} that $\Ub a \times \Ub b \cong \Ub(a\times b)$. Thus $$\begin{array}{lll} T_x\Ub a \times T_y\Ub b  & \cong T_{(x,y)}\big( \Ub a\times \Ub b\big)  & \mbox{(\cref{prop:isoProductTangentSpaces})} \\ & \cong T_{(x,y)}\Ub (a\times b) &\mbox{(\cref{lemma:PUUPLite})}\end{array}$$ so $X\times Y \in \Xf(\Ub (a\times b))$. Similarly, let $x\in \Ub a$, $y\in \Ub b$,  and suppose edges $\g_x\in \Sb^a_1$, $\eta_y\in \Sb^b_1$ are such that $(x,\r(x))\in a(\g_x)$ and $(y,\s(y))\in b(\eta_y)$ (\cref{remark:checkingJumpMap}).  Then we conclude that  $\big((x,y),(\r(x),\s(y))\big)\in (a\times b)(\g_x,\eta_y)$ (c.f.\ eq.\ \eqref{eq:BinProductRelationAssignment}), and hence  that   $$\mbox{graph}(\r\times \s) \subseteq \Lambda_{a\times b} \left(\discats_{(\g,\eta)\in \Sb^{a\times b}_1} (a\times b) (\g,\eta)\right).$$
	 
	 It is an easy verification that the projection maps $p_a:a\times b\rightarrow a$ and $p_b:a\times b\rightarrow b$ (c.f.\ \eqref{eq:projForProductHyPh}) of hybrid phase spaces define maps of deterministic  hybrid systems.  Relatedness of vector fields $T\Ub p_a \circ X\times Y = X\circ \Ub p_a$ follows immediately from \eqref{eq:defProductDetSystem} and \eqref{eq:BinProductManifoldAssignment}.  Relatedness of jump maps follows similarly.  If $$\mbox{graph}(\r\times \s)\subseteq\Lambda_{a\times b}\left(\discats_{(\g,\eta)\in \Sb_1^{a\times b}}(a\times b)(\g,\eta)\right),$$ then  $\big((x,y),(\r(x),\s(y))\big)\in (a\times b)(\g_x,\g_y)$  for each $(x,y)\in \Ub(a\times b)$ (c.f.\ \eqref{eq:defProductDetSystem}),   some $\eta_{(x,y)}=(\g_x,\g_y)\in \Sb^{a\times b}_1$. In this case,  $\big((x,\r(x)),(y,\s(y))\big)\in a(\g_x)\times b(\g_y)$ (c.f.\ \eqref{eq:BinProductRelationAssignment}), and in particular $(x,\r(x))\in a(\g_x)$.  Moreover $$(\pf_{a,\sF{dom}(\eta_{(x,y)})}\times \pf_{a,\sF{cod}(\eta_{(x,y)})}((x,y),(\r(x),\s(y))) = (x,\r(x)) \in a(\g_x),$$ which shows that $\Ub p_a \circ (\r\times \s) = \r\circ \Ub p_a$  (recall notation $p_a = (\p_a,\pf_a)$).  We conclude that  $(a\times b, X\times Y,\r\times \s)$ is a deterministic  hybrid system, and that the projection maps induce maps of deterministic hybrid systems. 
\end{remark}


\begin{remark}\label{remark:functorialityT&U}
	A mantra we will repeatedly use is ``by functoriality of $T$ and $\Ub$.'' The differential of a map of manifolds is a functor (\cite[\S 10]{tu}) and  $\Ub:\sF{HyPh}\rightarrow \sF{Man}$ is as well (\cref{prop:ForgetfulFunctor}).  The  proof of \cref{lemma:detHySysCategory} illustrates application of this phrase, and in future such applications we may circumvent  detailed calculation by citing this phrase.
\end{remark}
\begin{lemma}\label{lemma:detHySysCategory}
Deterministic  hybrid systems (\cref{def:detHySys}) and their	morphisms (\cref{def:detHySysMorphism}) form a category, $\sF{dHySys}$.
\end{lemma}

\begin{proof}
	We must show that objects have an identity morphism and that composition of maps is associative.  Let $(a,X,\r)$ be a deterministic  hybrid system. The identity map is the identity $id_a$ of hybrid phase space $a$,  which  induces a map $(a,X,\r)\xrightarrow{id_a}(a,X,\r)$ of deterministic hybrid systems since the diagrams $$\begin{array}{lll} \begin{tikzcd} \Ub a \arrow[r,"\Ub(id_a)"]\arrow[d,"X"] & \Ub a\arrow[d,"X"]\\
	T\Ub a \arrow[r,"T\Ub(id_a)"] & T\Ub a\end{tikzcd} & \mbox{and} & \begin{tikzcd} \Ub a \arrow[r,"\Ub(id_a)"]\arrow[d,"\r"] & \Ub a\arrow[d,"\r"]\\
	\Ub a \arrow[r,"\Ub(id_a)"] & \Ub a\end{tikzcd} 
\end{array}$$ both commute, showing that $(X,X)$ and $(\r,\r)$ are both $id_a$-related.  The diagrams commute because $\Ub id_a  = id_{\Ub a}$ and $T\Ub id_a = Tid_{\Ub a} = id_{T \Ub a}$ \textit{by functoriality of  $\Ub$} (``functoriality on identity'' (\cref{remark:functoriality})).

For associativity of composition, let $(a,X,\r)\xrightarrow{f} (b,Y,\s) \xrightarrow{g} (c,Z,\t)\xrightarrow{h}(d,W,\upsilon) $ be a string of composable  morphisms of deterministic  hybrid systems. We must show that $\big(h \circ g\big) \circ f = h\circ \big( g\circ f\big)$. This equality holds \textit{as morphisms of hybrid phase spaces}  (\cref{lemma:HyPhisaCategory}). So it suffices to check that the composition of morphisms of deterministic  hybrid systems is itself a morphism of deterministic  hybrid systems.

Each subdiagram of diagram $$\begin{tikzcd}[column sep = large]
	\Ub a \arrow[r,"\Ub f"]\arrow[rr,bend left,"\Ub (g\circ f)"] \arrow[d,"X"] & \Ub b\arrow[r,"\Ub g"]\arrow[d,"Y"] & \Ub c \arrow[d,"Z"]\\
	T\Ub a \arrow[r,swap,"T\Ub f"]\arrow[rr,bend right,swap,"T\Ub(g\circ f)"]  & T\Ub b\arrow[r,swap,"T\Ub g"]& T\Ub c\end{tikzcd}$$ commutes by functoriality of $T$ and $\Ub$. Therefore, the outer diagram $$\begin{tikzcd}[column sep = large] \Ub a\arrow[d,"X"] \arrow[rr,"\Ub(g\circ f)"] & & \Ub c\arrow[d,"Z"]\\
T\Ub a \arrow[rr,swap,"T\Ub(g\circ f)"] & & T\Ub c	
\end{tikzcd}$$  commutes.  A similar commuting diagram $$\begin{tikzcd}[column sep = large]
	\Ub a \arrow[d,"\r"]\arrow[r,"\Ub f"]\arrow[rr,bend left,"\Ub(g\circ f)"] & \Ub\arrow[d,"\s"] \arrow[r,"\Ub g"] & \Ub c\arrow[d,"\t"]\\
	\Ub a \arrow[r,swap,"\Ub f"]\arrow[rr,bend right,swap,"\Ub( g\circ f)"] & \Ub b \arrow[r,swap,"\Ub g"] & \Ub c \end{tikzcd}$$ shows that $\Ub\big(g \circ f\big)\circ \r  = \t\circ \Ub( g\circ f)$, and 
 hence that $g \circ f$ is a morphism of deterministic  hybrid systems. \end{proof}

 We now realize determinism in hybrid systems through \textit{executions}, the deterministic hybrid version of integral curves. They are defined as  a special class of maps of deterministic  hybrid systems (compare with \cref{def:integralCurveAsMap}).  To define executions, we first need to separate a special class of deterministic  hybrid systems. 

\begin{definition}\label{def:TuniversalDetHySys}
	
Let $\Tc = \{t_0<t_1<\ldots \}\subset \R$ be an increasing sequence of real numbers;  we define a $\Tc$-\textit{universal deterministic  hybrid system} $(\omega:\Sb^\omega\rightarrow\sF{RelMan},T,\tau)_\Tc$ as follows: \begin{enumerate}
	\item Hybrid phase space $\omega$ with source category  $\Sb^\omega \cong \N$: $$ 0 \xrightarrow{e_{1,0}} 1 \xrightarrow{e_{2,1}}2\rightarrow \cdots \rightarrow j \xrightarrow{e_{j+1,j}}j+1 \cdots$$  To each node $j\in \Sb^\omega_0$ we assign manifold $\omega(j)\defeq [t_j,t_{j+1}] \times\{j\}$.  To edge $e_{j+1,j}\in \Sb^\omega_1$ we assign the one-element relation $\omega(e_{j+1,j})=\left\{\big((t_{j+1},j),(t_{j+1},j+1)\big)\right\}$. 
	\item Vector field $T\in \Xf\left(\disu_{j\in \N} [t_j,t_{j+1}]\times\{j\}\right)$ is defined by  the constant vector field $T(t_0,j)\defeq\left.\frac{d}{dt}\right|_{(t_0,j)}.$
	\item Jump map $\t:\Ub\omega\rightarrow \Ub\omega$ is defined by $\t(t,j)= \left\{ \begin{array}{lll} (t,j) & \mbox{if} & t\in [t_j,t_{j+1}) \\ (t_{j+1},j+1) & \mbox{if} & t= t_{j+1}.\end{array}\right.$
\end{enumerate}
Notice that $\big((t_{j+1},j),(t_{j+1},j+1)\big) \in \omega(e_{j+1,j})$ so $(\omega,T,\t)$ is indeed a deterministic  hybrid system (\cref{remark:checkingJumpMap}).  When the sequence $\Tc=\{t_0<t_1<\cdots\}$ is clear or fixed ahead of time, we drop $\Tc$ in the subscript and simply write $(\omega,T,\t)$. 
\end{definition}

\begin{remark}\label{remark:TuniversalUniquelyDefinedByPartition}
An increasing sequence $\Tc=\{t_0,t_1,\ldots\}$ of real numbers uniquely defines a $\Tc$-universal system. 

	When $\Tc=\{t_0<\ldots<t_{k+1}\}$ is a finite set, we write $n_k$ instead of $\omega$, $\Sb^{n_k}$ is the finite category $$0\xrightarrow{e_{1,0}}1\rightarrow \cdots \rightarrow k-1\xrightarrow{e_{k,k-1}} k,$$  and $\Ub n_k  \cong \disu_{j=0}^{k}[t_j,t_{j+1}]\times \{j\}$.  In the case that $t_{k+1}=\infty$, the last interval in this union is $[t_{k},\infty)\times \{k\}$.  We abuse notation and continue to write $[t_k,t_{k+1}]$ when $t_{k}=\infty$. 
	\end{remark}
 
 Having demarcated our special class of deterministic  hybrid systems, we may now define executions---by analogy with \cref{def:integralCurveAsMap}---as maps of deterministic  hybrid systems.
\begin{definition}\label{def:deterministicExecution}
	Let $\Tc = \{t_0<t_1<\ldots\}\subset \R$.  A $\Tc$-\textit{execution} $(\epsilon,\ef):(\omega,T,\t)_\Tc\rightarrow(a,X,\rho)$ \textit{of deterministic  hybrid system} $(a,X,\rho)$ is a map of deterministic  hybrid dynamical systems (\cref{def:detHySysMorphism}) from  $\Tc$-universal deterministic  hybrid system $(\omega,T,\t)_\Tc$ (\cref{def:TuniversalDetHySys}). We may notate an execution $(\epsilon,\ef)$ by $e:(\omega,T,\t)_\Tc\rightarrow (a,X,\r)$. 
\end{definition}

Let us see how executions are the hybrid version of integral curve. On each interval $[t_j,t_{j+1}]$, the map $\ef_j:\omega(j)\rightarrow a(\epsilon(j))$ \textit{is} an integral curve for $X$ (c.f.\ \cref{def:integralCurveAsMap}).  This is the  continuous-time part.  At endpoint $(t_{j+1},j)\in [t_j,t_{j+1}]\times \{j\}$, we send point $\ef_j(t_{j+1},j)\mapsto \ef_{j+1}(t_{j+1},j+1)$. The condition that $\r\circ \Ub e= \Ub  e\circ \t$ requires  at point $(t_{j+1},j)\in [t_j,t_{j+1}]\times\{j\}$,  that  $\ef_{j+1}(t_{j+1},j+1) = \ef_{j+1}(\t(t_{j+1},j)) = \r(\ef_j(t_{j+1},j)$. This is the hybrid jump. 


A motivation for calling these systems deterministic  comes from the following.  

\begin{prop}\label{prop:uniqueExecutions} 
	Let $(a,X,\rho)$ be a deterministic  hybrid system and let   $\r$ be idempotent, i.e.\ $\r^2 = \r$.   Suppose that for each $x_0\in \Ub a\setminus \scJ_a=\{x\in \Ub:\, \r(x) = x\}$ (\cref{remark:jumpSet})  the maximal solution $\ph_{X,x_0}(t)$ of continuous-time system $(\Ub a, X)$ starting at $x_0$ (\cref{def:solutionToDynamicalSystem}) is either \begin{enumerate}
		\item complete $\ph_{X,x_0}:[0,\infty)\rightarrow \Ub a \setminus \scJ_a$ or 
		\item  bounded  $\ph_{X,x_0}:[0,t_{x_0})\rightarrow \Ub a\setminus \scJ_a$ and  $\ph_{X,x_0}(t_{x_0}) \in \scJ_a$. 
	\end{enumerate}   Then   $(a,X,\rho)$ has unique  executions. 
\end{prop}

\begin{proof}
	We construct the execution directly. If $\ph_{X,x_0}$ is complete, there is nothing to do: let $t_0=0$ and $t_1=\infty$.  An execution starting at $x_0$ is a solution of $(\Ub a,X)$ in the ordinary sense of continuous-time systems (\cref{def:solutionToDynamicalSystem}).   Otherwise, let   $\ph_{X,x_0}:[0,t_{x_0}]\rightarrow\Ub(a)$ be the  maximal integral curve starting at $x_0$.  We define $\Tc=\{t_i:\, i\in \N\}$ recursively:  set $t_0= 0$ and $t_1= t_{x_0}$. By assumption, $\ph_{X,x_0}(t_1)\in \scJ_a$. We set $x_1 \defeq \r(\ph_{X,x_0}(t_{x_1}))$ and $t_2 = t_1+ t_{x_1}$ where $\ph_{X,x_1}:[0,t_{x_1}]\rightarrow\Ub(a)$ is the maximal integral curve starting at $x_1$.  In general, we define $x_{j+1} \defeq  \r(\ph_{X,x_j}(t_{x_j}))$ where $t_{x_j}$ denotes the endpoint of the maximal integral curve $\ph_{X,x_j}:[0,t_{x_j}]\rightarrow\Ub(a)$ with initial condition $x_j$, and set $t_{j+1} \defeq t_j+t_{x_j}$. 
	
	Having thus defined set $\Tc=\{t_0,t_1,\ldots\}$, we obtain $
\Tc$-universal deterministic  hybrid system $(\omega,T,\tau)_\Tc$ (\cref{remark:TuniversalUniquelyDefinedByPartition}).  We define execution $(\epsilon, \ef):(\omega,T,\t)_\Tc\rightarrow (a,X,\r)$ as follows.   On nodes this map is defined $\epsilon(j)\defeq  \fs(x_j)$, where $\fs:\Ub a\rightarrow \Sb^a_0$  picks out the node $\fs(x)$ of which point $x$ is an element (\cref{lemma:sourceSet4Coproduct}) and $\ef_j(t,j):=\ph_{X,x_j}(t-t_j)$. By construction, $$  \r(\ef_k(t_{k+1},k))  =\r(\ph_{X,x_0}(t_{x_k})) =x_{j+1}= \ef_{k+1}(t_{k+1},k+1) = \ef_{k+1}(\t(t_{k+1},k)),$$ so  this is indeed a map of deterministic  hybrid systems (\cref{remark:checkingJumpMap}) and hence an execution.  \end{proof}
\begin{remark}\label{remark:uniqueExecutions}
	We will by default suppose that all deterministic hybrid systems $(a,X,\r)\in \sF{dHySys}$ have unique executions.  In other words, $\sF{dHySys}$ is the category whose objects are deterministic hybrid systems with unique executions. 
\end{remark}

\begin{example}\label{ex:bouncingBallExecution}
We return to the bouncing ball from \cref{ex:bouncingBallVanilla}, and consider trajectories in our formalism (c.f.\ \cite[\S1.2.2]{liberzonswitch}). 
It is not difficult to verify that the conditions of  \cref{prop:uniqueExecutions}  are here satisfied, so that this system has unique executions.  We build them explicitly.  At  $t_0=0$, suppose that $h(0) = 0$, $v(0) = 1/2$, and let $r\in (0,1)$ denote the coefficient of restitution (loss of energy at bounce).  Solving the differential equation $\left\{ \begin{array}{ll}\dh & = v\\
  \dv & = -1 \end{array}\right.$	
 gives $$v(t) = -(t-t_0) +v(t_0)=-t+1/2$$ and $$h(t) = -(t-t_0)^2/2+v(t_0)(t-t_0) = -t^2/2+1/2t.$$  To find the next bounce time, we set $h(t) = 0$ or $-t^2/2+1/2t = 0$ which occurs at $t=0$ and $t=1$.  Thus the first bounce time is  $t_1=1$, with downward velocity $v(t_1^-) = -1/2$, and after-bounce velocity $v(t_1^+) = -r\cdot (-1/2)= r/2$.  Letting $t_1$ play the role of $t_0$ above, for $t>t_1$, we have $v(t) = -(t-t_1) +r/2$ and $h(t) = -(t-t_1)^2/2 + r/2\cdot(t-t_1)$. Again solving for $h(t)=0$ gives $t=t_1$ and $t=1+r$.  We claim that the $k$-th bounce time is $$t_k=\diss_{j=0}^{k-1} r^j.$$ The zeroth jump time is $t_0=0$.  Arguing by  induction, we have $v(t_{k}^-) = -r^{k-1}/2$, $v(t_{k}^+) = r^{k}/2,$ $v(t) = -(t-t_k) + v(t_k^+)$, and $h(t) = -\frac{1}{2}(t-t_k)^2+ \frac{r^{k}}{2}(t-t_k)$. Setting $h(t) = 0$, we see that $h(t)=0$ at $t=t_k$ and at $$t=t_k+r^{k} = \diss_{j=0}^{k-1}r^j+r^{k} = \diss_{j=0}^{k} r^j,$$ the latter of which is the $(k+1)$th jump time.

Setting  $\Tc=\{t_0,t_1,\ldots\}\subset \R$ where $t_k\defeq\diss_{j=0}^kr^j$, this data entirely defines an execution  $$(\epsilon,\ef):(\omega,T,\t)_\Tc\rightarrow (c,Z,\mu),$$  where $(c,Z,\mu)$ is the deterministic hybrid system representing the bouncing ball in \cref{ex:bouncingBallVanilla}. The map $\epsilon:\Sb^n_0\rightarrow \Sb^c_0$ of source categories sends  each node $(j\in \Sb^n_0)\mapsto (0\in \Sb^c_0)$.   For $(t,k)\in [t_k,t_{k+1}]\times\{k\}$, we have $$\left\{\begin{array}{lll}v(t)&  = -(t-t_k)+v(t_k^+) & = -(t-t_k) +\frac{r^{k}}{2} \\ h(t) &  = -\frac{1}{2}(t-t_k)^2+\frac{r^{k}}{2}(t-t_k),& \end{array}\right.$$  and we set   $\ef_k(t,k) \defeq   (h(t),v(t))$.  Then $$\mu(\ef_k(t,k)) =\left\{\begin{array}{ll}  (0,-r \cdot v(t))& \mbox{if} \; h=0\;\mbox{and}\;v<0 \\ (h(t),v(t)) & \mbox{else.} \end{array}\right.$$ The condition that $h=0$ and $v<0$ is realized  at $t = t_{k+1}$.  On the other hand, $$\t(t,k) = \left\{\begin{array}{ll} (t,k) & \mbox{if}\; t\in [t_k,t_{k+1})\\ (t,k+1) &\mbox{if} \, t = t_{k+1},\end{array}\right.$$ from which  we see that $$\ef_k(\t(t,k)) =\ef_k(t,k)= (h(t),v(t))=\mu(h(t),v(t))= \mu(\ef_k(t,k)),$$ when $t\in [t_k,t_{k+1})\times\{k\}$  and  $$\ef_{k+1}(\t(t_{k+1},k)) = \ef_{k+1}(t_{k+1},k+1) = \left(0,\frac{1}{2}r^{k+1}\right) = \left(0,-r\cdot \frac{-r^k}{2}\right)  = \mu(0,v(t_{k+1}^-))= \mu(\ef_k(t_{k+1},k)),$$  when $t= t_{k+1}$. 
Therefore $\Ub(\epsilon,\ef)\circ \tau = \mu \circ \Ub(\epsilon,\ef)$, 
and we conclude that $(\epsilon,\ef):(\omega,T,\t)_\Tc\rightarrow (c,Z,\mu)$ is a deterministic  execution.
\end{example}

\subsection{Deterministic  Hybrid Open  Systems}

Now we discuss open deterministic  dynamics.  The extra piece of data we introduced in \cref{subsection:deterministicHybridSystems} was a jump map, somehow compatible with relations of the underlying phase space.  In the case of open systems, we also have a jump map this time from the underlying total space $\Ub a_{tot}$ to the state space $\Ub a_{st}$ (compare with the open control $X:\Ub a_{tot}\rightarrow T\Ub a_{st}$).  We simply need to reinterpret compatibility-with-relations, as now there are two spaces of constraints to consider. 

\begin{definition}\label{def:deterministicHyOS}
	We define a \textit{deterministic   hybrid open system} to be a triple $(a,X,\rho)$ where 
	 $(a_{tot}\xrightarrow{p_a}a_{st},X)$ is a hybrid open system  (\cref{def:hybridOS})
and   $\r:\Ub a_{tot}\rightarrow\Ub a_{st}$ is a set map (morphism in the category $\sF{Set}$), which we call the \textit{jump map}, satisfying the following inclusion:  
\begin{equation}\label{eq:ratEq} \left(  \Ub p_a \times id\right)\big(\mbox{graph}(\r)\big) \subseteq \Lambda_{a_{st}}\left(\discats_{\g\in \Sb^{a_{st}}_1} a_{st}(\g)\right),\end{equation} where $\Lambda_{a}:\discats_{\g\in \Sb^a_1}a(\g)\rightarrow \Ub a\times \Ub a$ is defined in \cref{remark:inclusionForDeterminism}, \eqref{eq:coproductRelations2Product}. 
\end{definition}

\begin{remark}\label{remark:checkJumpMapOpen}
 For future reference, we denote map \begin{equation} 
\Ub a_{tot}\times \Ub a_{st}\xrightarrow{  \Ub p_a\times id_{\Ub a_{st}}}\Ub a_{st}\times a_{st}.
 \end{equation}on the left-hand side of \eqref{eq:ratEq} by \begin{equation}
 \label{eq:defineTheta}	  \Theta_a\defeq    \Ub p_a\times   id_{\Ub a_{st}}.
 \end{equation}	
 Analogous to \cref{remark:checkingJumpMap}, we have an operational way of checking condition \eqref{eq:ratEq}: for each $x\in \Ub a_{tot}$, there is an edge $\g_x\in \Sb^{a_{st}}_1$ such that $\big(\Ub p_a (x), \r(x)\big)\in a_{st}(\g_x)$.
\end{remark}

\begin{definition}\label{def:mapdeterministicHybridOS}
	We define a \textit{map} $f:(a,X,\rho)\rightarrow (b,Y,\s)$ \textit{of deterministic  hybrid open systems} to be  a map $f:(a,X)\rightarrow (b,Y)$ of hybrid open systems (\cref{def:morphismHybridOS}) such that 
		$(\r,\s)$ are $f$-\textit{related}, namely $\Ub f_{st} \circ \r = \s \circ \Ub f_{tot}$. 
\end{definition}

\begin{remark}\label{remark:dHyOSaCategory}
	Deterministic  hybrid open systems and morphisms form a category, which we denote by $\sF{dHyOS}$. 
\end{remark}

\begin{remark}\label{remark:deterministicHyOSProducts}
	Deterministic  Hybrid Open Systems have products $(a,X,\r)$ and $(b,Y,\s)$: the product $(a\times b, X\times Y,\r\times \s)$ is also a deterministic  hybrid open system. Compare with  a similar fact   for deterministic  hybrid closed systems (\cref{remark:deterministicHySysProducts}).
\end{remark}

We introduce the notion of an augmented tangent bundle, which helps us package data for deterministic in hybrid systems. 
\begin{definition}\label{def:ctBundle}\label{def:cdBundle}
 Let $a:\Sb^a_1\rightarrow \sF{RelMan}$ be a hybrid phase space.  We define the \textit{continuous-discrete  bundle} $\Tb a$ (also: \textit{c.d.\ bundle}) by \begin{equation}\label{eq:ctBundle}
 	\Tb a\defeq T \Ub a\times \big\{x\in\Ub a\big\},
 \end{equation} the product of the tangent bundle of the underlying manifold $\Ub a$  (\cref{prop:ForgetfulFunctor}) and the underlying manifold as a set.  The appellation ``bundle'' is not accidental: $\Tb a$ comes equipped with a canonical projection $\varpi_a :\Tb a\rightarrow\Ub a$ defined by \begin{equation} \label{eq:definingCDBundleProj} \varpi_a\defeq \t_{\Ub a}\circ p_1,\end{equation} where $p_1:T\Ub a\times\big\{x\in  \Ub a\big\}\rightarrow T\Ub a$ is the canonical projection onto the first factor, and $\t_{\Ub a}:T\Ub a\rightarrow\Ub a$ is the canonical projection of the \textit{tangent} bundle of the underlying manifold $\Ub a$. 
\end{definition}

\begin{remark}\label{remark:deterministicHySysAsSection}
	A deterministic  hybrid system $\big(a,(X,\r)\big)$ is a pair where $a$ is a hybrid phase space and $(X,\r)$ is a section of the c.d.\ bundle satisfying some extra conditions, namely that $X$ is smooth  and $\r$ satisfies \eqref{eq:conditionFordeterministicJumps}. Compare this with continuous-time dynamical systems $(M,X)$, where $M$ is a manifold and $X\in \Xf(M)$ is a smooth section of the tangent bundle. We will encounter a general notion  of object and section of bundle in \cref{ch4}, \cref{subsection:abstractSections}. Similarly, a deterministic hybrid open system $\big(a,(X,\r)\big)$ is a pair where $a = a_{tot}\xrightarrow{p_a} a_{st}$ is a hybrid surjective submersion, and $(X,\r):\Ub a_{tot}\rightarrow \Tb a_{st}$ is a map to the c.d. bundle of the state space such that \begin{enumerate}
		\item $X:\Ub a_{tot} \rightarrow T\Ub a_{st} $ is smooth and $\Ub p_a = \t a_{st}\circ X$, and 
		\item $\Theta_{a}\big(\text{graph}(\r)\big) \subseteq \Lambda_{a_{st}}\left(\discats_{\g\in \Sb^{a_{st}}_1}a_{st}(g)\right),$ where $\Theta_a$ is as defined in \eqref{eq:defineTheta}. 
	\end{enumerate}
	 	\end{remark}

\begin{prop}\label{prop:cdBundleFunctorial}
	The assignment $\Tb a = T\Ub a\times \big\{x\in \Ub a\big\}$ in \eqref{eq:ctBundle} extends to a (covariant) functor $\Tb:\sF{HyPh}\rightarrow\sF{Set}$. 
\end{prop}
\begin{proof}[Proof Sketch.]
First, we have defined $\Tb a$ on objects in \eqref{eq:ctBundle}. Observe that  forgetting the smooth and topological structure of a manifold $M\mapsto\big\{x\in M\big\}$ is functorial.  We let $\{ \cdot\}$ denote this functor.  

	Let $a\xrightarrow{f}b$ be a morphism of hybrid phase spaces. We define \begin{equation}\label{eq:cdBundleFunctorial} \Tb f\defeq T \Ub f \times \{\Ub f\}.\end{equation} This map is functorial since $T$, $\Ub$, $\{\cdot\}$, and $\times$  are functorial. 
\end{proof}
We now define deterministic  control.
\begin{definition}\label{def:deterministicControlLax}\label{def:deterministicControlOnObjects}
	Let $a\defeq \big(a_{tot}\xrightarrow{p_a}a_{st}\big)$ be a hybrid surjective submersion. We define \footnotesize $$\begin{array}{ll}\dCrl(a) & \defeq \left\{(X,\r):\Ub a_{tot}\rightarrow \Tb  a_{st} \left| \t_{a_{st}}\circ X = \Ub p_a \,\mbox{and}\, \Theta_a(\mbox{graph}(\r)) \subseteq \Lambda_{a_{st}}\left( \discats_{\g\in\Sb^{a_{st}}_1} a_{st}(\g) \right) \right\},\right.
\end{array}
$$\normalsize  as the collection of  pairs $(X,\r)$ for which $(a,X,\r)$ is a deterministic  hybrid open system. Here $\t_{a_{st}}:T\Ub a_{st}\rightarrow\Ub a_{st}$ is the canonical projection of the tangent bundle, $\Theta_a:\Ub a_{tot} \times \Ub a_{st} \rightarrow\Ub a_{st}\times \Ub a_{st}$ is the map in \eqref{eq:defineTheta}, and  $\Lambda_{a_{st}}:\discats_{\g\in \Sb^{a_{st}}_1}a_{st}(\g)\rightarrow \Ub a_{st}\times \Ub a_{st}$ is the canonical  map defined in \cref{remark:inclusionForDeterminism}. 
	
	Having thus defined $\dCrl$ on objects, we now define $\dCrl$ on morphisms as a relation. 	For morphism $a\xrightarrow{f}b$ we define \small \begin{equation}\label{eq:defDetControlOnMorphisms} \dCrl(f) = \left\{ \big((X,\r),(Y,\s)\big)\in \dCrl(a)\times \dCrl(b)\left|\; (X,Y)\;\mbox{and}\;(\r,\s)\;\mbox{are}\; f\mbox{-related}\right.\right\}. \end{equation} \normalsize Alternatively, \small  $$\dCrl(f)= \left\{ \big((X,\r),(Y,\s)\big)\in \dCrl(a)\times \dCrl(b): (a,X,\r)\xrightarrow{f}(b,Y,\s)\,\mbox{is map in $\sF{dHyOS}$ (\cref{def:mapdeterministicHybridOS})}\right\}.$$\normalsize 
\end{definition}

\begin{prop}\label{prop:deterministicControlIsLax}
Deterministic  control (\cref{def:deterministicControlLax}) extends to a \textit{lax functor} 
 $\dCrl:\sF{HySSub}\rightarrow \sF{Rel}$  (\cref{def:Rel}).  Specifically, let $a\xrightarrow{f}b\xrightarrow{g} c$ be morphisms of hybrid surjective submersions. Then \begin{equation}\label{eq:laxFunctorialityForDetControl} \dCrl(g)\circ \dCrl(f)\subseteq \dCrl\big(g\circ f\big).\end{equation} 
  \end{prop}

\begin{proof}
	We have defined $\dCrl$ on objects and morphisms (\cref{def:deterministicControlLax} and \eqref{eq:defDetControlOnMorphisms}). 

 Lax functoriality \eqref{eq:laxFunctorialityForDetControl} follows immediately from the following two commutative diagrams:\small\begin{equation}\label{diagram:Related&CompatibleAreTheSame} \begin{tikzcd}[column sep = large, row sep = large]
	\Ub a_{tot} \arrow[r,swap,"\Ub f_{tot}"]\arrow[rr,bend left,"\Ub\big(g\circ f\big)_{tot}"] \arrow[d,"X"] & \Ub b_{tot}\arrow[r,swap,"\Ub g_{tot}"] \arrow[d,"Y"] & \Ub c_{tot}\arrow[d,"Z"] \\ T\Ub a_{st}\arrow[r,"T\Ub f_{st}"]\arrow[rr,bend right,swap, "T\Ub\big(g\circ f\big)_{st}"] & T\Ub b_{st}\arrow[r,"T\Ub g_{st}"] & T\Ub c_{st}
\end{tikzcd} \;\; \mbox{and} \;\;\begin{tikzcd}[column sep = large, row sep = large]
\Ub a_{tot} \arrow[d,"\r"]\arrow[rr,bend left,"\Ub\big(g\circ f\big)_{tot}"] \arrow[r,swap,"\Ub f_{tot}"] & \Ub b_{tot}\arrow[d,"\s"] \arrow[r,swap,"\Ub g_{tot}"] & \Ub c_{tot}\arrow[d,"\t"]\\
 \Ub a_{st}\arrow[rr,bend right,swap,"\Ub\big(g\circ f\big)_{st}"] \arrow[r,"\Ub f_{st}"] & \Ub b_{st} \arrow[r,"\Ub g_{st}"] & \Ub c_{st}.
\end{tikzcd}\end{equation}\normalsize  Each square commutes by assumption, and the outer diagrams by functoriality. Thus,  if $(X,Y)$ are $f$-related and $(Y,Z)$ are $g$-related then $(X,Z)$ are $g\circ f$-related. 
   A similar implication shows that $(\r,\t)$ are $g \circ f$-related.  This proves that $\dCrl(g)\circ \dCrl(f)\subseteq\dCrl\big(g\circ f\big)$.  Thus, for $\big((Y,\s),(Z,\t)\big)\in \dCrl(g)$ and $\big((X,\r),(Y,\s)\big)\in \dCrl(f)$ we have that $\big((X,\r),(Z,\t)\big)\in \dCrl\big(g\circ f\big)$. 
    \end{proof}
    
    \begin{remark}\label{remark:deterministicControlIsLAX}
    Laxness comes from unidirectionality of this implication: that $(X,Z)$ are $g\circ f$-related does not imply that both $(X,Y)$ are $f$-related and $(Y,Z)$ are $g$-related.  Similarly $(\r,\t)$ may be $g\circ f$-related without both $(\r,\s)$ being $f$-related and $(\s,\t)$ being $g$-related.  In general, the commuting of outer diagram $$\begin{tikzcd}
	\fx_1\arrow[r]\arrow[d] & \fx_2 \arrow[r]\arrow[d] & \fx_3\arrow[d]\\
	\fy_1\arrow[r] & \fy_2\arrow[r] & \fy_3
\end{tikzcd}$$ does not imply that both of the inner diagrams commute. 
 \end{remark}
\begin{example}\label{ex:strictInclusionOfCompositionOfRelations}
Consider the string of inclusions $\begin{tikzcd} \R\arrow[r,hookrightarrow,"\i_1"] & \R^2\arrow[r,hookrightarrow,"\i_2"] & \R^3\end{tikzcd}$, where $\i_1$ maps $x\mapsto (x,0)$ and $\i_2$ maps $(x,y)\mapsto (x,y,0)$.  On the one hand, $(g\circ f)$-relatedness of $(\r,\t)$ means that  $\t(x,0,0) = (\r(x),0,0)$ for $x\in \R$. On the other hand, $g$-relatedness of $(\s,\t)$  requires that $\t(x,y,0) = (\s(x,y),0)=(\s_1(x,y),\s_2(x,y),0)$. Then $(g\circ f)$-relatedness of $(\r,\t)$  imposes no condition on the second factor of $\s(x,y)$, and hence fails to ensure the equality $\t(x,y,0) = (\s(x,y),0)$ if $y\neq 0$ (a similar example for related vector fields is given in \cite[example 2.25]{lermanopennetworks}). \end{example}

We consolidate terminology: 
\begin{definition}
	We say that $(X,\r)\in \dCrl(a)$ and $(Y,\s)\in \dCrl(b)$ are $f$-\textit{related} for morphism $a\xrightarrow{f} b$ of hybrid surjective submersions if both $(X,Y)$ are $f$-related (\cref{def:morphismHybridOS}) and $(\r,\s)$ are $f$-related (\cref{def:mapdeterministicHybridOS}). 
\end{definition}

\subsection{Hybrid Interconnection and Deterministic Control}
Recall   that a morphism of hybrid surjective submersions is a pair of morphisms of hybrid phase spaces making a certain diagram commute (\cref{def:hybridSSubMorphism}).

\begin{definition}\label{def:HybridInterconnection}
Let $a\defeq \big(a_{tot}\xrightarrow{p_a}a_{st}\big)$ ,  $b\defeq \big(b_{tot}\xrightarrow{p_b}b_{st}\big)$ be hybrid surjective submersions (\cref{def:hybridSurjectiveSubmersion}). We define a \textit{hybrid interconnection} $i:a\rightarrow b$  to be a morphism  of hybrid surjective submersions (\cref{def:hybridSSubMorphism}) for which the map  $i_{st}:a_{st}\rightarrow b_{st}$ on state is an isomorphism of hybrid phase spaces. 
\end{definition}

It is easy to verify that hybrid surjective submersions with interconnection morphisms form a category.

\begin{definition}\label{def:HySSub-intCat}\label{remark:hybridInterconnectionsAreCategory}
	We define the category $\sF{HySSub_{int}}$ whose objects are hybrid surjective submersions (\cref{def:hybridSurjectiveSubmersion}) and whose morphisms are hybrid interconnections (\cref{def:HybridInterconnection}).  $\sF{HySSub_{int}}$ is a subcategory of $\sF{HySSub}$ with  the same objects. 
\end{definition}

\begin{remark}\label{remark:HyPhIsomorphism}
It will be useful for us to spell out the definition of an isomorphism of hybrid phase spaces.  An isomorphism in any category is an invertible morphism.  In the category of hybrid phase spaces, specifically,  an isomorphism $i = (\i,\ifrak):a\rightarrow b$ consists of an isomorphism $\i:\Sb^a\xrightarrow{\sim} \Sb^b$ of categories and  a diffeomorphism $\ifrak_\fs:a(\fs)\xrightarrow{\sim}b(\i(\fs))$  for each node $\fs\in \Sb^a_0$.  Consequently, one may easily check that the inverse is given by  $(\i,\ifrak)^{-1}= (\i^{-1},\ifrak^{-1})$.  \end{remark}

\begin{remark}\label{remark:introduceInterconnection}
The lax functor $\dCrl$ in  \cref{def:deterministicControlLax} does not in general map deterministic  controls to deterministic  control.  As we saw, the best we can hope for from an arbitrary map of hybrid surjective submersions $a\xrightarrow{f}b$  is a \textit{relation} of control. However, $\dCrl$ applied to interconnection does map control to control. 
\end{remark}

\subsection{Deterministic  Control as a Map}
\begin{prop}\label{prop:interconnectionMapOndeterministicHybridControl}\label{prop:dCrlInterconnectIsMap} Let $a\defeq\big( a_{tot}\xrightarrow{p_a} a_{st}\big)$ and $b\defeq\big( b_{tot}\xrightarrow{p_b} b_{st}\big)$ be two hybrid surjective submersions,   $i:a\rightarrow b$  a hybrid interconnection (\cref{def:HybridInterconnection}),  and $(Y,\s) \in \dCrl(b)$.  Let  \begin{equation}\label{eq:interconnectControl} X\defeq T\Ub(i_{st})^{-1}\circ Y\circ\Ub(i_{tot}) \;\;\; \mbox{and}\;\;\; \r\defeq\Ub(i_{st})^{-1}\circ \s\circ\Ub(i_{tot}). \end{equation} Then $(X,\r)\in \dCrl(a)$.  \end{prop}
Consequently,  for interconnection morphism  $i = (\i,\ifrak):a\rightarrow b$ there is a well-defined map  $\dCrl(i): \dCrl(b)  \rightarrow \dCrl(a) $ given by $(Y,\s)\mapsto \dCrl(i)(Y,\s) \defeq (X,\r)$, where $(X,\r)$ is defined in \eqref{eq:interconnectControl}.  Proposition \ref{prop:interconnectionMapOndeterministicHybridControl} guarantees that $\dCrl$ lands in the target.   We collect this fact in a definition:
\begin{definition}\label{def:interconnectionMapOndeterministicHybridControl}
Let $i=(\i,\ifrak):a\rightarrow b$ be an interconnection of hybrid surjective submersions (\cref{def:HybridInterconnection}). We define the map $\dCrl(i): \dCrl(b)\rightarrow \dCrl(a)$  by  \begin{equation}\label{eq:definingdcrlOnInterconnection} (Y,\s)\mapsto  \dCrl(i)(Y,\s)\defeq \big( T\Ub(i_{st})^{-1}\circ Y\circ\Ub(i_{tot}),  \Ub(i_{st})^{-1}\circ \s\circ\Ub(i_{tot})\big).\end{equation} \end{definition}

We need a lemma for the proof of \cref{prop:interconnectionMapOndeterministicHybridControl}.
\begin{lemma}\label{lemma:isoHyPhImpliesEqualityOnRelations}
	Let $(\i,\ifrak):a\rightarrow b$ be an isomorphism of hybrid phase spaces. Then for every arrow $\fs'\xrightarrow{\g}\fs\in \Sb^a_1$, there is equality of relations $$\big(\ifrak_{\fs'}\times \ifrak_\fs \big) \big( a(\g)\big)  = b(\i(\g)).$$ Moreover $$\big(\Ub i\times \Ub i\big) \left(\Lambda_a\left(\discats_{\g\in \Sb_1^a} a(\g) \right)\right)= \Lambda_b\left(\discats_{\eta\in \Sb_1^b} b(\eta)\right),$$
where $\Lambda_a:\discats_{\g\in \Sb^a_1}a(\g)\rightarrow \Ub a \times \Ub a$ (\cref{remark:inclusionForDeterminism}). \end{lemma} 
	\begin{proof}
		Since $(\i,\ifrak):a\rightarrow b$ is a morphism of hybrid phase spaces, for  edge $\fs'\xrightarrow{\g}\fs\in \Sb^{a}_1$,  we have inclusion  $$\big(\ifrak_{\fs'}\times \ifrak_\fs \big) \big(a(\g)\big) \subseteq b(\i(\g)).$$ We must show the opposite inclusion $b(\i(\g))\subseteq \big(\ifrak_{\fs'}\times \ifrak_\fs\big)\big(a(\g)\big).$
		
	The inverse $(\i,\ifrak)^{-1}$ is also a morphism of hybrid phase spaces equal to $(\i^{-1},\ifrak^{-1})$  (\cref{remark:HyPhIsomorphism}). Let $\eta \defeq \i(\g)$.  Then  we  have inclusion $$\big(\ifrak_{\i(\fs')}^{-1}\times \ifrak_{\i(\fs)}^{-1} \big)\big(b(\eta)\big) \subseteq  a\big(\i^{-1}(\eta)\big) = a(\g).$$ Applying $\ifrak_{\fs'}\times \ifrak_\fs$ to both sides: $$b(\eta)  = \big((\ifrak_{\fs'}\circ \ifrak_{\i(\fs')}^{-1})\times (\ifrak_\fs\circ \ifrak_{\i(\fs)}^{-1})\big)  \big(b(\eta)\big)=\left(\big(\ifrak_{\fs'}\times \ifrak_\fs\big)\circ \big(\ifrak_{\i(\fs')}^{-1}\times \ifrak_{\i(\fs)}^{-1}\big)\right) \big(b(\eta)\big)  \subseteq  \big(\ifrak_{\fs'}\times \ifrak_\fs  \big)\big(a(\g)\big),$$ which proves that $\big(\ifrak_{\fs'}\times \ifrak_\fs \big)\big(a(\g)\big) = b(\i(\g))$.
	
	We introduce  notation: let $\ifrak_\g\defeq \ifrak_{\sF{cod}(\g)}\times\ifrak_{\sF{dom}(\g)}$ where $\sF{dom}(\g)\xrightarrow{\g}\sF{cod}(\g)$ is edge in $\Sb^a_1$. Then $\ifrak_\g \big(a(\g)\big)  = b(\i(\g))$ for every $\g\in \Sb^a_1$ and the isomorphism $\i:\Sb^a\xrightarrow{\sim }\Sb^b$ imply that $\discats_{\g\in \Sb^a_1}a(\g) \cong \discats_{\eta\in \Sb^b_1}b(\eta)$ by the universal property of coproduct (\cref{def:coproducts}).  This can be verified by  the following diagram: $$\begin{tikzcd}[column sep = large, row sep = large]
	\discats_{\g\in \Sb^a_1}a(\g)\arrow[r,dashed,"\xi_a"]\arrow[rr,bend left,dashed,swap,"id"] & \discats_{\eta\in \Sb^b_1}b(\eta) \arrow[r,dashed,"\xi_b"] & \discats_{\g\in \Sb^a_1}a(\g) \\
	a(\g')\arrow[u,hookrightarrow,"in_{\g'}"]\arrow[rr,bend right,"id_{a(\g')}"]\arrow[r,"\ifrak_{\g'}"] & b(\i(\g'))\arrow[u,hookrightarrow,"in_{\i(\g')}"]\arrow[r,"\ifrak^{-1}_{\g'}"] & a(\i^{-1}(\i(\g'))),\arrow[u,hookrightarrow,"in_{\g'}"] 
\end{tikzcd}$$ where the maps $in_{\g'}$ and $in_{\i(\g')}$ are the canonical inclusions.  Thus, $\xi_a\left(\discats_{\g\in \Sb_1^a}a(\g)\right) = \discats_{\eta\in \Sb^b_1}b(\eta)$ and $\xi_b = \xi_a^{-1}$. 

Let $\zeta_a\defeq \Ub i\times \Ub i$, $\z_b\defeq \Ub i^{-1}\times \Ub i^{-1}$.   We then conclude from the universal property of coproduct (this time for $\Ub a$) and commuting diagram \small   $$\begin{tikzcd}[column sep = huge, row sep = large]
	\discats_{\g\in \Sb^a_1}a(\g)\arrow[r,"\xi_a"]\arrow[rr,bend left,"id"]\arrow[d,"\Lambda_a"] & \discats_{\eta\in \Sb^b_1}b(\eta)\arrow[r,"\xi_b"]\arrow[d,"\Lambda_b"] & \discats_{\g\in \Sb^a_1} a(\g)\arrow[d,"\Lambda_a"] \\
	\Ub a\times \Ub a\arrow[rr,dashed,bend right, "id_{\Ub a}\times id_{\Ub a}"] \arrow[r,dashed,"\z_a"] & \Ub b\times \Ub b\arrow[r,dashed,"\z_b"] & \Ub a\times \Ub a
\end{tikzcd}$$ \normalsize that $\big (\Ub i \times \Ub i\big)\left(\Lambda_a\left(\discats_{\g\in \Sb_1^a} a(\g)\right) \right) = \Lambda_b\left(\discats_{\eta\in \Sb_1^b}b(\eta)\right)$.
	\end{proof}
\begin{proof}[Proof of \cref{prop:interconnectionMapOndeterministicHybridControl}]
Let $(X,\r) \defeq \dCrl(\i,\ifrak)(Y,\s)$ as in \cref{prop:interconnectionMapOndeterministicHybridControl} and  \eqref{eq:interconnectControl}.  We must show that $\big(a,X,\r\big)$ is a deterministic  hybrid system: namely that $\Ub(p_a) = \t_{a_{st}}\circ  X$  and that $$\Theta_a\big(\mbox{graph}(\r)\big)\subseteq \Lambda_{a_{st}}\left(\discats_{\g\in \Sb^{a_{st}}_1} a_{st}(\g)\right),$$ where $\t_{a_{st}}:T\Ub a_{st}\rightarrow \Ub a_{st}$ is the canonical projection of the tangent bundle, and $\Theta_a$ is defined in \eqref{eq:defineTheta}. 

To demonstrate the equality  $\Ub(p_a) = \t_{a_{st}}\circ  X$, we refer to the diagram: \small  \begin{equation}\label{diagram:relatedVectorFields}  \begin{tikzcd}[column sep = large, row sep = large]
	& \Ub a_{tot} \arrow[dd,near start,swap,"\Ub(p_a)"] \arrow[dl,swap,"X"] \arrow[rr,"\Ub(i_{tot})"] & & \Ub b_{tot}\arrow[dd,swap,"\Ub(p_b)"] \arrow[dl,swap,"Y"]\\
	T\Ub a_{st}\arrow[rr,shift right,near end,swap,"T\Ub(i_{st})"]\arrow[dr,"\t_{a_{st}}"] & & T\Ub b_{st}\arrow[dr,"\t_{b_{st}}"]\arrow[ll,shift right, near start,swap,"T\Ub i_{st}^{-1}"]& \\
	& \Ub a_{st}\arrow[rr,"\Ub(i_{st})"] & & \Ub b_{st},
\end{tikzcd}\end{equation} \normalsize where each subdiagram---except the left triangle---is already known to commute. Equality \begin{equation}\label{eq:p1} \Ub(i_{st})\circ \Ub(p_a)= \Ub(p_b)\circ \Ub (i_{tot})\end{equation}  holds because $i$  is  a map of hybrid surjective submersions.   Equality \begin{equation}\label{eq:p2} \t_{b_{st}}\circ Y = \Ub(p_b)\end{equation}  follows by assumption that $(b,Y)$ is a hybrid open system. Equality \begin{equation}\label{eq:p3} X = T\Ub(i_{st})^{-1}\circ Y \circ \Ub(i_{tot})\end{equation}   follows by definition of $X$  (c.f.\ \eqref{eq:interconnectControl}). Finally, equality \begin{equation}\label{eq:p4} \Ub(i_{st})\circ \t_{a_{st}}\circ  T\Ub(i_{st})^{-1} = \t_{b_{st}}\end{equation} follows from the fact that $\t$ is natural (\cref{fact:naturalityOfProjTangentBundle}) and that $\Ub(i_{st})$ is a diffeomorphism (\cref{fact:isoHyPhForgetsToDiffeo}). 

Starting from $\Ub(p_a)$, we thus have a string of equalities  $$\begin{array}{lll} \Ub(p_a) & = \Ub(i_{st})^{-1}\circ \Ub(p_b)\circ \Ub(i_{tot}) & (\mbox{c.f.\ \eqref{eq:p1}})\\
& = \Ub(i_{st})^{-1} \circ \t_{b_{st}}\circ Y \circ \Ub(i_{tot}) & (\mbox{c.f.\ \eqref{eq:p2}}) \\ 
& = \Ub(i_{st})^{-1} \circ \t_{b_{st}} T\Ub(i_{st})\circ T\Ub(i_{st})^{-1}\circ Y \circ \Ub(i_{tot}) &  (T\;\mbox{is functorial})\\
& = \t_{a_{st}} \circ X & (\mbox{c.f.\ \eqref{eq:p3} and \eqref{eq:p4}}),
\end{array}$$ thus proving that $(a,X)$ is a hybrid open system (\cref{def:hybridOS}). 

To show that $(a,X,\r)$ is \textit{deterministic} hybrid open system,  we must show that $\Theta_a\big(\mbox{graph}(\r)\big) \subseteq \Lambda_{a_{st}}\left(\discats_{\g\in \fS(a_{st})_1} a_{st}(\g)\right)$, where  $\Theta_a\defeq \Ub(p_a)\times  id_{\Ub a_{st}} $ (c.f.\ \eqref{eq:defineTheta}).  Alternatively, we show that for each $x\in \Ub a_{tot}$, there is $\g_x\in \Sb^{a_{st}}_1$ such that $\big(\Ub p_a(x), \r(x)\big)\in a_{st}(\g_x)$ (\cref{remark:checkJumpMapOpen}).  For verification of this relation, we refer to the (not entirely commuting!)  diagram \begin{equation}\label{eq:diagramForLemmaInterconnectControlMap}\begin{tikzcd}
	\Ub a_{tot}\arrow[dd,near start,swap,"\r"]\arrow[rr,"\Ub(i_{tot})"]\arrow[dr,"\Ub(p_a)"] & & \Ub b_{tot}\arrow[dr,"\Ub(p_b)"]\arrow[dd,near start,swap,"\s"] & \\
	& \Ub a_{st}\arrow[rr,near start,"\Ub(i_{st})"]\arrow[dd,near start,"id_{\Ub a_{st}}"] & & \Ub b_{st}\arrow[dd,near start,"id_{\Ub b_{st}}"] \\
	\Ub a_{st}\arrow[dr,near start,swap,"id_{\Ub a_{st}}"]\arrow[rr,near start,"\Ub(i_{st})"] & & \Ub b_{st}\arrow[dr,near start,"id_{\Ub b_{st}}"] & \\
	& \Ub a_{st} \arrow[rr,near start,"\Ub(i_{st})"] & & \Ub b_{st}.
\end{tikzcd}\end{equation}
Equality \begin{equation}\label{eq:pi}\Ub(i_{st})\circ \r   = \s \circ\Ub(i_{tot})\end{equation} follows 
 by definition of $\r\defeq\Ub(i_{st})^{-1}\circ \s\circ \Ub(i_{tot})$ (c.f.\ \eqref{eq:interconnectControl}).  Equality \begin{equation}\label{eq:p11} \Ub(p_b)\circ \Ub(i_{tot}) = \Ub(i_{st})\circ  \Ub(p_a)\end{equation} holds   because $\Ub i:a\rightarrow b$  is a map of surjective submersions (\cref{def:hybridSSubMorphism}) and $\Ub$ is a functor (\cref{prop:ForgetfulFunctor}).

 Now let $x\in \Ub a_{tot}$, set $y\defeq \Ub i_{tot}(x)$, and let $\eta_y\in \Sb^{b_{st}}_1$ such that \begin{equation}\label{eq:random12345} (\Ub p_b (y),\s(y))\in b_{st}(\eta_y).\end{equation} Let $\g_x\defeq \i_{st}^{-1}(\eta_y)$.  Applying $\big(\ifrak_{\sF{dom}(\eta_y)}^{-1}\times \ifrak_{\sF{cod}(\eta_y)}^{-1}\big)_{st}$ to both sides of \eqref{eq:random12345}, we observe (c.f.\ \cref{lemma:isoHyPhImpliesEqualityOnRelations})  that  \begin{equation}\label{eq:inclusionABC} \left(\ifrak_{\sF{dom}(\eta_y)}^{-1}\big(\Ub p_b (y)\big)  , \ifrak_{\sF{cod}(\eta_{y})}^{-1} \big(\s(y)\big)\right)\in a_{st}(\g_x).\end{equation}  Now $$\begin{array}{ll} \Ub p_a(x) & = \Ub i_{st}^{-1} \circ \Ub p_b \circ \Ub i_{tot} (x) \\& = \Ub i_{st}^{-1}\circ \Ub p_b(y)  \\ & = \ifrak_{st,\sF{dom}(\eta_y)}^{-1}\big(\Ub p_b(y)\big).  \end{array}   $$  Similarly $$\begin{array}{ll} \r(x) & = \Ub i_{st}^{-1}\circ \s\circ \Ub i_{tot}(x)\\ & = \Ub_{st}^{-1} \circ \s(y)  \\& = \ifrak_{st,\sF{cod}(\eta_y)}^{-1}(\s(y)).	
\end{array}$$  In other words, together with \eqref{eq:inclusionABC}, we conclude that $(\Ub p_a (x), \r(x))\in a_{st}(\g_x)$, proving  that  $\dCrl(i):\dCrl(b)\rightarrow \dCrl(a)$ is  well defined. \end{proof}

 
 \begin{remark}
 We can also prove condition \eqref{eq:ratEq} abstractly. 
 Since $(b,Y,\s)$ is  a deterministic  hybrid open system (\cref{def:deterministicHyOS}), we have inclusion   $$\Theta_b\big( \mbox{graph}(\s)\big) \subseteq \Lambda_{b_{st}}\left(\discats_{\eta\in \Sb^{b_{st}}_1}b_{st}(\eta)\right),$$  where $\Theta_b\defeq  \Ub(p_b)\times id_{\Ub b_{st}}$ (c.f.\ \eqref{eq:defineTheta}).
 Since $i^{-1}:b_{st}\xrightarrow{\sim} a_{st}$ is an isomorphism of hybrid phase spaces, $\big(\ifrak_{\sF{dom}(\eta)}^{-1}\times \ifrak_{\sF{cod}(\eta)}^{-1}\big) \big(b_{st}(\eta)\big) = a_{st}\big(\i^{-1}(\eta)\big)$ (\cref{lemma:isoHyPhImpliesEqualityOnRelations}) for arrow $\big(\sF{dom}(\fd)\xrightarrow{\eta} \sF{cod}(d')\big)\in \Sb^{b_{st}}_1$. Therefore $$\discats_{\eta\in \Sb^{b_{st}}_1}  \big(\ifrak_{\sF{cod}(\eta)}^{-1}\times \ifrak_{\sF{dom}(\eta)}^{-1}\big) \big(b_{st}(\eta)\big) = \discats_{\eta\in \Sb^{b_{st}}_1} a_{st}\big(\i^{-1}(\eta)\big)= \discats_{\g\in \Sb^{a_{st}}_1}a_{st}(\g),$$ which we write (also by \cref{lemma:isoHyPhImpliesEqualityOnRelations})  as \begin{equation}\label{eq:inclusionForIsoHyPh}\big(\Ub(i_{st})^{-1}\times \Ub(i_{st})^{-1}\big) \left(\Lambda_{b_{st}}\left(\discats_{\eta\in \Sb^{b_{st}}_1}b_{st}(\eta) \right)\right) = \Lambda_{a_{st}}\left( \discats_{\g\in \Sb^{a_{st}}_1}a_{st}(\g)\right).\end{equation}

We notate  \begin{equation}\label{eq:123ForRef} \aleph_b\defeq \Ub p_b\circ \Ub i_{tot}\end{equation} for reference and compute: \footnotesize  $$\begin{array}{lll} 
\Theta_a(\mbox{graph}(\r)) & =  \left\{\big(\Ub p_a(x),\r(x)\big):\, x\in \Ub a_{tot}\right\} &  \mbox{since} \; (f\times g)(a,b) = (f(a),g(b))\\
	& =  \left\{\big(\Ub p_a(x),\Ub(i_{st})^{-1}\circ \s\circ \Ub i_{tot}(x)\big):\, x\in \Ub a_{tot}\right\}& (\mbox{c.f.\   \eqref{eq:interconnectControl}})  \\
	&  =  \left\{\big(\Ub(i_{st})^{-1}\circ\aleph_b(x),\Ub( i_{st})^{-1}\circ \s\circ \Ub i_{tot}(x)\big):\, x\in \Ub a_{tot}\right\} & (\mbox{\eqref{eq:p11}, \eqref{eq:123ForRef}})\\ 
	& =  \left\{\big(\Ub(i_{st})^{-1}\circ \Ub p_b(y),\Ub(i_{st})^{-1}\circ \s(y)\big):\, y\in im(\Ub i_{tot})\right\}& (\mbox{definition of}\; im(\Ub(\i,\ifrak)_{tot})) \\
	&= \big(\Ub(i_{st})^{-1}\times \Ub(i_{st})^{-1}\big)\left\{ \big( \Ub p_b (y),\s(y)\big):\, y\in im(\Ub i_{tot})\right\} & \mbox{since}\;(f\times g)(a,b) = (f(a),g(b))\\
	&  =   \big(\Ub(i_{st})^{-1}\times \Ub(i_{st})^{-1}\big) \circ \Theta_b\left( \mbox{graph}\left(\s\restriction_{im(\Ub(\i,\ifrak)_{tot})}\right)\right)& (\mbox{definition of graph.})
\end{array} $$ \normalsize   By assumption $$\Theta_b \left( \mbox{graph}\left(\s\right)\right)\subseteq \Lambda_{b_{st}}\left(\discats_{\eta\in \Sb^{b_{st}}_1}b_{st}(\eta)\right).  $$ Therefore we  conclude (c.f.\ \eqref{eq:inclusionForIsoHyPh}) that \small   $$\Theta_a(\mbox{graph}(\r))  \subseteq \big(\Ub(i_{st})^{-1}\times \Ub(i_{st})^{-1}\big)  \left(\Theta_b\left( \mbox{graph}\left(\s\right)\right)\right)\subseteq \Lambda_{a_{st}}\left(\discats_{\g\in \Sb^{a_{st}}_1}a_{st}(\g)\right),$$ \normalsize as desired. 
\end{remark}

\subsubsection{Examples of Deterministic Control on Interconnection}

	Now we provide a few examples of interconnection and control. Example \ref{ex:interconnectThermostat} models a thermostat, \cref{ex:BouncingBallInterconnect} a bouncing ball, while \cref{ex:stateDependentSwitching} and \cref{ex:timeDependentSwitching}  are general switched systems.  The examples are constructed to illustrate that a deterministic  hybrid system may be realized as an interconnection of deterministic  hybrid open systems, analogous to the way that  a vector field $(X:\R^n\rightarrow T\R^n) \in \Xf(\R^n)$ may be obtained as the interconnection of $n$ open systems $\{X_i:\R^n\rightarrow T\R_i\}_{i=1,\ldots,n}$ (\cite[example 3.2]{lermanopennetworks}). We recover the digital control hybrid system $(c,Z,\nu)$ from \cref{example:ThermostatVanilla} in \cref{ex:interconnectThermostat} and the bouncing ball $(c,Z,\mu)$ from \cref{ex:bouncingBallVanilla} in \cref{ex:BouncingBallInterconnect}.  

\begin{remark}\label{remark:outlineForExamples}
Because each example follows an identical template, we preface them with an outline. We present two hybrid phase spaces $a:\Sb^a\rightarrow\sF{RelMan}$ and  $b:\Sb^b\rightarrow \sF{RelMan}$.  From their product, we form hybrid surjective submersions $p_a:a\times b\rightarrow a$ and $p_b:a\times b\rightarrow b$ (\cref{remark:hyssubFromHyPh}).  Out of these two hybrid surjective submersions, we  build deterministic  hybrid open systems $(a\times b\rightarrow a, X,\r)$ and $(a\times b\rightarrow b, Y,\s)$, which amounts to defining pairs $(X,\r)\in\dCrl(a\times b\xrightarrow{p_a} a) , (Y,\s)\in \dCrl(a\times b\xrightarrow{p_b}b)$ (\cref{def:deterministicControlLax}), or sections of the c.d.\ bundle $\varpi_a:\Tb a\rightarrow\Ub a$ satisfying some  conditions (\cref{def:deterministicHyOS}, \cref{remark:deterministicHySysAsSection}). 

We then construct---with the product---a deterministic  hybrid open system (\cref{remark:deterministicHyOSProducts}) $$(a\times b\times a\times b\xrightarrow{p_a\times p_b} a\times b, X\times Y,\r\times \s),$$ and  an interconnection morphism $$\begin{tikzcd}[column sep = large] a\times b\arrow[d,"id_{a\times b}"]\arrow[r,swap,"i_{tot}"] & a\times b \times a\times b\arrow[d,"p_a\times p_b"] \\
a\times b\arrow[r,"i_{st}"] & a\times b,
\end{tikzcd}$$ where $i_{st} \defeq id_{a\times b}$ (\cref{def:HybridInterconnection}).  The interconnection map $i_{tot}$  sends point $$\big((x,y)\in (a\times b)(\fs_a,\fs_b) \big) \mapsto\big((x,y,x,y)\in (a\times b\times a \times b)(\fs_a,\fs_b,\fs_a,\fs_b)\big),\footnote{Implicit in this assignment is the map on nodes $\left((\fs_a,\fs_b)\in \Sb^{a\times b}\right) \mapsto \big((\fs_a,\fs_b,\fs_a,\fs_b)\in \Sb^{a\times b\times a\times b}\big). $} $$ while the projection map $p_a\times  p_b$ sends  point  $$\left((x,y,x',y')\in (a\times b\times a \times b)(\fs_a,\fs_b,\fs_a',\fs_b') \right) \mapsto\left( (x,y')\in (a\times b)(\fs_a,\fs_b')\right).\footnote{And maps node $(\fs_a,\fs_b,\fs_a',\fs_b')\mapsto (\fs_a,\fs_b')$.}$$ The deterministic hybrid  control is defined as follows  (\cref{remark:deterministicHyOSProducts}): we have  $$X\times Y:\Ub(a\times b\times a\times b)\rightarrow T\Ub(a\times b)$$ sending $(x,y,x',y')\mapsto \big(X(x,y),\, Y(x',y')\big)$ and $$\r\times \s:\Ub(a\times b\times a\times b)\rightarrow \Ub(a\times b)$$ sending $(x,y,x',y')\mapsto \big(\r(x,y),\,\s(x',y')\big)$. 

Finally, we apply \cref{prop:interconnectionMapOndeterministicHybridControl} to obtain a deterministic  hybrid system (in the sense of \cref{def:HySys}): $$\left(a\times b, \dCrl(\i,\ifrak)\big(X\times Y,\r\times \s\big) \right), $$ which, putting everything together, maps $$\begin{tikzcd}
	(x,y)\arrow[rr,mapsto,"\Ub(\i{,}\ifrak)_{tot}"]\arrow[rrrr,mapsto ,bend right,"\dCrl(\i{,}\ifrak)\left(X\times Y{,}\r\times \s\right)"] &  & (x,y,x,y)\arrow[rr,mapsto,"\left(X\times Y{,}\r\times \s\right)"] & &  \left(\big(X(x,y),\, Y(x,y)\big),\big(\r(x,y),\,\s(x,y)\big)\right).
\end{tikzcd}$$
\end{remark}

\begin{remark}
	For now, we acknowledge---but bracket addressing---the implicit  isomorphism $\Ub(a\times b)\cong \Ub a \times \Ub b$ which we used in this outline.  We will prove in \cref{lemma:PUUPLite} that $\Ub((\cdot)\times (\cdot)) \cong \Ub(\cdot) \times \Ub(\cdot)$, natural in each factor (see also \cref{prop:uFromHyPh2ManIsMonoidal}), so the examples are well defined. 
\end{remark}

\begin{example}\label{ex:interconnectThermostat} We model the thermostat introduced  in \cref{example:ThermostatVanilla} as an  interconnection of two deterministic  hybrid open systems. 
We start by defining hybrid phase spaces $a$ and $b$. In the following table, the first line is the source category $\Sb^a$ and $\Sb^b$, the second is the assignment of manifolds, and the third is assignment of relations.  We only display nontrivial relations, but every node $\fs_a\in \Sb^a$ has trivial relation $a(id_{\fs_a}) \defeq \Delta(a(\fs_a))$, and similarly for each $\fs_b\in \Sb^b$.
$$ \begin{array}{lll|crll}\Sb^a : & \{\fs_a\} &  & \Sb^b: &  \begin{tikzcd} 0\arrow[r,bend left,"e_{1,0}"] & 1\arrow[l,bend left,"e_{0,1}"] \end{tikzcd} &  \\
\hline 
 & a(\fs_a) & = \R & & b(j) & = \{j\},\;\;\; j =0,1 \\
 &   &   & & b(e_{1-j,j}) & = \{(j,1-j)\},\;\;\; j =0,1
	\end{array}$$	Phase space $a$ represents temperature and phase space $b$ represents digital control.

Now we consider the hybrid surjective submersions arising from the products $a\times b\xrightarrow{p_a} a$ and $a\times b\xrightarrow{p_b} b$, and define two hybrid deterministic  open systems.  In the following table, the first three lines depict the surjective submersions (1.\ the map source and target, 2.\ assignment on nodes, 3.\ assignment of points in manifold), the next two lines show the maps $X:\Ub(a\times b)\rightarrow T\Ub a $ and $Y:\Ub(a\times b)\rightarrow T\Ub b$ for vectors, and the last two lines depict the jump maps $\r:\Ub(a\times b)\rightarrow\Ub a$ and $\s:\Ub(a\times b)\rightarrow\Ub b$:
$$
\begin{array}{rrl|rrl} 
p_a: & a\times b \rightarrow & a &  p_b: & a \times b \rightarrow & b\\
& (\fs_a,j) \mapsto & \fs_a & & (\fs_a',j') \mapsto  & j'\\
 &(T,j) \mapsto & T & &(T',j')  \mapsto & j'\\
\hline 
	X: & \R \times \{0,1\} \rightarrow & T\R &   Y :&\R\times  \{0,1\} \rightarrow & T\{0,1 \}\\
	& (T,j)\mapsto & (-1)^{1-j} & &   (T',j') \mapsto & 0 	\\
	\hline 
	\r: & \R  \times \{0,1\} \rightarrow & \R  &  \s:&  \R\times \{0,1\}  \rightarrow & \{0,1\}\\
	& (T,j)\mapsto & T  & &   (T',j') \mapsto & \left\{\begin{array}{ll} 1-j' &\mbox{if}\; (-1)^{1-j'}T'\; \geq 1\\ j' &\mbox{else} \end{array}\right.	
	\end{array}$$
 These define a deterministic  open system since $(T,\r(T,j)) = (T,T) \in \Delta(\R)=a(id_{\fs_a})$ and $(j,\s(T,j)) = (j,1-j)\in b(e_{j-1,j})$ or $(j,\s(T,j)) = (j,j)\in \Delta(\{j\})=b(id_{j})$ for $j=0,\,1$ (\cref{remark:checkJumpMapOpen}).

 We combine both  open systems to form  the deterministic   hybrid open system with hybrid surjective submersion $(a\times b\times a\times b)\xrightarrow{p_a\times p_b}a\times b$ which on nodes maps $(\fs_a,j,\fs_a',j')\mapsto (\fs_a,j')$ and on manifolds maps points $(T,j,T',j')\mapsto (T,j')$.   Deterministic control $(X\times Y,\r\times \s)\in \dCrl(a\times b\times a \times b\rightarrow a\times b)$ is defined by $$\begin{array}{ll}  X\times Y (T,j,T',j') & \defeq \big(X(T,j),\,Y(T',j')\big)\\
 \r\times \s(T,j,T',j') & \defeq \big(\r(T,j),\, \s(T',j')\big).
\end{array}$$

We take hybrid interconnection $\begin{tikzcd} a\times b\arrow[d,"id_{a\times b}"]\arrow[r,"(\i{,}\mathfrak{i})_{tot}"] & a\times b\times a\times b\arrow[d,"p_a\times p_b"] \\ a\times b\arrow[r,"id_{a\times b}"] & a\times b\end{tikzcd}$  (\cref{remark:outlineForExamples}), and this induces a map on control $\dCrl(\i,\ifrak):\dCrl\left(a\times b\times a\times b\xrightarrow{p_a\times p_b} a\times b\right) \rightarrow\dCrl\left(a\times b\xrightarrow{id_{a\times b}} a\times b\right)$ defined by: $$\dCrl(\i,\mathfrak{i})(X\times Y)(T,j) \defeq (X(T,j),Y(T,j)) = \left((-1)^{1-j},0\right)$$ and  $$\dCrl(\i,\mathfrak{i})(\r\times \s) (T,j) = \left\{ \begin{array}{ll}  (T,1-j) &\mbox{if}\; (-1)^{1-j}T \geq  1\\ (T,j) &\mbox{else}.\end{array}\right.$$
We thus recover the thermostat from \cref{example:ThermostatVanilla} and \eqref{eq:vanillaThermostat}, and conclude that $$\left(a\times b, \dCrl(\i,\ifrak)\big(X\times Y,\r\times \s\big)\right) = (c,Z,\nu).$$
 \end{example}

 \begin{example}\label{ex:BouncingBallInterconnect}
 We now consider the bouncing ball as interconnection of two deterministic  hybrid open systems. This example decomposes a ``high dimensional'' hybrid system into components, by contrast with the thermostat which decomposes the hybrid system into continuous state with digital control (the distinction, however, is purely heuristic).   
 
 Here the hybrid phase space $a$ will correspond to position (height) and $b$ to velocity: 
$$ 
\begin{array}{lll|clll} \Sb^a: & \{\fs_a\} &  & \Sb^b: &  \begin{tikzcd} \fs_b \arrow[loop right,"e"] \end{tikzcd} &  \\
\hline 
 & a(\fs_a) & = \R^{\geq 0} & & b(\fs_b) & = \R \\
 &   &   & & b(e) & = \{(v,v')\in \R^2:\, v\cdot v'<0\}.
	\end{array} $$ Fix a coefficient of restitution $r\in (0,1)$. Then deterministic  hybrid open systems on projection of products are given as: \small 
	$$
\begin{array}{rrl|rrl} 
p_a: & a\times b \rightarrow & a &  p_b: &a\times b\rightarrow & b\\
& (\fs_a,\fs_b) \mapsto & \fs_a & & (\fs_a',\fs_b') \mapsto  & \fs_b'\\
 &(h,v) \mapsto & h & &(h',v')  \mapsto & v'\\
\hline 
	X: & \R^{\geq 0} \times \R  \rightarrow & T\R^{\geq 0} &   Y :&  \R \times \R^{\geq0}\rightarrow & T\R \\
	& (h,v)\mapsto &v \pdiv{}{h}  & &   (h',v') \mapsto & -\pdiv{}{v'} 	\\
	\hline
	\r: & \R^{\geq0} \times \R \rightarrow & \R^{\geq0}   &  \s:&  \R\times \R^{\geq0}  \rightarrow & \R \\
	& (h,v) \mapsto & h  & &   (h',v') \mapsto & \left\{\begin{array}{ll} -rv' &\mbox{if}\; h' = 0 \; \mbox{and} \; v'<0 \\ v' &\mbox{else}. \end{array}\right.	
	\end{array}$$\normalsize These are valid open systems (c.f.\ \cref{remark:checkingJumpMap}) because $(h,\rho(h,v)) = (h,h) = \Delta(\R^{\geq0})= a(id_{\fs_a})$ and $$(v',\s(h',v')) = \left\{\begin{array}{ll} (v',-rv') & \in b(e) \\ (v',v') & \in b(id_{\fs_b}).\end{array}\right.$$ 
	
From these we obtain deterministic  hybrid open systems $(a\times b\times a\times b, X\times Y, \r\times \s)$ where $$X\times Y(h,v,h',v') = v\pdiv{}{h} -\pdiv{}{v'}$$ and $$\r\times \s(h,v,h',v') = \left\{\begin{array}{ll} (h,-rv') & \mbox{if}\; v'<0 \;\mbox{and} \; h'=0\\ (h,v') &\mbox{else}. \end{array}\right.$$

	Interconnection $(\i,\mathfrak{i}):a\times b\rightarrow a\times b\times a\times b$, sends $ (h,v)  \mapsto (h,v,h,v)$. The map $$\dCrl(\i,\ifrak):\dCrl(p_a\times p_b)\rightarrow\dCrl(id_{a\times b})$$ on control gives us $$\begin{array}{ll}\dCrl(\i,\mathfrak{i})(X\times Y) (h,v) & = v\pdiv{}{h} -\pdiv{}{v}\\ \dCrl(\i,\mathfrak{i})(\r\times \s) (h,v) & = \r\times \s(h,v,h,v) =\left\{\begin{array}{ll}(0,-rv) & \mbox{if} \, h=0,\;\mbox{and}\; v<0\\ (h,v) &\mbox{else,}\end{array}\right.\end{array}$$ exactly the deterministic  hybrid system $(c,Z,\mu)$ in \cref{ex:bouncingBallVanilla}. 
\end{example}

\begin{example}\label{ex:stateDependentSwitching}
	Now we consider state dependent switched systems \begin{equation}\label{eq:switchedSystem}\dx = f_{\tilde{\s}(x)}(x)\in T_xM,\end{equation} for $x\in M$,  switch signal $\tilde{\s}:M\rightarrow \{1,\ldots,k\}$, and each $f_j\in \Xf(M)$. We decompose such a system as an interconnection of two deterministic  hybrid open systems. We will define two hybrid phase spaces $a$ and $b$ representing the state $x\in M$ and switching signal $\tilde{\s}(x)\in \{1,\ldots,k\}$, respectively. Consider complete graph $ G=\left\{i\in \{1,\ldots,k\}:\, \exists!\, e_{j,i}\, \forall i,j\in \{1,\ldots,k\}\right\}$.\footnote{The node set is $G_0=\{1,\ldots,k\}$ and edge set is $G_1 = \{e_{i,j}:\, i,j\in G_0\}$. When $i=j$, $e_{i,j}$ is the identity arrow.}  Let source category of $b$ be the path category on this graph, denoted $\fS(G)$.  Then:	$$ 
\begin{array}{lll|clll} \Sb^a: & \{\fs_a\} &  &\Sb^b: &   \fS(G) &  \\
\hline 
 & a(\fs_a) & =  M &  b(j)  & = \{j\} \\
 &   &   &   b(e_{j,i}) &  = \{(i,j)\}.	\end{array}$$
	
	And on the projection of  products we have deterministic  hybrid open systems $$
\begin{array}{rrl|rrl} 
p_a: & a\times b \rightarrow & a &  p_b: &a \times b \rightarrow & b\\
& (\fs_a,j) \mapsto & \fs_a & & (\fs_a,j') \mapsto  & j'\\
 &(x,j) \mapsto & x & &(x',j')  \mapsto & j'\\
\hline 
	X: & M\times \{1,\ldots,k\} \rightarrow & TM&   Y :&    M \times \{1,\ldots,k\} \rightarrow & T\{1,\ldots,k \}\\
	& (x,j)\mapsto & f_{j}(x) & &  (x',j') \mapsto & 0 	\\
	\hline 
	\r: & M  \times \{1,\ldots,k\} \rightarrow & \R  &  \s:&  M \times \{1,\ldots,k\}\rightarrow & \{1,\ldots,k\}\\
	& (x,j)\mapsto & x  & &   (x',j') \mapsto &  \tilde{\s}(x'),
	\end{array}$$ 
	which are valid deterministic  open systems since $(x,x) \in \Delta(M)$ and $(j',\s(x',j')) = (j',\tilde{\s}(x'))\in b(e_{\tilde{\s}(x'),j})$ (\cref{remark:checkJumpMapOpen}). 
	
	Taking product results in hybrid deterministic  open system $$\big(a\times b \times a\times b\rightarrow a\times b, X\times Y, \r\times \s\big),$$
	defined as 
	$$\begin{array}{lll}  X\times Y (x,j,x',j') & \defeq \big(X(x,j),\,Y(x',j')\big) & = (f_{j}(x),0),\\
 \r\times \s(x,j,x',j') & \defeq \big(\r(x,j),\, \s(x',j')\big)  & = (x,\tilde{\s}(x')). 
\end{array}$$ There is  induced map on control \begin{equation}\label{eq:switchedSystemAsInterconnection} (Z,\t)\defeq \dCrl(i)\big(X\times Y,\r\times\s\big), \end{equation} sending $(x,j)$ to  $Z(x,j) = (f_{j}(x),0)$ and  $\t(x,j)=  (x,\tilde{\s}(x))$.

This determined hybrid system $(c,Z,\t)$ is  the switched system in \eqref{eq:switchedSystem}: in this representation we notationally decouple switching and continuous-time dynamics.  More accurately: they are still coupled, but we have isolated continuous time dynamics and switching into separate components of a deterministic  hybrid system. 
	\end{example}

\begin{prop}
Let $\dx = f_{\tilde{\s}}(x)$ be a state-dependent switched system, with $x\in M$, $\tilde{\s}:M\rightarrow \{1,\ldots,k\}$, and $f_i\in \Xf(M)$ for $i=1,\ldots,k$.  This system arises as interconnection $\dCrl(i)(X\times Y,\r\times \s)$ of deterministic  hybrid open systems $$\begin{array}{l} \big(a\times b\rightarrow a, X, \r\big)\\
\big(a\times b \rightarrow b, Y, \s\big) 	
\end{array}$$
 as defined in \cref{ex:stateDependentSwitching}, \eqref{eq:switchedSystemAsInterconnection}. 	
\end{prop}

Now we consider time-dependent switching, which turns out to be a special case of state-dependent switching. 
\begin{example}\label{ex:timeDependentSwitching}

Consider system $\dx = f_{\tilde{\s}(t)}(x)$ with $\tilde{\s}:\R\rightarrow \{1,\ldots,k\}$. Here we  keep $b$ the same hybrid phase spaces as in \cref{ex:stateDependentSwitching}, but this time $a(\fs_a) = M\times \R$.  
Then open systems are defined  $$ 
	\begin{array}{lrlclrl} X: & M\times \R  \times \{1,\ldots,k\} \rightarrow & TM\times \R  &  & Y :& M\times \R \times \{1,\ldots, k\} \rightarrow & T\{1,\ldots ,k \}\\
	& (x,t,j)\mapsto & f_j(x) & & &  (x',t',j') \mapsto & 0 	
	\end{array}$$
	and functions $\r,\s$ as 
	 $$ \begin{array}{lrlclrl} \r: & M\times \R  \times \{1,\ldots,k\} \rightarrow & M\times \R  &  & \s:&  M\times \R  \times \{1,\ldots, k\} \rightarrow & \{1,\ldots ,k \}\\
	& (x,t,j)\mapsto & (x,t)  & & &  (x',t',j') \mapsto & \tilde{\s}(t)	
	\end{array}$$

As in the previous example, we obtain time-dependent switched system as interconnection.
\end{example}
 \begin{remark}\label{remark:timeSwitchingSpecialCaseStateSwitching}
 	We observe how  \cref{ex:timeDependentSwitching} is a special case of \cref{ex:stateDependentSwitching}.  It is well known (\cite[\S3.3.1]{liberzonswitch}) that a time-dependent continuous-time dynamical system $X\in \Xf(M)$ of dimension $n$ is secretly a time-\textit{in}dependent dynamical system $\tilde{X}\in \Xf(\tilde{M})$ in disguise, of dimension $n+1$. Indeed, the state manifold may be given by $\tilde{M} = M\times \R$ and the  $(n+1)$th variable  given constant dynamics $\dx_{n+1} = 1$. 
 \end{remark}

\subsection{Monoidal Structure of Hybrid Phase Spaces}
We show in this section that products in hybrid phase spaces behave well with forgetful functor $\Ub :\sF{HyPh}\rightarrow\sF{Man}$, namely that $\Ub$ is a strong monoidal functor. Concretely $\Pi\circ \Ub \cong  \Ub \circ \Pi$. First we develop some facts about $\dCrl$.

\subsubsection{Deterministic  Control as a Functor}
	Deterministic control is only lax functorial on arbitrary maps of hybrid surjective submersion (\cref{prop:deterministicControlIsLax}), but functorial on interconnection: 
\begin{prop}\label{prop:dCrlInterconnectFunctorial}
	The map  $\dCrl$ defined on interconnection in \cref{prop:interconnectionMapOndeterministicHybridControl} extends to  a contravariant functor $$\dCrl:\left(\sF{HySSub_{int}}\right)^{op}\rightarrow\sF{Set}.$$ 
\end{prop}
\begin{proof} Let $a\xrightarrow{f} b\xrightarrow{g}c$ be hybrid interconnections.  
	We must show that $\dCrl(id_a) = id_{\dCrl(a)}$  and that $\dCrl\big(g\circ f\big)= \dCrl(f)\circ\dCrl(g)$. Fix $(X,\r)\in \dCrl(a)$.       Deterministic  control $\dCrl$ in \cref{def:interconnectionMapOndeterministicHybridControl}, \eqref{eq:definingdcrlOnInterconnection} defines the left vertical arrow in each  diagram $$\begin{array}{ccc} 
	\begin{tikzcd}[column sep = large] \Ub a_{tot}\arrow[d,"?"] \arrow[r,"\Ub id_{a_{tot}}"] & \Ub a_{tot}\arrow[d,"X"] \\ T\Ub a_{st} & T\Ub a_{st}\arrow[l,"T\Ub id^{-1}_{a_{st}}"] \end{tikzcd} & \mbox{and} &  
	\begin{tikzcd}[column sep = large] \Ub a_{tot}\arrow[d,"?"] \arrow[r,"\Ub id_{a_{tot}}"] & \Ub a_{tot}\arrow[d,"\r"] \\ T\Ub a_{st} & \Ub a_{st}.\arrow[l,"\Ub id^{-1}_{a_{st}}"] \end{tikzcd}
\end{array}
	$$ Since  $\Ub$ and $T$  are functors,  $T\Ub id_{a_{st}}^{-1} = id_{T\Ub a_{st}}$ and $\Ub id_{a_{st}} = id_{\Ub a_{st}}$, which implies that $\dCrl(id_a)(X,\r) = (X,\r)$.	
	Now let  $(Z,\t)\in \dCrl(c)$,  and consider the following diagrams 	$$	\begin{array}{ccc}	\begin{tikzcd}[column sep = large]
\Ub a_{tot}\arrow[r,"\Ub f_{tot}"] \arrow[rr,bend left,"\Ub\big(g\circ f\big)_{tot}"]\arrow[d] & \Ub b_{tot} \arrow[r,"\Ub g_{tot}"]\arrow[d] & \Ub c_{tot}  \arrow[d,"Z"] \\
T\Ub a_{st} & T\Ub b_{st}\arrow[l,shift right,"T\Ub f_{st}^{-1}"]& T\Ub c_{st} \arrow[l,shift right,"T\Ub g_{st}^{-1}"]\arrow[ll,bend left,"T\Ub\big(g\circ f)\big)_{st}^{-1}"] 	\end{tikzcd} & \mbox{and} & \begin{tikzcd}[column sep = large]
\Ub a_{tot} \arrow[r,"\Ub f_{tot}"] \arrow[rr,bend left,"\Ub\big(g\circ f\big)_{tot}"]\arrow[d] & \Ub b_{tot}\arrow[r,"\Ub g_{tot}"]\arrow[d] & \Ub c_{tot} \arrow[d,"\r"] \\
\Ub a_{st}  & \Ub b_{st} \arrow[l,"\Ub f_{st}^{-1}"]& \Ub c_{st},  \arrow[l,"\Ub g_{st}^{-1}"]\arrow[ll,bend left,"\Ub\big(g\circ f\big)_{st}^{-1}"] 	\end{tikzcd} 
\end{array}$$ where left and center vertical arrows defined by \eqref{eq:definingdcrlOnInterconnection}. 
 Each subdiagram commutes by functoriality of $T$ and $\Ub$ (\cref{remark:functorialityT&U}), so  $\dCrl(g)\big(\dCrl(f) (Z,\t)\big)  = \dCrl\big(g\circ f\big) (Z,\t)$, and $\dCrl(g)\circ  \dCrl(f)  = \dCrl\big(g\circ f\big)$. 
\end{proof}


\subsubsection{Monoidality of functor $\Ub$}
We  turn to what will be a key component for networks of deterministic  hybrid systems: the forgetful functor $\Ub:\sF{HyPh}\rightarrow \sF{Man}$ is a \textit{monoidal} functor.  We have already defined the monoidal structure of $\sF{HyPh}$ (\cref{def:monoidalCategoryHyPh}). That of $\sF{Man}$ is similarly constructed: we define the monoidal product $\otimes_{\sF{Man}}$ of $\sF{Man}$ to be the cartesian product of manifolds, and the monoidal unit $1_{\sF{Man}}$ to be a terminal object, a one point discrete manifold. 

First a lemma.

\begin{lemma}\label{lemma:PUUPLite}
	There is an isomorphism $\g_{(a,b)}:\Ub(a\times b)\rightarrow \Ub a\times \Ub b$, natural in $a$ and $b$. 
\end{lemma}
In other words,  there is natural isomorphism $$  
 \nat{\sF{HyPh\times HyPh}}{\sF{HyPh}}{\Ub\big((\cdot) \times (\cdot)\big)}{\Ub(\cdot)\times \Ub(\cdot)}{\g}$$ between bifunctors  $\Ub\big((\cdot) \times (\cdot)\big), \, \Ub(\cdot)\times \Ub(\cdot):\sF{HyPh}\times \sF{HyPh}\rightarrow \sF{HyPh}.$  Thus, for   pair of maps $a\xrightarrow{f} a'$, $ b\xrightarrow{g} b'$ of hybrid phase spaces, we have commuting diagram $$\begin{tikzcd}[row sep = large, column sep = large]
	\Ub(a\times b)\arrow[d,swap,"\Ub\big(f\times g\big) "]\arrow[rr,"\g_{(a{,}b)}"] & &  \Ub a\times \Ub b\arrow[d,"\Ub f\times \Ub g"]
	\\
	\Ub(a'\times b')\arrow[rr,"\g_{(a'{,}b')}"] & &  \Ub a'\times \Ub b'.
\end{tikzcd}$$

  \begin{proof} Recall that $\Ub(a)\defeq \discats_{\fc\in \Sb^a}a(\fc)$ is a coproduct in the category of manifolds.  We  define the map  $\g_{(a,b)}:\Ub(a\times b)\rightarrow \Ub a \times \Ub b$ as the  canonical map $\Omega$ in  \cref{prop:canonicalMapCoproduct2Product}. As a map of \textit{sets}, $\g_{(a,b)}$ is a bijection (\cref{prop:coproductProductCommuteInSet}).  Therefore, we must show that this map is smooth.

    Recall the variant construction of $\g\defeq\g_{(a,b)}$ in \eqref{eq:constructVarOmega1}: the map (drop subscript $(a,b)$) $$\g:\discats_{(\fc,\fd)\in \Sb^{a\times b}_0} (a\times b)(\fc,\fd)\rightarrow \left(\discats_{\fc\in \Sb^a_0}a(\fc)\right)\times \left(\discats_{\fd\in \Sb^b_0}b(\fd)\right) $$ is  uniquely defined by collection of maps $$\left\{ a(\fc')\times b(\fd')\xrightarrow{\g_{(\fc',\fd')}} \left(\discats_{\fc\in \Sb^a_0}a(\fc)\right)\times\left(\discats_{\fd\in\Sb^b_0} b(\fd)\right)\right\}_{\fc'\in \Sb^a_0,\, \fd'\in \Sb^b_0}.$$
And each map $\g_{(\fc',\fd')}$ is uniquely defined by maps $$\begin{array}{lll} a(\fc')\times b(\fd') \xrightarrow{\g_{(\fc',\fd')}^1} \discats_{\fc\in \Sb^a_0}a(\fc) & \mbox{and} & a(\fc')\times b(\fd') \xrightarrow{\g_{(\fc',\fd')}^2} \discats_{\fd\in \Sb^b_0}b(\fd).\end{array}$$  These maps are defined by $\g_{(\fc',\fd')}^1\defeq i_{\fc'}\circ p_1$ and $\g_{(\fc',\fd')}^2\defeq i_{\fd'}\circ p_2$ where $i_{\fc'}:a(\fc')\hookrightarrow\Ub a$ is the canonical injection and $p_1:a(\fc')\times b(\fd')\rightarrow a(\fc')$ is the canonical projection. They induce map $\g_{(\fc',\fd')} = i_{\fc'}\times i_{\fd'}$.    Since both  $i_{\fc'}:a(\fc')\hookrightarrow\Ub a$ and $i_{\fd'}:b(\fd')\hookrightarrow\Ub b$ are open embeddings (\cref{prop:coproductCanonicalInjectionOpen}), the  induced map $\g_{(\fc',\fd')}:a(\fc')\times b(\fd')\hookrightarrow \Ub a\times \Ub b$ is also an open embedding (\cref{prop:productOpenEmbeddings}).  Consequently, the map $\g$ in diagram $$\begin{tikzcd}[column sep = large]
	\Ub(a\times b)\arrow[r,dashed,"\g"] & \Ub a\times \Ub b\\
	a(\fc')\times b(\fd')\arrow[u,hookrightarrow,"i_{\fc'{,}\fd'}"] \arrow[ur,hookrightarrow,swap,"i_{\fc'}\times i_{\fd'}"]  & 
\end{tikzcd}$$
is also an open embedding.  Indeed, every $x\in \Ub(a\times b)$ is contained in an open set $\Oc_x$ which is wholly contained in the image of $i_{\fc',\fd'}$.   Thus $\g:\Ub(a\times b)\rightarrow\Ub a\times \Ub b$ is a local diffeomorphism (e.g.\ \cite[Proposition 5.1]{leeManifolds} and inverse function theorem), and therefore a diffeomorphism.

For naturality, let $a\xrightarrow{f} a'$ and $b\xrightarrow{g} b'$ be two maps of hybrid phase spaces. We must show that the diagram $$\begin{tikzcd}[column sep = large]
	\Ub(a\times b)\arrow[r,"\g_{(a{,}b)}"]\arrow[d,swap,"\Ub\big(f\times g\big)"] & \Ub a\times \Ub b\arrow[d,"\Ub f\times \Ub g"]\\
	\Ub(a'\times b')\arrow[r,"\g_{(a'{,}b')}"] & \Ub a'\times \Ub b'
\end{tikzcd}$$ commutes. 

The diagram 
$$\begin{tikzcd}[column sep = large]
	a\times b\arrow[r,"p_a"] \arrow[d,swap,"f\times g"] & a\arrow[d,"f"] \\a'\times b'\arrow[r,"p_{a'}"] & a'
\end{tikzcd}$$ in $\sF{HyPh}$ commutes. Therefore, the diagram $$\begin{tikzcd}[column sep = large]
	\Ub(a\times b)\arrow[r,"\Ub(p_a)"] \arrow[d,swap,"\Ub\big(f\times g\big)"] & \Ub a\arrow[d,"\Ub f"] \\ \Ub(a'\times b')\arrow[r,"\Ub(p_{a'})"] & \Ub a'
\end{tikzcd}$$ commutes as well, since $\Ub$ is a functor (\cref{lemma:functorsPreserveDiagrams}, \cref{prop:ForgetfulFunctor}).  There is an analogous commuting diagram, where $b$, $b'$, $g$ replace $a$, $a'$, and $f$, respectively.

Therefore, in diagram $$\begin{tikzcd}
	\Ub(a\times b) \arrow[drr,shift left,bend left,"\upsilon_1"] \arrow[drr,shift right,bend left,swap,"\upsilon_2"]\arrow[ddr,shift left,bend right,"\z_1"]\arrow[ddr,shift right,bend right,swap,"\z_2"] \arrow[dr,shift left,dashed,"\chi_1"] \arrow[dr,dashed, shift right, swap,"\chi_2"] & & \\ 
	& \Ub a'\times \Ub b' \arrow[r,"p_{\Ub b'}"]\arrow[d,"p_{\Ub a'}"] & \Ub b'\\
	& \Ub a', & 
\end{tikzcd}$$ where $$\begin{array}{llll} \upsilon_1  & \defeq  \Ub(p_{b'}) \circ \Ub\big(f\times g\big), & \upsilon_2 & \defeq \Ub g\circ \Ub(p_b)\\  \z_1 & \defeq \Ub(p_{a'}) \circ \Ub\big(f\times g\big), & \z_2 & \defeq \Ub f\circ \Ub(p_a), 	
\end{array}$$ the canonically induced maps $\chi_1= \chi_2$ are equal since $\upsilon_1 = \upsilon_2$ and $\z_1 = \z_2$. In other words, $$\begin{array}{lll}\chi_1= \g_{(a',b')}\circ \Ub\big(f\times g\big) & \mbox{and} & \chi_2 = \big(\Ub f\times \Ub g\big)\circ \g_{(a,b)},\end{array}$$  proving naturality.   \end{proof}

We restate \cref{lemma:PUUPLite} in a way which will be directly useful to us later. 
\begin{prop}\label{prop:uFromHyPh2ManIsMonoidal}
 The forgetful functor $\Ub:(\sF{HyPh},\times, 1_{\sF{HyPh}})\rightarrow(\sF{Man},\times  , 1_{\sF{Man}})$ is a strong monoidal functor (\cref{def:strongMonoidalFunctor}). We denote the isomorphism $\eta_{(a,b)}:\Ub(a)\times \Ub(b) \xrightarrow{\sim} \Ub(a\times b)$, $\eta = \g^{-1}$ (\cref{fact:inverseOfNaturalIsoIsNatural}). 
\end{prop}
\begin{proof}
	This follows immediately from \cref{def:strongMonoidalFunctor}, \cref{fact:inverseOfNaturalIsoIsNatural}, and \cref{lemma:PUUPLite}. 
\end{proof}
We will additionally need the following fact:
\begin{remark}\label{remark:mapProductOfControlsToControlPoduct}\label{prop:mapFromProductdCrl2dCrlProductLite}
	There is a map $\scP_{a,b}: \dCrl(a)\times \dCrl(b)\rightarrow\dCrl(a\times b)$. 	Recall that there is  natural isomorphism $\g_{(a,b)}:\Ub(a\times b)\xrightarrow{\sim}\Ub a\times \Ub b$  (\cref{lemma:PUUPLite}).  Let  $$\left((X,\r),(Y,\s)\right) \in \dCrl(a)\times \dCrl(b),$$  $z\in \Ub(a\times b)$, and set $(x,y)\defeq\g_{(a,b)_{tot}}(z) \in \Ub a_{tot}\times \Ub b_{tot}$.   We define \small\begin{equation}\label{eq:defineProductOfControl} \begin{array}{ll} \scP_{a,b}\left(\big(X,\r\big),\big(Y,\s\big)\right)(z) & \defeq \left(T\g_{(a,b)_{st}}^{-1}\circ \big(X\times Y\big) \circ \g_{(a,b)_{tot}}(z),\; \g_{(a,b)_{st}}^{-1}\circ \big(\r\times \s \big) \circ \g_{(a,b)_{tot}}(z)\right)\\
& = \left( T\g_{(a,b)_{st}}^{-1} \big(X(x),Y(y)\big), \g_{(a,b)_{st}}^{-1}\big(\r(x),\s(y)\big)\right). 
\end{array}\end{equation}\normalsize
Compare this definition with \eqref{eq:interconnectControl}. The verification that $\scP_{a,b}\big((X,\r),(Y,\s)\big) \in \dCrl(a\times b)$---that $\varpi_{\Ub (a\times b)_{st}}\circ T\g_{(a,b),st}^{-1}\big((X(x),Y(y)\big) = \Ub\big(p_a\times p_b\big)$ and a similar check for the jump map---is formally identical to (the proof of) \cref{prop:dCrlInterconnectIsMap}, considering this time diagram 
\small\begin{equation}\label{eq:diagramProductSortOfLikeInterconnect}
\begin{tikzcd}[column sep = huge, row sep = huge]
	\Ub(a\times b)_{tot}\arrow[dd,"\Ub p_{a\times b}"]\arrow[rr,"\g_{tot}"]\arrow[dr,"\scP"]& & \Ub a_{tot}\times b_{tot}\arrow[dd,near start, swap,"\Ub p_a\times\Ub p_b"]\arrow[dr,"v\times w"] & \\
	& \Tb(a\times b)_{st}\arrow[rr,near start,shift right,swap,"\Tb\g_{st}"]\arrow[dl,"\varpi_{a\times b}"]& & \Tb a_{st}\times \Tb b_{st}\arrow[ll,near start, shift right, swap,"\Tb\eta_{st}"] \arrow[dl,"\varpi_a\times \varpi_b"]\\
	\Ub(a\times b)_{st} & & \Ub a_{st}\times \Ub b_{st}\arrow[ll,"\eta_{st}"] & 
\end{tikzcd}
\end{equation}\normalsize where  $v = (X,\r)$, $w = (Y,\s)$, $\scP= \scP_{a,b}(v,w)$, $\g_\bullet:\Ub(a\times b)_\bullet\xrightarrow \Ub a_\bullet \times \Ub b_\bullet$,  $\eta_\bullet$ its inverse, and $\Tb \g = T\g\times \g$.  
\end{remark}

We extend \cref{lemma:PUUPLite} and \cref{remark:mapProductOfControlsToControlPoduct} to arbitrary finite products.  The arguments are essentially the same, and therefore omitted.  
\begin{prop}\label{lemma:PUUPGeneral}
	Let $\fX$ be a finite set and $a_\fx$ a hybrid phase space for each $\fx\in \fX$.  Then there is isomorphism $$\g_\fX:\Ub\left(\discatp_{\fx\in \fX} a_\fx\right)\xrightarrow{\sim}\discatp_{\fx\in \fX}\Ub a_\fx$$ natural in $a_\fx$.  
\end{prop}

\begin{prop}\label{prop:mapFromProductdCrl2dCrlProductGeneral}
	Let $\fX$ be a finite set and $a_\fx$ a hybrid phase space for each $\fx\in \fX$.  Then there is map $\scP_\fX:\discatp_{\fx\in \fX}\dCrl(a_\fx)\rightarrow \dCrl\left(\discatp_{\fx\in \fX} a_\fx\right)$.  Consider  collection $(v_\fx,\r_\fx)_{\fx\in \fX}\in \discatp_{\fx\in \fX}\dCrl(a_\fx)$, and let  $(z_\fx)_{\fx\in \fx}\defeq \g_{\fX,tot}(z)$ for $z\in \Ub\left(\discatp_{\fx\in \fX}a_\fx\right)$ (\cref{lemma:PUUPGeneral}). The the map is defined thus: \begin{equation}\label{eq:definingProductOfdeterministicControlArbitrary}
		\begin{array}{ll}\scP_\fX\left((v_\fx,\r_\fx)_{\fx\in \fX}\right)(z) & \defeq \left(T\g_{\fX,st}^{-1}\circ\left(\discatp_{\fx\in \fX} v_\fx\right)\circ \g_{X,tot}(z), \g_{\fX,st}^{-1}\circ \left(\discatp_{\fx\in \fX}\r_\fx\right)\circ \g_{X,tot}(z)\right) \\
		& = \left(T\g_{X,st}^{-1}\discatp_{\fx\in \fX}v_\fx(z_\fx), \g_{X,st}^{-1}\left(\discatp_{\fx\in \fX}\r_\fx(z_\fx)\right)\right).
		\end{array}
	\end{equation}
\end{prop}

\subsubsection{Monoidality of Deterministic Control}
Before stating the main theorem, we conclude this subsection by showing that $\dCrl$ on interconnection is monoidal.  We already showed in \cref{prop:dCrlInterconnectFunctorial} that $\dCrl$ is a functor. We show the functor also preserves monoidal unit and that the map $\scP_{a,b}:\dCrl(a)\times \dCrl(b)\rightarrow\dCrl(a\times b)$ is natural in $a$ and $b$ (c.f. \cite[\S11.2]{maclane}).  

\begin{prop}\label{prop:dCrlInterconnectMonoidal}
	The functor $\dCrl:\left(\sF{HySSub_{int}},\times,1\right)^{op}\rightarrow \big(\sF{Set},\times,1_{\sF{Set}}\big)$ from \cref{prop:dCrlInterconnectFunctorial} is  monoidal  (\cref{def:monoidalCat}).  
\end{prop}

\begin{remark}\label{remark:productInterconnectionIsInterconnect}
We note that  $f\times g$ is an interconnection when both $f$ and $g$ are. This follows directly from \cref{prop:productIsosIsIso}.
\end{remark}

\begin{proof}
We start by verifying that $\dCrl(1_{\sF{HySSub}}) = 1_{\sF{Set}}$: the tangent space $T\Ub(1)$ is the zero space, and hence the  space of maps $\{X:\Ub(1)\rightarrow T\Ub(1)\}$ is the zero space also. Similarly, $\Ub(1)$ terminal implies that a  map $\r:\Ub(1)\rightarrow\Ub(1)$  can only be the identity; there is thus only one map $\r:\Ub(1)\rightarrow\Ub(1)$.   Hence $\dCrl(1_{\sF{HySSub}})=  \{(0,1)\}$, with only one element.
	
	We now show  that there is map $\scP_{a,b}:\dCrl(a)\times \dCrl(b)\rightarrow \dCrl(a\times b)$ natural in $a, b\in \sF{HySSub}$.  Naturality means that the diagram $$\begin{tikzcd}
	\dCrl(a')\times \dCrl(b')\arrow[r,"\scP_{a'{,}b'}"]\arrow[d,swap,"\dCrl(f)\times \dCrl(g)"] & \dCrl(a'\times b')\arrow[d,"\dCrl\big(f\times g\big)"] \\
	\dCrl(a)\times \dCrl(b)\arrow[r,"\scP_{a{,}b}"] & \dCrl(a\times b)
\end{tikzcd}$$ commutes, for hybrid interconnections $a\xrightarrow{f}a'$ and $b\xrightarrow{g}b'$. 

We have already defined the map $\scP_{a,b}$ for hybrid surjective submersions $a$ and $b$   in  \cref{remark:mapProductOfControlsToControlPoduct} as:  	\begin{equation}\label{eq:defineProductOfControlAbstract}
	\scP_{a,b}(\cdot,\cdot) \defeq  \Tb \g_{(a,b)_{st}}^{-1}\circ \big((\cdot)\times (\cdot)\big) \circ 	\g_{(a,b)_{tot}}. 
	\end{equation}
	
	Naturality of $\scP$ follows from naturality of $\gamma$.  Let $a\xrightarrow{f}a' $ and   $b\xrightarrow{g}b' $ be interconnections of hybrid surjective submersions (\cref{def:HybridInterconnection}), and let $(X',\r')\in \dCrl(a')$,  $(Y',\s')\in \dCrl(b')$.  Set  \begin{equation}\label{eq:someDefs}\begin{array}{llll} (X,\r) \defeq \dCrl(f)(X',\r') & \mbox{and} & (Y,\s)\defeq \dCrl(g)(Y',\s') & (\mbox{c.f. \cref{def:interconnectionMapOndeterministicHybridControl}}).\end{array}\end{equation}  We must show that \begin{equation}\label{eq:monoidalNaturality} \dCrl\big(f\times g\big) \left(\scP_{a',b'}\big((X',\r'),(Y',\s')\big)\right)= \scP_{a,b}\left(\big(X,\r\big),\big(Y,\s\big) \right).\end{equation}  Indeed, consider the following diagram, which we claim commutes: 
	
		$$\small \begin{tikzcd}
	\Ub(a'\times b')_{tot}\arrow[rr,"\xi_1"]\arrow[dd,"\xi_2"] & & \Ub a'_{tot}\times \Ub b'_{tot} \arrow[dd,near start,"\xi_5"] & \\
	& \Ub(a\times b)_{tot}\arrow[ul,"\xi_3"] \arrow[rr,near start,"\xi_4"]\arrow[dd,near start,swap,shift right,"\xi_7"] \arrow[dd,near start,shift left,"\xi_7'"] & & \Ub a_{tot}\times \Ub b_{tot}\arrow[ul,"\xi_6"]\arrow[dd,"\xi_8"] \\
	\Tb\Ub(a'\times b')_{st}\arrow[dr,"\xi_{10}"]& & \Tb\Ub a'_{st}\times \Tb\Ub b'_{st} \arrow[dr,"\xi_{11}"]\arrow[ll,near start,"\xi_9"] & \\ 
	& \Tb\Ub(a\times b)_{st} & & \Tb\Ub a_{st}\times \Tb\Ub b_{st},\arrow[ll,"\xi_{12}"] 
	 \end{tikzcd}$$  where \small  
	  $$\begin{array}{ll|ll|ll} \xi_1 & = \g_{(a',b')_{tot}} & \xi_5 & = (X',\r')\times (Y',\s')  & \xi_9 & = \Tb\g_{(a',b')_{st}}^{-1} \\
\xi_2 & = \scP_{a',b'}\big((X',\r'),(Y',\s')\big) & \xi_{6} & = \Ub(f_{tot})\times \Ub(g_{tot}) &  \xi_{10} & = \Tb\Ub\big(f\times g\big)_{st}^{-1} \\
\xi_3 & = \Ub\big(f\times g\big)_{tot} & \xi_7& \stackrel{?}{=} \xi_7' & \xi_{11} &  = \Tb\Ub(f_{st})^{-1}\times \Ub(\g_{st})^{-1} \\
\xi_4 & = \g_{(a,b)_{tot}}  & \xi_8 & =  (X,\r) \times (Y,\s)    &  \xi_{12}& = \Tb\g_{(a,b)_{st}}^{-1}.
\end{array}$$\normalsize

We want to show $\xi_7=\xi_7'$, where $\xi_7= \mbox{left-hand side of}$ \eqref{eq:monoidalNaturality} and $\xi_7' = \mbox{right-hand side of}$ \eqref{eq:monoidalNaturality}. We argue that each square face in the diagram commutes.  First, $\xi_7 = \xi_{10}\circ \xi_2\circ \xi_3$ by definition of $\dCrl\big(f\times g\big)$ (\cref{def:interconnectionMapOndeterministicHybridControl}, \cref{remark:productInterconnectionIsInterconnect}). Equalities $\xi_2 = \xi_9\circ \xi_5\circ \xi_1$ and $\xi_7' = \xi_{12} \circ \xi_8\circ \xi_4$ follow from  definition of $\scP_{(\cdot),(\cdot)} $ (c.f.\ \eqref{eq:defineProductOfControl}). Equality  $\xi_1\circ \xi_3 = \xi_6\circ \xi_4$  follows by naturality of $\g_{tot}$ (c.f.\ \cref{lemma:PUUPLite}).  Similarly,  $\xi_{10}\circ \xi_9= \xi_{12}\circ \xi_{11}$ follows by naturality of $\g_{st}^{-1}$ and functoriality of $\Tb$ (\cref{prop:cdBundleFunctorial}).    Finally,  $\xi_8 = \xi_{11}\circ\xi_5\circ \xi_6$ follows by definition of  $(X,\r)$ and $(Y,\s)$ (c.f.\ \eqref{eq:someDefs}). Since each face in the diagram commutes, this proves that $\xi_7 = \xi_7'$, and hence that $\scP_{a,b}$ is natural in $a$ and $b$. \end{proof}

\section{Networks of Deterministic  Hybrid Open Systems}\label{section:NetworksHybrid}

We now have the ingredients to both define networks of deterministic  hybrid open systems, and state the main theorem (\cref{theorem:mainTheoremFordeterministicHybridSystems}) which says that a collection of morphisms of deterministic open systems induces a morphism of deterministic open systems.  
\begin{definition}\label{def:networkDetHyOS}
	We define a  \textit{network} $\left(\big\{\Hc_{\fX,\fx}\big\}_{\fx\in \fX}, i_{a,\fX}:a\hookrightarrow\discatp_{\fx\in \fX}\Hc_{\fX,\fx}\right)$ \textit{of hybrid open systems} to be a pair where $\Hc_\fX:\fX\rightarrow\sF{HySSub}$ is a list of hybrid surjective submersions (\cref{def:categoryOfLists}, \cref{def:hybridSurjectiveSubmersion}) and $i_{a,\fX}:a\rightarrow\discatp_{\fx\in \fX}\Hc_{\fX,\fx}$ is a hybrid interconnection (\cref{def:HybridInterconnection}). 
\end{definition}

And morphisms:
\begin{definition}\label{def:networkDetHyOSMorphism}
	Let $\left(\big\{\Hc_{\fX,\fx}\big\}_{\fx\in \fX}, i_{a,\fX}:a\hookrightarrow\discatp_{\fx\in \fX}\Hc_{\fX,\fx}\right)$ and $\left(\big\{\Hc_{\fY,\fy}\big\}_{\fy\in \fY}, i_{b,\fY}:b\hookrightarrow\discatp_{\fy\in \fY}\Hc_{\fY,\fy}\right)$ be two networks of hybrid open systems, denoted by $(\Hc_\fX, i_{a,\fX})$ and $(\Hc_\fY, i_{b,\fY})$. We define a \textit{morphism} $$(\Hc_\fX, i_{a,\fX})\xrightarrow{\big((\ph,\Phi),z\big)}(\Hc_\fY, i_{b,\fY})$$ \textit{of networks} to be a pair where $(\ph,\Phi):\Hc_\fX\rightarrow\Hc_\fY$ is a morphism of lists of hybrid surjective submersions (\cref{def:categoryOfLists}) and $z:b\rightarrow a$ is a morphism of surjective submersions (\cref{def:hybridSSubMorphism}) compatible with $(\ph,\Phi)$.  Namely, the diagram $$\begin{tikzcd}[column sep = large]
	\discatp_{\fy\in \fY}\Hc_{\fY,\fy}\arrow[r,swap,"\Pi(\ph{,}\Phi)"{name = f1, below}] & \discatp_{\fx\in \fX}\Hc_{\fX,\fx}\\
	b\arrow[u,hookrightarrow,"i_{b{,}\fY}"] \arrow[r,"z"{name = f2, above}] & a\arrow[u,hookrightarrow,"i_{a{,}\fX}"]\ar[shorten <= 7pt, shorten >= 7pt, Rightarrow, from = f2, to = f1]
\end{tikzcd}$$ commutes, where $\Pi(\ph,\Phi)$ is as in \cref{prop:NietzscheProp}.

\end{definition}

\begin{theorem}\label{theorem:mainTheoremFordeterministicHybridSystems}
	Let \small $$\left(\big\{\Hc_{\fX,\fx}\big\}_{\fx\in \fX}, i_{a,\fX}:a\hookrightarrow \discatp_{\fx\in \fX}\Hc_{\fX,\fx}\right)\;\mbox{and}\left(\big\{\Hc_{\fY,\fy}\big\}_{\fy\in \fY},i_{b,\fY}:b\hookrightarrow\discatp_{\fy\in \fY}\Hc_{\fY,\fy}\right)$$\normalsize be two networks of deterministic  hybrid open systems (\cref{ex:HybridOpenSystemsInstanceOfAbstractOpenSystem}).  
A morphism $$(\Hc_\fX,(\i,\ifrak)_{a,\fX})\xrightarrow{\big((\ph,\Phi),z\big)} (\Hc_\fY,(\i,\ifrak)_{b,\fY})$$  of networks  of deterministic  hybrid open systems (\cref{def:networkDetHyOSMorphism})	induces a 1-morphism
$$\begin{tikzcd}[column sep = large, row sep = large]
\discatp_{\fy\in \fY}\dCrl(\Hc_{\fY,\fy})\arrow[d,swap,"\dCrl (i_{b,\fY})\circ\scP_\fY"]\arrow[r,"\dCrl(\Phi)"{name = foo1, below}] & \discatp_{\fx\in \fX}\dCrl(\Hc_{\fX,\fx})\arrow[d,"\dCrl (i_{a,\fX})\circ \scP_\fX"]\\
\dCrl(b)\arrow[r,"\dCrl(z)"{name =foo2, above}] & \dCrl(a),\ar[shorten <= 7pt, shorten >= 7pt, Rightarrow, from = foo1, to = foo2]
\end{tikzcd}
$$ where $$\dCrl(\Phi) = \left\{\big((w_\fy)_{\fy\in \fY},(v_\fx)_{\fx\in \fX}\big)\in \discatp_{\fy\in \fY}\dCrl(\Hc_{\fY,\fy})\times \dCrl(\Hc_{\fX,\fx}):\, (w_{\ph(\fx)},v_\fx)\in \dCrl(\Phi_\fx)\;\forall \; \fx\in \fX\right\}.$$
Thus $$\left( \dCrl\big(i_{b,\fY} )\left(\scP_\fY\big((w_\fy)_{\fy\in \fY}\big)\right), \dCrl\big(i_{a,\fX}\big)\left(\scP_\fX\big((v_\fx)_{\fx\in \fX}\big)\right)\right)$$ are $\Pi(z)$-related (\cref{def:detHySysMorphism}) whenever $(w_{\ph(\fx)},v_\fx)$ are $\Phi_\fx$-related  for all $\fx\in \fX$. 
\end{theorem}
We prove this theorem in \cref{ch4}. For now, we interpret this result in the example of a map of networks.  As we may understand invariance of a subsystem as a map of systems, so too, we see how to realize  invariance  in a network as a map of networks.

\subsection{Networked Thermostats Example}
We give an example of \cref{theorem:mainTheoremFordeterministicHybridSystems} to show the invariance of networked thermostats.  We consider two rooms with heat flow between each of them, and their own thermostats. 
We define a morphism of lists (\cref{def:categoryOfLists}), then the networks,  then the morphism of networks.  
\begin{example}
	Let $c$ be the hybrid phase space in \cref{example:ThermostatVanilla}, $c(i) = \R\times \{i\}$ for $i = 0,1$, $c(e_{i,1-i}) = \big\{(x,1-i,x',i)\in (\R\times\{0,1\})^2:\, x = x'\big\}$ and define lists of hybrid surjective submersions as follows. Let $\fX = \{1,2\}$, $\fY = \{\star\}$  and define $\Hc_\fX:\fX\rightarrow\sF{HySSub}$, $\Hc_\fY:\fY\rightarrow\sF{HySSub}$ by $\Hc_\fX(i) = \Hc_\fY(\star) = \big( c\times c \xrightarrow{p_1}c\big)$, where $p_1:\sF{HyPh}\rightarrow\sF{HyPh}$ is the projection onto the first factor (\cref{prop:BinProductHyPh}). 
	
	We define morphism of lists $\begin{tikzcd}
	\fX\arrow[r,"\ph"] \arrow[dr,swap,"\Hc_\fX",""{name = foo,above}] & \fY\arrow[d,"\Hc_\fY"]\ar[shorten <= 3pt, shorten >= 3pt, Rightarrow, to = foo,"\Phi"] \\
	 & \sF{HySSub} 
\end{tikzcd}$ by $\ph(i) = \star$ for $i = 1,2$ and $\Phi_i:\Hc_\fX(i)\rightarrow\Hc_\fY(\star)$ by $$\begin{tikzcd}(x,j,x',j')\arrow[d,mapsto,"p_1"]\arrow[r,mapsto,"\Phi_{1,tot}"] & (x,j,-x',1-j')\arrow[d,mapsto,"p_1"] \\
(x,j)\arrow[r,mapsto, "\Phi_{1,st}"] & (x,j)	
\end{tikzcd} \;\mbox{and}\; \begin{tikzcd} (x,j,x',j')\arrow[r,mapsto,"\Phi_{2,tot}"]\arrow[d,mapsto,"p_1"] & (-x,1-j,x',j')\arrow[d,mapsto,"p_1"] \\ (x,j)\arrow[r,mapsto,"\Phi_{2,st}"] & (-x,1-j).\end{tikzcd}
 $$This defines the morphism of lists.  The functorial extension by $\Pi$ to a morphism of hybrid surjective submersions  $\Pi(\ph,\Phi):\Pi(\Hc_\fY)\rightarrow\Pi(\Hc_\fX)$ (\cref{prop:NietzscheProp}, \eqref{eq:NietzscheEq}) is defined by $$\begin{tikzcd}[column sep = large]
	(x,j,x',j')\arrow[r,mapsto,"\Pi(\ph{,}\Phi)_{tot}"]\arrow[d,mapsto] & (x,j,-x',1-j',-x,1-j,x',j')\arrow[d,mapsto]\\
	(x,j)\arrow[r,mapsto,"\Pi(\ph{,}\Phi)_{st}"] & (x,j,-x,1-j).
\end{tikzcd}$$

We now define networks.  As we have a list of hybrid surjective submersions $\Hc_\fX$ and $\Hc_\fY$, we need  interconnections $a\rightarrow\Pi(\Hc_\fY)$ and $b\rightarrow\Pi(\Hc_\fX)$.  Observe that  $\Pi(\Hc_\fY) = \Hc_\fY(\star) = c\times c\xrightarrow{p_1}c$.  We define interconnection $(c\xrightarrow{id_c}c)\xrightarrow{i_{c,\fY}} (c\times c\xrightarrow{p_1}c)$ by $$\begin{tikzcd}[column sep = large]
	(x,j)\arrow[r,mapsto,"i_{c{,}\fY{,}tot}"]\arrow[d,mapsto,"id_c"] & (x,j,x,j)\arrow[d,mapsto,"p_1"] \\ 
	(x,j)\arrow[r,mapsto,"i_{c{,}\fY{,}st}"] & (x,j). 
\end{tikzcd}$$
Similarly, we define $(c\times c\xrightarrow{id_{c\times c}} c\times c)\xrightarrow{i_{c\times c, \fX}} (c\times c\times c\times c \xrightarrow{p_1\times p_1}c\times c) = \Pi(\Hc_\fX)$ by $$\begin{tikzcd}[column sep = large]
	(x,j,x',j')\arrow[r,mapsto,"i_{c\times c{,}\fX{,}tot}"]\arrow[d,mapsto] & (x,j,x',j',x',j',x,j)\arrow[d,mapsto]\\
	(x,j,x',j')\arrow[r,mapsto ,"i_{c\times c{,}\fX{,}st}"] & (x,j,x',j').
\end{tikzcd}$$

Finally, to define morphism of networks, we must define a map $z:(c\xrightarrow{id_c}c)\rightarrow(c\times c\xrightarrow{id_{c\times c}} c\times c)$ of hybrid surjective submersions compatible with each network, i.e.\ so that $$\Pi(\ph,\Phi) \circ i_{c,\fY} = (\i,\ifrak)_{c\times c,\fX}\circ z.$$ We define $z:c\rightarrow c\times c$ by $(x,j)\mapsto (x,j,-x,1-j)$ (on both state and total space). For notational ease, we denote \begin{equation}\label{eq:notateThermosatMorphisms} z\defeq (\z,\zf), \; i_\fY \defeq i_{c,\fY},\; i_\fX \defeq i_{c\times c,\fX}.\end{equation}   
 It is easy to verify that these data define a morphism  
$$(\Hc_\fX,i_\fX)\xrightarrow{\big((\ph,\Phi),z\big)} (\Hc_\fY,i_\fY)$$ of networks. Indeed $$\begin{tikzcd}[column sep = large]
(x,j)\arrow[rr,mapsto,"z_{tot}"]\arrow[d,mapsto,"i_{\fY{,}tot}"]& & (x,j,-x,1-j)\arrow[d,mapsto,"i_{\fX{,}tot}"]\\
(x,j,x,j) \arrow[rr,mapsto,"\Pi(\ph{,}\Phi)_{tot}"] & &(x,j,-x,1-j,-x,1-j,x,j),
\end{tikzcd}$$
so indeed $\Pi(\ph,\Phi) \circ i_\fY = i_\fX \circ z$. The map $z$ embeds a state into the ``antidiagonal.'' 

We now define $(\tilde{X},\tilde{\r})\in \dCrl(\Hc_\fY(\star))$, $(X_i,\r_i)\in \dCrl(\Hc_\fX(i))$ for $i =1,2$ so that $\big((X_i,\r_i),(\tilde{X},\tilde{\r})\big)$ are $\Phi_i$-related. Define  deterministic  control by \small\begin{equation}\label{eq:allthecontrolfornetworkexample} \begin{array}{lll} \tilde{X}(x,j,x',j') \defeq (-1)^{1-j} + f(x,x') &  \mbox{and} &  \tilde{\r}(x,j,x',j') \defeq \left\{\begin{array}{lll} 1-j & \mbox{if} & (-1)^{1-j}x \geq 1 \\ j & \mbox{else,} & \end{array}\right. \\
X_1(x,j,x',j') \defeq (-1)^{1-j} + f(x,-x') &  \mbox{and} & \r_1 (x,j,x',j') \defeq \left\{\begin{array}{lll} 1-j & \mbox{if} & (-1)^{1-j}x \geq 1 \\ j & \mbox{else,} & \end{array}\right.\\
X_2(x,j,x',j') \defeq (-1)^{1-j} - f(-x,x') &  \mbox{and} & \r_2 (x,j,x',j') \defeq \left\{\begin{array}{lll} 1-j & \mbox{if} & (-1)^{1-j}x \geq 1 \\ j & \mbox{else}. & \end{array}\right.\end{array}\end{equation}\normalsize

Here  $f:\R^2\rightarrow\R$ is any smooth map.  The first term $(-1)^{1-j}$ of each vector field  represents the heater, and the second term $f(x,x')$ represents heat flow from the outside environment.

  That the relevant maps are related is a computation.  We check $\Phi_2$-relatedness of $(\tilde{X},X_2)$: \small
\begin{equation}\label{diagram:relatednessNetworksExample} \begin{tikzcd}[column sep = large]
	(x,j,x',j')\arrow[r,"\Phi_{2,tot}"]\arrow[d,"\tilde{X}"] & (-x,1-j,x',j')\arrow[d,mapsto ,"X_2"] \\
	(-1)^{1-j} + f(x,x')\arrow[r,mapsto,"T\Phi_{2,st}"] & -\big((-1)^{1-j}+ f(x,x')\big) = (-1)^{1-(1-j)} - f(-(-x),x'),   
\end{tikzcd}\end{equation}\normalsize since $(-1)^j = (-1)^{-j}$ for $j= 0,1$. $\Phi_2$-relatedness of $(\tilde{\r},\r_2)$ is similar: $$1-\tilde{\r}(x,j,-x',1-j') =\left\{\begin{array}{ll} j & \mbox{if}\, (-1)^{1-j}x \geq 1\\1- j & \mbox{else}\end{array}\right.$$ and $$\r_2(-x,1-j,x',j') = \left\{\begin{array}{ll} 1-(1-j) & \mbox{if}\, (-1)^{1-(1-j)}(-x)\geq 1 \\  1-j & \mbox{else}\end{array}\right.$$ which are equal, so $\Phi_{2,st}\circ \tilde{\r} = \r_2 \circ \Phi_{2,tot}$. 

Since each $\big((\tilde{X},\tilde{\r}),(X_i,\r_i)\big)$ is $\Phi_i$-related, \cref{theorem:mainTheoremFordeterministicHybridSystems} says that $$\left(\dCrl(i_\fY)\big(\tilde{X},\tilde{\r}\big), \dCrl(i_\fX)\big((X_1,\r_1)\times (X_2,\r_2)\big)\right)$$ are $z$-related.  Thus the antidiagonal is invariant under the dynamics of $\dCrl(i_\fX)\big((X_1,\r_1)\times(X_2,\r_2)\big)$.

We verify relatedness directly.  On the one hand, \small\begin{equation}\label{eq:networkeq1} \dCrl(i_\fY)\big(\tilde{X},\tilde{\r}\big)(x,j) = (\tilde{X},\tilde{\r})(x,j) = \left((-1)^{1-j}+f(x,x),\left\{\begin{array}{ll} 1-j & \mbox{if}\; (-1)^{1-j}x \geq  1\\ j & \mbox{else,} \end{array}\right.\right).\end{equation} Pushing forward by $\Tb\Pi(\ph,\Phi)_{st}$ (\cref{prop:cdBundleFunctorial}, \eqref{eq:cdBundleFunctorial}), we obtain \small \begin{equation}\label{eq:networkseq2} \begin{array}{ll}\Tb\Pi(\ph,\Phi)_{st}\dCrl(i_\fY)\big(\tilde{X},\tilde{\r}\big)(x,j) &  = \left((-1)^{1-j}+ f(x,x),\left\{\begin{array}{ll} 1-j &\mbox{if} \; (-1)^{1-j}x\geq 1, \\ j &\mbox{else} \end{array}\right. \right. \\
 & (-1)^{j}-f(x,x),
 \left.\left\{\begin{array}{ll} j & \mbox{if} \; (-1)^{1-j}x \geq 1\\ 1-j & \mbox{else} \end{array}\right.\right). \end{array}\end{equation}\normalsize
On the other hand,  $$\small \begin{array}{ll}\dCrl(i_\fX)(X_1\times X_2,\r_1\times \r_2)z(x,j) & =\left( X_1(x,j,-x,1-j),\r_1(x,j,-x,1-j), X_2(-x,1-j,x,j), \r_2(-x,1-j,x,j)\right).
\end{array}
$$\normalsize Reference to \eqref{eq:allthecontrolfornetworkexample} and \eqref{diagram:relatednessNetworksExample} shows that \eqref{eq:networkeq1} equals \eqref{eq:networkseq2}.

\end{example}

\chapter{Abstract Networks of Systems}\label{ch4}

\newcommand{\open}{open }
\section{Introduction} 
In this chapter we develop a categorical framework for understanding network of systems.  We  generalize the notion of system as a pair (object and section of a bundle) in \cref{subsection:abstractSections}, and then build upon this version of system by generalizing the notions of network from \cite{lermanopennetworks} and from \cref{ch3}, \cref{section:NetworksHybrid} in \cref{subsection:monoidallyfibered}. 
 Running examples include manifolds, continuous-time dynamical systems, open systems, and networks of open systems, as defined and worked out in \cite{lermanopennetworks}, and hybrid phase spaces, hybrid systems, hybrid open systems and networks of hybrid open systems, as defined and worked out in \cref{ch3}.

\begin{assumption}\label{assumption:locallySmall}\label{assumption:concrete}
Throughout this chapter, all categories are concrete (\cref{remark:concreteCategory}) and   locally small (\cref{def:locallySmall}). \end{assumption} 

Our goal is to prove a generalized version of \cref{theorem:mainTheoremFordeterministicHybridSystems} and \cref{theorem:mainTheoremLermanOpenNetworks}, which crudely says that a morphism of networks induces a 1-morphism $$\begin{tikzcd}
	\discatp_{\fy\in \fY}\G(\fy)\arrow[r,""{below, name = foo1}]\arrow[d] & \discatp_{\fx\in \fX} \G(\fx) \arrow[d]\\ \G(b)\arrow[r,""{above, name = foo2}]  & \G(a) \ar[shorten <= 10pt, shorten >= 10 pt, Rightarrow, from = foo1 , to = foo2]
\end{tikzcd}$$ in the double category $\sF{Set}^\square$.   More concretely, a collection of relations holding between pairs  $(Y,X) \in \G(\fy')\times \G(\fx')$  induces a relation in $\G(b)\times \G(a)$.  Alternatively, a collection of pairs of related systems induces a pair of related systems.  The vertical arrows take a collection of systems, by interconnection, to a single system.  This succinct result is meant to convey the idea of building a map of systems from a collection of maps of subsystems.  

We have not yet assigned content to the symbol $\G$.  In this chapter, we will define $\G$ as \textit{abstract sections} of some natural transformation, and prove the main theorem (\cref{theorem:mainTheoremAbstract}) in two parts.  After constructing $\G$ which assigns sections to some split epimorphisms---namely, epimorphisms coming from a natural transformation---in \cref{subsection:fiberedCategories},   we introduce the notion of interconnection,  formally a class of morphisms in an arrow category which are isomorphisms on state.  We  show that a certain commuting square $$\begin{tikzcd}
	a \arrow[r,""{below, name = foo1}] & b\\
	c \arrow[u]\arrow[r,""{above, name = foo2}] & d\arrow[u] \ar[shorten <= 5pt, shorten >= 5pt, Rightarrow, from = foo2, to = foo1]
\end{tikzcd}$$ (where vertical arrows are interconnection maps) induces a 1-morphism $$\begin{tikzcd}
	\G(a) \arrow[r,""{below, name = foo1}] \arrow[d]& \G(b)\arrow[d]\\
	\G(c)\arrow[r,""{above, name = foo2}] & \G(d) \ar[shorten <= 5pt, shorten >= 5pt, Rightarrow, from = foo1, to = foo2]
\end{tikzcd}$$ in $\sF{Set}^\square$. This is the first half of the main theorem. 

 In \cref{subsection:monoidallyfibered}, we introduce and work with a monoidal structure on the category.  This is a key component of networks, used for putting  many  (separate) systems together into one.  Interconnection is then used to relate them together, or to \textit{interconnect} them.  This is the networks piece of the puzzle.   We  show that a collection of relations induce a relation on the product, diagrammatically indicated by 1-morphism
$$\begin{tikzcd}
\discatp_{\fy\in \fY}\G(\fy)\arrow[r,""{below,name = foo1}] \arrow[d] & \discatp_{\fx\in \fX} \G(\fx)\arrow[d]\\
\G\left(\bigotimes_{\fy\in \fY}\fy\right) \arrow[r,""{above, name = foo2}] & \G\big(\bigotimes_{\fx\in \fX}\fx\big) \ar[shorten <= 10pt, shorten >= 10pt, Rightarrow, from = foo1, to = foo2]
\end{tikzcd}$$ in $\sF{Set}^\square$, which gives the second  half of the main theorem.  We combine results and interpret them for systems by proving concrete results about morphisms of networks of continuous-time open systems, networks of hybrid open systems, and finally networks of deterministic hybrid open systems.

Before working out a pure category theory for networks, we develop an abstract notion of system as object and section of some split epimorphism.    Recall, for reference, that a continuous-time  dynamical system  is a pair $(M,X)$ where $M$ is a manifold and $X\in \Xf(M)$ is a vector field on the manifold  (\cref{def:ctDySys}).   In this case, the manifold is the object and $X$ is a section of the tangent bundle $TM\xrightarrow{\t_M}M$. More generally, we consider a natural transformation $\t:\Tc\Rightarrow\Uc$ between two functors which is \textit{split epimorphism}, and take a system to be an object $\fc$ and section of the epimorphism $\t_\fc:\Tc\fc\rightarrow\Uc\fc$. In the case of manifolds, both $\Tc$ and $\Uc$ are endofunctors on $\sF{Man}$, where the source functor $\Tc$ is the tangent endofunctor and the target functor $\Uc$ is the identity. 

 So far, this formalism appears unnecessarily abstract.  We make use of the extra generality when considering hybrid systems, which we have defined as a hybrid phase space together with a vector field on a manifold.  In this case, our source category is $\sF{HyPh}$, target category $\sF{Man}$, $\Uc$ is the forgetful functor $\Ub:\sF{HyPh}\rightarrow\sF{Man}$ from \cref{prop:ForgetfulFunctor}, and $\Tc$ is the tangent endofunctor on $\sF{Man}$ composed with $\Ub$.  We thus pin down a way of describing different kinds of systems by using functors, natural transformations, and typical categorical nonsense to make sense of dynamics in one category via another proxy category ($\sF{HyPh}$ and $\sF{Man}$, respectively, in our example of hybrid systems). Secondly, we isolate what appears to us as particular to the very notion of dynamical system, itself. 
 
 We motivate the latter endeavor. We contrast our perspective of dynamical system as space and vector field, as opposed to  trajectory or flow.  This distinction  is like the difference between a differential equation and  its solution.  At the end of the day, we care about flows, but we find it convenient to work with vector fields and related vector fields; in much the same way as solving a differential equation is usually harder than merely having one.  Additionally, the focus on vector field makes it very easy to abstract: a smooth section of the tangent bundle is a section of a map in the relevant category.  By extension, an abstract system is an object with a section.  The downside of this abstract definition is interpreting dynamics in time, or ``solutions'' to the system.

 Our optimism that there is a categorical interpretation of time comes  from the Yoneda version of existence and uniqueness (\cref{prop:existenceAndUniquenessRepresentable}): the forgetful functor $\upsilon:\sF{DySys}\rightarrow \sF{Set}$ from the category of complete dynamical systems which forgets dynamics and smooth structure is representable.  This formulation is unique to \textit{complete} systems, but the notion of completeness is not essential to a theory of continuous-time dynamical systems: integral curves $\g:(-\e,\e)\rightarrow M$ from a bounded domain are perfectly fine.  Representability translates to initiality of an object in the category of elements $\disg_{\sF{DySys}}\upsilon,$ but the similar forgetful functor from the category of (possibly incomplete) continuous-time dynamical systems is not representable.   Lack of completeness means there may not be any map from $(\R,\frac{d}{dt})$ to the desired dynamical system; at best there is at most one. We therefore interpret solution of a system as a map from some ``quasi-initial'' object in the category of elements: a map from an object which is unique when it exists.

\section{Fibered Transformations}\label{subsection:fiberedCategories}\label{subsection:abstractSections}

We begin with the notion of fibered transformation. 
\begin{definition}\label{def:CfiberedInDandSections}\label{def:fiberedTransformationInD}
Let $\fC, \fD$ be locally small concrete categories,  $\Tc,\Uc:\fC\rightarrow\fD$   functors, and  $\t:\Tc\Rightarrow \Uc$ a natural transformation. We say that $\t$ is $\fD$-\textit{fibered} (or simply \textit{fibered}) if for each object $\fc\in \fC_0$, $\Tc\fc\xrightarrow{\t_\fc}\Uc \fc$ is a split epimorphism (\cref{def:splitEpi}).
For a $\fD$-fibered transformation, we define  $\t$-\textit{sections} by the set of right inverses of $\t_\fc$: 
 \begin{equation}\label{eq:abstractSections} \G_\t(\fc)\defeq \left\{\big(\Uc\fc\xrightarrow{X} \Tc\fc\big)\in \fD_1:\, \t_{\fc}\circ X =id_{\Uc\fc}\right\}.\end{equation}
	 \end{definition}
	 
	 \begin{remark}\label{remark:whySplit?}
	 	Since $\t_\fc:\Tc\fc\rightarrow\Uc\fc$ is split epi, $\t$-sections $\G_\t(\fc)$ are guaranteed to be nonempty (\cref{def:splitEpi}). 
	 \end{remark}
	
We make sure this definition passes a few sanity tests.

\begin{example}\label{ex:identityFunctorForfiberedCategories}
Let $\fC=\fD=\sF{Man}$, $\Uc = id_{\sF{Man}}$, and $\Tc:\sF{Man}\rightarrow\sF{Man}$ assign the tangent bundle $TM$ to each manifold $M$.  This assignment is functorial (\cref{prop:differentialIsFunctorial}). Moreover, the canonical projection of the tangent bundle $\t_M:TM\rightarrow M$ assembles into a natural transformation (\cref{fact:naturalityOfProjTangentBundle}).  Finally, the  projection $\t_M:TM\rightarrow M$ is a split epimorphism (\cref{fact:tangentProjSplitEpi}).  Therefore the natural transformation $\t$ is $\sF{Man}$-fibered, or in this case, simply fibered. 
\end{example}

The next example pertains to and is used in our development of hybrid systems. 
\begin{example}\label{ex:tauSectionForHyPh} Let $\fC = \sF{HyPh}$, $\fD = \sF{Man}$, and let $\Uc:\fC\rightarrow\fD$ be the forgetful functor $\Ub:\sF{HyPH}\rightarrow\sF{Man}$ (\cref{prop:ForgetfulFunctor}). Finally, let $\Tc:\fC\rightarrow\fD$ be the composition $\Tc = T\circ \Ub$, where $T$ is the tangent endofunctor in \cref{ex:identityFunctorForfiberedCategories}. 
Similarly, let $\t:\Tc\Rightarrow\Uc$ be defined on components as the canonical projection of the (underlying) tangent bundle $\t_a:T\Ub a\rightarrow\Ub a$, which is again a split epimorphism  (\cref{fact:tangentProjSplitEpi}). We conclude that  $\t$ is $\sF{Man}$-fibered.  
\end{example}
	
\begin{definition}\label{def:DSystemsFromC}\label{def:tauSystem}
	Let $\Tc,\Uc:\fC\rightarrow \fD$ be functors, and $\t:\Tc\Rightarrow \Uc$ a $\fD$-fibered transformation (\cref{def:fiberedTransformationInD}). We define a $\t$-\textit{system} to be a pair $(\fc,X)$ where $\fc\in \fC_0$ is an object in $\fC$ and $X\in \G_\t(\fc)$ is a  $\t$-section. We also define a  \textit{morphism $(\fc,X)\xrightarrow{f}(\fd,Y)$ of $\t$-sections} to be a morphism $\fc\xrightarrow{f}\fd$ in $\fC$ such that $(X,Y)$ are $f$-related, i.e.\ $\Tc f \circ X = Y\circ \Uc f$.  The collection of $\t$-systems and morphisms make up a category, which we denote by $\t$-$\sF{Sys}$.
\end{definition}

\begin{example}\label{ex:discSysAsprojCsysInC}
	A discrete (time) dynamical system $(f:X\rightarrow X, x_0)$ may be thought of as an endomorphism $f$ with a choice of base point $x_0\in X$ (\cite[\S2.1]{riehlb}). We will recover the base point in \cref{ex:iteratedMapasAbstractSolution}. For now, we think of such a system  only as an endomorphism $X\xrightarrow{f}X$ and interpret in the lens of \cref{def:DSystemsFromC}. 
	
	Let $\fC$ be a concrete category with products and $\fD = \fC$. Let $\Tc = (\cdot)\times (\cdot)$ and $\Uc = id_\fC$.\footnote{Formally, $\Tc = (\cdot)\times (\cdot)$ is the composition of functors $\fC\xrightarrow{\Delta}\fC\times \fC\xrightarrow{\times}\fC$ sending $\fc\mapsto (\fc,\fc)\mapsto \fc\times \fc$.}  Finally, let $\t = p_1$ be the projection onto the first factor. For object $X\in \fC$, we have projection $X\times X\xrightarrow{p_1}X$ onto the first factor. An abstract $\t$-system $(X,f)$ is an object $X\in \fC$ with section $f:X\rightarrow X\times X$ (satisfying $p_1\circ f = id_X$).  Generally, it is more convenient to ignore the ``section'' part and treat $f$ simply as a (any) morphism $X\xrightarrow{f}X$ (precisely, as $p_2\circ f$).
\end{example}

\begin{example}\label{ex:dySysAsManSysInMan}
	Let $\fC=\fD=\sF{Man}$, $\Uc = id_{\sF{Man}}$, $\Tc=T$ the tangent functor, and $\t:\Tc\Rightarrow\Uc$ the canonical projection of the tangent bundle (\cref{ex:identityFunctorForfiberedCategories}).  Then a $\t$-system $(M,X)$  is a continuous-time dynamical system (\cref{def:ctDySys}). 
\end{example}
\begin{example}\label{ex:hybridSystemAsHyPhSystemInMan} For $\nat{\fC}{\fD}{\Tc}{\Uc}{\t}$ as in \cref{ex:tauSectionForHyPh}, with $\fC=\sF{HyPh}$, a $\t$-system is a hybrid system $(a,X)$  (\cref{def:hybridSys}). 
\end{example}

The notion of open system  (\cref{def:openDySys}) generalizes as well. In the setting of continuous-time systems, we considered the subcategory $\sF{SSub}\subset\sF{Arrow(Man)}$ whose objects are surjective submersions. The motivation there comes from control theory, where  typical   control systems map $f:M\times U\rightarrow TM$ satisfying $\t_M\circ f = p_1$. At this moment, we only ask for a subcategory of the arrow category (later we will ask it to also be cartesian (\cref{assumption:3}, \cref{def:abstractNetworks})!).

\begin{definition}\label{def:COpenSystemInDRelativeToA}\label{def:abstractOpenSections}\label{def:abstractOpenSystem}
	Let $\fA\subseteq\sF{Arrow(C)}$ be a subcategory of the arrow category of $\fC$ and $\nat{\fC}{\fD}{\Tc}{\Uc}{\t}$  a $\fD$-fibered transformation of functors $\Tc,\Uc:\fC\rightarrow\fD$ (\cref{def:fiberedTransformationInD}). 
   We define an \textit{$\fA$-open $\t$-system} to be a pair $(\fa\xrightarrow{p_\fa }\fa', X$) where $p_\fa\in \fA$ and $X$ is a morphism $X:\Uc\fa\rightarrow \Tc\fa'$ in $\fD$ such that $\Uc p_\fa  = \t_{\fa} \circ X$ (we say that $X$ is $p_a$-\textit{compatible}).  We define the set $$\G_\t(p_\fa)\defeq \big\{X:\Uc \fa\rightarrow\Tc\fa':\, \Uc p_\fa = \t_\fa\circ X\big\},$$ and call $X\in \G_\t(p_\fa)$ an \textit{\open $\t$-section},  \textit{\open section of $\t$}, or \textit{abstract} $\t$-\textit{control}. Also, we may refer to $X$ itself as an (abstract) open $\t$-system. 
\end{definition}
\begin{remark}\label{remark:openSectionsNonemtpy}
	For any object $\fa\xrightarrow{p_\fa}\fa'$ in $\fA$, $\G_\t(p_\fa)$ is guaranteed to be nonempty.  First, $\G_\t(\fa')\neq \varnothing$ as (ordinary) $\t$-sections since $\t:\Tc\Rightarrow \Uc$ is $\fD$-fibered (\cref{def:fiberedTransformationInD}, eq.\ \eqref{eq:abstractSections}). Thus, for section $X'\in \G_\t(\fa')$, we define $X\defeq X'\circ \Uc p_\fa$ which is  an open $\t$-section     because $$\Uc p_\fa = id_{\Uc\fa'} \circ \Uc p_\fa = \t_{\fa'} \circ X'\circ \Uc p_\fa=\t_{\fa'}\circ X.$$
\end{remark}
\begin{notation}\label{notation:OpenSectionsEtc.}
For $p_\fa\in \fA$, we will generally denote the domain $\sF{dom}(p_\fa)$ of $p_\fa$ by $\fa_{tot}$ (\textit{total} space) and the codomain $\sF{cod}(p_\fa)$ by $\fa_{st}$  (\textit{state} space), consistent with   convention in \cite[definition 2.14]{lermanopennetworks}.  By extension, we denote objects $p_\fa\in\fA$ by $\fa$, open $\t$-sections of $\fa_{tot}\xrightarrow{p_\fa}\fa_{st}$ by $\G_\t(\fa)$, and open systems by  $(\fa,X)$.
\end{notation}
\begin{example}\label{ex:recoverAbstractSystemFromOpen}
	Let $\nat{\fC}{\fD}{\Tc}{\Uc}{\t}$ be any $\fD$-fibered natural transformation (\cref{def:fiberedTransformationInD}), and $\fA\subset \sF{Arrow(C)}$ the arrow subcategory of only identity morphisms:  $$\fA\defeq \big\{ id_\fc:\, \fc\in \fC_0\big\}.$$ Then  $\fA$-open $\t$-systems (\cref{def:abstractOpenSystem}) are $\t$-systems (\cref{def:tauSystem}).  \end{example}

\begin{example}\label{ex:discreteOpenasAbstract}
	Let $\nat{\fC}{\fC}{\Tc}{\Uc}{\t}$ be as in \cref{ex:discSysAsprojCsysInC} for discrete-time dynamical systems and let $\fA \subset \sF{Arrow(C)}$ be the collection of surjections: an arrow $(X'\xrightarrow{p_X}X)\in \fA$ as long as $p_X:X'\twoheadrightarrow X$ is a surjective map of sets. Then an $\fA$-open $\t$-system $(p_X, f)$ is a pair where $p_X:X'\rightarrow X$ is a surjective map and $f:X'\rightarrow X$ is any map. Compatibility of $f$ and $p_X$ in \cref{def:abstractOpenSystem} is vacuous because properly speaking, an open system $f$ is required to make the  diagram commute  $$\begin{tikzcd}
	X'\arrow[r,"f"]\arrow[dr,swap,"p_X"] & X\times X,\arrow[d,"\t_X"]\\
	& X.
\end{tikzcd}$$We treat $f$ as the map to the second factor.  In other words, maps $p_X:X'\rightarrow X$ and $f:X'\rightarrow X$ induce a map $X'\rightarrow X\times X$, which by abuse of notation we have written in the diagram above simply as $f$ itself.  
\end{example}

\begin{example}\label{ex:ordinaryOpenSystemsInstanceOfAbstract}
	Let $\nat{\fC}{\fD}{\Tc}{\Uc}{\t}$ be as in \cref{ex:dySysAsManSysInMan} ($\fC=\sF{Man}$, etc.) and let $\fA\subset \sF{Arrow(C)}$ be $\fA=\sF{SSub}$.  Then  $\fA$-open $\t$-systems $(a,X)$ are open systems (\cref{def:openDySys}). 
\end{example}
\begin{example}\label{ex:HybridOpenSystemsInstanceOfAbstractOpenSystem}
Let $\nat{\fC}{\fD}{\Tc}{\Uc}{\t}$ be as in  \cref{ex:tauSectionForHyPh} ($\fC = \sF{HyPh}$ etc.) and $\fA = \sF{HySSub}$ (\cref{def:categoryHySSub}).  An  $\fA$-open $\t$-system $(a,X)$ is   a hybrid open system (\cref{def:hybridOS}). 
\end{example}
We now consider a modified version of \cref{ex:HybridOpenSystemsInstanceOfAbstractOpenSystem} which accounts for determinism. 
\begin{example}\label{ex:detHyOSInstanceOfAbstractOpenSystem}
	Let $\fC = \sF{HySSub}$, $\fD= \sF{Set}$, $\Tc a \defeq \Tb a$ the c.d.\ bundle (\cref{def:cdBundle}, \cref{prop:cdBundleFunctorial}), and $\Uc a\defeq \big\{x\in \Ub a\big\}$, the set of points in the underlying manifold $\Ub a$.  Define projection $\varpi:\Tc a\Rightarrow \Uc a$ on objects   $\varpi_a:\Tc a\rightarrow \{x\in \Ub a\}$ by $\varpi \defeq \{\t_{\Ub a}\circ p_1\}$ where $p_1$ is projection onto the first factor, $\t_M:TM\rightarrow M$ is the canonical projection of the tangent bundle 
(c.f.\ eq.\ 	\eqref{eq:definingCDBundleProj}), and $\{\cdot\}$ is the forgetful functor which returns the underlying set.  This map is natural in $a$ since $\varpi$ is defined as the composition of canonical maps, and easily seen to be a split epimorphism. Therefore $\nat{\fC}{\fD}{\Tc}{\Uc}{\varpi}$ is $\fD$-fibered.  For $\fA =\sF{HySSub}$, a deterministic hybrid open system $(a,X,\r)$  (\cref{def:deterministicHyOS})  is an abstract $\fA$-open $\varpi$-system.  However, not all $\fA$-open $\varpi$-systems are deterministic hybrid open systems: since the ambient category $\fD = \sF{Set}$, sections in general need not satisfy the smoothness condition for $X$ and jump-compatibility condition for $\r$ (\cref{def:detHySys}).  We call attention to this asymmetry, as we will need to be careful when proving \cref{theorem:mainTheoremFordeterministicHybridSystems} in \cref{subsection:mainDeterministicTheorem}.  
\end{example}
\begin{definition}\label{def:COpenRelatedness}\label{def:relatedOpenSections}
	Let $f:\fa\rightarrow\fa'$ be a morphism in $\fA$ (\cref{def:arrowCategory}), so that  diagram \begin{equation}\label{diagram:arrowDiagram} \begin{tikzcd}
\fa_{tot}\arrow[d,"p_\fa"]\arrow[r,"f_{tot}"] & \fa'_{tot}\arrow[d,"p_{\fa'}"] \\
\fa_{st}\arrow[r,"f_{st}"] & \fa_{st}' \end{tikzcd}\end{equation} commutes.  Let  $X\in \G_\t(\fa)$, $X'\in \G_\t(\fa')$	be  $\fA$-open $\t$-systems  (\cref{def:abstractOpenSystem}, \cref{notation:OpenSectionsEtc.}).  We say that $(X,X')$ are $f$-\textit{related} if $\Tc f_{st} \circ X = X'\circ \Uc f_{tot}$, i.e.\ if the diagram $\begin{tikzcd} \Uc\fa_{tot}\arrow[d,"X"]\arrow[r,"
\Uc f_{tot}"] & \Uc\fa_{tot}' \arrow[d,"X'"] \\
\Tc\fa_{st}\arrow[r,"\Tc f_{st}"] & \Tc\fa_{st}'
\end{tikzcd}$ commutes.
\end{definition}

\begin{definition}\label{def:COpenSystemMorphism}
	Let $(\fa,X)$ and $(\fa',X')$ be two open systems (\cref{def:COpenSystemInDRelativeToA}).  We define a \textit{morphism} $f:(\fa,X)\rightarrow (\fa',X')$ \textit{of $\fA$-open $\t$-systems} to be a morphism $f:\fa\rightarrow \fa'$ in $\fA$ (c.f.\ \eqref{diagram:arrowDiagram})  such that $(X,X')$ are $f$-related (\cref{def:COpenRelatedness}).
\end{definition}
\begin{remark}\label{remark:commutativityConditionsForMorphismAbstractSystem}
	In addition to $f$-relatedness, a map $f:(\fa,X)\rightarrow(\fb,Y)$ of $\fA$-open $\t$-systems consists of numerous commutating diagrams, namely each subdiagram in the following:$$\begin{tikzcd}[column sep = large,row sep = large]
& \Uc \fa_{tot}\arrow[dl,swap,"X"]\arrow[rr, "\Uc f_{tot}"]\arrow[dd,near start,"\Uc p_\fa"]  & & \Uc \fb_{tot}\arrow[dl,swap,"Y"]\arrow[dd,near start,"\Uc p_\fb"] \\
\Tc \fa_{st}\arrow[rr,near start,"\Tc f_{st}"]\arrow[dr,"\t_{\fa_{st}}"] & & \Tc\fb_{st}\arrow[dr,"\t_{\fb_{st}}"] & \\ & 
\Uc\fa_{st}\arrow[rr,"\Uc f_{st}"] & & \Uc \fb_{st}.
\end{tikzcd}
 $$
\end{remark}

Equalities $\t_{\fa_{st}}\circ X = \Uc p_\fa$ and $\t_{\fb_{st}}\circ Y = \Uc p_\fb$ hold by assumption $X\in \G_\t(\fa)$, $Y\in \G_\t(\fb)$. Equality $\Uc f_{st} \circ \Uc p_\fa = \Uc p_\fb \circ \Uc f_{tot}$ holds since $f:\fa\rightarrow \fb$ is a morphism in $\fA$ and $\Uc$ is a functor.  Equality $\Uc f_{st}\circ \t_{\fa_{st}} = \t_{\fb_{st}} \circ \Tc f_{st}$ follows by naturality of $\t$. Finally, $\Tc f_{st}\circ X = Y \circ \Uc f_{tot}$ is  $f$-relatedness of $(X,Y)$ (\cref{def:relatedOpenSections}). 

\begin{remark}\label{remark:catTauSystem}
There is a category $\t$-$\sF{Sys}$ of $\t$-(open)-systems  (\cref{def:tauSystem}, \cref{def:COpenSystemInDRelativeToA}) and their morphisms (\cref{def:COpenSystemMorphism}). Since $\fC$ is concrete, there is also forgetful functor $\upsilon:\t\text{-}\sF{Sys}\rightarrow \sF{Set}$ sending $(\fc,X)\mapsto \{x\in \fc\}$.\footnote{To be precise, we are composing two functors $\t\text{-}\sF{Sys}\rightarrow\fC\rightarrow\sF{Set}$, where the first drops the $\t$-section and the second is  faithful functor $\scU:\fC\rightarrow\sF{Set}$ (\cref{remark:concreteCategory}). The assignment, properly speaking, is $(\fc,X)\mapsto \scU(\fc).$}
\end{remark}

Still missing from the theory of continuous-time systems in our abstraction is the notion of ``solution''.  Taking a hint from \cref{prop:existenceAndUniquenessRepresentable}, we may trying defining a solution for abstract systems as a morphism from an initial object in the relevant category of elements.  Unfortunately, the notion of execution of hybrid systems (\cref{def:deterministicExecution})---our deterministic hybrid version of solution---is not representable as ``morphism from initial object'' since the domain of executions are not initial in the relevant category of elements for deterministic hybrid systems. The problem is that a morphism may not always exist.  For example, when the set of jump times is $\Tc =\{0,1,2,3,4,\ldots\}$, there is no morphism from $(\omega,T,\t)_\Tc$ to the bouncing ball system $(c,Z,\varsigma)$ in \cref{ex:bouncingBallExecution}. It turns out that this is not so unfortunate, as the same ``problem'' exists in the category of (possibly incomplete) dynamical systems.  For example, there is no morphism $(\R,\dx =1) \rightarrow(\R,\dx = x^2)$ (such systems exhibit finite escape).

\begin{definition}\label{def:quasiInitial}
	Let $\fC$ be a category and $\fc\in \fC_0$ an object.  We say that $\fc$ is \textit{quasi-initial} if for any object $\fc'\in \fC_0$, there is \textit{at most} one morphism $\fc\xrightarrow{f}\fc'$ in $\fC$. Furthermore, if there is  morphism $\fc'\xrightarrow{g}\fc$, then we require that $g$ be a monomorphism (\cref{def:monomorphism}). 
\end{definition}

Unlike initial objects (\cref{def:initialObject}), quasi-initial objects need not be unique up to unique isomorphism, nor even isomorphic.  

We are now ready to introduce a notion of solution.

\begin{definition}\label{def:abstractSolution}
	Let $(\fc,X)\in \t\text{-}\sF{Sys}$ be an abstract $\t$-system.  We define a \textit{solution} of $(\fc,X)$ to be a morphism $((\sF{i},I),0)\rightarrow ((\fc,X),x_0)$ from quasi-initial object $((\sF{i},I),0)\in \disg_{\t\text{-}\sF{Sys}} \upsilon$ in the category of elements of forgetful functor $\upsilon:\t\text{-}\sF{Sys}\rightarrow\sF{Set}$ (\cref{remark:catTauSystem}).
\end{definition}

\begin{example}\label{ex:iteratedMapasAbstractSolution} This example comes from \cite[\S2.2]{riehlb}.
 Let $f:X\rightarrow X$ be a discrete dynamical system (\cref{ex:discSysAsprojCsysInC}) and suppose that the natural numbers $\N\in \fC$.  Consider successor map $\s:\N\rightarrow\N$ defined by $\s(n)\defeq n+1$. This map defines a discrete dynamical system as well, and a map $(\N,\s)\xrightarrow{\a}(X,f)$ of systems satisfies \begin{equation}\label{eq:discRelated} \a\circ\s = f\circ \a.\end{equation}  Choosing initial point $x_0$ as the image of $0$ under $\a$, we have entirely determined the map $\a$: for $\a$-relatedness in \eqref{eq:discRelated} implies that $\a(1) = a(\s(1)) = f(x_0)$ and in general, $\a(n) = \underbrace{f\circ \cdots \circ f}_{\text{n-times}}(x_0).$  Therefore, the map $$\left(\big(\N,\s\big),0\right)\xrightarrow{\a}\left(\big(X,f\big),x_0\right)$$  in $\disg_{\sF{dSys}}\upsilon$ is a solution of $(X,f)$ in the sense of \cref{def:abstractSolution}.
\end{example}
\begin{example}\label{ex:intCurveasAbstractSolution}
	Let $(M,X)$ be a continuous-time dynamical system (\cref{ex:dySysAsManSysInMan}). An integral curve $\g:(-\e,\e)\rightarrow M$ of $(M,X)$ sends time $0\mapsto \g(0) = x_0$ to initial condition in $M$. Consequently there is morphism $$\left(\big((-\e,\e),\frac{d}{dt}\big),0\right)\xrightarrow{\g}\left(\big(M,X\big),x_0\right)$$ in the category $\disg_{\sF{DySys}} \upsilon$ where $\sF{DySys}$ denotes the category of continuous-time dynamical systems and $\upsilon:(M,X)\mapsto \big\{x\in M\big\}$ is the forgetful functor.  By existence and uniqueness, this map is a solution in the sense of \cref{def:abstractSolution}. 
\end{example}

\begin{example}\label{ex:ExecutionasAbstractSolution}
Similarly, $\Tc$-universal systems induce abstract solutions (\cref{prop:uniqueExecutions}, \cref{remark:uniqueExecutions}) in $\disg_{\sF{dHySys}}\upsilon$, where $\sF{dHySys}$ is the category of deterministic hybrid systems (\cref{lemma:detHySysCategory}) and $\upsilon:(a,X,\r)\mapsto\big\{x\in \Ub a\big\}$. 
\end{example}

We now interpret $\t$-systems as an element in   category of \textit{related} elements.  Recall the category $\sF{Rel}$ (\cref{def:Rel}), whose objects are sets and morphisms are relations.  First we observe that open  $\t$-sections $\G_\t(\cdot)$ extend to a lax functor.

\begin{prop}\label{prop:openSectionsLaxFunctorial} Let $\nat{\fC}{\fD}{\Tc}{\Uc}{\t}$ be $\fD$-fibered natural transformation (\cref{def:fiberedTransformationInD}) and $\fA\subseteq \sF{Arrow(C)}$ an arrow subcategory of $\fC$. Then the assignment of $\G_\t$ on objects to $\fA$-open $\t$-sections (\cref{def:abstractOpenSections}) extends to a lax functor $\G_\t:\fA\rightarrow\sF{Rel}$. In particular (\cref{ex:functorToRel}), for morphisms $\fa\xrightarrow{f}\fa'\xrightarrow{g}\fa''$ in $\fA$, there is inclusion \begin{equation}\label{eq:strictInclusion}
\G_\t(g)\circ \G_\t(f) \subseteq \G_\t(g\circ f).
\end{equation}
\end{prop}

\begin{proof}[Proof Sketch]
	We have defined $\G_\t$ on objects (\cref{def:COpenSystemInDRelativeToA}).  Let $\fa,\fa'\in \fA$ and $f:\fa\rightarrow\fa'$ be a morphism in $\fA$. We define relation \begin{equation}\label{eq:definingOpenSectionsRelationOnMorphisms} \G_\t(f)\defeq \big\{(X,X')\in \G_\t(\fa)\times \G_\t(\fa'):\;(X,X')\,\mbox{are f-related}\big\}.\end{equation}
	
	To show inclusion \eqref{eq:strictInclusion},  suppose that $(X',X'')\in \G_\t(g)$ and $(X,X')\in \G_\t(f)$ for morphisms $\fa\xrightarrow{f}\fa'\xrightarrow{g}\fa''$.  Then each sub-diagram of \begin{equation}\label{diagram:538a}\begin{tikzcd}[column sep = large]
	\Uc \fa_{tot}\arrow[r,"\Uc f_{tot}"]\arrow[d,"X"] \arrow[rr,bend left,"\Uc(g_{tot}\circ f_{tot})"] & \Uc \fa_{tot}'\arrow[d,"X'"]\arrow[r,"\Uc g_{tot}"] & \Uc \fa_{tot}'' \arrow[d,"X''"] \\
	\Tc\fa_{st}\arrow[r,"\Tc f_{st}"]\arrow[rr,bend right, swap, "\Tc(g_{st}\circ f_{st})"] & \Tc\fa_{st}'\arrow[r,"\Tc g_{st}"] & \Tc \fa_{st}''
\end{tikzcd}\end{equation} commutes: e.g.\ $\Tc f_{st} \circ X = X'\circ \Uc f_{tot}$ since $(X,X')$ are $f$-related, and $\Uc(g_{tot}\circ f_{tot}) = \Uc g_{tot} \circ \Uc f_{tot}$ since $\Uc$ is a functor.  Therefore, the outer diagram $$\begin{tikzcd}[column sep = huge] \Uc \fa_{tot}\arrow[r,"\Uc(g_{tot}\circ f_{tot})"] \arrow[d,"X"] & \Uc \fa_{tot}''\arrow[d,"X''"]\\
\Tc\fa_{st}\arrow[r,"\Tc(g_{st}\circ f_{st})"] & \Tc\fa_{st}''	
\end{tikzcd}$$ also commutes, and hence $(X,X'') \in \G_\t(g\circ f)$. 
	 \end{proof}
	\label{remark:generalizingProofOfLaxnessForTauSections}
	 See the proof of \cref{prop:deterministicControlIsLax} for a similar argument.

	 \begin{remark}\label{remark:openSystemsAreCategoryOfRelatedElements}
	We now formally observe that open systems (\cref{def:COpenSystemInDRelativeToA}) and their morphisms (\cref{def:COpenSystemMorphism}) form a category $\fA\fO\fS$, \textit{$\fA$-open ($\t$)-systems}, the category of related elements $\disg_\fA \G_\t$ (\cref{def:categoryOfRelatedElements}, \cref{prop:openSectionsLaxFunctorial}): objects are pairs $(\fa,X)$ where $\fa\in \fA_0$ and $X\in \G_\t(\fa)$. Morphisms $(\fa,X)\rightarrow (\fb,Y)$ are morphisms $\fa\xrightarrow{f}\fb$ in $\fA$ such that $(X,Y)\in \G_\t(f)$ (c.f.\ \eqref{eq:definingOpenSectionsRelationOnMorphisms}).   \end{remark}

We switch gears and turn to interconnection. We have seen used for connecting disparate open systems together.  Abstractly, interconnection is a special morphism of arrows. 
 
	\begin{definition}\label{def:CInterconnectionInA}
		Let  $\nat{\fC}{\fD}{\Tc}{\Uc}{\t}$  a $\fD$-fibered transformation  and $\fA\subseteq \sF{Arrow(C)}$ a subcategory of the arrow category of $\fC$.  We say that a morphism $\big(\fa'\xrightarrow{f}\fa\big)\in \fA_1$ is an \textit{$\fA$-interconnection} (or simply \textit{interconnection}) if $f_{st}:\fa'_{st}\xrightarrow{\sim}\fa_{st}$ is an isomorphism in $\fC$.
	\end{definition}
	
	\begin{remark}\label{remark:CInterconnectionInASubcategory}
		There is subcategory $\fA_{int}$ whose objects are the same as $\fA$ but whose morphisms are $\fA$-interconnections (\cref{def:CInterconnectionInA}). 
	\end{remark}
	
	We saw that $\G_\t:\fA\rightarrow\sF{Rel}$ is a lax functor (\cref{prop:openSectionsLaxFunctorial}).  We can define $\G_\t$ on interconnections differently, so that $\G_\t$ is strictly (contravariantly) functorial. 
	\begin{prop}\label{prop:functorOpenSectionsOnInterconnect}
		There is contravariant functor $\G_\t:\big(\sF{A_{int}}\big)^{op}\rightarrow\sF{Set}$.
	\end{prop}
	Temporarily, we use the same symbol $\G_\t$ for both the strict (contravariant) functor  in \cref{prop:functorOpenSectionsOnInterconnect} and for the lax functor of \cref{prop:openSectionsLaxFunctorial}.  In \cref{prop:doubleFunctorASquare2SetSquare}, we will tie both usages together into a  double functor.
	\begin{proof}[Proof Sketch]
	The definition of $\G_\t$ on objects remains the same (\cref{def:COpenSystemInDRelativeToA}).  We now define $\G_\t$ on interconnection morphisms.  		Let $f:\fa'\rightarrow \fa$ be an $\fA$-interconnection (\cref{def:CInterconnectionInA}), and $X\in \G_\t(\fa)$ an \open $\t$-section (\cref{def:abstractOpenSections}). We define \begin{equation}\label{eq:openSectionsOnInterconnection}\G_\t(f) X\defeq \Tc f_{st}^{-1}\circ X \circ \Uc f_{tot}.\end{equation}
		
		That $X'\defeq \G_\t(f) X\in \G_\t(\fa')$  follows a diagram chase identical to that of diagram \eqref{diagram:relatedVectorFields} (the analog of \eqref{eq:p4} follows by naturality of $\t$).  Functoriality of $\G_\t$ on interconnection is a straightforward generalization of \cref{prop:dCrlInterconnectFunctorial}, where $ \Uc= \Ub $ and $\Tc = T\Ub$. 
	\end{proof}
	
	\begin{remark}\label{remark:compareWithdeterministicControl}
		Modify \cref{ex:HybridOpenSystemsInstanceOfAbstractOpenSystem}, where $\fA = \sF{HySSub}$, and let    $a\xrightarrow{i}b$ be an interconnection of hybrid surjective submersions.  Then the map $\G_\varpi(i)$ \eqref{eq:openSectionsOnInterconnection} is the map $\dCrl(i)$  in \cref{def:interconnectionMapOndeterministicHybridControl}, \eqref{eq:definingdcrlOnInterconnection}. 
		
		A cautionary note: $\dCrl$ is defined identically to $\G_\varpi$ on \textit{morphisms}, but not on objects.  In general, $\dCrl(a) \subsetneq \G_\varpi(a)$ (\cref{ex:detHyOSInstanceOfAbstractOpenSystem}).  Deterministic control are sections of $\varpi$ satisfying some extra conditions (\cref{def:deterministicControlLax}). 
	\end{remark}
	For a subcategory $\fA\subseteq \sF{Arrow(C)}$ of the arrow category of $\fC$, recall that $\fA$ may be interpreted as a double category (\cref{ex:arrowCatAsDoubleCat}). 
		 We next consider an augmented version of $\fA$  with interconnection,  as a double category. 
	 \begin{definition}\label{def:doubleCategoryASquare} 
	 	Let  $\fA\subseteq \sF{Arrow(C)}$ be a subcategory of the arrow category of category $\fC$. We define double category $\fA^\square$ by the following: the object category is $\fA^\square_0= \fA_{int}$ (\cref{remark:CInterconnectionInASubcategory}).  $1$-objects are morphisms in $\fA$, and $1$-morphisms are commuting squares: \begin{equation}\label{diagram:arrowDiagramInAsquare} \begin{tikzcd}\fc\arrow[r,"f"{name=two,below}] & \fd\\
	 	\fc'\arrow[u,swap,"\a"]\arrow[r,"g"{name=one,above}] & \fd'.  \arrow[u,"\b"] \ar[shorten  <= 3pt, shorten >= 3pt, Rightarrow, from = one, to = two]
 \end{tikzcd}\end{equation} 
In $\fC$,  diagram \eqref{diagram:arrowDiagramInAsquare} amounts to the commuting of  diagram \begin{equation}\label{diagram:cubeForASquare}\begin{tikzcd}
 	\fc_{tot}\arrow[rr,"f_{tot}"]\arrow[dr,"p_\fc"] & & \fd_{tot} \arrow[dr,"p_\fd"] & \\
 	& \fc_{st}\arrow[rr,near start,"f_{st}"] & & \fd_{st}\\
 	\fc_{tot}'\arrow[uu,swap,"\a_{tot}"]\arrow[dr,"p_{\fc'}"]\arrow[rr,near start,"g_{tot}"] & & \fd'_{tot}\arrow[uu,near start,swap,"\b_{tot}"]\arrow[dr,"p_{\fd'}"] & \\
 	 & \fc'_{st}\arrow[uu,near start,swap,"\a_{st}"]\arrow[rr,"g_{st}"] & & \fd_{st}',\arrow[uu,swap,"\b_{st}"] 
 \end{tikzcd}\end{equation} with the extra condition that $\a_{st}:\fc_{st}'\xrightarrow{\sim}\fc_{st}$ and $\b_{st}:\fd_{st}'\xrightarrow{\sim}\fd_{st}$ are both isomorphisms. 
 \end{definition}
 
 \begin{example}\label{ex:doubleCategoryHySSubSquare}
 	When $\fA = \sF{HySSub}$, $0$-morphisms of $\fA^\square$ are hybrid interconnections (\cref{def:HybridInterconnection}) and $1$-objects are morphisms of hybrid surjective submersions.  1-morphisms \eqref{diagram:cubeForASquare} are commuting  diagrams.
 \end{example}
 
 \begin{prop}\label{prop:doubleFunctorASquare2SetSquare}
 Abstract \open sections $\G_\t$ (\cref{def:abstractOpenSections}) extend to a  double functor $\G_\t:\fA^\square\rightarrow\sF{Set}^\square$. 
 \end{prop} See 	\cite[Lemma 8.12]{lermanopennetworks}. The diagram chase is identical. 
 \begin{proof} We apply $\G_\t$ to diagram \eqref{diagram:arrowDiagramInAsquare}.   Let $\left(X \in \G_\t(\fc),Y\in \G_\t(\fd)\right)$ be $f$-related (\cref{prop:openSectionsLaxFunctorial}, \eqref{eq:definingOpenSectionsRelationOnMorphisms}) and consider diagram\begin{equation}\label{diagram:cubeForDoubleFunctor}
\begin{tikzcd}
\Uc\fc_{tot}\arrow[dr,"X"]\arrow[rr,"\Uc f_{tot}"] & & \Uc\fd_{tot}\arrow[dr,"Y"] & \\
& \Tc\fc_{st}\arrow[dd,near start,"\Tc\a_{st}^{-1}"]\arrow[rr,near start,"\Tc f_{st}"] & & \Tc \fd_{st}\arrow[dd,near start,"\Tc\b_{st}^{-1}"]\\
\Uc\fc_{tot}'\arrow[dr,"X'"] \arrow[rr,near start,"\Uc g_{tot}"]\arrow[uu,near start,"\Uc \a_{tot}"] & & \Uc\fd_{tot}'\arrow[dr,"Y'"]\arrow[uu,near start,swap,"\Uc\b_{tot}"] & \\
& \Tc \fc_{st}'\arrow[rr,"\Tc g_{st}"] & & \Tc\fd_{st}'.
\end{tikzcd}
 \end{equation}  In diagram \eqref{diagram:cubeForDoubleFunctor}, $X'\defeq \G_\t(\a)X$ and $Y'\defeq \G_\t(\b)Y$ (c.f.\ \eqref{eq:openSectionsOnInterconnection}).
 We must show that $(X',Y')$ are $g$-related, i.e.\ that $\Tc g_{st}\circ X' = Y'\circ \Uc g_{tot}$. We compute:
 $$\begin{array}{lll}\Tc g_{st}\circ X' & = \Tc g_{st}\circ \Tc \a_{st}^{-1}\circ  X\circ \Uc\a_{tot} & (\mbox{by \eqref{eq:openSectionsOnInterconnection}})\\
 & = \Tc\b_{st}^{-1}\circ \Tc f_{st}\circ X \circ \Uc \a_{tot} & (\mbox{by \eqref{diagram:cubeForASquare} and functoriality of $\Tc$})\\
 & = \Tc\b_{st}^{-1}\circ Y \circ \Uc f_{tot} \circ \Uc \a_{tot} & (\mbox{since}\, (Y,X)\in \G_\t(f))\\
 & = \Tc\b_{st}^{-1}\circ Y \circ \Uc\b_{tot} \circ \Uc g_{tot} & (\mbox{by \eqref{diagram:cubeForASquare} and functoriality of $\Uc$}) \\
 & = Y'\circ \Uc g_{tot} & (\mbox{by \eqref{eq:openSectionsOnInterconnection}})
 \end{array}$$
 proving that $(X',Y')$ are $g$-related. 
 \end{proof}

 \section{Monoidally Fibered Transformations and Networks}\label{subsection:monoidallyfibered}
 In \cref{subsection:abstractSections}, we worked out the notion of abstract section, open system, and interconnection. Interconnection---a class of morphisms of arrows which is isomorphism on state---by itself is an apparently unmotivated notion.  For us, interconnection is the way we interconnect a collection of subsystems into one, and hence build networks.  There are two steps to this process.  We start with a collection of spaces and combine them with some sort of product.  Then we interconnect, by defining some interconnection map to the product. Both steps together comprise our formalism of networks. 
 
 Having established the requisite theory for abstract systems and interconnection, we now turn to working out what we need for products.     We assume all material from \cref{subsection:fiberedCategories} and briefly identify relevant notions from \cref{subsection:MonoidalCat}. Start with  $\fD$-fibered natural transformation $\nat{\fC}{\fD}{\Tc}{\Uc}{\t}$ (\cref{def:fiberedTransformationInD}).  We assume that both $\fC$ and $\fD$ are   monoidal categories (\cref{def:monoidalCat}), and that $\fD$ is cartesian  (\cref{def:monoidalProductCartesian}). Let $\fA\subseteq \sF{Arrow(C)}$ be a subcategory of the arrow category of $\fC$,   and $\fA^\square$ the  double category in \cref{def:doubleCategoryASquare}. Consider $\fX$-indexed collection $\big\{\Ac_\fx\big\}_{\fx\in \fX}$ of $\fA$-objects and recall (c.f.\ \eqref{eq:productFinSet})    the product assignment \begin{equation}\label{eq:productFinSet2} \Pi(\Ac_\fX)\defeq \disbop_{\fx\in \fX}\Ac_\fx. \end{equation} Suppose, moreover,  that $\Pi:\big(\sF{FinSet/A}^\Leftarrow\big)^{op}\rightarrow\fA$ is functorial.   When $\fA$ is cartesian, for example, this assignment is functorial (\cref{prop:productIsFunctorialForCartesianMonoidal}). Finally, for the arrow category $\fA_\fD \defeq \sF{Arrow(D)}$ of $\fD$, functors $\Tc,\,\Uc:\fC\rightarrow\fD$ extend to functors $\Tc_*,\,\Uc_*:\fA\rightarrow\fA_\fD$ which are also strong monoidal (\cref{prop:arrowMonoidalProductPreserving}).
 
 \begin{notation}\label{notation:listOfAobjects}
   For $\fX$-indexed $\fA$-objects $\Ac_\fX:\fX\rightarrow\fA$ or $\big\{\Ac_\fx\big\}_{\fx\in \fX}$, we denote $\Ac_\fx \defeq \big(\Ac_{\fx,tot}\xrightarrow{p_\fx} \Ac_{\fx,st}\big)$. Therefore, $$\Pi(\Ac_\fX) =\left(\Pi\big(\Ac_{\fX,tot}\big)\xrightarrow{\Pi(p_\fx)}\Pi\big(\Ac_{\fX,st}\big)\right).$$
\end{notation}

\begin{example}\label{ex:manifoldsAndOpenSystemsAsSatisfyingAssumptions}
Recall \cref{ex:identityFunctorForfiberedCategories} and \cref{ex:ordinaryOpenSystemsInstanceOfAbstract} where $\fC = \fD = \sF{Man}$ are the category of manifolds, $\Tc = T$ is the tangent endofunctor, $\Uc =id_{\sF{Man}}$ is the identity functor,  $\t:\Tc\Rightarrow\Uc$ is the canonical projection of the tangent bundle, and $\fA = \sF{SSub}$ is the category of surjective submersions.  The category of manifolds is cartesian  (\cref{fact:productManifolds}, \cref{prop:productManifoldsCategorical}), as is $\sF{SSub}$ (\cite[\S4]{lermanopennetworks}), and the tangent functor $\Tc$ is strong monoidal (\cref{prop:isoProductTangentSpaces}). 
\end{example}

We now restate and isolate  our assumptions for later reference:
\begin{assumption}\label{assumption:3}
 Categories  $(\fA,\otimes_\fA,1_\fA)$ and  	$(\fD,\otimes_\fD,1_\fD)$ are monoidal and cartesian monoidal, respectively (\cref{def:monoidalProductCartesian}), and the product assignment $\Pi:\big(\sF{FinSet/A}^\Leftarrow\big)^{op}\rightarrow\fA$ is functorial. \end{assumption}
\begin{remark}\label{remark:showAFunctorial}
We will show functoriality of $\Pi$ in our examples by observing that $(\fA,\otimes_\fA,1_\fA)$ is itself cartesian.  Cartesianality of $\fA$ is used only for functoriality while we need cartesianality of $\fD$ in order to ensure that agreement of ``dynamics on components''  is sufficient for network coherence (\cref{prop:openSectionsProductRelatedAreRelated}). \end{remark}

\begin{assumption}\label{assumption:2}
	Monoidal functors $\Uc,\Tc:(\fC,\otimes_\fC,1_\fC)\rightarrow(\fD,\otimes_\fD,1_\fD)$ are \textit{strong monoidal} (\cref{def:strongMonoidalFunctor}). 
\end{assumption}

We require one more condition on the natural transformation  $\t:\Tc\Rightarrow\Uc$, that $\t$ respects monoidal products with  both monoidal functors $\Tc$ and $\Uc$:

\begin{assumption}\label{assumption:1}
	The $\fD$-fibered (\cref{def:fiberedTransformationInD}) natural transformation $\nat{\fC}{\fD}{\Tc}{\Uc}{\t}$ is   monoidal  (\cref{def:monoidalTransformation}).
\end{assumption} We recall   monoidality of natural transformation $\t:\Tc\Rightarrow \Uc$  means that for  natural  transformations (in this case, isomorphisms) $\eta^\Xc: \Xc(\cdot)\otimes_\fD(\cdot) \Rightarrow\Xc((\cdot)\otimes_\fC(\cdot))$ of monoidal functors  $\Xc=\Tc,  \Uc$ (\cref{def:monoidalFunctor}), and  for every pair of objects $\fc,\fc'\in \fC_0$, we have commuting diagram $$\begin{tikzcd}
	\Tc\fc\otimes_\fD\Tc\fc\arrow[r,"\eta^\Tc_{\fc{,}\fc'}"] \arrow[d,"\t_\fc\otimes\t_{\fc'}"] & \Tc(\fc\otimes_\fC\fc)\arrow[d,"\t_{\fc\otimes\fc'}"] \\
	\Uc\fc\otimes_\fD\Uc\fc'\arrow[r,"\eta^{\Uc}_{\fc{,}\fc'}"] & \Uc(\fc\otimes_\fC\fc').
\end{tikzcd}$$

\begin{remark}\label{remark:redundancyOfAssumption3}
	We observe that monoidality of transformation $\t:\Tc\Rightarrow \Uc$ is a consequence of cartesianality 
	of monoidal category $\fD$ and strong  monoidality of functors $\Tc,\Uc:\fC\rightarrow \fD$.   Indeed, let $\g_{\fc,\fc'}^\Xc:\Xc(\fc\times \fc')\xrightarrow{\sim}\Xc\fc\times \Xc\fc'$ denote the natural inverse of $\eta^{\Xc}_{\fc,\fc'}$ for $\Xc = \Tc,\Uc$.  The following diagrams commute: \small$$\begin{array}{ccc}\begin{tikzcd}
	\Tc(\fc\times\fc')\arrow[r]\arrow[d,"\t_{\fc\times\fc'}"]  & \Tc\fc\times \Tc\fc'\arrow[d,"\t_\fc\times\t_{\fc'}"]\\ \Uc(\fc\times\fc')\arrow[r] & \Uc\fc\times\Uc\fc'
\end{tikzcd}& \mbox{and} & \begin{tikzcd}[column sep = large,row sep = large]
 	\Tc(\fc\times \fc')\arrow[rr,bend left,swap,"id_{\Tc(\fc\times \fc')}"] \arrow[r,swap,"\g^\Tc_{\fc{,}\fc'}"] \arrow[d,"\t_{\fc\times\fc'}"] & \Tc\fc\times \Tc\fc'\arrow[r,swap,"\eta^\Tc_{\fc{,}\fc'}"]\arrow[d,"\t_\fc\times\t_{\fc'}"] & \Tc(\fc\times \fc')\arrow[d,"\t_{\fc\times \fc'}"] \\
 	\Uc(\fc\times\fc')\arrow[rr,bend right, "id_{\Uc(\fc\times\fc')}"] \arrow[r,"\g_{\fc{,}\fc'}^\Uc"] & \Uc\fc\times\Uc\fc' \arrow[r,"\eta^\Uc_{\fc{,}\fc'}"] & \Uc(\fc\times \fc').
 \end{tikzcd}\end{array}$$\normalsize
 The first diagram commutes by the universal property of product (\cref{def:products}) and naturality of $\t$ (applied to projection maps $p_\fc:\fc\times\fc\rightarrow\fc$ and $p_{\fc'}:\fc\times\fc'\rightarrow\fc'$).  This diagram appears as the left-hand square in the diagram on the right, and the outer diagram obviously commutes. We readily conclude that the inner diagram on the right does as well.  

\end{remark}
\begin{remark}
	Very loosely speaking, we may unify monoidal assumptions \ref{assumption:3}, \ref{assumption:2}, and \ref{assumption:1} as follows: the categories are cartesian monoidal, the functors are (strong) monoidal, and the natural transformation is monoidal. We noted in \cref{remark:showAFunctorial} that $\fA$ need not be cartesian, though it is in the examples we consider. 
\end{remark}

Having stated a slew of assumptions, we reflect on where we are going.  We want to build a theory of  abstract networks of open systems.  As in our previous examples, an abstract notion of network is something like `product + interconnection'.  The product is over a collection of $\fA$-objects, and  interconnection is a map into this product.  Use of a category $\fA\subseteq \sF{Arrow(C)}$ is the ``open'' part (e.g.\ \cref{def:abstractOpenSystem}) and functoriality of $\Pi$ ensures that we can make networks into a category. Now for the precise definition.

\begin{definition}\label{def:abstractNetworks} Fix cartesian subcategory $\fA\subseteq \sF{Arrow(C)}$. 
We define an \textit{abstract network of $\fA$-open $\t$-systems} (or simply: \textit{network of open systems}) to be a pair $$\left(\big\{\Ac_\fx\big\}_{\fx\in \fX}, \i_\fX:\fa\rightarrow \discatp_{\fx\in \fX} \Ac_\fx\right),$$ where $\Ac_\fX:\fX\rightarrow \fA$ is a finite indexed collection of $\fA$-objects (\cref{def:categoryOfLists}) and $\i_\fX:\fa\rightarrow \Pi(\Ac_\fX)$ is an interconnection morphism (\cref{def:CInterconnectionInA}, eq.\ \eqref{eq:productFinSet2}).   We will write $(\Ac_\fX, \i_\fX:\fa\rightarrow \Pi(\Ac_\fX))$ or just $(\Ac_\fX,\i_\fX)$ as shorthand for such a network (and $\sF{dom}(\i)$ for the source of morphism $\i_\fX$ if it is not otherwise specified).   
	\end{definition}

	\begin{remark}\label{remark:whyCallItNetworkOfSYSTEMS?}
	We previously defined $\fA$-open \textit{system} to be a pair $(\fa,X)$ where $\fa\defeq (\fa_{tot}\xrightarrow{p_\fa}\fa_{st}) \in \fA\subseteq \sF{Arrow(C)}$ and $X\in \G_\t(\fa)$ (\cref{def:COpenSystemInDRelativeToA}), while no sections appear in the definition of \textit{networks} of open systems.  We will see that the  ``systems'' part of \cref{def:abstractNetworks} comes from applying  \open sections functor  $\G_\t$ to some 1-morphism in $\fA^\square$ involving the interconnection $\i_\fX:\fa\rightarrow\Pi(\Ac_\fX)$.
\end{remark}

There is also a notion of morphisms  of networks. 
\begin{definition}\label{def:abstractNetworksMorphism}
	Let $(\Ac_\fX,\i_\fX:\fa\rightarrow \Pi(\Ac_\fX))$ and $(\Ac_\fY,\i_\fY:\fa'\rightarrow\Pi(\Ac_\fY))$ be two abstract networks of $\fA$-open $\t$-systems (\cref{def:abstractNetworks}).  We define a  \textit{morphism} $\big((\ph,\Phi),f\big):(\Ac_\fX,\i_\fX)\rightarrow (\Ac_\fY,\i_\fY)$ \textit{of  networks} to be a  morphism $\begin{tikzcd}[column sep = large]\fX\arrow[r,"\ph"]\arrow[dr,swap,"\Ac_\fX",""{name = foo, above}] & \fY\arrow[d,"\Ac_\fY"]\ar[shorten <= 5pt, Rightarrow, to = foo,"\Phi"] \\
  & \fA
\end{tikzcd}$ of $\fX$-indexed $\fA$-objects (i.e.\ in the category $(\sF{FinSet/A})^\Leftarrow$ (\cref{def:categoryOfLists})) and a morphism $f:\fa'\rightarrow\fa $ in $\fA$ such that $\begin{tikzcd}[column sep = large, row sep = large]
	\Pi\big(\Ac_\fY\big)\arrow[r,swap,"\Pi(\ph{,}\Phi)",""{name = foo1,below}] &\Pi\big(\Ac_\fX\big)\\
	\fa'\arrow[u,swap,"\i_\fY"]\arrow[r,"f",""{name = foo2,above}] & \fa \arrow[u,"\i_\fX"]\ar[shorten <= 15pt, shorten >= 15pt, Rightarrow, from = foo2, to = foo1]
\end{tikzcd}$ is a 1-morphisms in $\fA^\square$. 
\end{definition}

\begin{remark}\label{remark:networksAreACategory}
	It is easy to see that networks of $\fA$-open $\t$-systems (\cref{def:abstractNetworks}) and morphisms of networks  (\cref{def:abstractNetworksMorphism}) form a category.  The verification is formally very similar to \cref{lemma:HyPhisaCategory}. 
\end{remark}
The next result partly explains the motivation behind the name system (\cref{remark:whyCallItNetworkOfSYSTEMS?}).  A morphism of networks induces a 1-morphism in $\sF{Set}^\square$ (\cref{prop:doubleFunctorASquare2SetSquare}):
\begin{prop}\label{prop:halfOfMainTheorem1stHalf}
	Let $\big((\ph,\Phi),f\big):(\Ac_\fX,\i_\fX:\fa\rightarrow\Pi(\Ac_\fX))\rightarrow (\Ac_\fY,\i_\fY:\fa'\rightarrow\Pi(\Ac_\fY))$ be a morphism of abstract networks.  Then there is 1-morphism in $\sF{Set}^\square$ \begin{equation}
		\begin{tikzcd}[column sep = large, row sep = large]
			\G_\t\left(\Pi\big(\Ac_\fY\big)\right)\arrow[r,swap,"\G_\t(\Pi(\ph{,}\Phi))"{name = foo4},""{name = foo1, below}]\arrow[d,"\G_\t(\i_\fY)"] & \G_\t\left(\Pi\big(\Ac_\fX\big)\right) \arrow[d,swap,"\G_\t(\i_\fX)"] \\
			\G_\t(\fa')\arrow[r,"\G_\t(f)"{name = foo3},""{name = foo2, above}] & \G_\t(\fa) \ar[shorten <= 5pt, shorten >= 5pt, Rightarrow, from = foo4, to = foo3]
		\end{tikzcd}
	\end{equation}
	\end{prop}
\begin{proof}
	This follows directly from \cref{def:abstractNetworksMorphism} and \cref{prop:doubleFunctorASquare2SetSquare}. 
\end{proof}
Proposition \ref{prop:halfOfMainTheorem1stHalf} brings to light the significance of interconnection in our notion of network.  Notice that the interconnection induces a map of open $\t$-sections on a product to open $\t$-sections of another space.  We still need some way to take a \textit{collection} of open $\t$-sections of a collection of spaces  to an open $\t$-section of the product.  The next result gives us exactly that.

\begin{lemma}\label{lemma:strongMonoidalInducesMonoidalFunctor}
Suppose \cref{assumption:2} and \cref{assumption:1} hold, namely  that  $\fD$-fibered (\cref{def:fiberedTransformationInD}) transformation
$\nat{\fC}{\fD}{\Tc}{\Uc}{\t}$ is  monoidal  (\cref{def:monoidalTransformation}) and both $\Tc$ and $\Uc$ are  strong monoidal (\cref{def:strongMonoidalFunctor}).    Let $\fA\subseteq \sF{Arrow(C)}$  and $\Ac_\fX:\fX\rightarrow \fA$ be a list of $\fA$-objects. Then there is a map $\scP_\fX:\discatp_{\fx\in \fX}\G_\t(\Ac_\fx) \rightarrow\G_\t\left(\Pi\big(\Ac_\fX\big)\right)$. 
\end{lemma}
\begin{proof}
Let \begin{equation}\label{diagrams:naturalIsosForStrongMonoidal} \nat{\fC\times \fC}{\fD}{\Tc(\cdot) \otimes_\fD \Tc(\cdot)}{\Tc\big((\cdot)\otimes_\fC(\cdot)\big)}{\eta}\;\;\mbox{and}\;\;\nat{{\fC\times \fC}}{\fD}{\Uc\big((\cdot)\otimes_\fC(\cdot)\big)}{\Uc(\cdot)\otimes_\fD \Uc(\cdot)}{\g}\end{equation} be the natural isomorphisms of \cref{def:strongMonoidalFunctor} (also recall \cref{fact:inverseOfNaturalIsoIsNatural}).  These extend naturally to  transformations \begin{equation}\label{eq:natTransformForMonoidalProductsGeneral} \begin{array}{lll}\g_\fX:\Uc\big(\bigotimes_{\fx\in \fX}(\cdot)\big)\Rightarrow \bigotimes_{\fx\in \fX}\Uc(\cdot) & \mbox{and} & \eta_\fX: \bigotimes_{\fx\in \fX}\Tc(\cdot)\Rightarrow	\Tc \big(\bigotimes_{\fx\in \fX}(\cdot)\big)\end{array}\end{equation} for functors $\fC^\fX\rightarrow\fD$. Let $(v_\fx)_{\fx\in \fX}\in \discatp_{\fx\in \fX}\G_\t(\Ac_\fx)$ be an $\fX$-indexed collection of \open sections. By assumption, there is isomorphism $$ \Uc\left(\disbop_{\fx\in \fX}\Ac_{\fx,tot}\right) \xrightarrow{\g_{\fX,tot}} \disbop_{\fx\in \fX}\Uc\big(\Ac_{\fx,tot}\big).$$ Since $\otimes_\fD$ is a functor, there is also map \begin{equation}\label{eq:monoidalProductOpenSectionsV_x} V_\fX\defeq \left(\disbop_{\fx\in \fX}\Uc\big(\Ac_{\fx,tot}\big)\xrightarrow{ \disbop_{\fx\in \fX}v_\fx} \disbop_{\fx\in \fX}\Tc\big(\Ac_{\fx,st}\big)\right).\end{equation} Finally, again there is isomorphism $$ \disbop_{\fx\in \fX}\Tc\big(\Ac_{\fx,st}\big)\xrightarrow{\eta_{\fX,st}}\Tc\left(\disbop_{\fx\in \fX}\Ac_{\fx,st}\right).$$
	
	We thus define the map $$\scP_\fX\left((v_\fx)_{\fx\in \fX}\right):\Uc\left( \Pi\big(\Ac_{\fx,tot}\big)\right)\rightarrow\Tc\left(\Pi\big(\Ac_{\fX,st}\big)\right)$$ by  \begin{equation}\label{eq:defineSCP} \scP_\fX\left((v_\fx)_{\fx\in \fX}\right)\defeq \eta_{\fX,st}\circ V_\fX\circ \g_{\fX,tot}.\end{equation}
	
	For notational convenience, set $v_\scP\defeq \scP_\fX\big((v_\fx)_{\fx\in \fX}\big)$.  	Verification that $v_\scP\in \G_\t\big(\Pi(\Ac_\fX)\big)$---i.e.\ that $\t_{\Pi(\Ac_\fX)_{st}}\circ v_\scP  = \Uc  p_{\Pi(\Ac_\fx)}$---uses the following diagram: \small  $$\begin{tikzcd}[column sep = large, row sep = large]
	\Uc \left( \disbop_{\fx\in \fX} \Ac_\fX(\fx)_{tot}\right)\arrow[dd,swap,"\Uc p_{\Pi(\Ac_\fX)}"]\arrow[dr,near end,"v_\scP"]\arrow[rr,"\g_{\fX,tot}"] & & \disbop_{\fx\in \fX}\Uc \Ac_{\fx,tot} \arrow[dd,near start,swap,"\bigotimes_{\fx\in \fX}\Uc p_\fx"]\arrow[rd,"\bigotimes_{\fx\in \fX}v_\fx"] & \\
	& \Tc\left(\disbop_{\fx\in \fX}\Ac_{\fx,st}\right)\arrow[rr,near start,shift right,swap,"\eta_{\fX,st}^{-1}"]\arrow[dl,"\t_{\Pi(\Ac_{\fX,st})}"] & & \disbop_{\fx\in \fX}\Tc\Ac_{\fx,st}\arrow[dl,"\bigotimes_{\fx\in \fX}\t_{\fx,st}"]\arrow[ll,near start,shift right,swap,"\eta_{\fX,st}"] \\
	\Uc\left(\disbop_{\fx\in \fX}\Ac_{\fx,st}\right) & & \disbop_{\fx\in \fX}\Uc\Ac_{\fx,st},\arrow[ll,"\g_{\fX,st}^{-1}"] & 
\end{tikzcd}
	$$\normalsize where $$\t_{\fx,st}:\Tc \Ac_{\fx,st}\rightarrow \Uc\Ac_{\fx,st}$$ is the epimorphism from natural transformation $\t:\Tc\Rightarrow\Uc$ at object $\Ac_{\fx,st}$, and $$ p_\fx:\Ac_{\fx,tot}\rightarrow \Ac_{\fx,st}$$ is the object $\Ac_\fx$ of $\fA$ (\cref{notation:listOfAobjects}). 
	
The following subdiagrams commute:  $$\begin{array}{lll}
	v_\scP& = \eta_{\fX,st}\circ V_\fX \circ \g_{\fX,tot} & (\mbox{definition of}\,  \scP, \, \eqref{eq:defineSCP})\\
	\bigotimes_{\fx\in \fX}\Uc p_\fx & = 	\bigotimes_{\fx\in \fX} \t_{\fx,st} \circ 	V_\fX & (v_\fx \in \G_\t\left(\Ac_\fX(\fx)\right)\, \forall\, \fx, \eqref{eq:monoidalProductOpenSectionsV_x}, \mbox{and $\otimes$ is a functor})\\
	\g_{\fX,st}\circ \Uc p_{\Pi(\Ac_\fX)} & = 	\bigotimes_{\fx\in \fX} \Uc p_\fx\circ \g_{\fX,tot} & (\Uc \, \mbox{is strong monoidal, and $\g$ is natural})\\
	\g_{\fX,st}^{-1} \circ 	\bigotimes_{\fx\in \fX}\t_{\fx,st} & = \t_{\Pi(\Ac_\fX)} \circ \eta_{\fX,st} & (\t\, \mbox{is monoidal transformation}, \\
	 &  & \mbox{c.f.\ \cref{def:monoidalTransformation} and \cref{assumption:1}}). 
\end{array}$$
Thus, the final triangle diagram commutes: $\Uc p_{\Pi(\Ac_\fX)} = \t_{\Pi(\Ac_\fX)_{st}} \circ \scP_\fX\big((v_\fx)_{\fx\in \fX}\big)$, and hence proves that   $$\scP_\fX:\discatp_{\fx\in \fX}\G_\t\big(\Ac_\fx\big) \rightarrow\G_\t\left(\bigotimes_{\fx\in \fX}\Ac_\fx\right)$$ is well defined. 
\end{proof}

The next proposition is the last piece in our puzzle, the ``second half of our main theorem'', which says that individually related open $\t$-sections assemble to an open $\t$-section on the product.

\begin{prop}\label{prop:openSectionsProductRelatedAreRelated}
	Let $(\ph,\Phi):(\Ac_\fX:\fX\rightarrow \fA)\rightarrow(\Ac_\fY:\fY\rightarrow\fA)$ be a morphism of lists (\cref{def:categoryOfLists}) and suppose that $(w_\fy)_{\fy\in \fY}\in \discatp_{\fy\in \fY}\G_\t(\Ac_\fY(\fy))$ and $(v_\fx)_{\fx\in \fX}\in \discatp_{\fx\in \fX}\G_\t(\Ac_\fX(\fx))$ are two collections of open sections (\cref{notation:OpenSectionsEtc.}) with the following property: $(w_{\ph(\fx)},v_\fx)$ are $\Phi_\fx$-related for each $\fx\in \fX$ (\cref{def:relatedOpenSections}). Suppose, further, that assumptions \ref{assumption:3}, \ref{assumption:2}, and \ref{assumption:1}   hold. Then $$\left(\scP_\fY\big((w_\fy)_{\fy\in \fY}\big),\scP_\fX\big((v_\fx)_{\fx\in \fX}\big)\right)$$ are $\Pi(\ph,\Phi)$-related. 
\end{prop}

We express \cref{prop:openSectionsProductRelatedAreRelated} differently.  First a definition. 
\begin{definition}\label{def:productRelationOfOpenSections}
Let $(\ph,\Phi):(\Ac_\fX:\fX\rightarrow\fA)\rightarrow(\Ac_\fY:\fY\rightarrow\fA)$ be a morphism of lists. 
	We define the relation $\G_\t(\ph,\Phi)\subseteq \discatp_{\fy\in \fY}\G_\t(\Ac_\fy)\times \discatp_{\fx\in \fX}\G_\t(\Ac_\fx)$ by $$\small \G_\t(\ph,\Phi)\defeq \left\{\big((w_\fy)_{\fy\in \fY},(v_\fx)_{\fx\in \fX}\big)\in \discatp_{\fy\in \fY}\G_\t(\Ac_\fy)\times \discatp_{\fx\in \fX}\G_\t(\Ac_\fx):\, (w_{\ph(\fx)},v_\fx)\in  \G_\t(\Phi_\fx)\;\forall \;\fx\in \fX \right\}.$$\normalsize Recall that $(w_{\ph(\fx)},v_\fx)\in \G_\t(\Phi_\fx)$ means that the pair of $\t$-\open sections $(w_{\ph(\fx)},v_\fx)$ are $\Phi_\fx$-related (\cref{def:relatedOpenSections}).  Contrast $\G_\t(\ph,\Phi)$ with $\G_\t(\Pi(\ph,\Phi))$, the latter of which is the relation $$\G_\t(\Pi(\ph,\Phi)) =\left\{\big(W,V\big)\in  \G_\t(\Pi(\Ac_\fY))
	\times\G_\t(\Pi(\Ac_\fX)):\,(W,V)\,\mbox{are}\, \Pi(\ph,\Phi)\mbox{-related}\right\}.$$ 
\end{definition} 
Then 
\cref{prop:openSectionsProductRelatedAreRelated} says there is $1$-morphism \begin{equation}\label{diagram:2CellProductRelatedRelatedInProduct}\begin{tikzcd} [column sep = large, row sep = large]
	\discatp_{\fy\in \fY}\G_\t(\Ac_\fy)\arrow[d,"\scP_\fY"] \arrow[r,"\G_\t(\ph{,}\Phi)"{name = foo1, below}] & \discatp_{\fx\in \fX}\G_\t(\Ac_\fx) \arrow[d,swap,"\scP_\fX"] \\
	\G_\t\left(\disbop_{\fy\in \fY}\Ac_\fy\right)\arrow[r,"\G_\t(\Pi(\ph{,}\Phi))"{name = foo2, above}] & \G_\t\left(\disbop_{\fx\in \fX}\Ac_\fx\right).\ar[shorten <=15pt,shorten >=15pt,Rightarrow,from=foo1,to=foo2]\end{tikzcd}
\end{equation} in $\sF{Set}^\square$.

We introduce  notation.  For indexed collection of  $\Ac_\fX:\fX\rightarrow \fA$ of $\fA$-objects and functor $\Uc:\fC\rightarrow\fD$, there is indexed collection of  $\sF{Arrow(D)}$-objects $\Ac_\fX^\Uc:\fX\rightarrow \sF{Arrow(D)}$ defined by post-composition  $\Ac_\fX^\Uc\defeq \Uc_*\Ac_\fX$ (c.f.\ \eqref{eq:extendingFunctorToFunctorOnArrows}).  Thus, for morphism 	$(\ph,\Phi):(\Ac_\fX^\Uc:\fX\rightarrow \sF{Arrow(D)})\rightarrow (\Ac_\fY^\Uc:\fX\rightarrow \sF{Arrow(D)})$ of indexed $\sF{Arrow(D)}$-objects (\cref{def:categoryOfLists}), there is morphism $\Pi_\Uc(\ph,\Phi):\Pi(\Ac_\fY^\Uc)\rightarrow \Pi(\Ac_\fX^\Uc)$  since $(\fD,\otimes_\fD,1_\fD)$ is cartesian (c.f.\ \cref{prop:NietzscheProp}). Here $\Pi(\Ac_\fX^\Uc)\defeq \discatp_{\fx\in \fX}\Uc\big(\Ac_{\fx}\big)$. 
\begin{proof}
	We must show that $\Tc\big(\Pi(\ph,\Phi)_{st}\big)\circ w_\scP = v_\scP\circ \Uc \big(\Pi(\ph,\Phi)_{tot}\big)$, where $v_\scP\defeq \scP_\fX\left((v_\fx)_{\fx\in \fX}\right)$, $w_\scP\defeq \scP_\fY\left((w_\fy)_{\fy\in \fY}\right) $ (\cref{lemma:strongMonoidalInducesMonoidalFunctor}). We consider the following diagram \begin{equation}\label{diagram:provingRelatednessOfProductOfOpenSections}
 	\begin{tikzcd}[column sep = large, row sep = large]
 		\Uc\big(\Pi(\Ac_\fY)_{tot}\big)\arrow[rr,"\xi_1"]\arrow[dr,"\xi_3"]\arrow[dd,"\xi_2"] & & \Uc\big(\Pi(\Ac_\fX)_{tot}\big)\arrow[dd,near start,"\xi_4"]\arrow[dr,"\xi_5"] & \\
 		& \Pi(\Ac_\fY^\Uc)_{tot}\arrow[dd,near start,"\xi_6"]\arrow[rr,near start,"\xi_7"] & & \Pi(\Ac_\fX^\Uc)_{tot}\arrow[dd,"\xi_8"]\\
 		\Tc\big(\Pi(\Ac_\fY)_{st}\big) \arrow[rr,near start,"\xi_9"] & & \Tc\big(\Pi(\Ac_\fX)_{st}\big)  &\\
 		& \Pi(\Ac^\Tc_\fY)_{st}\arrow[rr,"\xi_{11}"]\arrow[ul,"\xi_{10}"] & & \Pi(\Ac^\Tc_\fX)_{st},\arrow[ul,"\xi_{12}"]
 	\end{tikzcd}
 \end{equation}
 where---recalling that $\g$ and $\eta$ are the natural isomorphisms from \eqref{eq:natTransformForMonoidalProductsGeneral}  of strong monoidal functors $\Tc$ and $\Uc$---we have \small 
 $$\begin{array}{ll|ll|ll|ll} \xi_1 & = \Uc\Pi(\ph,\Phi)_{tot} &\xi_4 & = v_\scP&\xi_7 & = \Pi_\Uc(\ph,\Phi)_{tot}   &   \xi_{10} & = \eta_{\fY,st}\\ 
 \xi_2 & = w_\scP &   \xi_5 & = \g_{\fX,tot}& \xi_8 & =  \discatp_{\fx\in \fX}v_\fx   & \xi_{11} &  =\Pi_\Tc(\ph,\Phi)_{st} \\
\xi_3 & = \g_{\fY,tot} & \xi_{6} & = \discatp_{\fy\in \fY}w_\fy  &
  \xi_9 & = \Tc\Pi(\ph,\Phi)_{st}&  \xi_{12}& = \eta_{\fX,{st}}. \\
\end{array}$$\normalsize 
Equalities \begin{equation}\label{eq:fzs1}
 \xi_2 = \xi_{10}\circ \xi_6\circ \xi_3 \;\mbox{and}\; \xi_4 = \xi_{12}\circ\xi_8\circ \xi_5	
 \end{equation} follow by definition of $\scP_\fY$ and $\scP_\fX$ (\cref{lemma:strongMonoidalInducesMonoidalFunctor}).  Equalities \begin{equation}\label{eq:fzs2}
 \xi_7\circ \xi_3=\xi_5\circ \xi_1 \; \mbox{and} \;	\xi_9\circ\xi_{10} = \xi_{12}\circ \xi_{11}
 \end{equation} follow by naturality of $\g$ and $\eta$ (\cref{def:strongMonoidalFunctor}, \eqref{diagrams:naturalIsosForStrongMonoidal}). Finally, \begin{equation}\label{eq:fzs3}
 	\xi_{11}\circ\xi_6 = \xi_8\circ \xi_7
 \end{equation} follows  by assumption that $(w_{\ph(\fx)},v_\fx)$ are $\Phi_\fx$-related and that $\otimes_\fD$ is  cartesian.  In more detail,  consider diagram \small$$\begin{tikzcd}
 \discatp_{\fy\in \fY}\Uc\Ac_{\fy,tot}\arrow[rr,dashed,"\z_1"]\arrow[dd,"\z_2"]\arrow[dr,"\z_3"] & & \discatp_{\fx\in \fX}\Uc\Ac_{\fx,tot}\arrow[dd,near start,"\z_4"]\arrow[dr,"\z_5"] & \\
 & \Uc\Ac_{\ph(x'),tot}\arrow[dd,near start,"\z_6"]\arrow[rr,near start,"\xi_7"] & & \Uc\Ac_{\fx',tot}\arrow[dd,"\z_8"] \\
 \discatp_{\fy\in \fY}\Tc\Ac_{\fy,st}\arrow[rr,dashed,near start,"\z_9"]\arrow[dr,"\z_{10}"] & & \discatp_{\fx\in \fX}\Tc\Ac_\fX(\fx)_{st}\arrow[dr,"\z_{11}"] & \\
 & \Tc\Ac_{\ph(\fx'),st}\arrow[rr,"\z_{12}"] & & \Tc\Ac_{\fx',st},
\end{tikzcd}$$ \normalsize
where 
 \small 
 $$\begin{array}{ll|ll|ll|ll} \z_1 & = \Pi_\Uc(\ph,\Phi)_{tot} (=\xi_7) &\z_4 & = \bigsqcap_{\fx\in \fX}v_\fx &\z_7 & = \Uc\Phi_{\fx',tot}   &   \z_{10} & =p_{\Tc\Ac_{\ph(\fx'),{st}}}\\ 
 \z_2 & = \bigsqcap_{\fy\in\fY}w_\fy &   \z_5 & = p_{\Uc\Ac_{\fx',{tot}}}& \z_8 & = v_{\fx'}  & \z_{11} &  =  p_{\Tc\Ac_{\fx',st}} \\
\z_3 & = p_{\Uc\Ac_{\ph(\fx'),tot}} & \z_{6} & = w_{\ph(\fx')}  &
  \z_9 & = \Pi_\Tc(\ph,\Phi) (=\xi_{11})&  \z_{12}& = \Tc\Phi_{\fx',st}. \\
\end{array}$$\normalsize 
Then $\Phi_\fx$-relatedness of $(w_{\ph(\fx')},v_{\fx'})$ for every $\fx'\in \fX$ means that $\z_{12}\circ \z_6 = \z_8\circ \z_7$.  Cartesianality of $\otimes_\fD$  implies that the back face also commutes: in other words $\z_9\circ\z_2 = \z_4 \circ \z_1$, or $\xi_{11}\circ \xi_6= \xi_8\circ \xi_7$ (eq.\ \eqref{eq:fzs3}).

We have thus shown that every face of diagram \eqref{diagram:provingRelatednessOfProductOfOpenSections} commutes, and in particular that $$\xi_9\circ \xi_2 = \xi_4\circ \xi_1.$$
We conclude that $\left(\scP_\fY\left((w_\fy)_{\fy\in \fY}\right),\scP_\fX\left((v_\fx)_{\fx\in \fX}\right)\right)$ are $\Pi(\ph,\Phi)$-related (\cref{def:relatedOpenSections}).  \end{proof}

We collect results and state the main theorem. 

\begin{theorem}\label{theorem:mainTheoremAbstract} Let $\nat{\fC}{\fD}{\Tc}{\Uc}{\t}$ be a $\fD$-fibered  (\cref{def:fiberedTransformationInD})  monoidal transformation (\cref{def:monoidalTransformation}, \cref{assumption:1}) between strong monoidal functors $\Tc$ and $\Uc$ (\cref{assumption:2}).  Suppose  that monoidal product $\otimes_\fD$ is cartesian, and let $\fA\subseteq \sF{Arrow(C)}$ be a category for which $\Pi:\big(\sF{FinSet/A}^\Leftarrow\big)^{op}\rightarrow\fA$ is functorial (\cref{assumption:3}). 
 Let $$\big((\ph,\Phi),f\big):(\Ac_\fX,\i_\fX:\fa\hookrightarrow\Pi(\Ac_\fX))\rightarrow(\Ac_\fY,\i_\fY:\fa'\hookrightarrow\Pi(\Ac_\fY))$$ be a morphism of networks of $\fA$-open $\t$-systems (\cref{def:abstractNetworksMorphism}).  Then there is induced  1-morphism \begin{equation}\label{diagram:main2Cell} \begin{tikzcd}[column sep = large,row sep = large]
 \discatp_{\fy\in \fY}\G_\t(\Ac_{\fY,\fy})\arrow[r,"\G_\t(\ph{,}\Phi)"{name = foo1, below}]\arrow[d,"\G_\t(\i_\fY)\circ \scP_{\fY}"] & \discatp_{\fx\in \fX}\G_\t(\Ac_{\fX,\fx})\arrow[d,swap,"\G_\t(\i_\fX)\circ \scP_\fX"]\\
 \G_\t(\fa')\arrow[r,swap,"\G_\t(f)"{name = foo2, above}] & \G_\t(\fa)\ar[shorten <= 10pt, shorten >= 10pt, Rightarrow,from = foo1, to = foo2]
 \end{tikzcd}\end{equation} in $\sF{Set}^\square$, where   (\cref{def:productRelationOfOpenSections})
 $$\small \G_\t(\ph,\Phi)\defeq \left\{\big((w_\fy)_{\fy\in \fY},(v_\fx)_{\fx\in \fX}\big)\in \discatp_{\fy\in \fY}\G_\t(\Ac_\fy)\times \discatp_{\fx\in \fX}\G_\t(\Ac_\fx):\, (w_{\ph(\fx)},v_\fx)\in  \G_\t(\Phi_\fx)\;\forall \;\fx\in \fX \right\}.$$  \normalsize
\end{theorem}

\begin{proof}
	By \cref{prop:openSectionsProductRelatedAreRelated} (c.f.\ \eqref{diagram:2CellProductRelatedRelatedInProduct}), we have 1-morphism  $$\begin{tikzcd}[column sep = large,row sep = large]
 \discatp_{\fy\in \fY}\G_\t(\Ac_{\fY,\fy})\arrow[r,"\G_\t(\ph{,}\Phi)"{name = foo1, below}]\arrow[d,"\scP_\fY"] & \discatp_{\fx\in \fX}\G_\t(\Ac_{\fX,\fx})\arrow[d,swap,"\scP_\fX"]\\
 \G_\t\left(\Pi\big(\Ac_\fY\big)\right)\arrow[r,swap,"\G_\t(\Pi(\ph{,}\Phi))"{name = foo2, above}] & \G_\t\left(\Pi\big(\Ac_\fX\big)\right),\ar[shorten <= 10pt, shorten >= 10pt,Rightarrow,from = foo1, to = foo2]
 \end{tikzcd}$$
 and by \cref{prop:halfOfMainTheorem1stHalf}, we have 1-morphism  $$\begin{tikzcd}[column sep = large, row sep = large]
			\G_\t\left(\Pi\big(\Ac_\fY\big)\right)\arrow[r,swap,"\G_\t(\Pi(\ph{,}\Phi))"{name = foo4},""{name = foo1, below}]\arrow[d,"\G_\t(\i_\fY)"] & \G_\t\left(\Pi\big(\Ac_\fX\big)\right) \arrow[d,swap,"\G_\t(\i_\fX)"] \\
			\G_\t(\fa')\arrow[r,"\G_\t(f)"{name = foo3},""{name = foo2, above}] & \G_\t(\fa).\ar[shorten <= 5pt, shorten >= 5pt, Rightarrow, from = foo4, to = foo3]
		\end{tikzcd}$$
		
		Applying vertical composition in $\sF{Set}^\square_1$, the result follows immediately. 
\end{proof}

\section{Examples of Networks of Open Systems}
This section generates a slew of concrete instances of \cref{theorem:mainTheoremAbstract}.  
\begin{remark}\label{remark:outlineExampleMainTheorem} An outline for  these examples is as follows. We define  categories $\fC$, $\fD$, subarrow category $\fA\subseteq \sF{Arrow(C)}$,  functors $\Tc, \Uc:\fC\rightarrow\fD$, and natural transformation $\t:\Tc\Rightarrow \Uc$. Then we verify that assumptions \ref{assumption:3}, \ref{assumption:2}, and \ref{assumption:1} hold. We conclude by citing \cref{theorem:mainTheoremAbstract}.
 \end{remark}
\subsection{Networks of  Open Systems}\label{subsec:normalExampleforMain}

\begin{example}\label{ex:concreteOpenSystems}
	A special case of \cref{theorem:mainTheoremAbstract} was stated in \cref{theorem:mainTheoremLermanOpenNetworks} from  \cite[Theorem 9.3]{lermanopennetworks} for networks of $\sF{Man}$-open systems (\cref{def:abstractNetworks}).  We must show that a morphism \small $$\big((\ph,\Phi),f\big):\left(\Sc_\fX:\fX\rightarrow \sF{SSub},\i_\fX:a\rightarrow \discatp_{\fx\in \fX}\Sc_\fX(\fx)\right)\rightarrow\left(\Sc_\fY:\fY\rightarrow\sF{SSub},\i_\fY:b\rightarrow\discatp_{\fy\in \fY}\Sc_\fY(\fy)\right) $$\normalsize of networks of open systems (\cref{def:abstractNetworksMorphism}, \cref{def:SSub}) induces a 1-morphism (\cref{theorem:mainTheoremAbstract})  \begin{equation}\label{eq:1-morphismOS}\begin{tikzcd}[column sep = large, row sep = large]
	\discatp_{\fy\in \fY}\sF{Crl}(\Sc_\fY(\fy))\arrow[r,"\sF{Crl}(\ph{,}\Phi)"{name = foo1, below}]\arrow[d,"\sF{Crl}(\i_\fY)\circ \scP_\fY"] & \discatp_{\fx\in \fX}\sF{Crl}(\Sc_\fX(\fx))\arrow[d,swap,"\sF{Crl}(\i_\fX)\circ \scP_\fX"] \\
	\sF{Crl}(b)\arrow[r,"\sF{Crl}(f)"{name = foo2, above}] & \sF{Crl}(a) \ar[shorten <= 7pt, shorten >= 7pt, Rightarrow, from = foo1, to = foo2]\end{tikzcd}\end{equation}
in $\sF{Set}^\square$.  In this context, $\sF{Crl}$ is the open sections functor $\G_\t$ (\cite[c.f.\ (2.7)]{lermanopennetworks}).
	\begin{proof}[Proof of \cref{theorem:mainTheoremLermanOpenNetworks}]
	We follow the outline in \cref{remark:outlineExampleMainTheorem}. 
		Let $\fC = \fD= \sF{Man}$ be the category of smooth manifolds,  $\Tc= T$ the tangent endofunctor, $\Uc= id_{\sF{Man}}$ the identity functor, and $\t:\Tc\Rightarrow \Uc$ the canonical projection of the tangent bundle.   Let $\fA\subseteq\sF{Arrow(Man)}$ be $\fA\defeq \sF{SSub}$ and observe that $\sF{Crl}$ as defined in \cref{def:controlOnSubmersions} is $\G_\t$ (\cref{def:abstractOpenSections}).  The category $\sF{Man}$ of manifolds has products (\cref{fact:productManifolds}), and this defines cartesian monoidal structure, which $\fA$ also inherits (\cite[\S4]{lermanopennetworks}, \cref{assumption:3}). Naturality of isomorphism $T(M\times N)\cong TM\times TN$ (\cref{corollary:TProductDiffeoProductT})  follows from commutative diagram $$\begin{tikzcd}
	M\arrow[d,"f"] & M\times N\arrow[l,swap,"p_M"] \arrow[r,"p_N"]\arrow[d,"f\times g"] & N\arrow[d,"g"]\\
	M' & M'\times N'\arrow[l,swap,"p_{M'}"] \arrow[r,"p_{N'}"] & N',
\end{tikzcd}$$ functoriality of $T$ (\cref{prop:differentialIsFunctorial}), and universal property of product (c.f.\ \cref{prop:isoProductTangentSpaces}). This implies that $\Tc$ is strong monoidal (\cref{assumption:2}).  Finally, the projection $\t_M:TM\rightarrow M$ is split epimorphism (\cref{fact:tangentProjSplitEpi}) and $\t$ is easily seen to be monoidal (\cref{assumption:1}), e.g.\ now from naturality of $\t$ (\cref{fact:naturalityOfProjTangentBundle}) and commuting diagram $$\begin{tikzcd}
	TM\arrow[d,"\t_M"] & T(M\times N)\arrow[l,swap,"Tp_M"] \arrow[r,"Tp_N"]\arrow[d,"\t_{M\times N}"] & TN\arrow[d,"\t_N"]\\
	M & M\times N\arrow[l,swap,"p_{M}"] \arrow[r,"p_{N}"] & N.
\end{tikzcd}$$ We thus obtain 1-morphism in  \eqref{eq:1-morphismOS}. 	\end{proof}	 

\end{example}

\subsection{Networks of Hybrid Open Systems}\label{subsec:hybridExampleforMain}
\begin{example}\label{ex:networkOfHyOS}
We extend \cref{ex:concreteOpenSystems} for (non-deterministic) hybrid open systems.  This result was proven directly in  \cite[Theorem 6.19]{lermanSchmidt1}. Here we prove as a corollary of \cref{theorem:mainTheoremAbstract}. 
\begin{theorem}\label{prop:mainTheoremForHybridSystems} Let  $$\left(\big\{\Hc_{\fX,\fx}\big\}_{\fx\in \fX}, i_{a,\fX}:a\hookrightarrow \discatp_{\fx\in \fX}\Hc_{\fX,\fx}\right)\;\;\mbox{and}\;\;\left(\big\{\Hc_{\fY,\fy}\big\}_{\fy\in \fY},i_{b,\fY}:b\hookrightarrow\discatp_{\fy\in \fY}\Hc_{\fY,\fy}\right)$$\normalsize be two networks of hybrid open systems (\cref{ex:HybridOpenSystemsInstanceOfAbstractOpenSystem}).  
A morphism $$(\Hc_\fX,i_{a,\fX})\xrightarrow{\big((\ph,\Phi),z\big)} (\Hc_\fY,i_{b,\fX})$$  of networks  of hybrid open systems (\cref{def:hybridOS}, \cref{def:abstractNetworksMorphism})	induces a 1-morphism

$$\begin{tikzcd}[column sep = large, row sep = large]\discatp_{\fy\in \fY}\sF{Crl}_\Ub(\Hc_{\fY,\fy})\arrow[r,"\sF{Crl}_\Ub(\Phi)"{name = foo1, below}]\arrow[d,swap,"\sF{Crl}_\Ub (i_{b{,}\fY})\circ \scP_\fY"] & \discatp_{\fx\in \fX}\sF{Crl}_\Ub(\Hc_{\fX, \fx})\arrow[d,"\sF{Crl}_\Ub (i_{a{,}\fX})\circ \scP_\fX"]\\
\sF{Crl}_\Ub(b)\arrow[r,"\sF{Crl}_\Ub (z)"{name = foo2, above}] & \sF{Crl}_\Ub(a)  \ar[shorten <= 5pt, shorten >= 5pt, Rightarrow, from = foo1, to = foo2]
\end{tikzcd}
$$  where (c.f.\ \cref{def:abstractOpenSections}) $$\sF{Crl}_\Ub (a)\defeq  \big\{X: \Ub a_{tot}\rightarrow T\Ub a_{st}:\, p_a = \t_{a_{st}}\circ X\big\},$$ and $\sF{Crl}_\Ub(\ph,\Phi)$ is as in \cref{def:productRelationOfOpenSections}. 
\end{theorem}

\end{example}
\begin{proof}
Here $\fC = \sF{HyPh}$, $\fD = \sF{Man}$,  $\fA = \sF{HySSub}$,  $\Uc = \Ub:\fC\rightarrow\fD$ is the forgetful functor (\cref{prop:ForgetfulFunctor}), and $\Tc \defeq  T\circ \Uc$ is the tangent endofunctor composed with $\Ub$. 
 The  transformation $\nat{\sF{HyPh}}{\sF{Man}}{T\circ \Ub}{\Ub}{\t}$ is the canonical projection  $\t_a:T\Ub a\rightarrow \Ub a$  of the tangent bundle of the underlying manifold $\Ub a$, and is  $\fD$-fibered.  We have seen that $\sF{HyPh}$ is a cartesian monoidal category (\cref{prop:BinProductHyPh}, \cref{def:monoidalCategoryHyPh}) and that $\sF{HySSub}$ is cartesian as well (\cref{fact:HySSubIsCartesian}), which shows that  \cref{assumption:3} is satisfied. 
 
  The forgetful functor  $\Ub:\sF{HyPh}\rightarrow\sF{Man}$ is strong monoidal (\cref{prop:uFromHyPh2ManIsMonoidal}) and since the endofunctor $T:\sF{Man}\rightarrow\sF{Man}$ is strong monoidal, so is $T\circ \Ub$ (\cref{assumption:2}). Finally, the natural transformation is monoidal---satisfying \cref{assumption:1}---again, by naturality and the universal property.  Therefore, the result follows as a direct application of \cref{theorem:mainTheoremAbstract}. 
\end{proof}

\subsection{Networks of Deterministic Hybrid Open Systems}\label{subsection:mainConcreteTheorem}\label{subsection:mainDeterministicTheorem}
We now restate and finally prove \cref{theorem:mainTheoremFordeterministicHybridSystems}.  While we cite \cref{theorem:mainTheoremAbstract}, the proof is not as  immediate a consequence as the preceding two examples.

\begin{theorem}\label{prop:mainTheoremFordeterministicHybridSystemsPrime}
	Let $$\left(\big\{\Hc_{\fX,\fx}\big\}_{\fx\in \fX}, i_{a,\fX}:a\hookrightarrow \discatp_{\fx\in \fX}\Hc_{\fX,\fx}\right)\;\;\mbox{and}\;\;\left(\big\{\Hc_{\fY,\fy}\big\}_{\fy\in \fY},i_{b,\fY}:b\hookrightarrow\discatp_{\fy\in \fY}\Hc_{\fY,\fy}\right)$$ be two networks of deterministic hybrid open systems (\cref{ex:HybridOpenSystemsInstanceOfAbstractOpenSystem}).  
A morphism $$(\Hc_\fX,i_{a,\fX})\xrightarrow{\big((\ph,\Phi),z\big)} (\Hc_\fY,i_{b,\fX})$$  of networks  of deterministic hybrid open systems (\cref{def:hybridOS}, \cref{def:abstractNetworksMorphism})	induces a 1-morphism
$$\begin{tikzcd}[column sep = large, row sep = large]
\discatp_{\fy\in \fY}\dCrl(\Hc_{\fY,\fy})\arrow[d,swap,"\dCrl i_{b,\fY}\circ\scP_\fY"]\arrow[r,"\dCrl(\ph{,}\Phi)"{name = foo1, below}] & \discatp_{\fx\in \fX}\dCrl(\Hc_{\fX,\fx})\arrow[d,"\dCrl i_{a,\fX}\circ \scP_\fX"]\\
\dCrl(b)\arrow[r,"\dCrl (z)"{name =foo2, above}] & \dCrl(a).\ar[shorten <= 5pt, shorten >= 5pt, Rightarrow, from = foo1, to = foo2]
\end{tikzcd}
$$
Thus $$\left( \dCrl \big( i_{b,\fY} \big)\left(\scP_\fY\big((w_\fy)_{\fy\in \fY}\big)\right), \dCrl \big(i_{a,\fX}\big)\left(\scP_\fX\big((v_\fx)_{\fx\in \fX}\big)\right)\right)$$ are $\Pi(z)$-related (\cref{def:detHySysMorphism}) whenever $(w_{\ph(\fx)},v_\fx)$ are $\Phi_\fx$-related  for all $\fx\in \fX$. 
\end{theorem}
\begin{remark}
	Observe that the first half of \cref{prop:mainTheoremFordeterministicHybridSystemsPrime} (\cref{theorem:mainTheoremFordeterministicHybridSystems}) is identical to the first half of \cref{prop:mainTheoremForHybridSystems}. In other words, networks of hybrid open systems are the same as networks of deterministic hybrid open systems.  The difference appears in open sections (\cref{remark:whyCallItNetworkOfSYSTEMS?}).  In \cref{prop:mainTheoremForHybridSystems},  $\G_\t = \sF{Crl}_\Ub$ while in \cref{prop:mainTheoremFordeterministicHybridSystemsPrime}, we will see that $ \dCrl\subset\G_\t$.  Moreover, in \cref{prop:mainTheoremForHybridSystems}, the target category $\fD$ was $\sF{Man}$, while in \cref{prop:mainTheoremFordeterministicHybridSystemsPrime} (\cref{theorem:mainTheoremFordeterministicHybridSystems}), the target category is $\sF{Set}$. 
\end{remark}
\begin{proof} Let $\fC=\sF{HyPh}$, $\fD=\sF{Set}$, $\fA = \sF{HySSub}$, as in \cref{ex:networkOfHyOS}.  We immediately observe that \cref{assumption:3} is satisfied. 
Let functor  $\Uc:\sF{HyPh}\rightarrow\sF{Set}$ be defined by $\Uc a \defeq \{x\in \Ub a\}$, the set of points in the underlying manifold of $a$; in other words, as the composition of forgetful functors $\sF{HyPh}\xrightarrow{\Ub}\sF{Man}\rightarrow\sF{Set}$, where the second functor forgets the smooth structure.  We define functor $\Tc:\sF{HyPh}\rightarrow\sF{Set}$ on objects by $$\Tc a \defeq \{v\in T\Ub a\}\times \{s\in \Ub a\} = \{(v,s):\, v\in T\Ub a\, \mbox{and}\, s\in \Ub a\}.$$  It is a quick check to verify that $\Tc$ is  a functor.  For morphism $a\xrightarrow{f}b$ of hybrid phase spaces, $\Tc f\defeq T\Ub f\times \Ub f$. We define natural transformation $\upsilon:\Tc\Rightarrow \Uc$ on objects by $\upsilon_a: \Tc a\rightarrow\Uc a$ by $\upsilon_a \defeq \t_{\Ub a} \circ p_1$, where $p_1: X\times Y\rightarrow X$ is canonical projection to first factor, and $\t_{\Ub a}$ is the canonical projection of tangent bundle. Since both $\t$ and $p_1$ are split epimorphisms, the composition $\upsilon$ is as well. The definition of $\upsilon$ on objects assembles into a natural transformation $\nat{\sF{HyPh}}{\sF{Set}}{\Tc}{\Uc}{\upsilon}$ which is $\sF{Set}$-fibered.  We know that $\Uc$ is strong monoidal since $\Ub$ is (\cref{prop:uFromHyPh2ManIsMonoidal}), and to see that $\Tc$ is strong monoidal, we compute: \small $$\begin{array}{lll} 
\Tc a\times \Tc b & = \{u\in T\Ub a\}\times\{r\in \Ub a\}\times \{v\in T\Ub b\}\times \{s\in \Ub b\} & \\
& =   \{u\in T\Ub a\}\times \{v\in T\Ub b\}\times\{r\in \Ub a\}\times \{s\in \Ub b\} & \\
& = \{(u,v)\in T\Ub a \times T\Ub b\}\times \{(r,s)\in \Ub a\times \Ub b\} & \\
& \cong \{w\in T(\Ub a\times \Ub b)\}\times \{t\in \Ub(a\times b)\} & (\mbox{by \cref{lemma:PUUPLite}})\\
& = \{w \in T\Ub(a\times b)\}\times \{t\in \Ub(a\times b)\} & (\mbox{also by \cref{lemma:PUUPLite}})\\ & = \Tc(a\times b), & \end{array}$$ \normalsize verifying  that \cref{assumption:2} holds. Finally, it  is not difficult to see that $\upsilon:\Tc\Rightarrow \Uc$ is monoidal (\cref{assumption:1}).  Therefore,  \cref{theorem:mainTheoremAbstract} implies a morphism of networks of deterministic hybrid systems induces 1-morphism $$\begin{tikzcd}[column sep = large]
	\discatp_{\fy\in \fY}\G_\upsilon(\Hc_{\fY,\fy})\arrow[r,"\G_\upsilon(\ph{,}\Phi)"{name = foo1, below}]\arrow[d,swap,"\G_\upsilon(i_{b{,}\fY})\circ \scP_\fY"] & \discatp_{\fx\in \fX}\G_\upsilon(\Hc_{\fX,\fx})\arrow[d,"\G_\upsilon(i_{a{,}\fX})\circ \scP_\fX"]\\
	\G_\upsilon(b)\arrow[r,"\G_\upsilon(z)"{name = foo2, above}] & \G_\upsilon(a) \ar[shorten <= 5pt, shorten >= 5pt, Rightarrow, from = foo1, to = foo2]
\end{tikzcd}$$

To complete the proof, we must check  that $\dCrl(a)\subseteq \G_\upsilon(a)$,   that there is well defined  map $\scP_\fX:\discatp_{\fx\in \fX}\dCrl(\Hc_{\fX,\fx})\rightarrow \dCrl\left(\discatp_{\fx\in \fX}\Hc_{\fX,\fx}\right)$, and that $\dCrl(i):\dCrl(b)\rightarrow \dCrl(a)$ for hybrid interconnection $a\xrightarrow{i}b$.  The inclusion $\dCrl(a)\subseteq \G_\upsilon(a)$ follows by definition of $\dCrl$, since we have defined $\dCrl(a) = \big\{X\in \G_\upsilon(a):\, \mbox{satisfying some conditions}\big\}$. That $\scP_\fX:\discatp_{\fx\in \fX}\dCrl(\Hc_{\fX,\fx})\rightarrow \dCrl\left(\discatp_{\fx\in \fX}\Hc_{\fX,\fx}\right)$ is the statement of \cref{prop:mapFromProductdCrl2dCrlProductGeneral}, and the map $\dCrl(\i,\ifrak):\dCrl(b)\rightarrow \dCrl(a)$ for interconnection comes from \cref{prop:interconnectionMapOndeterministicHybridControl} (\cref{def:interconnectionMapOndeterministicHybridControl}). 
\end{proof}

\subsection{Networks of Discrete Open Systems}

We conclude with a statement of our main result for discrete-time systems. We have lightly touched upon the notion of discrete-time systems:  recall \cref{ex:discSysAsprojCsysInC} and \cref{ex:discreteOpenasAbstract}, where $\fC$ is a concrete category, and $\fA$ is the subcategory of $\sF{Arrow(C)}$ whose objects are surjections.  Echoing the example of a vector field $X\in \Xf(\R^2)$ as interconnection of open systems (c.f.\ \cref{subsubNetworksVanillaInIntro}), we may view a discrete-time system $f:X\times Y\rightarrow X\times Y$ on a product as the interconnection of two discrete-time open systems $f_1:X\times Y\rightarrow X$ and $f_2:X\times Y \rightarrow Y$. We do not develop this particular viewpoint, but all the ingredients are at our disposal to do so.  A direct application of \cref{theorem:mainTheoremAbstract} gives us a discrete-time networks theorem:
\begin{theorem}\label{theorem:mainForDiscrete}
	A morphism  \small$$\left(\big\{\Dc_{\fY,\fy}\big\}_{\fy\in \fY}, \i_\fY:b\hookrightarrow \discatp_{\fy\in \fY}\Dc_{\fY,\fy}\right)\xrightarrow{\big((\ph,\Phi),f\big)}\left(\big\{\Dc_{\fX,\fx}\big\}_{\fx\in \fX}, \i_\fX:a\hookrightarrow \discatp_{\fx\in \fX}\Dc_{\fX,\fx}\right)$$\normalsize of networks of discrete-time open system induces a $1$-morphism $$\begin{tikzcd}[column sep = large]
\discatp_{\fy\in \fY}\G_\t(\Dc_{\fY{,}\fy})\arrow[r,swap,"\G_\t(\ph{,}\Phi)"{name =foo1, below}]\arrow[d,swap,"\G_\t(\i_\fY)\circ \scP_\fY"] & \discatp_{\fx\in \fX}\G_\t(\Dc_{\fX{,}\fx})\arrow[d,"\G_\t(\i_\fX)\circ\scP_\fX"]\\
\G_\t(b)\arrow[r,"\G_\t(f)"{name = foo2, above}]\ar[shorten <= 5pt, shorten >= 5pt, Rightarrow, from = foo1, to = foo2] & \G_\t(a).
\end{tikzcd}
$$
\end{theorem}

Notationally, we merely replaced $\Ac$ in \cref{theorem:mainTheoremAbstract} with $\Dc$ in \cref{theorem:mainForDiscrete} to represent the indexed assignment of discrete-time open systems.  The power of a developed categorical theory: in appearances we did almost nothing, but the array of application is vast.

\chapter{Conclusion: Why Morphisms of  Systems?}\label{ch5}

\section{Introduction}
We emphasized   category theory as a primary  motivation for studying maps of networks of dynamical systems: the Yoneda embedding justifies investigating mathematical objects through their (collection of) morphisms.   Irrespective of networks or even hybrid systems, a  reasonable question is whether ordinary continuous-time dynamical systems can  \textit{in practice}  be better understood through  morphisms.   We saw, for example, that existence and uniqueness for complete continuous-time systems can even be given a wholly categorical formulation (\cref{prop:existenceAndUniquenessRepresentable}), but this observation does not properly extend the theory of dynamical systems.  Aside from connecting two apparently disparate mathematical fields, the categorical formulation tells us little we did not already know about the theory of dynamical systems.  Still,  we argue that a map-centric perspective can yield new insight into systems themselves.  We consider, in particular, the notion of Lyapunov stability which describes a property of equilibria points:    solutions starting close enough remain close enough for all time.  We show that under certain hypotheses, stable points are sent to stable points under maps of dynamical systems.  This is reminiscent of Lyapunov stability theorem, which makes an assertion in the reverse direction: a point (in the domain of some map) sent to a stable point is in fact stable. Our result provides a practitioner with the analogous ability to determine  stability when explicit solutions cannot be found. 

Working with the notion of continuous-time dynamical system as manifold and vector field pair $(M,X)$, we  introduce the \textit{solution map} which sends a point $x_0\in M$ to the solution of $X$ passing through $x_0$ at time $0$. We then review Lyapunov stability and interpret it as continuity of the solution map, with respect to the appropriate topology.    While this abstraction is not new, little use of its generality has, as far as we know, been used in the dynamical systems literature.  An immediate upshot, for example, is that it allows for a natural (useful) description of arbitrary and even unbounded trajectories as stable. In control theory, where one often cares about driving a system to some desired---not necessarily equilibrium---trajectory,  error (deviation of state from desired trajectory) may be used as proxy for the underlying system.  In this setting, what is sought is that the dynamical system representing error has solutions which go to zero, and moreover,  that zero is a stable equilibrium, thus guaranteeing that a control algorithm is robust with respect to uncertainty, noise, or disturbances.  From this perspective, not much is gained from the added generality. 

However, we are now able to consider stability in the context of \textit{maps} of systems.  It makes sense to speak of composition as preserving continuity, as long as we are careful about working in the right topology.  We detail a topology in the space of maps of dynamical systems, and use this to prove \cref{theorem:pushStableToStable} which says that an open map between dynamical systems sends bounded stable points to stable  points.  We end with an example illustrating the usefulness of this result, by mapping a \textit{linear} system (whose stability properties are entirely known by eigenvalues of the matrix representing  its dynamics) to a nonlinear system. Though the nonlinear system can be explicitly solved for (in particular, simply by pushing the linear solution forward), and therefore stability determined through other means, linearization still fails to detect stability.  Even though  continuity is a local concept,  local-in-a-topology-on-$M$ or -$TM$ is different than local-in-the-space-of-maps-of-systems, which explains why \cref{theorem:pushStableToStable} can answer stability questions which linearization of a  vector field cannot.  

\subsection{Review of Complete Dynamical Systems}

Recall the notion of complete dynamical systems (\cref{def:completeDySys}), those for which a solution existence through each point at all time, in this chapter let  $\sF{DySys}$ denote the category of complete dynamical systems.

\begin{definition}\label{def:solutionMap}
	A complete dynamical system $(M,X)\in \sF{DySys}$ defines map $$\ph_{X,(\cdot)}:M\rightarrow\sF{DySys}\left(\big(\R,\frac{d}{dt}\big),\big(M,X\big)\right)$$  by sending $x_0\mapsto \ph_{X,x_0}(\cdot)$, the integral curve of $(M,X)$ passing through $x_0\in M$ at times $t=0$. We call $\ph_{X,(\cdot)}$ the \textit{solution  map} of $(M,X)$, and $\ph_{X,x_0}(\cdot)$---the solution map evaluated at point $x_0\in M$---the \textit{solution of} $(M,X)$ \textit{with initial condition} $x_0$. 
\end{definition}

When the system $(M,X)$ is fixed, we drop the dependence of $\ph_X$ on vector field $X$ and simply write $\ph$. 

\begin{definition}\label{def:equilibria}
	A point $x_e\in M$ is said to be an \textit{equilibrium point} of $(M,X)$ if $X(x_e) = 0\in T_{x_e}M$. 
\end{definition}
\begin{remark}
	The name comes from the fact that solutions starting at equilibria go nowhere: $\ph_{X,x_e}(t) \equiv x_e$ for all $t\in \R$ when $X(x_e)= 0$. 
\end{remark}

\subsection{Lyapunov Stability}
\begin{fact}
	Any second countable smooth manifold is metrizable.  We will by default let $d_M:M\times M\rightarrow\R^{\geq 0}$ denote a metric on manifold $M$. 
\end{fact}

\begin{notation}
	Let $(M,X)$ be a dynamical system.  It will be convenient to consolidate notation for the set of integral curves, i.e.\ the set of maps of dynamical systems from $(\R,\frac{d}{dt})$:  we let $$\scM_{X}\defeq \sF{DySys}\left(\big(\R,\frac{d}{dt}\big),\big(M,X\big)\right).$$ (Script M is for ``morphism.") Because $X\in \Xf(M)$, specifying the vector field alone is sufficient for disambiguation. Leaving the choice of system open, $\scM_{(\bullet)} = \sF{DySys}\left(\big(\R,\frac{d}{dt}\big),\bullet\right).$
\end{notation}

\begin{lemma}\label{lemma:inducedMetric}
	A metric $d_M:M\times M\rightarrow\R^{\geq 0}$ induces a metric $\d_X:\scM_X\times \scM_X\rightarrow\R^{\geq0}\cup\{\infty\}$ on the collection of integral curves for $(M,X)$. 
\end{lemma}
\begin{remark}
	Calling $\d_X$ a metric is misleading since two curves may be $\d_X$-infinitely far apart.  For the linear system $\dx = x$, for example, two integral curves $e^tx$ and $e^tx'$ have infinite $\d_X$-distance whenever $x\neq x'$. Nonetheless, we ultimately care about the topology which this ``metric'' generates.  For a space where every point is $\infty$-separated ($\{e^{(\cdot)}x:\R\rightarrow\R:\, x\in \R\}$ for example), the topology is discrete.   
\end{remark}
\begin{proof}
	Let $\ph_x,\ph_y\in \scM_X$ be two integral curves of $(M,X)$. We define $$\d_X(\ph_x,\ph_y)\defeq \dissup_{t\geq 0} d_M(\ph_x(t),\ph_y(t)).$$ 
	Then  $$\begin{array}{ll} \dissup_{t\geq 0 } \d_X(\ph_x(t),\ph_z(t))  & \leq \dissup_{t\geq 0} \big(d_M(\ph_x(t),\ph_y(t)) + d_M(\ph_y(t),\ph_z(t))\big) \\ &  \leq \dissup_{t\geq 0}d_M(\ph_x(t),\ph_y(t))+ \dissup_{t\geq 0}d_M(\ph_y(t),\ph_z(t)) \\ & = \d_X(\ph_x,\ph_y) + \d_X(\ph_y,\ph_z).\end{array} $$
	It is immediate that $\d_X(\ph_x,\ph_y)=\d_X(\ph_y,\ph_x)\geq 0$, with---by existence and uniqueness---equality when and only when $x=y$.\end{proof}

\begin{remark}\label{remark:topologyOnIntegralCurves} The metric $\d_X$ induces a topology on $\scM_X$ generated by base open sets $$B_\e(\ph_x)\defeq \big\{\ph_y\in \scM_X:\, \d_X(\ph_x,\ph_y)<\e\big\},$$ for $\e>0$ and $x\in M$. 
\end{remark}

\begin{definition}\label{def:lyapunovStabilityContinuity}
	Let $x_e\in M$ be an equilibrium point (\cref{def:equilibria}).  The point $x_e$ is said to be \textit{Lyapunov stable} if the solution map $\ph_{X,(\cdot)}:M\rightarrow\scM_X$ is continuous at $x_e$, w.r.t.\ the topology on $\scM_X$ defined in \cref{remark:topologyOnIntegralCurves}. 
\end{definition}

\begin{remark} 
	This definition captures the notion that a solution which starts close to a stable equilibrium will remain nearby for all (positive) time.  A more  standard but equivalent definition of Lyapunov stability uses the $\d$-$\e$ criterion:   for any $\e>0$ there is a $\d_\e>0$ such that $\d_X\big(\ph_{X,x_e},\ph_{X,x_0}\big)<\e$ whenever $d_M(x_e,x_0)< \d_\e$. 
\end{remark}
In fact, there is nothing sacrosanct about equilibria points in this definition:
\begin{definition}\label{def:stabilityContinuity}  We say that a point $x_o\in M $ is  \textit{stable} if the solution map $\ph_{X,(\cdot)}:M\rightarrow\scM_X$ is continuous at $x_o$.
\end{definition}
\begin{remark}\label{remark:upshotsOfGeneralStability}
	There are two advantages of this definition.  First, it may apply to any arbitrary point (and therefore integral curve) of a dynamical system $(M,X)$.  For example, every point of dynamical system $(\R,\dx = -x)$ is stable in the sense of \cref{def:stabilityContinuity}.  Secondly, stability is not restricted to bounded solutions.  For example, every point of $(\R,\dx = 1)$ is stable, even though the solution $\ph_{X,x_0}(t) = x_0+t$ is unbounded.  Yet, in both cases, stability still captures the notion we want: solutions which start close to each other remain close.   
\end{remark}

\section{Open Maps Preserving Stability} 
 Recall that the maps of dynamical systems preserve integral curves (\cref{def:integralCurveAsMap}).  They also preserver equilibria.  Let $f:(M,X)\rightarrow (N,X)$ be a map of dynamical systems and $x_e\in M$ an equilibrium.  Then  linearity of the differential implies that $$0 = Tf (0) = Tf X(x_e) = Y(f(x_e)),$$ which further implies that that $f(x_e)$ is an equilibrium of $(N,Y)$.  Alternatively, since maps of dynamical systems send  integral curves   to integral curves, $f_*\ph_{X,x_e}$ is a constant curve, so $0=\frac{d}{dt} f_*\ph_{X,x_e}(t)  = Y (\ph_{Y,f(x_e)}(t))$.

 Under some conditions on the map of systems, stability is also preserved. 
 
 First a definition: 
 \begin{definition}\label{def:stablePoints}
 	Let $(M,X)$ be a dynamical system and $x_0\in M$ a point.  We say that $x_0$ is \textit{bounded} if its solution $\ph_{X,x_0}$ is bounded, i.e.\ if $\d(\ph_{X,x_0}, \underline{x})<\infty$ for any constant map $\underline{x}:\R\rightarrow M$ defined by $\underline{x}(t)=x$, for $x\in M$. 
 \end{definition}
 \begin{remark}
 	We defined the metric $\d_X$ in \cref{lemma:inducedMetric} on $\scM_X$, so technically $\d_X(\ph_{X,x_0}, \cdot)$ may only take $\ph_{X,x'}$ as an argument, for $x'\in M$. We can extend the induced metric between solutions to a \textit{pseudo}metric on curves of $M$---the space $\Cc(\R,M)$---in the obvious way: $$\d(\phi,\psi)\defeq \dissup_{t\geq 0}d_M(\phi(t),\psi(t)).$$This definition is a pseudometric because two distinct continuous curves may agree on $\R^{\geq 0}$. 
 \end{remark}
  \begin{theorem}\label{prop:pushStableToStable}\label{theorem:pushStableToStable}
  Let $(M,X)\xrightarrow{f}(N,Y)$ be a map of dynamical systems for which $f$ is open: $f(\Oc)$ is an open set in $N$ whenever $\Oc$ is open in $M$.  Suppose, further, that  $x_0\in M$ is stable and bounded.  Then $f(x_0)$ is stable. 	
  \end{theorem}
 
The proof of this theorem requires a lemma, interesting in its own right. 
  \begin{lemma}\label{lemma:pushForwardContinuousAtBounded}
  	Let $f:(M,X)\rightarrow (N,Y)$ be a map of systems.  Then the pushforward $$\begin{array}{rrl} f_*: & \scM_X & \rightarrow\scM_Y\\ & \ph_{X,x} & \mapsto f\circ \ph_{X,x} = \ph_{Y,f(x)}\end{array}$$  is continuous at bounded curves.
  \end{lemma}
  
We need a secondary lemma to prove \cref{lemma:pushForwardContinuousAtBounded}: 
\begin{lemma}\label{lemma:continuousDelta}
	Let $f:M\rightarrow N$ be a continuous map between manifolds and fix $\e>0$.  Then there is \textit{continuous} function $\d_\e:M\rightarrow\R^{>0}$ such that $d_M(x,x')<\d_\e(x)$ implies that $d_N(f(x),f(x'))< \e$.  
\end{lemma}

We call attention to our dual use of $\d$ as both a metric on $\scM_X$ and a function $M\rightarrow \R^{\geq0}$.  In this proof, $\d$ and all its variants only refer to the latter function. 

\begin{proof}
	Since $f$ is continuous, there is \textit{a} function \begin{equation}\label{eq:tildeDelta} \tilde{\d}:M\rightarrow\R^{>0}\end{equation} (not necessarily continuous) such  that $d_M(x,x')<\tilde{\d}(x)$ implies that $d_N(f(x),f(x'))<\e/2$.  Let  $\scB = \left\{B_{\tilde{\d}(x_0)/2}(x_0):\, x_0\in M_0\right\}$ be a locally finite open cover of $M$ and  $\big\{\r_{x_0}:M\rightarrow[0,1]	:\, x_0\in M_0\big\}$ be a partition of unity subordinate to $\scB$, where $M_0\subset M$. 
	
	We  define function $\d_\e:M\rightarrow\R^{>0}$ by \begin{equation}\label{eq:DefDelta} \d_\e(\cdot) \defeq \frac{1}{2}\diss_{x_0\in M_0}\tilde{\d}(x_0)\r_{x_0}(\cdot),\end{equation}  which is smooth, and therefore continuous, as long since each $\r_{x_0}(\cdot)$ is. We must  verify that this $\d_\e$  satisfies the delta-epsilon constraint, namely that $d_N(f(x),f(x'))<\e$ whenever $d_M(x,x')<\d_\e(x)$. 
	
	Let $x\in M$ and set $$\begin{array}{ll} \ox& \defeq \displaystyle\arg\max_{x_0\in M_0}\left\{\tilde{\d}(x_0):\, \r_{x_0}(x)\neq 0\right\},\\\overline{\d}& \defeq \tilde{\d}(\ox)= \dismax_{x_0\in M_0}\big\{\tilde{\d}(x_0):\, \r_{x_0}(x)\neq 0\big\}.\end{array}$$  Observe that  $\d_\e(x)\leq \frac{1}{2}\overline{\d}$ (c.f.\ \eqref{eq:DefDelta}), and suppose   that $d_M(x,x')<\d_\e(x)$. Since $d_M$ is a metric, \begin{equation}\label{eq:epsilonWhatevs} d_M(x,x')\leq  d_M(x,\ox)+d_M(\ox,x') \leq d_M(x,\ox) + \big( d_M(\ox,x)+d_M(x,x')\big).\end{equation}   Since $\rho_{\ox} (x)\neq 0$ and $\mbox{supp}(\rho_{\ox} )\subseteq B_{\overline{\d}/2}(\ox)$, we see that $d_M(x,\ox)<\frac{1}{2}\overline{\d} = \frac{1}{2}\tilde{\d}(\ox)<\tilde{\d}(\ox)$ which implies (c.f.\ \eqref{eq:tildeDelta}) that \begin{equation}\label{eq:epsilon1} d_N(f(x),f(\ox))<\e/2\end{equation}  Similarly,  $\big( d_M(\ox,x)+d_M(x,x')\big)< \frac{1}{2}\overline{\d}+  \frac{1}{2}\overline{\d} = \overline{\d} = \tilde{\d}(\ox)$  which implies (c.f.\ second inequality of eq.\ \eqref{eq:epsilonWhatevs}) that \begin{equation}\label{eq:epsilon2} d_N(f(\ox),f(x'))< \e/2.\end{equation}  Inequalities eq.\ \eqref{eq:epsilon1} and eq.\ \eqref{eq:epsilon2}  together imply that $$d_N(f(x),f(x'))\leq d_N(f(x),f(\ox))+d_N(f(\ox),f(x'))< \e/2+\e/2 = \e,$$  and hence $\d_\e(x)$ satisfies the delta-epsilon constraint. 
\end{proof}
  
  \begin{proof}[Proof of \cref{lemma:pushForwardContinuousAtBounded}]  	 Fix $\hat{\e}>0$ and let $\ph\in \scM_X$ be bounded, so that $\dissup_{t\geq 0 } d_M(\ph(t),x)< \infty$ for any $x\in M$. We must show that there is $\hat{\d}>0$ for which $f_*\left(B_{\hat{\d}}(\ph)\right)\subseteq B_{\hat{\e}}(f_*\ph)$.  Let $\d_{\hat{\e}}:M\rightarrow \R^{>0}$ be a continuous function  satisfying delta-epsilon condition for $\hat{\e}$ (\cref{lemma:continuousDelta}), so that $$d_N\left(f\big(\ph(t)\big),f(x)\right)<\hat{\e}$$ whenever   $$d_M(\ph(t),x)<\d_{\hat{\e}}(\ph(t)).$$  The closure $\Lc_\ph\defeq  \overline{\left\{\ph(t):\, t\geq 0 \right\}}$ is compact and $\d_{\hat{\e}}(\cdot)$ is continuous, so the minimum $$\d(\ph(t^*)) = \dismin\big\{\d_{\hat{\e}}(x):\, x\in \Lc_\ph\big\}$$ is achieved for some $t^*\geq 0$; call it $$\hat{\d}\defeq \d_{\hat{\e}}(\ph(t^*)).$$  Then we readily conclude that $$f_*(B_{\hat{\d}}(\ph))\subseteq B_{\hat{\e}}(f_*\ph), $$	as required
  \end{proof}
    
  \begin{proof}[Proof of \cref{prop:pushStableToStable}]
	Suppose that $x_0\in M$ is stable and bounded.  To show that $f(x_0)\in M$ is stable, we must show that the solution map $\ph_Y:N\rightarrow \scM_Y$ is continuous at $f(x_0)$. Let $\Oc\subseteq \scM_Y$ be open containing $\ph_{Y,f(x_0)}$ and consider the commutative diagram $$\begin{tikzcd}[column sep = large, row sep= large]
	M\arrow[r,"\ph_X"]\arrow[d,"f"] & \scM_X \arrow[d,"f_*"]\\ N\ar[r,"\ph_Y"] & \scM_Y 
	\end{tikzcd}$$  
since $\ph_Y\circ f = f_*\circ \ph_X$, we have that $f^{-1}\circ \ph_Y^{-1} (\Oc) = \ph_X^{-1}\circ f_*^{-1}(\Oc)$ and therefore $$\ph_Y^{-1}(\Oc)\supseteq f \left(f^{-1}\big(\ph_Y^{-1}(\Oc)\big)\right)  = f\left(\ph_X^{-1}\big(f_*^{-1}(\Oc)\big)\right),$$ which is open because $f_*$ is continuous (\cref{lemma:pushForwardContinuousAtBounded}), $\ph_X$ is continuous at $x_0$ by assumption (\cref{def:stabilityContinuity}), and $f$ is open by assumption. \end{proof}

  \begin{example}
  	Consider the nonlinear dynamical system $(\R,\dx = -x^3)$ and map of systems $$\begin{tikzcd}
\R\arrow[d,"-x"]\arrow[r,"f"] & \R \arrow[d,"-x^3"]\\
T\R\arrow[r, "Tf"] & T\R, 	
\end{tikzcd}$$ where  $$f(x) = \frac{1}{\sqrt{\log\left(\frac{1}{x^2}\right)+1}}.$$ 

As  observed previously, $(\R,\dx = -x)$ is stable and $f$ is open at $x=1$. Therefore  \cref{prop:pushStableToStable} implies that $f(1)=1$ is stable in $(\R,\dx = -x^3)$.

Contrast with a traditional method:
  linearization at $(x,t) = (1,0)$ does not (cannot) prove stability   of system \begin{equation}\label{eq:pi314} \dx = \left\{\begin{array}{ll} -x ^3 - t  & \mbox{if}\;\; x\geq 0\\ x^3-t&\mbox{else.} \end{array}\right.\end{equation} This system is more appropriately represented in $\R^2$ with variables $(t,x)$,  $\dt = 1$, and $\dx$ given in \eqref{eq:pi314}, and no map $\R\rightarrow\R^2$ can be open. 
  \end{example}
  
  \begin{example}
  	Consider constant-time system $\dx = 1$, whose solution is given by $x(t) = x_0 + t$.  While obviously stable, we observe this fact as a result of \cref{theorem:pushStableToStable}.  Consider map of systems $$\begin{tikzcd}
	 \R^{>0}\arrow[r,"-\log(x)"]\arrow[d,"-x"] & \R\arrow[d,"1"]\\
T\R^{>0}\arrow[r,"-\frac{1}{x}"] & T\R,  
\end{tikzcd}$$ which is an open map.  Since $\dx = -x$ is stable, we conclude that $\dx = 1$ is as well. 
  \end{example}

\cleardoublepage\phantomsection
\addcontentsline{toc}{chapter}{References}
\bibliographystyle{plain}
\bibliography{main.bib}

\end{document}